\documentclass[10pt,a4paper]{article}

\pdfoutput=1

\usepackage{a4}
\usepackage{amsmath}
\usepackage{amssymb}
\usepackage{amsthm}
\usepackage{mathrsfs}
\usepackage{eucal}
\usepackage{graphicx}
\usepackage{epsfig}
\usepackage[curve]{xypic}
\usepackage{stmaryrd}
\usepackage[intoc]{nomencl}
\usepackage{dsfont}
\usepackage{fancyhdr}
\usepackage{titlesec}
\usepackage{subfigure}
\usepackage[german,english]{babel}
\usepackage{enumitem}
\usepackage{color}
\usepackage{import}
\usepackage{pdflscape}
\usepackage{ifpdf}
\usepackage{mathdots}
\usepackage{mathtools}
\usepackage[bbgreekl]{mathbbol}
\usepackage{amsfonts}
\usepackage{blkarray}
\usepackage{etoolbox}
\usepackage{needspace}

\DeclareSymbolFontAlphabet{\mathbb}{AMSb}
\DeclareSymbolFontAlphabet{\mathbbl}{bbold}

\usepackage{hyperref}
\hypersetup {
  pdfpagemode = {None}, 
  pdftitle = {A version of scale calculus and the associated Fredholm theory},
  pdfauthor = {Andreas Gerstenberger}
}

\usepackage[capitalise,noabbrev]{cleveref}

\mathchardef\mhyphen="2D

\makeatletter
\def\blfootnote{\xdef\@thefnmark{}\@footnotetext}
\makeatother

\title{A version of scale calculus and the associated Fredholm theory\blfootnote{This work is supported by DFG Projekt LA 2448/2-1}}

\author{Andreas Gerstenberger}

\date{\today}

\begin{document}

\renewcommand{\H}{\mathbb{H}}
\newcommand{\D}{\mathbb{D}}
\newcommand{\field}{\mathds{k}}
\newcommand{\R}{\mathds{R}}
\newcommand{\C}{\mathds{C}}
\newcommand{\Z}{\mathds{Z}}
\newcommand{\N}{\mathds{N}}
\newcommand{\Q}{\mathds{Q}}
\newcommand{\Ztwo}{\Z_2}
\def\set#1{\{\s@t#1&\}}
\def\s@t#1&#2&{#1\;|\;#2}
\def\infset#1{{\inf\{\s@t#1&\}}}
\def\s@t#1&#2&{#1\;|\;#2}
\def\supset#1{{\sup\{\s@t#1&\}}}
\def\s@t#1&#2&{#1\;|\;#2}
\def\minset#1{{\min\{\s@t#1&\}}}
\def\s@t#1&#2&{#1\;|\;#2}
\def\maxset#1{{\max\{\s@t#1&\}}}
\def\s@t#1&#2&{#1\;|\;#2}
\newcommand{\definedas}{\mathrel{\mathop:}=}
\newcommand{\defines}{=\mathrel{\mathop:}}
\newcommand{\definedequiv}{\mathrel{\mathop:}\Leftrightarrow}
\newcommand{\Hol}{\operatorname{Hol}}

\newcommand{\tensoralg}{\mathrm{T}}
\newcommand{\symmalg}{\mathrm{S}}
\newcommand{\extalg}{\Lambda}
\newcommand{\Pol}{\operatorname{Pol}}
\newcommand{\extmult}{\mathrm{e}}
\newcommand{\insertion}{\iota}
\newcommand{\Sym}{\operatorname{Sym}}

\newcommand{\category}[1]{\text{\textnormal{\textbf{#1}}}}
\newcommand{\catC}{\category{C}}
\newcommand{\kAlg}{\text{\textnormal{$\field$-\textbf{Alg}}}}
\newcommand{\abGrp}{\text{\textnormal{\textbf{ab-Grp}}}}
\newcommand{\Top}{\text{\textnormal{\textbf{Top}}}}
\newcommand{\pchTop}{\text{\textnormal{\textbf{pch-Top}}}}
\newcommand{\mor}[3]{\operatorname{Mor}_{#1}({#2},{#3})}
\newcommand{\Mor}{\operatorname{Mor}}
\renewcommand{\hom}{\operatorname{Hom}}
\newcommand{\Hom}{\operatorname{Hom}}
\newcommand{\Iso}{\operatorname{Iso}}
\newcommand{\Aut}{\operatorname{Aut}}
\newcommand{\End}{\operatorname{End}}
\newcommand{\equ}{\operatorname{Equ}}

\newcommand{\open}{\text{\textnormal{\textbf{Open}}}}
\newcommand{\sheaf}{\mathscr}
\renewcommand{\O}{\mathcal{O}}
\newcommand{\Prshf}[2]{\text{\textnormal{\textbf{Prshf}}}({#1},{#2})}
\newcommand{\Shf}[2]{\text{\textnormal{\textbf{Shf}}}({#1},{#2})}
\newcommand{\homol}[3]{H^{#1}({#2};{#3})}
\newcommand{\stack}{\mathfrak}

\newcommand{\ie}{i.\,e.~}
\newcommand{\st}{s.\,t.~}
\newcommand{\wrt}{w.\,r.\,t.~}
\newcommand{\Wlog}{W.\,l.\,o.\,g.~}
\newcommand{\smallwlog}{w.\,l.\,o.\,g.~}
\newcommand{\eg}{e.\,g.~}
\newcommand{\cf}{cf.~}
\renewcommand{\iff}{iff~}
\newcommand{\etc}{etc.~}
\newcommand{\iid}{i.\,i.\,d.~}
\newcommand{\fs}{f.\,s.~}
\newcommand{\fa}{f.\,a.~}
\newcommand{\zB}{z.\,B.~}
\renewcommand{\dh}{d.\,h.~}
\renewcommand{\ae}{a.\,e.~}

\newcommand{\matrices}{\mathrm{Mat}}
\newcommand{\trace}{\operatorname{Tr}}
\newcommand{\supertrace}{\operatorname{Str}}
\newcommand{\superdet}{\operatorname{Sdet}}
\newcommand{\transpose}[1]{^\mathrm{t}{#1}}
\newcommand{\supertranspose}[1]{^\mathrm{st}{#1}}
\newcommand{\parity}{\mathrm{p}}
\newcommand{\commutator}[2]{[#1,#2]}
\newcommand{\supercommutator}[2]{\commutator{#1}{#2}_\mathrm{s}}
\newcommand{\im}{\operatorname{im}}
\newcommand{\envelopingalg}[1]{\mathcal{U}(#1)}

\newcommand{\Cf}{\mathcal{C}}
\newcommand{\Cffinord}{\mathcal{C}_{\mathrm{fo}}}
\newcommand{\Gammafinord}{\Gamma_{\mathrm{fo}}}
\newcommand{\Omegafinord}{\Omega_{\mathrm{fo}}}
\newcommand{\Xfinord}{\mathfrak{X}_{\mathrm{fo}}}
\newcommand{\DiffOpfinord}{\mathrm{DO}_{\mathrm{fo}}}
\newcommand{\DiffOp}{\mathrm{DO}}
\newcommand{\TDO}{\mathrm{TDO}}
\newcommand{\TDOfinord}{\mathrm{TDO}_{\mathrm{fo}}}
\newcommand{\functionalforms}{\mathcal{F}}
\newcommand{\euler}{\mathfrak{E}}
\newcommand{\eulerlagrange}{\mathfrak{E}}
\newcommand{\interioreuler}{\mathfrak{I}}
\newcommand{\functional}[1]{\mathfrak{#1}}
\newcommand{\functionals}{\mathfrak{Func}}
\newcommand{\contactideal}{\mathcal{C}}
\newcommand{\pdr}{\mathcal{R}}
\renewcommand{\d}{\mathrm{d}}
\newcommand{\Lie}{\mathcal{L}}
\newcommand{\DiffeoLoc}{\mathrm{Diff}_{\mathrm{loc}}}
\newcommand{\Diff}{\mathrm{Diff}}
\newcommand{\laplace}{\triangle}
\newcommand{\dvol}{\mathrm{dvol}}
\newcommand{\vol}{\mathrm{vol}}
\newcommand{\Lagrangian}{\mathcal{L}}
\newcommand{\dbar}{\overline{\partial}}
\newcommand{\ind}{\operatorname{ind}}
\newcommand{\inj}{\mathrm{inj}}
\newcommand{\dist}{\mathrm{dist}}
\newcommand{\coker}{\operatorname{coker}}
\newcommand{\corank}{\operatorname{corank}}

\newcommand{\permutations}{\mathcal{S}}
\newcommand{\sign}{\operatorname{sign}}
\newcommand{\GL}{\mathrm{GL}}
\newcommand{\U}{\mathrm{U}}
\newcommand{\SU}{\mathrm{SU}}
\newcommand{\SO}{\mathrm{SO}}

\newcommand{\projRinfty}{\mathrm{p}}
\newcommand{\hor}{\mathrm{Hor}}
\newcommand{\conn}{\mathrm{Conn}}
\newcommand{\der}{\mathrm{Der}}
\newcommand{\graph}{\mathrm{Graph}}
\newcommand{\prolongation}{\operatorname{pr}}
\newcommand{\totalvf}[1]{\operatorname{tot}({#1})}
\newcommand{\evolutionaryvf}[1]{{#1}_{\mathrm{ev}}}
\newcommand{\evol}{\mathrm{Evol}}
\newcommand{\diag}{\mathrm{diag}}
\newcommand{\framebundle}[1]{\mathcal{F}(#1)}
\newcommand{\orthogonalframebundle}[1]{\mathcal{O}(#1)}
\newcommand{\Dh}{D^{\mathrm{h}}}
\newcommand{\Dv}{D^{\mathrm{v}}}
\newcommand{\Dastv}{D^{\ast\mathrm{v}}}

\newcommand{\inv}{^{-1}}
\newcommand{\id}{\mathrm{id}}
\newcommand{\eval}{\operatorname{ev}}
\newcommand{\ev}{\eval}
\newcommand{\supp}{\operatorname{supp}}
\newenvironment{notyetdone}{{\large\textbf{Not yet done: }}}{}
\newcommand{\interior}[1]{\operatorname{int}(#1)}
\newcommand{\compact}{\mathrm{c}}
\newcommand{\cpct}{\compact}
\newcommand{\bounded}{\mathrm{b}}
\newcommand{\bdd}{\bounded}
\newcommand{\loc}{\mathrm{loc}}
\newcommand{\double}{\mathrm{d}}
\newcommand{\Max}{\mathrm{max}}
\renewcommand{\i}{\mathbf{i}}
\newcommand{\connectedsum}{\operatorname{\#}}
\newcommand{\imaginaryPart}{\operatorname{Im}}
\newcommand{\realPart}{\operatorname{Re}}
\newcommand{\onto}{\twoheadrightarrow}
\newcommand{\into}{\hookrightarrow}
\newcommand{\pr}{\mathrm{pr}}
\newcommand{\Cpr}{\mathrm{Pr}}
\newcommand{\codim}{\operatorname{codim}}
\newcommand{\closure}{\operatorname{cl}}

\newcommand{\Sphere}[1]{\mathrm{S}^{#1}}
\newcommand{\RiemCurv}{\mathrm{Rm}}
\newcommand{\SecCurv}{\mathrm{Sec}}
\newcommand{\RicCurv}{\mathrm{Ric}}
\newcommand{\ScalCurv}{\mathrm{Scal}}
\newcommand{\Exp}{\operatorname{Exp}}

\newcommand{\norml}{\|}
\newcommand{\normr}{\|}

\newcommand{\cl}{\operatorname{cl}}

\newcommand{\ocirc}[1]{\overset{\mathclap{\;\tiny{\circ}}}{#1}}


\theoremstyle{plain}
	\newtheorem{theorem}{Theorem}[section]
	\newtheorem{proposition}{Proposition}[section]
	\newtheorem{lemma}{Lemma}[section]
	\newtheorem{corollary}{Corollary}[section]
	\newtheorem{conjecture}{Conjecture}[section]
\theoremstyle{definition}
	\newtheorem{definition}{Definition}[section]
	\newtheorem{construction}{Construction}[section]
	\newtheorem{example}{Example}[section]
	\newtheorem{exercise}{Exercise}[section]
	\newtheorem{convention}{Convention}[section]
	\newtheorem{remark}{Remark}[section]
\theoremstyle{remark}
	\newtheorem{comment}{Commentary}[section]
	\newtheorem*{claim}{Claim}

\AtBeginEnvironment{theorem}{\Needspace{5\baselineskip}}
\AtBeginEnvironment{proposition}{\Needspace{5\baselineskip}}
\AtBeginEnvironment{lemma}{\Needspace{5\baselineskip}}
\AtBeginEnvironment{corollary}{\Needspace{5\baselineskip}}
\AtBeginEnvironment{definition}{\Needspace{5\baselineskip}}

\maketitle

\selectlanguage{english}

\thispagestyle{empty}

\begin{abstract}
This article provides a version of scale calculus geared towards a notion of (nonlinear) Fredholm maps between certain types of Fr{\'e}chet spaces, retaining as many as possible of the properties Fredholm maps between Banach spaces enjoy, and the existence of a constant rank theorem for such maps.
It does so by extending the notion of linear Fredholm maps from \cite{1407.3185} and \cite{1209.4040} to a setting where the Nash-Moser inverse function theorem can be applied and which also encompasses the necessary examples such as the reparametrisation action and (nonlinear) elliptic partial differential operators.
\end{abstract}

\clearpage

\pagenumbering{roman}
\tableofcontents

\clearpage

\pagenumbering{arabic}

\setcounter{section}{-1}
\section{Introduction}

One of the more notorious problems in fields of geometric analysis such as the study of holomorphic curves in symplectic geometry is the fact that while the maps between spaces of functions (or rather their completions to Banach manifolds of maps of, say, some Sobolev class) defined by the partial differential operator under question, such as a nonlinear Cauchy-Riemann operator, are usually Fr{\'e}chet differentiable, the same cannot be said about reparametrisation actions.
As a standard example one can even take the action $S^1\times W^{k,p}(S^1, \R) \to W^{k,p}(S^1,\R)$, $(g,f) \mapsto (h \mapsto f(g\cdot h))$, which is continuous but not Fr{\'e}chet differentiable.
$W^{k,p}(S^1,\R)$ here denotes the Banach space of maps $S^1 \to \R$ of Sobolev class $(k,p)$ for some $k\in\N$ and $1 < p < \infty$.
In recent years, one very promising approach has been the sc-calculus that is part of the theory of polyfolds developed by H.~Hofer, K.~Wysocki and E.~Zehnder, see the most recent \cite{1407.3185} and the references therein.
There the authors develop a new structure on a Banach space, called an sc-structure, and differentiability for maps between subsets of such spaces, called sc-differentiability.
The notion of sc-structure builds on the older notion of a scale of Banach spaces, \ie a sequence $(E_i)_{i \in \N_0}$ of Banach spaces together with continuous inclusions $\iota_i : E_{i+1} \to E_i$, each of which has dense image.
Asking in addition for each $\iota_i$ to be compact results in an sc-scale.
Examples of such sc-scales would be $E_i \definedas \mathcal{C}^i(S^1,\R)$ or $F_i \definedas W^{1+i,p}(S^1,\R)$ for some $p > 1$. \\
Under this new notion of differentiability, reparametrisation actions such as the one above become smooth.
This sc-calculus also comes with a notion of linear and nonlinear Fredholm map, see \cite{1209.4040} for a closer examination of these. \\
Unfortunately, the notions of sc-Banach spaces and of sc-Fredholm operators and maps have a few not so desirable features.
For example, the sc-scales $(E_i)_{i\in\N_0}$ and $(F_i)_{i\in\N_0}$ above satisfy $\bigcap_{i\in\N_0} E_i = \mathcal{C}^\infty(S^1,\R) = \bigcap_{i\in\N_0} F_i$ and both these equalities turn $\mathcal{C}^\infty(S^1, \R)$ into the same Fr{\'e}chet space.
On the other hand, these sc-scales are not equivalent under the standard notion of an isomorphism between sc-Banach spaces. \\
Also, the notion of a nonlinear Fredholm map is fairly complicated, making questions such as whether or not the composition of two Fredholm maps is Fredholm quite intransparent. \\
In contrast in the setting of Banach spaces, one possible definition is that a continuous linear operator is Fredholm \iff it is invertible modulo compact operators.
This definition has the advantage that isomorphisms are obviously Fredholm and since compact operators form an ideal in the space of all bounded operators, it is very easy to see that the composition of two Fredholm operators is Fredholm again. \\
The most straightforward generalisation of this to Fr{\'e}chet differentiable nonlinear maps between Banach spaces is that such a map is Fredholm \iff its differential at every point is a Fredholm operator.
It then is obvious that every diffeomorphism is Fredholm and by the chain rule the composition of two Fredholm maps is Fredholm again, resulting in a naturally coordinate invariant notion of a Fredholm map. \\
Because for maps between Banach spaces there exists a very strong inverse function theorem, there consequently also exists a constant rank theorem for such Fredholm maps. \\
Unfortunately, the situation drastically changes when moving from Banach to Fr{\'e}chet spaces, where the invertibility of the differential of a map at a point does no longer imply the existence of a local inverse for the map itself.
And even invertibility of the differential at all points does not suffice.
So while the theory of Fredholm operators between Fr{\'e}chet spaces largely mirrors the theory of Fredholm operators between Banach spaces, defining a Fredholm map to be a map whose differential at every point is Fredholm results in a class of maps that does not allow a constant rank theorem. \\
Fortunately, the situation is not completely hopeless, for at least for a certain class of Fr{\'e}chet spaces and a class of maps between them that satisfy a certain boundedness condition, there does exist the famous Nash-Moser inverse function theorem (see the excellent article \cite{MR656198} for the theorem, its proof, and plenty of (counter-) examples), guaranteeing the existence of a local inverse, provided the differential is invertible as a family of operators, within this given class of maps.

In this article I will join these two theories, of sc-calculus and the setting where the Nash-Moser inverse function theorem applies, and develop a theory of nonlinear Fredholm maps that follows the one in the settting of Banach spaces as closely as possible.
This also involves a natural equivalence relation (basically a more formalised version of tame equivalence from \cite{MR656198}) on sc-scales that makes the above two examples equivalent, removing a lot of ambiguity in the choice of a concrete sc-scale.

More concretely, after a brief excursion in \cref{Section_Differentiation} into differentiation in locally convex topological vector spaces, where the main definitions and results used in the remainder of this text are collected, first the linear theory of Fredholm operators is developed in \cref{Section_Linear_theory}.
After defining the types of spaces, called $\overline{\text{sc}}$-Fr{\'e}chet spaces, and morphisms between them that are the basic elements of the theory, this theory of Fredholm operators starts by defining what a compact operator is in this setting.
In short, the definition of these operators is as a strengthening of the definition of an $\text{sc}^+$-operator from \cite{1407.3185} that behaves well under equivalence, called a strongly smoothing operator, and which subsequently are shown to form an ideal.
The theory then mirrors the Banach space setting by defining a Fredholm operator as an operator that is invertible modulo strongly smoothing operators, making the Fredholm property evidently stable under composition and under perturbation by strongly smoothing operators.
It culminates in the main structure theorem on Fredholm operators, characterising them as operators with the property that the kernel and cokernel are finite dimensional, split the domain and target space, respectively, and \st the restriction of the operator to a complement of the kernel defines an isomorphism onto the image. \\
In \cref{Section_Nonlinear_maps}, which is fairly technical and consists mainly of results on well definedness and independence of choices for the definitions, the notions of differentiability from \cref{Section_Differentiation} are applied to a scale setting. \\
In \cref{Section_Nash_Moser}, in preparation for the nonlinear theory, augmented versions of an $\overline{\text{sc}}$-Fr{\'e}chet space are defined, which join the previous notion with the tameness conditions from \cite{MR656198} and \cite{MR546504}.
Subsequently a version of the Nash-Moser inverse function theorem (extending the theorem from \cite{MR656198} by results from \cite{MR546504}) that is adapted to these notions is stated and proved. \\
In \cref{Section_Fredholm_maps}, the nonlinear Fredholm theory is then built by following the same scheme as in \cref{Section_Linear_theory}, but in a family version: First a notion of strongly smoothing family of operators is defined and shown to be an ideal under composition of families of operators.
Then a Fredholm family of operators is defined to be one that is invertible modulo families of strongly smoothing operators and shown to have the standard properties: Compositions of Fredholm families are Fredholm, strongly smoothing perturbations of Fredholm families are Fredholm, and the Fredholm index is locally constant, behaves additively under composition and is invariant under strongly smoothing perturbations.
And finally a Fredholm map is a map whose differential is a Fredholm family of operators. \\
Last but not least it is then shown that the Nash-Moser inverse function theorem implies the main theorems on Fredholm maps by virtually identical proofs as for maps between open subsets of Banach spaces: The constant rank theorem, finite dimensional reduction and the Sard-Smale theorem. \\
In the final \cref{Section_Summary}, the results from the previous parts are summarised and collected into an application friendly framework.
To get a more detailed overview of the main definitions, results and examples of this article it might actually be advisable to check out this part of the article first.

\clearpage
\section{Notions of differentiability}\label{Section_Differentiation}

In this section I will give a quick overview over some of the notions and results about differentiation in locally convex vector spaces.
This is not meant as an exhaustive treatment and will cover only the elements of the theory that are used in later parts.
For a more complete picture of the topic and its history the reader is referred \eg to the textbooks \cite{MR0440592} or \cite{MR0488118}.

One of the first things any student attending a calculus course hears when it comes to differentiability is that a function is differentiable if near every point (after a shift) it can be approximated by a linear function (equivalently, if it can be approximated by an affine function).
Formalising this, let $X$ and $Y$ be Hausdorff locally convex topological vector spaces, let $U \subseteq X$ be an open subset and let $f : U \to Y$ by a continuous function.
For a point $x \in U$ consider the shifted function
\begin{align*}
\tilde{f} : U &\to Y \\
x' &\mapsto f(x') - f(x)\text{.}
\end{align*}
The goal is to ``approximate'' this function by a linear function $L \in L_{\mathrm{c}}(X,Y)$,
which only becomes meaningful after one has given a precise definition of ``approximate''.
So let $V \subseteq X$ be a convex balanced neighbourhood of $0$ \st $x + V \subseteq U$.
For every $t \in (0,1]$, one can consider the rescaling map
\begin{align*}
\phi_t : V\times Y &\to U\times Y \\
(u, y) &\mapsto \left(x + tu, ty\right)
\end{align*}
and pull back both $\tilde{f}$ and $L$ by $\phi_t$:
\begin{align*}
\phi_t^\ast \tilde{f} : V &\to Y \\
u &\mapsto \frac{1}{t}\left( f(x + tu) - f(x) \right)\text{,}
\end{align*}
whereas $\phi_t^\ast L = L|_V$.
``Approximate'' then means that the difference
\begin{align*}
r^f_x(\cdot, t) \definedas \phi_t^\ast \tilde{f} - \phi_t^\ast L : V &\to Y \\
u &\mapsto \frac{1}{t}\left( f(x + tu) - f(x) \right) - Lu
\end{align*}
``goes to zero as $t \to 0$''. \\
One can now follow several strategies to give precise meaning to this phrase. \\
As a first possibility, one can consider this as a one parameter family of continuous functions $(r^f_x(\cdot, t))_{t\in (0,1]} \subseteq \mathcal{C}(V, Y)$ and after equipping the space $\mathcal{C}(V, Y)$ with a concrete topology one arrives at a precise definition of what ``approximate'' is supposed to mean:
Setting $r^f_x(\cdot, 0) \equiv 0$, $r^f_x(\cdot, t) \to r^f_x(\cdot, 0)$ as $t \to 0$ in the given topology on $\mathcal{C}(V, Y)$. \\
Alternatively, one can consider this one parameter family as a function $r^f_x : V\times [0,1] \to Y$ and ask for this to be either a continuous or a uniformly continuous function, eliminating the need for a choice of topology on $\mathcal{C}(V, Y)$. \\
As a third possibility one can ask for $r^f_x$ to be either continuous or uniformly continuous when restricted to subsets of $V\times [0,1]$ of the form $A\times [0,1]$, where $A$ comes from a chosen class of subsets such as finite, compact or bounded subsets. \\
Some of these choices will now be explored over the next few sections.

\Needspace{25\baselineskip}
\subsection{Definitions and various characterisations}

\Needspace{15\baselineskip}
\subsubsection{Nonlinear maps between topological vector spaces}

I will use the terminology from \cite{MR3154940}: \\
Let $X$ and $Y$ be Hausdorff locally convex topological vector spaces over a ground field $\mathds{k}$, which is $\R$ or $\C$.
$L_{\mathrm{c}}(X, Y)$ denotes the space of continuous linear maps and will be equipped with the topology of bounded convergence (or bounded-open topology).
It is generated (\ie these sets form a neighbourhood basis of $0$) by the subsets
\[
N(A,U) \definedas \{L \in L_{\mathrm{c}}(X,Y) \;|\; L(A) \subseteq U\}\text{,}
\]
where $A \subseteq X$ is a bounded subset (\ie for every neighbourhood $V \subseteq X$ of $0$ there exists $c \in \mathds{k}$ \st $A \subseteq cV$) and $U \subseteq Y$ is a convex balanced ($\forall\, c\in\mathds{k}$ with $|c| \leq 1$, $cU \subseteq U$) neighbourhood of $0$.
If $(X, \|\cdot\|_X)$ and $(Y, \|\cdot\|_Y)$ are normed spaces (so in particular for Banach spaces), this coincides with the probably more familiar operator norm topology.
I will denote the associated operator norm by
\begin{align*}
\|\cdot\|_{L_{\mathrm{c}}(X,Y)} : L_{\mathrm{c}}(X, Y) &\to [0,\infty) \\
L &\mapsto \sup \bigl\{\|Lx\|_Y \,\bigl|\; x \in X, \|x\|_X = 1\bigr\}\text{.}
\end{align*}

I will also repeatedly use the following properties for subsets of a locally convex topological vector space and (not necessarily linear) maps between such subsets:
\begin{definition}\label{Definition_Boundedness}
Let $X$ and $Y$ be locally convex topological vector spaces, let $U \subseteq X$ be a subset and let $f : U \to Y$ be a map.
\begin{enumerate}[label=\arabic*.,ref=\arabic*.]
  \item A subset $A \subseteq X$ is called \emph{bounded} \iff for every neighbourhood $V \subseteq X$ of $0$ there exists $c \in \mathds{k}$ \st $A \subseteq cV$. \\
A subset $A \subseteq U$ is called bounded if it is bounded as a subset of $X$.
  \item $X$ is called \emph{locally bounded} \iff there exists a bounded neighbourhood of $0$ in $X$.
  \item $f$ is called \emph{bounded} if it maps bounded subsets of $U \subseteq X$ to bounded subsets of $Y$.
  \item $f$ is called \emph{locally bounded} if for every $x \in U$ there exists a neighbourhood $V \subseteq U$ of $x$ in $U$ \st $f|_{V} : V \to Y$ is bounded.
  \item $f$ is said to have \emph{bounded image} if $f(U)\subseteq Y$ is a bounded subset.
  \item $f$ is said to have \emph{locally bounded image} if for every $x \in U$ there exists a neighbourhood $V \subseteq U$ of $x$ in $U$ \st $f|_{V} : V \to Y$ has bounded image.
\end{enumerate}
\end{definition}

\begin{remark}\label{Remark_Bounded_maps}
\begin{enumerate}[label=\arabic*.,ref=\arabic*.]
  \item In finite dimensions, a subset $A \subseteq X$ is compact \iff it is bounded and closed.
  \item $X$ is locally bounded \iff it is normable. \\
This follows by combining Propositions 2.8 and 3.33 in \cite{MR3154940}.
  \item\label{Remark_Bounded_maps_3} The space of smooth sections of a vector bundle with its Fr{\'e}chet topology (\ie the inverse limit of the topologies defined by the $\mathcal{C}^k$-norms, \cf \cref{Subsection_Reparametrisation_action}) is not locally bounded. \\
This follows immediately from the existence, for any $k \in \N_0$, of a sequence $(u^k_n)_{n\in\N_0}$ of smooth sections whose $\mathcal{C}^k$-norms are bounded (independently of $n\in\N_0$), but whose $\mathcal{C}^{k+1}$-norms go to infinity when $n$ goes to infinity.
The construction is completely analogous to the construction of the section $u^k_n$ from \cref{Example_Counterexample_reparametrisation_action}.
  \item In finite dimensions, if $U$ is closed, then any continuous map $f : U \to Y$ is bounded. \\
For a bounded subset of $U$ then has compact closure in $U$, hence its image is contained in a compact, hence bounded (\cf \cite{MR3154940}, Theorem 4.28), subset of $Y$, so is bounded itself.
  \item In finite dimensions, any continuous map $f : U \to Y$ is locally bounded, but obviously not every continuous map has bounded image.
  \item Any continuous linear map $L : X \to Y$ is bounded (\cf \cite{MR3154940}, Proposition 2.18).
  \item If $f$ has (locally) bounded image, then $f$ is (locally) bounded.
  \item If $X$ is locally bounded, then $f$ has locally bounded image \iff $f$ is locally bounded.
\end{enumerate}
\end{remark}

\begin{definition}[and Lemma]\label{Definition_Function_space_topologies}
Let $X$ and $Y$ be locally convex topological vector spaces and let $U \subseteq X$ be an open subset.
\begin{enumerate}[label=\arabic*.,ref=\arabic*.]
  \item $\mathcal{C}(U, Y)$ is the vector space of all continuous functions from $U$ to $Y$.
  \item $\mathcal{C}^{\mathrm{b}}(U, Y)$ is the vector space of all bounded continuous functions from $U$ to $Y$.
  \item $\mathcal{C}^{\mathrm{bi}}(U, Y)$ is the vector space of all continuous functions from $U$ to $Y$ with bounded image.
  \item The subsets of the form
\[
N(C,W) \definedas \{g \in \mathcal{C}(U,Y) \;|\; g(C) \subseteq W\}\text{,}
\]
where $W \subseteq Y$ is a convex balanced neighbourhood of $0$ in $Y$ and $C \subseteq U$ is a finite subset, form a neighbourhood base at $0$ of a topology on $\mathcal{C}(U,Y)$ called the \emph{finite-open} topology. \\
It turns $\mathcal{C}(U,Y)$ into a Hausdorff locally convex topological vector space denoted $\mathcal{C}_{\mathrm{fo}}(U,Y)$.
  \item The subsets of the form
\[
N(K,W) \definedas \{g \in \mathcal{C}(U,Y) \;|\; g(K) \subseteq W\}\text{,}
\]
where $W \subseteq Y$ is a convex balanced neighbourhood of $0$ in $Y$ and $K \subseteq U$ is a compact subset, form a neighbourhood base at $0$ of a topology on $\mathcal{C}(U,Y)$ called the \emph{compact-open} topology. \\
It turns $\mathcal{C}(U,Y)$ into a Hausdorff locally convex topological vector space denoted $\mathcal{C}_{\mathrm{co}}(U,Y)$.
  \item The subsets of the form
\[
N(B,W) \definedas \{g \in \mathcal{C}^{\mathrm{b}}(U,Y) \;|\; g(B) \subseteq W\}\text{,}
\]
where $W \subseteq Y$ is a convex balanced neighbourhood of $0$ in $Y$ and $B \subseteq U$ is a bounded subset, form a neighbourhood base at $0$ of a topology on $\mathcal{C}^{\mathrm{b}}(U,Y)$ called the \emph{bounded-open} topology. \\
It turns $\mathcal{C}^{\mathrm{b}}(U,Y)$ into a Hausdorff locally convex topological vector space denoted $\mathcal{C}^{\mathrm{b}}_{\mathrm{bo}}(U,Y)$.
  \item The subsets of the form
\[
N(A,W) \definedas \{g \in \mathcal{C}^{\mathrm{bi}}(U,Y) \;|\; g(A) \subseteq W\}\text{,}
\]
where $W \subseteq Y$ is a convex balanced neighbourhood of $0$ in $Y$ and $A \subseteq U$ is a bounded subset, form a neighbourhood base at $0$ of a topology on $\mathcal{C}^{\mathrm{bi}}(U,Y)$ called the \emph{arbitrary-open} topology. \\
It turns $\mathcal{C}^{\mathrm{bi}}(U,Y)$ into a Hausdorff locally convex topological vector space denoted $\mathcal{C}^{\mathrm{bi}}_{\mathrm{ao}}(U,Y)$.
\end{enumerate}
\end{definition}
\begin{proof}
The proofs that the above topologies are well defined with the stated properties all use \cite{MR3154940}, Theorem 3.2. \\
For let $A \subseteq U$ be a (finite, compact, bounded) subset and let $W \subseteq Y$ be a convex balanced neighbourhood of $0$.
Then $\tilde{N}(A,W)$ is convex and balanced because $W$ is.
It is also absorbent: If $g \in \mathcal{C}(U,Y)$ ($g \in \mathcal{C}^{\mathrm{b}}(U,Y)$, $g \in \mathcal{C}^{\mathrm{bi}}(U,Y)$), then $g(A) \subseteq Y$ is bounded, because it is finite or compact (\cf \cite{MR3154940}, Theorem 4.28), or by assumption.
So there exists $c \in \mathds{k}$ \st $g(A) \subseteq cW$.
Hence $g \in \tilde{N}(A,cW) = c\tilde{N}(A,W)$ and thus $\tilde{N}(A,W)$ is absorbent.
Moreover, $\frac{1}{2}\tilde{N}(A,W) = \tilde{N}(A, \frac{1}{2}W)$ and $\tilde{N}(A,W)\cap \tilde{N}(A',W') \supseteq \tilde{N}(A\cup A', W\cap W')$.
Finally, $\bigcap \mathcal{B}_0 = \{0\}$, for let $0 \neq g \in \mathcal{C}(U,Y)$ ($g \in \mathcal{C}^{\mathrm{b}}(U,Y)$, $g \in \mathcal{C}^{\mathrm{bi}}(U,Y)$).
Choose $x \in U$ \st $g(x)\neq 0$ and let $A \definedas \{x\}$.
$W' \definedas Y \setminus \{g(x)\}$ is an open neighbourhood of $0$ so by \cite{MR3154940}, Proposition 3.1, there exists a convex balanced neighbourhood $U \subseteq Y$ of $0$ \st $W \subseteq W'$.
Then $g \not\in \tilde{N}(A, W)$.
\end{proof}

\begin{definition}
Let $X$ and $Y$ be locally convex topological vector spaces, let $U \subseteq X$ be a subset and let $f : U \to Y$ be a map.
\begin{enumerate}[label=\arabic*.,ref=\arabic*.]
  \item $f : U \to Y$ is called \emph{uniformly continuous} if for every neighbourhood $W$ of $0$ in $Y$ there exists a neighbourhood $V$ of $0$ in $U$ \st $f(x) - f(y) \in W$ for all $x,y \in U$ with $x-y \in V$.
  \item $f$ is called \emph{locally uniformly continuous} if for every $x \in U$ there exists a neighbourhood $V \subseteq U$ of $x$ in $U$ \st $f|_{V} : V \to Y$ is uniformly continuous.
  \item $\mathcal{C}^{\mathrm{uc}}(U,Y)$, $\mathcal{C}^{\mathrm{uc},\mathrm{b}}(U,Y)$ and $\mathcal{C}^{\mathrm{uc},\mathrm{bi}}(U,Y)$ are the vector spaces of all function from $U$ to $Y$ that are uniformly continuous, bounded uniformly continuous and uniformly continuous with bounded image, respectively. \\
If equipped with the corresponding finite-open, compact-open, bounded-open and arbitrary-open topologies, they are Hausdorff locally convex topological vector spaces denoted $\mathcal{C}^{\mathrm{uc}}_{\mathrm{fo}}(U,Y)$, $\mathcal{C}^{\mathrm{uc}}_{\mathrm{co}}(U,Y)$, $\mathcal{C}^{\mathrm{uc},\mathrm{b}}_{\mathrm{bo}}(U,Y)$ and $\mathcal{C}^{\mathrm{uc},\mathrm{bi}}_{\mathrm{ao}}(U,Y)$, respectively.
\end{enumerate}
\end{definition}

\begin{remark}
Obviously, every (locally) uniformly continuous map is continuous.
\end{remark}

An important result about uniform continuity is the following rather well known (at least for metric spaces) statement relating continuity and uniform continuity on compact subsets.
\begin{proposition}\label{Lemma_Continuity_implies_uniform_continuity}
Let $X$ and $Y$ be Hausdorff locally convex topological vector spaces, let $K\subseteq X$ be a compact subset and let $g : K \to Y$ be a continuous function.
Then $g$ is uniformly continuous.
\end{proposition}
\begin{proof}
Consider the function
\begin{align*}
G : K\times K &\to Y \\
(x,y) &\mapsto g(x) - g(y)
\intertext{and the embedding}
\Phi : K\times K &\to K\times X \\
(x,y) &\mapsto (x,x-y)
\end{align*}
and define $\tilde{K} \definedas \Phi(K\times K)$ and $\tilde{G} \definedas G\circ \Phi\inv|_{\tilde{K}} : \tilde{K} \to Y$.
Let $W \subseteq Y$ be an arbitrary neighbourhood of $0$.
$\tilde{G}(K\times \{0\}) = \{0\} \subseteq W$, so by continuity of $g$ and hence $\tilde{G}$, for every $x \in K$ there exists a neighbourhood (which can be assumed to be of the form) $V'_x\times V_x \subseteq K\times X$ of $(x,0)$ \st $\tilde{G}(\tilde{K}\cap V'_x\times V_x) \subseteq W$.
Now as in the proof of \cref{Lemma_Topological_Lemma}, there exist finitely many $x_1, \dots, x_r \in K$ \st $K \subseteq \bigcup_{i=1}^rV'_{x_i}$.
Define $V \definedas \bigcap_{i=1}^r V_{x_i}$.
Then $\tilde{G}(\tilde{K}\cap K\times V) \subseteq \tilde{G}\left(\tilde{K}\cap \bigcup_{i=1}^r V'_{x_i} \times V\right) = \bigcup_{i=1}^r\tilde{G}\left(\tilde{K}\cap V'_{x_i} \times V\right) \subseteq W$.
On the other hand, $\tilde{G}(\tilde{K}\cap K\times V) = G(\Phi\inv(\tilde{K}\cap K\times V)) = G(\{(x,y) \in K\times K \;|\; x \in K, x-y \in V\}) = \{g(x) - g(y) \;|\; (x,y)\in K\times K, x-y\in V\}$.
Or in other words, $g(x) - g(y) \in W$ for all $x,y\in K$ with $x-y\in V$, so $g$ is uniformly continuous.
\end{proof}

\begin{corollary}
Let $X$ and $Y$ be Hausdorff locally convex topological vector spaces, let $U \subseteq X$ be an open subset and let $f : U \to Y$ be a map. \\
If $X$ is finite dimensional, then $f$ is locally uniformly continuous \iff $f$ is continuous.
\end{corollary}
\begin{proof}
One direction is trivial.
And in the other direction, because $X$ is finite dimensional, every point in $U$ has a compact neighbourhood $K$.
So if $f$ is continuous, by \cref{Lemma_Continuity_implies_uniform_continuity} $f|_K$ is uniformly continuous.
\end{proof}

\Needspace{15\baselineskip}
\subsubsection{Weak Fr{\'e}chet differentiability}

\begin{definition}\label{Definition_Differentiability_I}
Let $X$ and $Y$ be Hausdorff locally convex topological vector spaces, let $U \subseteq X$ be an open subset and let $f : U \to Y$ be a continuous map. \\
Given a point $x \in U$, a continuous linear map $L \in L_{\mathrm{c}}(X,Y)$ and a convex balanced neighbourhood $V \subseteq X$ of $0$ \st $x + V \subseteq U$, define a function
\begin{align*}
r^f_x : V \times [0,1] &\to Y \\
(u,t) &\mapsto \begin{cases} \frac{1}{t} \left(f(x + tu) - f(x)\right) - Lu & t > 0 \\ 0 & t = 0 \end{cases}\text{.}
\end{align*}
$f$ is called
\begin{enumerate}[label=\arabic*.,ref=\arabic*.]
  \item\label{Definition_Differentiability_I_1} \emph{G{\^a}teaux, or pointwise weakly Fr{\'e}chet, differentiable at $x$ with derivative} $Df(x) \definedas L$ \iff there exist $L \in L_{\mathrm{c}}(X,Y)$ and $V \subseteq X$ as above \st for every finite subset $C \subseteq V$ the map
\[
r^f_x|_{C\times [0,1]} : C\times [0,1] \to Y
\]
is continuous.
  \item\label{Definition_Differentiability_I_2} \emph{compactly weakly Fr{\'e}chet differentiable at $x$ with derivative} $Df(x) \definedas L$ \iff there exist $L \in L_{\mathrm{c}}(X,Y)$ and $V \subseteq X$ as above \st for every compact subset $K \subseteq V$ the map
\[
r^f_x|_{K\times [0,1]} : K\times [0,1] \to Y
\]
is continuous.
  \item\label{Definition_Differentiability_I_3} \emph{weakly Fr{\'e}chet differentiable at $x$ with derivative} $Df(x) \definedas L$ \iff there exist $L \in L_{\mathrm{c}}(X,Y)$ and $V \subseteq X$ as above \st the map
\[
r^f_x : V\times [0,1] \to Y
\]
is continuous.
\end{enumerate}
If one of the above holds for some $V \subseteq X$, then it also holds for all $V' \subseteq V$. \\
If an $L$ exists \st \labelcref{Definition_Differentiability_I_1}, \labelcref{Definition_Differentiability_I_2}~or \labelcref{Definition_Differentiability_I_3}~above holds, then it is unique and $f$ is simply called \emph{pointwise, compactly or just weakly Fr{\'e}chet differentiable at $x$}, respectively.
\end{definition}

\begin{remark}
In the definition of weak Fr{\'e}chet differentiability, asking just for continuity of $r^f_x$ along $V\times \{0\}$, this is also called \emph{Leslie differentiability}, \cf \cite{DodsonGalanisVassiliou}.
\end{remark}

\begin{remark}
Another natural class of subsets, along with the finite, compact or arbitrary ones used in the definition above, is the class of bounded subsets.
So the obvious question is why not repeat the above definition for this class of subsets?
This will not be examined in this text any further because of the following reasons:
First, it would not be used in this text anyway, outside of this section. \\
Second, for locally bounded (\ie normable) Hausdorff locally convex topological vector spaces it coincides with weak Fr{\'e}chet differentiability. \\
And third, for not locally bounded spaces it seems that further distinctions/assumptions need to be made for the analogous results to those in this and the subsequent subsections to hold, which tipped the scale in favour of excluding it.
\end{remark}

\begin{remark}\label{Remark_Rest_continuity_versus_continuity_at_0_1}
Note that the continuity of $r^f_x|_{V\times (0,1]} : V\times (0,1] \to Y$ is automatic by virtue of the continuity assumption on $f$.
One could also drop this assumption and instead ask for continuity of $r^f_x$ at the points of $A\times \{0\}$, where $A \subseteq V$ is finite, compact, or all of $V$, respectively.
But for the purposes of this article this just presents an unnecessary complication (the question of whether differentiability in various forms implies continuity will not be addressed here).
\end{remark}

\begin{remark}
By \cite{MR3154940}, Proposition 3.1, the existence of a convex balanced neighbourhood $V$ of $0$ in the previous definition is not a restriction and one can furthermore \smallwlog assume that $V$ is open or that $V$ is closed.
\end{remark}

\begin{example}\label{Example_Gateaux_but_not_weakly_Frechet}
This is a standard example that can be found in many calculus textbooks. \\
Consider the function
\begin{align*}
f : \R^2 &\to \R \\
(x,y) &\mapsto \begin{cases} \frac{x^3y}{x^4 + y^2} & (x,y) \neq (0,0) \\ 0 & (x,y) = (0,0) \end{cases}\text{.}
\end{align*}
$f$ is continuous and G{\^a}teaux differentiable at $(0,0)$ but not weakly Fr{\'e}chet and hence by \cref{Theorem_Relations_between_notions_of_differentiability} also not compactly weakly Fr{\'e}chet or Fr{\'e}chet differentiable.

We have $f(0,y) = 0$ and for $x \neq 0$, $|f(x,y)| = |x| \frac{|y/x^2|}{1 + |y/x^2|^2} \leq |x|$, so $f$ is continuous. \\
Let $L : \R^2 \to \R$ be the zero map.
Then
\[
r^f_{(0,0)}((x,y), t) = t\frac{x^3y}{t^2x^4 + y^2}
\]
$r^f_{(0,0)}((x,0), t) \equiv 0$ and for $y \neq 0$, $|r^f_{(0,0)}((x,y), t)| \leq t|x^3/y| \underset{t \to 0}{\longrightarrow} 0$, so $f$ is G{\^a}teaux differentiable at $(0,0)$ with derivative $Df(0,0) = 0$.
On the other hand, $r^f_{(0,0)}((1,t), t) = t\frac{t}{t^2 + t^2} = \frac{1}{2}$ whereas $r^f_{(0,0)}((1,0), t) \equiv 0$.
So $r^f_{(0,0)}$ is not continuous and $f$ is not weakly Fr{\'e}chet differentiable at $(0,0)$.
\end{example}

\begin{definition}\label{Definition_Weakly_strongly_continuous_bounded_derivative}
Let $X$ and $Y$ be Hausdorff locally convex topological vector spaces, let $U \subseteq X$ be an open subset and let $f : U \to Y$ be a continuous map.
For a subset $A \subseteq U$, if $f$ is (pointwise, compactly) weakly Fr{\'e}chet differentiable at $x$ for all $x \in A$, then $f$ is called \emph{(pointwise, compactly) weakly Fr{\'e}chet differentiable on $A$}, respectively. \\
Let a topological space $\tilde{U}$ together with a continuous map $\iota : \tilde{U} \to U$ be given.
If $f$ is (pointwise, compactly) Fr{\'e}chet differentiable on $\iota(\tilde{U})$ and if furthermore the map
\begin{enumerate}[label=\arabic*.,ref=\arabic*.]
  \item 
\begin{align*}
\tilde{U} &\to L_{\mathrm{c}}(X,Y) \\
x &\mapsto Df(\iota(x))
\end{align*}
is continuous, then $f$ is called \emph{strongly continuously} (pointwise, compactly) weakly Fr{\'e}chet differentiable \emph{along $\iota$}, respectively. \\
In case $\iota = \id_U$, $f$ is just called strongly continuously (pointwise, compactly) weakly Fr{\'e}chet differentiable.
  \item 
\begin{align*}
\tilde{U}\times X &\to Y \\
(x, u) &\mapsto Df(\iota(x))u
\end{align*}
is continuous, then $f$ is called \emph{weakly continuously} (pointwise, compactly) weakly Fr{\'e}chet differentiable \emph{along $\iota$}, respectively. \\
In case $\iota = \id_U$, $f$ is just called weakly continuously (pointwise, compactly) weakly Fr{\'e}chet differentiable.
  \item If
\begin{align*}
\tilde{U} &\to L_{\mathrm{c}}(X,Y) \\
x &\mapsto Df(\iota(x))
\end{align*}
is locally bounded, then $f$ is said to have \emph{locally bounded derivative along $\iota$}. \\
In case $\iota = \id_U$, $f$ is just said to have locally bounded derivative.
\end{enumerate}
\end{definition}

\begin{theorem}[Chain rule]\label{Theorem_Chain_rule_I}
Let $X$, $Y$ and $Z$ be Hausdorff locally convex topological vector spaces and let $U\subseteq X$ and $V \subseteq Y$ be open subsets.
Let furthermore $f : U \to Y$ and $g : V \to Z$ be continuous functions with $f(U) \subseteq V$.
Given $x \in U$ and a topological space $\tilde{U}$ together with a continuous function $\iota : \tilde{U} \to U$, the following hold:
\begin{enumerate}[label=\arabic*.,ref=\arabic*.]
  \item If $f$ is (compactly) weakly Fr{\'e}chet differentiable at $x$ and $g$ is (compactly) weakly Fr{\'e}chet differentiable at $f(x)$, then the composition $g\circ f : U \to Z$ is (compactly) weakly Fr{\'e}chet differentiable at $x$ with
\[
D(g\circ f)(x) = Dg(f(x))\circ Df(x)\text{.}
\]
  \item If $f$ is weakly continuously (compactly) weakly Fr{\'e}chet differentiable along $\iota$ and $g$ is weakly continuously (compactly) weakly Fr{\'e}chet differentiable along $f\circ \iota$, then $g \circ f$ is weakly continuously (compactly) weakly Fr{\'e}chet differentiable along $\iota$. \\
If $Y$ is locally bounded, then the same holds for ``strongly continuously'' in place of ``weakly continuously''.
\end{enumerate}
\end{theorem}
\begin{proof}
\begin{enumerate}[label=\arabic*.,ref=\arabic*.]
  \item Note that for $t > 0$, $r^f_x(u,t)$ (and analogously for $r^g$ and $r^{g\circ f}$) is uniquely determined by the equation
\[
f(x + tu) = f(x) + t(Df(x)u + r^f_x(u,t))\text{.}
\]
One hence calculates for $t > 0$.
\begin{align*}
g\circ f(x + tu) &= g(f(x) + t(Df(x)u + r^f_x(u,t))) \\
&= g(f(x)) + t\bigl[ Dg(f(x))\left(Df(x)u + r^f_x(u,t)\right) \;+ \\
&\quad\; \qquad\qquad\quad\; +\; r^g_{f(x)}(Df(x)u + r^f_x(u,t), t) \bigr] \\
&= g(f(x)) + t\bigl[ Dg(f(x))\circ Df(x)u \;+ \\
&\quad\; \qquad\qquad\quad\; +\; Dg(f(x))r^f_x(u,t) \;+ \\
&\quad\; \qquad\qquad\quad\; +\; r^g_{f(x)}(Df(x)u + r^f_x(u, t), t) \bigr]\text{.}
\end{align*}
It follows that
\begin{equation}\label{Equation_Chain_rule_I}
r^{g\circ f}_{x}(u, t) = Dg(f(x))r^f_x(u,t) + r^g_{f(x)}(Df(x)u + r^f_x(u, t), t)\text{,}
\end{equation}
whenever both sides of the equation are defined.
\begin{enumerate}[label=\arabic*.,ref=\arabic*.]
  \item Let $f$ and $g$ be weakly Fr{\'e}chet differentiable at $x$ and $f(x)$, respectively.
For the right hand side in the above formula to be well defined, if $r^g_{f(x)}$ is defined on $V'\times [0,1]$ and $r^f_x$ is defined on $V\times [0,1]$, one has to find a convex balanced neighbourhood $V'' \subseteq V \subseteq X$ of $0$ \st $Df(x)u + r^f_x(u,t) \in V'$ for all $(u,t) \in V'' \times [0,1]$.
$r^{g\circ f}_x$ is then defined on $V''\times [0,1]$ and satisfies the above formula. \\
Now the function
\begin{align*}
\phi : V\times [0,1] &\to Y \\
(u,t) &\mapsto Df(x)u + r^f_x(u,t)
\end{align*}
is continuous and satisfies $\{0\}\times [0,1] \subseteq \phi\inv(0) \subseteq \phi\inv(V')$.
By \cref{Lemma_Topological_Lemma} below there hence exists a neighbourhood $V'' \subseteq V \subseteq X$ of $0$ \st $V''\times [0,1] \subseteq \phi\inv(V')$ and by \cite{MR3154940}, Proposition 3.1, one can assume that $V''$ is convex and balanced. \\
If $f$ and $g$ are weakly Fr{\'e}chet differentiable at $x$ and $f(x)$, respectively, then the right hand side of the above formula for $r^{g\circ f}_x$ extends continuously by $0$ to $t=0$.
So $r^{g\circ f}_x$ is continuous as a composition of continuous functions.
  \item Let $f$ and $g$ be compactly weakly Fr{\'e}chet differentiable at $x$ and $f(x)$, respectively, and let $V$ and $V'$ be convex balanced neighbourhoods of $0$ in $X$ and $Y$, respectively, as in \cref{Definition_Differentiability_I}.
After possibly making $V$ smaller (and again using \cite{MR3154940}, Proposition 3.1), one can also assume that $f(x + V) \subseteq f(x) + V'$ and $Df(x)(V) \subseteq V'$, using continuity of $f$ and $Df(x)$, respectively.
Let $K \subseteq V$ be compact.
The goal is to show that $r^{g\circ f}_x|_{K\times [0,1]}$ is continuous.
Since $r^{g\circ f}_x|_{K\times (0,1]}$ is already continuous by definition, it suffices to show that
$r^{g\circ f}_x|_{K\times [0,\delta)}$ is continuous for an arbitrarily small $\delta = \delta(K) > 0$.
By assumption, the map
\begin{align*}
\phi : K\times [0,1] &\to Y \\
(u,t) &\mapsto Df(x)u + r^f_x(u,t)
\end{align*}
is continuous and satisfies $\phi(K \times \{0\}) = Df(x)(K) \subseteq Df(x)(V) \subseteq V'$, hence $K\times \{0\} \subseteq \phi\inv(V')$.
By \cref{Lemma_Topological_Lemma} there exists $\delta = \delta(K) > 0$ \st $K\times [0,2\delta) \subseteq \phi\inv(V')$.
This implies that in the above formula for $r^{g\circ f}_x$ the right hand side is well defined for $(u,t) \in K\times [0,2\delta)$.
The map
\begin{align*}
K\times [0,\delta] &\to Y \\
(u,t) &\mapsto Df(x)u + r^f_x(u,t)
\end{align*}
is continuous because $Df(x)$ is and because $r^f_x : K\times [0,1] \to Y$ is by assumption.
It hence has compact image and one can conclude as before that $r^{g\circ f}_x|_{K\times [0,\delta)}$ is continuous as a composition of continuous functions.
\end{enumerate}
  \item For weak continuity, the map
\begin{align*}
\tilde{U} \times X &\to Y \\
(x,u) &\mapsto D(g\circ f)(\iota(x))u
\end{align*}
is given by the composition
\[
(x,u) \mapsto (f(\iota(x)), Df(\iota(x))u) \mapsto Dg(f(\iota(x)))Df(\iota(x))u\text{,}
\]
hence continuous. \\
If $Y$ is locally bounded, then the composition $L_{\mathrm{c}}(X,Y)\times L_{\mathrm{c}}(Y, Z) \to L_{\mathrm{c}}(X,Z)$ is continuous, which shows the last statement.
\end{enumerate}
\end{proof}

\begin{lemma}\label{Lemma_Topological_Lemma}
Let $A$, $B$ and $K$ be topological spaces with $K$ compact.
\begin{enumerate}[label=\arabic*.,ref=\arabic*.]
  \item\label{Lemma_Topological_Lemma_1} Let $\phi : A\times K \to B$ be a continuous function, let $a_0\in A$ and let $V \subseteq B$ be an open subset with $\{a_0\}\times K \subseteq \phi\inv(V)$.
Then there exists a neighbourhood $U \subseteq A$ of $a_0$ \st $U\times K \subseteq \phi\inv(V)$.
  \item\label{Lemma_Topological_Lemma_2} If $\phi : A \times K \to \R$ is a continuous function, then so is the function
\begin{align*}
\max_K \phi : A &\to \R \\
a &\mapsto \max \{\phi(a,k) \;|\; k\in K\}\text{.}
\end{align*}
\end{enumerate}
\end{lemma}
\begin{proof}
\begin{enumerate}[label=\arabic*.,ref=\arabic*.]
  \item Let $\phi$, $a_0$ and $V$ be as in the statement of the lemma.
Because $\phi$ is continuous, for every $k \in K$ there exist neighbourhoods $U_k \subseteq A$ of $a_0$ and $W_k \subseteq K$ of $k$ with $U_k\times W_k \subseteq \phi\inv(V)$.
Because $K$ is compact, one can choose finitely many $k_1, \dots k_d \in K$ \st $W_{k_1}, \dots, W_{k_d}$ cover $K$.
Let $U \definedas \bigcap_{j=1}^d U_{k_j}$.
Then $U\times K \subseteq U\times \bigcup_{j=1}^d W_{k_j} \subseteq \bigcup_{j=1}^d U_{k_j}\times W_{k_j} \subseteq \phi\inv(V)$.
  \item Let $a_0 \in A$, $b_0 \definedas \max_K\phi(a_0)$ and let $\varepsilon > 0$.
By definition $\{a_0\}\times K \subseteq \phi\inv((-\infty, b_0 + \varepsilon))$ and hence by \labelcref{Lemma_Topological_Lemma_1}~there exists a neighbourhood $U'\subseteq A$ of $a_0$ \st $U'\times K \subseteq \phi\inv((-\infty, b_0 + \varepsilon))$.
Hence $\max_K(a) < b_0 + \varepsilon$ for all $a \in U'$. \\
Now pick $k_0 \in K$ \st $b_0 = \phi(a_0,k_0)$.
By continuity of $\phi$ there exists a neighbourhood $U'' \subseteq A$ of $a_0$ \st $|\phi(a,k_0) - \phi(a_0,k_0)| < \varepsilon$ for all $a \in U''$.
Hence $\max_K\phi(a) \geq \phi(a,k_0) > \phi(a_0,k_0) - \varepsilon = b_0 - \varepsilon$ for all $a \in U''$.
$U \definedas U' \cap U''$ then satisfies $|\max_K(a) - \max_K(a_0)| < \varepsilon$ for all $a \in U$.
\end{enumerate}
\end{proof}

\Needspace{15\baselineskip}
\subsubsection{Fr{\'e}chet differentiability}

\begin{definition}\label{Definition_Differentiability_II}
Let $X$ and $Y$ be Hausdorff locally convex topological vector spaces, let $U \subseteq X$ be an open subset and let $f : U \to Y$ be a continuous map. \\
Given a point $x \in U$, a continuous linear map $L \in L_{\mathrm{c}}(X,Y)$ and a convex balanced neighbourhood $V \subseteq X$ of $0$ \st $x + V \subseteq U$, define a function
\begin{align*}
r^f_x : V \times [0,1] &\to Y \\
(u,t) &\mapsto \begin{cases} \frac{1}{t} \left(f(x + tu) - f(x)\right) - Lu & t > 0 \\ 0 & t = 0 \end{cases}\text{.}
\end{align*}
$f$ is called
\begin{enumerate}[label=\arabic*.,ref=\arabic*.]
  \item\label{Definition_Differentiability_II_1} \emph{G{\^a}teaux, or pointwise Fr{\'e}chet, differentiable at $x$ with derivative} $Df(x) \definedas L$ \iff there exist $L \in L_{\mathrm{c}}(X,Y)$ and $V \subseteq X$ as above \st the map
\begin{align*}
[0,1] &\to \mathcal{C}_{\mathrm{fo}}(V, Y) \\
t &\mapsto r^f_x(\cdot, t)
\end{align*}
is continuous.
  \item\label{Definition_Differentiability_II_2} \emph{compactly Fr{\'e}chet differentiable at $x$ with derivative} $Df(x) \definedas L$ \iff there exist $L \in L_{\mathrm{c}}(X,Y)$ and $V \subseteq X$ as above \st the map
\begin{align*}
[0,1] &\to \mathcal{C}_{\mathrm{co}}(V, Y) \\
t &\mapsto r^f_x(\cdot, t)
\end{align*}
is continuous.
  \item\label{Definition_Differentiability_II_3} If $f$ has locally bounded image, then $f$ is called \emph{Fr{\'e}chet differentiable at $x$ with derivative} $Df(x) \definedas L$ \iff there exist $L \in L_{\mathrm{c}}(X,Y)$ and $V \subseteq X$ as above \st the map
\begin{align*}
[0,1] &\to \mathcal{C}^{\mathrm{bi}}_{\mathrm{ao}}(V, Y) \\
t &\mapsto r^f_x(\cdot, t)
\end{align*}
is well defined and continuous.
\end{enumerate}
If one of the above holds for some $V \subseteq X$, then it also holds for all $V' \subseteq V$. \\
If an $L$ exists \st \labelcref{Definition_Differentiability_II_1}, \labelcref{Definition_Differentiability_II_2}~or \labelcref{Definition_Differentiability_II_3}~above holds, then it is unique and $f$ is simply called \emph{pointwise, compactly or just Fr{\'e}chet differentiable at $x$}, respectively.
The notions of \emph{weakly} and \emph{strongly continuously (pointwise, compactly) Fr{\'e}chet differentiable}, as well as of \emph{locally bounded derivative}, are defined in complete analogy to \cref{Definition_Weakly_strongly_continuous_bounded_derivative}.
\end{definition}

\begin{remark}\label{Remark_Rest_continuity_versus_continuity_at_0_2}
Note that, similar to \cref{Remark_Rest_continuity_versus_continuity_at_0_1}, one could also ask for $t \mapsto r^f_x(\cdot, t)$ to be continuous just at $t = 0$, instead of continuity for all $t \in [0,1]$.
Indeed, at least for Fr{\'e}chet differentiability, this is equivalent only under the assumption that $f$ is locally uniformly continuous (\cf also \cref{Remark_Assumption_uniform_continuity}).
But for the purposes of this article, the distinction is immaterial since Fr{\'e}chet differentiability won't be used anyway.
\end{remark}

At first glance, the notions of weak Fr{\'e}chet differentiability and Fr{\'e}chet differentiability might seem quite different, but the following proposition shows that in suitable reformulations these notions are quite similar.

\begin{proposition}\label{Proposition_Characterisation_Differentiability_I}
In the notation of \cref{Definition_Differentiability_II},
\begin{enumerate}[label=\arabic*.,ref=\arabic*.]
  \item\label{Proposition_Characterisation_Differentiability_I_1} The following are equivalent:
\begin{enumerate}[label=(\alph*),ref=(\alph*)]
  \item $f$ is pointwise Fr{\'e}chet differentiable at $x$.
  \item There exists a convex balanced neighbourhood $V \subseteq X$ of $0$ \st for all finite subsets $C \subseteq V$ the map
\[
r^f_x|_{C\times [0,1]} : C\times [0,1] \to Y
\]
is uniformly continuous.
\end{enumerate}
  \item\label{Proposition_Characterisation_Differentiability_I_2} The following are equivalent:
\begin{enumerate}[label=(\alph*),ref=(\alph*)]
  \item $f$ is compactly Fr{\'e}chet differentiable at $x$.
  \item There exists a convex balanced neighbourhood $V \subseteq X$ of $0$ \st for all compact subsets $K \subseteq V$ the map
\[
r^f_x|_{K \times [0,1]} : K\times [0,1] \to Y
\]
is uniformly continuous.
\end{enumerate}
  \item\label{Proposition_Characterisation_Differentiability_I_3} If $f$ has locally bounded image and is locally uniformly continuous, then the following are equivalent:
\begin{enumerate}[label=(\alph*),ref=(\alph*)]
  \item $f$ is Fr{\'e}chet differentiable at $x$.
  \item There exists a convex balanced neighbourhood $V \subseteq X$ of $0$ \st the map
\[
r^f_x : V\times [0,1] \to Y
\]
is uniformly continuous.
\end{enumerate}
\end{enumerate}
\end{proposition}
\begin{proof}
The proofs in all three cases are actually almost identical, so will be treated all at once: \\
Let $V \subseteq X$ be a convex balanced neighbourhood of $0$ either as in the definition of Fr{\'e}chet differentiability or as in the statement of the proposition and let $A \subseteq V$ be a (finite, compact) subset.
Furthermore, assume that $f|_A$ is uniformly continuous and has bounded image, either by choosing $V$ appropriately (possible since it was assumed that $f$ has locally bounded image and is locally uniformly continuous) or by using \cref{Lemma_Continuity_implies_uniform_continuity} in case $A$ is (finite or) compact. \\
\begin{claim}
$r^f_x|_{A\times [\delta_0,1]} : A\times [\delta_0,1] \to Y$ is uniformly continuous for every $\delta_0 \in (0,1]$.
\end{claim}
\begin{proof}
One computes for $(u,t),(u',t')\in (0,1]\times V$ that
{\allowdisplaybreaks
\begin{align}
r^f_x(u, t) - r^f_x(u', t') &= \frac{1}{t}(f(x+tu) - f(x)) - Lu \;-\nonumber \\
&\quad\; -\; \frac{1}{t'}(f(x+t'u') - f(x)) + Lu'\nonumber \\
&= \frac{t-t'}{tt'}f(x) + L(u'-u) \;+\nonumber \\
&\quad\; +\; \frac{1}{tt'}(t'f(x+tu) - tf(x+t'u'))\nonumber \\
&= \frac{t-t'}{tt'}f(x) + L(u'-u) \;+\nonumber \\
&\quad\; +\; \frac{1}{tt'}(t'f(x+tu) - t'f(x+t'u') \;+\nonumber \\
&\quad\; +\; t'f(x+t'u') - tf(x+t'u'))\text{,}\nonumber
\intertext{so}
\begin{split}\label{Proposition_Characterisation_Differentiability_I_Eq1}
r^f_x(u, t) - r^f_x(u', t') &= \frac{t'-t}{tt'}(f(x+t'u') - f(x)) + L(u'-u) \;+ \\
&\quad\; +\; \frac{1}{t}(f(x+tu) - f(x+t'u'))\text{.}
\end{split}
\end{align}
}
Given any convex balanced neighbourhood $W \subseteq Y$ of $0$ one can find convex balanced neighbourhoods $W_1,W_2,W_3 \subseteq Y$ of $0$ \st $x_1 + x_2 + x_3 \in W$ for all $(x_1,x_2,x_3) \in W_1\times W_2\times W_3$ (just by continuity of addition).
Now if $t, t' \in [\delta_0,1]$, then $\frac{1}{tt'}$ and $\frac{1}{t}$ are bounded below away from $0$.
Also, $\{f(x+t'u') - f(x) \;|\; (u',t') \in A\times [0,1]\} \subseteq Y$ is a bounded subset by the assumption that $f|_A$ has bounded image.
Consequently, $M_1 \definedas \{\frac{1}{tt'}(f(x+t'u') - f(x)) \;|\; t,t' \in [\delta_0,1], u' \in A\} \subseteq Y$ is a bounded subset and there exists $\delta_1 > 0$ \st $(t-t')M_1 \subseteq W_1$ for all $t,t' \in [0,1]$ with $|t-t'| < \delta_1$. \\
$L$ is a continuous linear and hence uniformly continuous map, so there exists a (convex balanced) neighbourhood $V_1 \subseteq X$ of $0$ \st $L(u'-u) \in W_2$ for all $u,u' \in X$ with $u-u' \in V_1$.
Since it was assumed that $f|_A$ is uniformly continuous, so is the function $A\times [0,1] \to Y$, $(u,t) \mapsto f(x + tu)$.
Hence there exists $\delta_2 > 0$ and a (convex balanced) neighbourhood $V_2 \subseteq X$ of $0$ \st $f(x + tu) - f(x+t'u') \in \delta_0 W_3$ for all $(u,t), (u',t') \in A\times [0,1]$ with $|t-t'| < \delta_2$ and $u-u' \in V_2$.
Consequently, $\frac{1}{t}(f(x+tu) - f(x+t'u')) \in W_3$ for all $(u,t), (u',t') \in A\times [\delta_0,1]$ with $|t-t'| < \delta_2$ and $u-u' \in V_2$. \\
Let $\delta \definedas \min\{\delta_1,\delta_2\}$ and $\tilde{V} \definedas V_1 \cap V_2$.
Then $r^f_x(u, t) - r^f_x(u', t') \in W_1 + W_2 + W_3 \subseteq W$ for all $(u,t), (u',t') \in A\times [\delta_0,1]$ with $(u,t) - (u',t') = (u-u',t-t') \in \tilde{V}\times (-\delta,\delta)$.
\end{proof}
\begin{claim}
$(0,1] \to \mathcal{C}^{\mathrm{bi}}_{\mathrm{ao}}(V, Y)$ ($\mathcal{C}_{\mathrm{co}}(V, Y)$, $\mathcal{C}_{\mathrm{fo}}(V, Y)$), $t \mapsto r^f_x(\cdot, t)$, is continuous.
\end{claim}
\begin{proof}
Let $t_0 \in (0,1]$.
By definition of the topology on $\mathcal{C}^{\mathrm{bi}}_{\mathrm{ao}}(V, Y)$ ($\mathcal{C}_{\mathrm{co}}(V, Y)$, $\mathcal{C}_{\mathrm{fo}}(V, Y)$), to show continuity at $t = t_0$, for every (compact, finite) subset $A \subseteq V$ and any convex balanced neighbourhood $W \subseteq Y$ of $0$, one needs to find $\delta > 0$ \st $r^f_x(\cdot, t) - r^f_x(\cdot, t_0) \in N(A, W)$ for all $t \in (0,1]$ with $|t - t_0| < \delta$ ($N(A,W)$ is as in \cref{Definition_Function_space_topologies}).
In other words, that $r^f_x(u, t) - r^f_x(u, t_0) \in W$ for all $u \in A$ and $t \in (0,1]$ with $|t - t_0| < \delta$.
But this follows immediately from the previous claim, which shows uniform continuity of $r^f_x|_{A\times (t_0/2,1]}$ and because $(u,t) - (u,t_0) = (0,t-t_0)$.
\end{proof}
``Fr{\'e}chet differentiability $\Rightarrow$ uniform continuity'': \\
It remains to show that the first claim above also holds for $\delta_0 = 0$.
So let again $W \subseteq Y$ be a convex balanced neighbourhood of $0$ and choose a convex balanced neighbourhood $W' \subseteq Y$ of $0$ \st $W' + W' \subseteq W$.
By the assumtion of (pointwise, compact) Fr{\'e}chet differentiability there exists $\delta' \in (0,1]$ \st $r^f_x(\cdot, t) \in N(A,W')$ for all $t \in [0,\delta']$ or in other words that $r^f_x(u, t) \subseteq W'$ for all $(u,t) \in A\times [0,\delta']$.
Observe that for all $t,t' \in [0,1]$ with $|t-t'| < \delta'/2$, either $t,t' \in [0,\delta']$ and/or $t,t' \in [\delta'/2,1]$.
Also, by the claim, $r^f_x|_{A\times [\delta'/2,1]} : A\times [\delta'/2,1] \to Y$ is uniformly continuous, so there exists $\delta'' > 0$ and a convex balanced neighbourhood $\tilde{V} \subseteq X$ of $0$ \st $r^f_x(u, t) - r^f_x(u', t') \in W$ for all $(u,t), (u',t') \in A\times [\delta'/2,1]$ with $(u-u',t-t') \in \tilde{V}\times (-\delta'',\delta'')$.
Let $\delta\definedas \min\{\delta'/2,\delta''\}$.
It follows that if $(u,t), (u',t') \in A\times [0,\delta']$, then $r^f_x(u, t) - r^f_x(u', t') \in W' + W' \subseteq W$.
And if $(u,t), (u',t') \in A\times [\delta'/2, 1]$ with $(u-u',t-t') \in \tilde{V}\times (-\delta,\delta)$, then $r^f_x(u, t) - r^f_x(u', t') \in W$.
This shows that $r^f_x|_{A\times [0,1]}$ is uniformly continuous.

``Uniform continuity $\Rightarrow$ Fr{\'e}chet differentiability'': \\
This part of the proof is identical to the proof of the second claim, where now one can use that $t \mapsto r^f_x(\cdot, t)$ is uniformly continuous for all $t\in [0,1]$ (and not just for $t \in [\delta_0,1]$, for some $\delta_0 \in (0,1]$).
\end{proof}

\begin{remark}\label{Remark_Assumption_uniform_continuity}
The condition that $f$ be locally uniformly continuous in \cref{Proposition_Characterisation_Differentiability_I}, \labelcref{Proposition_Characterisation_Differentiability_I_3}~might seem somewhat strong and arbitrary, considering that the definition of Fr{\'e}chet differentiability did not assume it.
But note (\eg by taking $u=u'$ in \cref{Proposition_Characterisation_Differentiability_I_Eq1}) that the proof of the first and hence also of the second claim in the proof of \cref{Proposition_Characterisation_Differentiability_I} does make use of this assumption.
So for maps that are not locally uniformly continuous, continuity of $[0,1] \to \mathcal{C}^{\mathrm{bi}}_{\mathrm{ao}}(V, Y)$, $t \mapsto r^f_x(\cdot,t)$, for all $t\in [0,1]$ is a priori stronger than just continuity at $t=0$.
\end{remark}

\begin{theorem}[Chain rule]\label{Theorem_Chain_rule_II}
Let $X$, $Y$ and $Z$ be Hausdorff locally convex topological vector spaces and let $U\subseteq X$ and $V \subseteq Y$ be open subsets.
Let furthermore $f : U \to Y$ and $g : V \to Z$ be continuous functions with $f(U) \subseteq V$.
In the following, where Fr{\'e}chet differentiability is concerned, assume furthermore that $f$ has locally bounded image and is locally uniformly continuous. \\
Given $x \in U$ and a topological space $\tilde{U}$ together with a continuous function $\iota : \tilde{U} \to U$, the following hold:
\begin{enumerate}[label=\arabic*.,ref=\arabic*.]
  \item If $f$ is (compactly) Fr{\'e}chet differentiable at $x$ and $g$ is (compactly) Fr{\'e}chet differentiable at $f(x)$, then the composition $g\circ f : U \to Z$ is (compactly) Fr{\'e}chet differentiable at $x$ with
\[
D(g\circ f)(x) = Dg(f(x))\circ Df(x)\text{.}
\]
  \item If $f$ is weakly continuously (compactly) Fr{\'e}chet differentiable along $\iota$ and $g$ is weakly continuously (compactly) Fr{\'e}chet differentiable along $f\circ \iota$, then $g \circ f$ is weakly continuously (compactly) Fr{\'e}chet differentiable along $\iota$. \\
If $Y$ is locally bounded, then the same holds for ``strongly continuously'' in place of ``weakly continuously''.
\end{enumerate}
\end{theorem}
\begin{proof}
Using \cref{Proposition_Characterisation_Differentiability_I}, the proof is identical to the proof of \cref{Theorem_Chain_rule_I}, just replacing the word ``continuous'' by the words ``uniformly continuous''.
\end{proof}

\Needspace{25\baselineskip}
\subsection{Interrelations}

\begin{proposition}\label{Proposition_Characterisation_Differentiability_II}
In the notation of \cref{Definition_Differentiability_I,Definition_Differentiability_II}, assume that the topology on $Y$ is given by a chosen translation invariant metric $d$.
\begin{enumerate}[label=\arabic*.,ref=\arabic*.]
  \item The following are equivalent:
\begin{enumerate}[label=(\alph*),ref=(\alph*)]
  \item $f$ is pointwise Fr{\'e}chet differentiable at $x$.
  \item There exists a convex balanced neighbourhood $V \subseteq X$ of $0$ \st for all $u \in V$
\[
\lim_{t \to 0} d\left(0,r^f_x(u, t)\right) = 0\text{.}
\]
  \item There exists a convex balanced neighbourhood $V \subseteq X$ of $0$ \st
\begin{multline*}
\forall\, u_0 \in V, \varepsilon > 0 \;\;\exists\, \delta = \delta(u_0,\varepsilon) > 0 \;: \\
d\left(0, r^f_x(u_0,t)\right) < \varepsilon \quad\forall\, t < \delta\text{.}
\end{multline*}
\end{enumerate}
  \item The following are equivalent:
\begin{enumerate}[label=(\alph*),ref=(\alph*)]
  \item $f$ is compactly Fr{\'e}chet differentiable at $x$.
  \item There exists a convex balanced neighbourhood $V \subseteq X$ of $0$ \st for all compact subsets $K \subseteq V$
\[
\lim_{t \to 0} \sup\left\{ \left.d\left(0, r^f_x(u, t)\right) \;\right|\; u \in K \right\} = 0\text{.}
\]
  \item There exists a convex balanced neighbourhood $V \subseteq X$ of $0$ \st
\begin{multline*}
\forall\, K \subseteq V \text{ compact}, \varepsilon > 0 \;\;\exists\, \delta = \delta(K,\varepsilon) > 0 \;: \\
d\left(0,r^f_x(u,t)\right) < \varepsilon \quad\forall\, u \in K, t < \delta\text{.}
\end{multline*}
\end{enumerate}
  \item The following are equivalent:
\begin{enumerate}[label=(\alph*),ref=(\alph*)]
  \item $f$ is weakly Fr{\'e}chet differentiable at $x$
  \item There exists a convex balanced neighbourhood $V \subseteq X$ of $0$ \st
\begin{multline*}
\forall\, u_0 \in V, \varepsilon > 0 \;\;\exists\, \delta = \delta(u_0,\varepsilon) > 0, V'\subseteq V \text{ an open neighbourhood of $u_0$} \;: \\
d\left(0,r^f_x(u,t)\right) < \varepsilon \quad\forall\, u \in V', t < \delta\text{.}
\end{multline*}
\end{enumerate}
  \item Assume that $f$ has locally bounded image and is locally uniformly continuous.
Then the following are equivalent:
\begin{enumerate}[label=(\alph*),ref=(\alph*)]
  \item $f$ is Fr{\'e}chet differentiable at $x$.
  \item There exists a convex balanced neighbourhood $V \subseteq X$ of $0$ \st
\[
\lim_{t \to 0} \sup\left\{ \left.d\left(0,r^f_x(u, t)\right) \;\right|\; u \in V \right\}= 0\text{.}
\]
  \item There exists a convex balanced neighbourhood $V \subseteq X$ of $0$ \st
\begin{multline*}
\forall\, \varepsilon > 0 \;\;\exists\, \delta = \delta(\varepsilon) > 0 \;:\;
d\left(0,r^f_x(u,t)\right) < \varepsilon \;\forall\, u \in V, t < \delta\text{.}
\end{multline*}
\end{enumerate}
\end{enumerate}
\end{proposition}
\begin{proof}
All of the statements in the proposition are obtained by simply writing out the definitions in terms of the metric on $Y$.
\end{proof}

\begin{remark}
For the following, remember that an embedding of topological vector spaces is an injective continuous linear map.
It does not need to be an embedding of topological spaces or have closed image. \\
For the definition of a compact linear operator, see \cref{Definition_Compact_operator}.
\end{remark}

\begin{proposition}\label{Lemma_Strong_Gateaux_implies_Frechet}
Let $X$ and $Y$ be Hausdorff locally convex topological vector spaces, let $A \subseteq U\subseteq X$ be subsets with $U$ open and let $f : U \to Y$ be a continuous map.
Let furthermore $X_1$ be another Hausdorff locally convex topological vector space that comes with a compact embedding of topological vector spaces $\kappa : X_1 \hookrightarrow X$ and define $U_1 \definedas \kappa\inv(U)$, $A_1 \definedas \kappa\inv(A)$. \\
If $f$ is compactly Fr{\'e}chet differentiable on $A$, then $f_1 \definedas f \circ \kappa|_{U_1} : U_1 \to Y$ is Fr{\'e}chet differentiable on $A_1$ with
\[
Df_1(x) = Df(\kappa(x))\circ \kappa\quad \forall\, x \in A_1\text{.}
\]
\end{proposition}
\begin{proof}
For simplicity I will drop $\kappa$ from the notation and consider $A_1, U_1, X_1$ as subsets of $X$, although not with the induced topologies.
For $x \in U_1$, by compactness of $\kappa$ there exists a neighbourhood $U'_1 \subseteq U_1$ of $x$ \st $K \definedas \overline{U}'_1 \subseteq U \subseteq X$ is compact.
By \cref{Lemma_Continuity_implies_uniform_continuity}, $f|_K$ is uniformly continuous and it has compact, hence bounded, image.
Since $\kappa$, as a continuous linear map, is uniformly continuous, this shows that $f_1 = f\circ \kappa|_{U_1}$ is locally uniformly continuous and has locally bounded image. \\
It is clear from the definition that for all $x \in A_1$ and some convex balanced neighbourhood $V_1 \subseteq X_1$ of $0$
\[
r^{f_1}_x = r^f_x|_{V_1\times [0,1]}\text{.}
\]
Denote by $K$ the closure of $V_1$ in $X$, which can be assumed to be compact by compactness of $\kappa$, and $K \subseteq V$ for $V_1$ small enough.
By assumption, the function $r^f_x|_{K\times [0,1]} : K \times [0,1] \to Y$ is continuous and hence by \cref{Lemma_Continuity_implies_uniform_continuity} it is uniformly continuous.
It follows that $r^{f_1}_x$ is uniformly continuous and so by \cref{Proposition_Characterisation_Differentiability_I}, $f_1$ is Fr{\'e}chet differentiable.
\end{proof}

\begin{theorem}\label{Theorem_Relations_between_notions_of_differentiability}
Let $X$ and $Y$ be locally convex topological vector spaces, let $U \subseteq X$ be an open subset, $x \in U$, and let $f : U \to Y$ be a continuous map.
\begin{enumerate}[label=\arabic*.,ref=\arabic*.]
  \item 
\begin{align*}
\text{$f$ weakly Fr{\'e}chet differentiable at $x$} &\Rightarrow \\
&\hspace{-1.2cm}\Rightarrow \text{$f$ compactly weakly Fr{\'e}chet differentiable at $x$} \\
&\hspace{-1.2cm}\Rightarrow \text{$f$ pointwise weakly Fr{\'e}chet differentiable at $x$.}
\end{align*}
  \item If the topology on $X$ is compactly generated, in particular if $X$ is first countable (equivalently, if $X$ is metrisable), then
\begin{multline*}
\text{$f$ compactly weakly Fr{\'e}chet differentiable at $x$} \Leftrightarrow \\
\Leftrightarrow \text{$f$ weakly Fr{\'e}chet differentiable at $x$.}
\end{multline*}
  \item Each of the other implications ``$f$ pointwise weakly Fr{\'e}chet differentiable at $x$ $\Rightarrow$ $f$ (compactly) weakly Fr{\'e}chet differentiable at $x$'' in general is false.
  \item $f$ pointwise weakly Fr{\'e}chet differentiable at $x$ $\Leftrightarrow$ $f$ pointwise Fr{\'e}chet differentiable at $x$.
  \item $f$ compactly weakly Fr{\'e}chet differentiable at $x$ $\Leftrightarrow$ $f$ compactly Fr{\'e}chet differentiable at $x$.
  \item Assume that $f$ has locally bounded image.
Then \\
$f$ Fr{\'e}chet differentiable at $x$ $\Rightarrow$ $f$ weakly Fr{\'e}chet differentiable at $x$, but the converse implication in general is false.
  \item If $X$ is finite dimensional, then $f$ Fr{\'e}chet differentiable at $x$ $\Leftrightarrow$ $f$ weakly Fr{\'e}chet differentiable at $x$.
\end{enumerate}
\end{theorem}
\begin{proof}
\begin{enumerate}[label=\arabic*.,ref=\arabic*.]
  \item This is completely obvious.
  \item $X$ metrisable and $X$ first countable are equivalent by \cite{MR3154940}, Theorem 3.35, and first countable spaces are compactly generated, \cf \cite{MR0464128}, Lemma 46.3.
By \cite{MR0464128}, Lemma 46.4, it remains to show that if, in the notation of \cref{Definition_Differentiability_I}, $\tilde{K} \subseteq V\times [0,1]$ is compact, then $r^f_x|_{\tilde{K}} : \tilde{K} \to Y$ is continuous.
Denoting $K \definedas \pr_1(\tilde{K}) \subseteq V$, which is compact, by definition, $r^f_x|_{K\times [0,1]} : K\times [0,1] \to Y$ is continuous.
Since $\tilde{K} \subseteq K\times [0,1]$, the claim follows.
  \item This was already shown by way of \cref{Example_Gateaux_but_not_weakly_Frechet}.
  \item This is immediate from \cref{Proposition_Characterisation_Differentiability_I} and \cref{Lemma_Continuity_implies_uniform_continuity}.
  \item Ditto.
  \item Let $V \subseteq X$ be as in the definition of $f$ Fr{\'e}chet differentiable.
Then continuity of $r^f_x : V\times [0,1] \to Y$ at points $(u,t) \in V\times (0,1]$ is automatic by virtue of its definition (and because $f$ is assumed to be continuous).
And to show continuity at $(u,0)$ for some $u \in V$, let $W \subseteq Y$ be a convex balanced neighbourhood of $0$.
By definition of the topology on $\mathcal{C}^{\mathrm{bi}}(V, Y)$, there exists $\delta > 0$ \st $r^f_x(\cdot, t') - r^f_x(\cdot, 0) = r^f_x(\cdot, t') \in N(V,W)$ for all $t'\in[0,1]$ with $|t' - 0| = |t'| < \delta$.
But this implies in particular that for all $(u',t')\in V\times [0,1]$ with $(u',t') - (u,0) = (u'-u, t') \in X\times [0,\delta]$, $r^f_x(u', t') - r^f_x(u, 0) = r^f_{x'}(u', t') \in W$.
So $r^f_x$ is continuous at $(u,t)$. \\
That the converse implication in general is false will be shown by virtue of an example in \cref{Subsection_Reparametrisation_action}.
  \item Apply \cref{Lemma_Strong_Gateaux_implies_Frechet} with $X_1 \definedas X$, $\kappa \definedas \id_X$, which is compact precisely in finite dimensions to show ``compactly weakly Fr{\'e}chet'' $\Rightarrow$ ``Fr{\'e}chet''.
\end{enumerate}
\end{proof}

\begin{proposition}\label{Lemma_Weakly_continuous_implies_strongly_continuous}
Let $X$, $Y_0$, $Y_1$ and $Z$ be Hausdorff locally convex topological vector spaces, let $\kappa : Y_1 \hookrightarrow Y_0$ be a compact embedding of topological vector spaces, let $U \subseteq X$ be an open subset and let $\Phi : U\times Y_0 \to Z$ be a continuous map that is linear in the second factor.
I.\,e.~for all $(x,v)\in U\times Y_0$, $\Phi(x,v) = \Phi_x(v)$, where $\Phi_x \in L_{\mathrm{c}}(Y_0, Z)$. \\
Then the map
\begin{align*}
U &\to L_{\mathrm{c}}(Y_1, Z) \\
x &\mapsto \Phi_x\circ \kappa
\end{align*}
is continuous.
\end{proposition}
\begin{proof}
Let $A$ be a bounded subset of $Y_1$ and let $V$ be a convex balanced neighbourhood of $0$ in $Z$.
Also let $N(A,V) = \{L \in L_{\mathrm{c}}(Y_1,Z) \;|\; L(A) \subseteq V\} \subseteq L_{\mathrm{c}}(Y_1,Z)$ be the corresponding neighbourhood of $0 \in L_{\mathrm{c}}(Y_1, Z)$.
It has to be shown that there exists a neighbourhood $W \subseteq U$ of $x_0$ \st $\Phi_x\circ \kappa - \Phi_{x_0}\circ \kappa \in N(A,V)$, \ie $(\Phi_x - \Phi_{x_0})(\kappa(A)) \subseteq V$, for all $x \in W$. \\
Let $K$ be the closure of $\kappa(A)$ in $Y_0$, which by assumption is compact and let $x_0 \in U$.
By assumption, the map $\Phi : U\times Y_0 \to Z$ is continuous and hence so is the map
\begin{align*}
\phi : U\times K &\to Z \\
(x,v) &\mapsto \Phi_x(v) - \Phi_{x_0}(v)\text{.}
\end{align*}
It then obviously suffices to find a neighbourhood $W \subseteq U$ of $x_0$ \st $W\times K \subseteq \phi\inv(V)$.
But $\phi(\{x_0\}\times K) = \{0\} \subseteq V$, so by \cref{Lemma_Topological_Lemma} the existence of such a neighbourhood $W \subseteq U$ of $x_0$ follows.
\end{proof}

\begin{proposition}\label{Lemma_Strong_Gateaux_implies_Frechet_II}
Let $X$ and $Y$ be Hausdorff locally convex topological vector spaces. \\
Let furthermore $X_1$ be another Hausdorff locally convex topological vector space that comes with a compact embedding of topological vector spaces $\kappa : X_1 \hookrightarrow X$ and define $U_1 \definedas \kappa\inv(U)$. \\
If $\tilde{U}$ is a topological space together with a continuous map $\iota_1 : \tilde{U} \to U_1$ \st $f$ is weakly continuously compactly weakly Fr{\'e}chet differentiable along $\iota \definedas \kappa\circ\iota_1$, then $f_1 \circ \kappa|_{U_1} : U_1 \to Y$ is strongly continuously Fr{\'e}chet differentiable along $\iota_1$.
\end{proposition}
\begin{proof}
For simplicity I will drop $\kappa$ from the notation and consider $U_1, X_1$ as subsets of $X$, although not with the induced topologies. \\
By \cref{Lemma_Strong_Gateaux_implies_Frechet}, $f_1$ is Fr{\'e}chet differentiable on $\iota_1(\tilde{U})$ and $Df_1(x) = Df(\kappa(x))\circ \kappa$. \\
By assumption, the map
\begin{align*}
\phi : \tilde{U}\times X &\to Y \\
(x,u) &\mapsto Df(\iota(x))u
\end{align*}
is continuous.
\cref{Lemma_Weakly_continuous_implies_strongly_continuous} then implies that the map
\begin{align*}
\tilde{U} &\to L_{\mathrm{c}}(X_1, Y) \\
x &\mapsto Df(\iota(x))\circ \kappa = Df_1(\iota_1(x))
\end{align*}
is continuous, \ie $f_1$ is strongly continuously Fr{\'e}chet differentiable along $\iota_1$.
\end{proof}

\Needspace{25\baselineskip}
\subsection{Main example: The reparametrisation action}\label{Subsection_Reparametrisation_action}

Let $(\Sigma,g)$ be a closed $n$-dimensional Riemannian manifold and let $\pi : F \to \Sigma$, $\pi_1 : F_1 \to \Sigma$ and $\pi_2 : F_2 \to \Sigma$ be real (or complex) vector bundles equipped with euclidean (or hermitian) metrics and metric connections.
Let furthermore $B \subseteq \R^r$ (for some $r\in\N_0$) be an open subset and let
\begin{align*}
\phi : B \times \Sigma &\to B\times \Sigma \\
(b, z) &\mapsto (b, \phi_b(z))
\intertext{and}
\Phi : B \times F_1 &\to B\times F_2 \\
(b, \xi) &\mapsto (b, \Phi_b(\xi))
\end{align*}
be smooth families of maps, where $\phi$ is a family of diffeomorphisms and $\Phi$ covers $\phi\inv$.
I.\,e.~$\phi$ and $\Phi$ are smooth, $\phi_b \in \operatorname{Diff}(\Sigma)$ and $\Phi_b \in \mathcal{C}^\infty(F_1, F_2)$ for all $b \in B$, and
\[
\xymatrix{
B\times F_2 \ar@{}[rd]|-{\circlearrowleft} \ar[d]_-{\id_B\times \pi} & B\times F_1 \ar[d]^-{\id_B\times \pi} \ar[l]_-{\Phi} \\
B\times \Sigma \ar[r]^-{\phi} & B\times \Sigma
}
\]
commutes.
Furthermore, assume that $\Phi$ is linear in the fibres of $F_1$ and $F_2$, \ie for every $b\in B$, $\Phi_b : F_1 \to F_2$ defines a vector bundle morphism covering $\phi\inv_b : \Sigma \to \Sigma$.

For $k\in \N_0$, let $\Gamma^k(F)$ be the space of $k$-times continuously differentiable sections of $F$ equipped with the $\mathcal{C}^k$-topology.
It can be defined for example via the (complete) norm
\begin{align*}
\|\cdot\|_k : \Gamma^k(F) &\to [0,\infty) \\
u &\mapsto \sum_{j=0}^k \sup\{|\nabla^j u(z)| \;|\; z \in \Sigma\}\text{.}
\end{align*}
Also, denote by $\Gamma(F)$ the (Fr{\'e}chet) space of smooth sections equipped with the inverse limit of the $\mathcal{C}^k$-topologies, for $k\in\N_0$. \\
Finally, for $k\in\N_0$ and $1 < p < \infty$, let $W^{k,p}(F)$ be the Sobolev space of sections of $F$ of class $(k,p)$.
For concreteness' sake, let it be defined as the completion of $\Gamma(E)$ \wrt the norm
\begin{align*}
\|\cdot\|_{k,p} : \Gamma(F) &\to [0,\infty) \\
u &\mapsto \sum_{j=0}^k \left(\int_\Sigma |\nabla^j u|^p\, \dvol_g\right)^{1/p}\text{.}
\end{align*}
Also assume that $kp > \dim_\R \Sigma$, so that the Sobolev embedding and multiplication theorems hold, in particular one can identify $W^{k,p}(F)$ with a subset of $\Gamma^0(F)$ and consider the elements of $W^{k,p}(F)$ as continuous sections of $F$. \\
Define
\begin{align*}
\Gamma_B(F) &\definedas B\times \Gamma(F)\text{,} & \Gamma^k_B(F) &\definedas B\times \Gamma^k(F) \quad\text{and} & W^{k,p}_B(F) &\definedas B\times W^{k,p}(F)
\intertext{and set for $b\in B$}
\Gamma_b(F) &\definedas \{b\}\times \Gamma(F)\text{,} & \Gamma^k_b(F) &\definedas \{b\}\times \Gamma^k(F) \quad\text{and} & W^{k,p}_b(F) &\definedas \{b\}\times W^{k,p}(F)\text{.}
\end{align*}
All of these are equipped with the product topologies and one has vector bundles (via projection onto the first factor)
\begin{align*}
\rho : \Gamma_B(F) &\to B\text{,} & \rho^k : \Gamma^k_B(F) &\to B\quad\text{and} & \rho^{k,p} : W^{k,p}_B(F) &\to B
\end{align*}
Note that $\Gamma_B(F)$, $\Gamma^k_B(F)$ and $W^{k,p}_B(F)$ are open subsets of the Hausdorff locally convex topological vector spaces $\Gamma_{\R^r}(F)$, $\Gamma^k_{\R^r}(F)$ and $W^{k,p}_{\R^r}(F)$, respectively.

There then are vector bundle morphisms
\begin{align*}
\Psi : \Gamma_B(F_1) &\to \Gamma_B(F_2)\text{,} & \Psi^k : \Gamma^k_B(F_1) &\to \Gamma^k_B(F_2)\quad\text{and} & \Psi^{k,p} : W^{k,p}_B(F_1) &\to W^{k,p}_B(F_2)\text{,}
\intertext{defined in all three cases by}
& & (b,u) &\mapsto (b,\Phi_b^\ast u)\text{,} & &
\intertext{where}
& & \Phi_b^\ast u &\definedas \Phi_b\circ u\circ \phi_b\text{.}\hspace{-4cm} & &
\end{align*}

\begin{proposition}\label{Proposition_Reparametrisation_action_cts}
In the above notation,
\begin{enumerate}[label=\arabic*.,ref=\arabic*.]
  \item\label{Proposition_Reparametrisation_action_cts_1} $\Psi^k : \Gamma^k_B(F_1) \to \Gamma^k_B(F_2)$ is continuous and has locally bounded image, but in general is not locally uniformly continuous.
  \item\label{Proposition_Reparametrisation_action_cts_2} $\Psi^{k,p} : W^{k,p}_B(F_1) \to W^{k,p}_B(F_2)$ is continuous and has locally bounded image, but in general is not locally uniformly continuous.
  \item\label{Proposition_Reparametrisation_action_cts_3} $\Psi : \Gamma_B(F_1) \to \Gamma_B(F_2)$ is continuous and locally bounded, but in general has neither locally bounded image nor is it locally uniformly continuous.
\end{enumerate}
\end{proposition}
\begin{proof}
\begin{claim}
There exist continuous functions $C^k, C^{k,p} : B \to (0,\infty)$ \st
\begin{align*}
\|\Phi_b^\ast u\|_k &\leq C^k(b)\|u\|_k &&\forall\, u \in \Gamma^k(F_1)
\intertext{and}
\|\Phi_b^\ast u\|_{k,p} &\leq C^{k,p}(b)\|u\|_{k,p} &&\forall\, u \in W^{k,p}(F_1)\text{.}
\end{align*}
If for every $b \in B$, $\Phi_b$ is a diffeomorphism (\ie a vector bundle isomorphism covering $\phi_b\inv$), then there exist continuous functions $c^k, c^{k,p}, C^k, C^{k,p} : B \to (0,\infty)$ \st
\begin{align*}
c^k(b)\|u\|_k &\leq \|\Phi_b^\ast u\|_k \leq C^k(b)\|u\|_k &&\forall\, u \in \Gamma^k(F_1)
\intertext{and}
c^{k,p}(b)\|u\|_{k,p} &\leq \|\Phi_b^\ast u\|_{k,p} \leq C^{k,p}(b)\|u\|_{k,p} &&\forall\, u \in W^{k,p}(F_1)\text{.}
\end{align*}
\end{claim}
\begin{proof}
For a fixed $b \in B$ and $u\in\Gamma^k(F_1)$, a simple calculation gives $\|\Phi_b^\ast u\|_k \leq C^k(b)\|u\|_k$, where $C^k(b)$ is a universal expression in the terms $\max\{|\nabla^j \Phi_b(z)| \;|\; z \in \Sigma\}$, for $j = 0, \dots, k$.
By \cref{Lemma_Topological_Lemma}, this depends continuously on $b\in B$. \\
For the other inequality, apply the same reasoning to $\Phi_b\inv$. \\
The case $b \in B$ and $u \in W^{k,p}(F_1)$ is identical.
\end{proof}
\begin{enumerate}[label=\arabic*.,ref=\arabic*.]
  \item\label{Proposition_Reparametrisation_action_cts_Proof_1} The claim that $\Psi^k$ is in general not locally uniformly continuous will be shown by way of \cref{Example_Counterexample_reparametrisation_action} below. \\
If $K \subseteq B$ is a compact subset, let $C^k(K) \definedas \max\{C^k(b) \;|\; b \in K\} \in (0,\infty)$.
By the above claim, for any $\rho > 0$,
\[
\Psi^k(K\times B^{\Gamma^k(F_1)}(0,\rho)) \subseteq K \times B^{\Gamma^k(F_2)}(0,C^k(K)\rho)\text{,}
\]
which is bounded.
Since every point in $\Gamma^k_B(F_1)$ has a neighbourhood of the form $K\times B^{\Gamma^k(F_1)}(0,\rho)$ for some compact subset $K\subseteq B$ and some $\rho > 0$, this shows that $\Psi^k$ has locally bounded image. \\
For the proof of continuity of $\Psi^k$ I will follow \cite{MR2644764}, esp.~Section 2.2. \\
So let $(b_0,u_0) \in \Gamma^k_B(F_1)$ and let $\varepsilon > 0$.
For any $(b,u) \in \Gamma^k_B(F_1)$,
\begin{align*}
\|\Psi^k(b,u) - \Psi^k(b_0,u_0)\| &= \|(b-b_0, \Phi_b^\ast u - \Phi_{b_0}^\ast u_0)\| \\
&= |b-b_0| + \|\Phi_b^\ast u - \Phi_{b_0}^\ast u_0\|_k
\end{align*}
and for any $u'_0 \in \Gamma(F)$,
\begin{align*}
\|\Phi_b^\ast u - \Phi_{b_0}^\ast u_0\|_k &\leq \|\Phi_b^\ast u - \Phi_b^\ast u'_0\|_k + \|\Phi_b^\ast u'_0 - \Phi_{b_0}^\ast u'_0\|_k + \|\Phi_{b_0}^\ast u'_0 - \Phi_{b_0}^\ast u_0\|_k \\
&= \|\Phi_b^\ast (u - u'_0)\|_k + \|\Phi_b^\ast u'_0 - \Phi_{b_0}^\ast u'_0\|_k + \|\Phi_{b_0}^\ast (u'_0 - u_0)\|_k \\
&\leq C^k(b)\underbrace{\|u - u'_0\|_k}_{\mathclap{\leq\; \|u - u_0\|_k + \|u_0 - u'_0\|_k}} + C^k(b_0)\|u'_0 - u_0\|_k + \|\Phi_b^\ast u'_0 - \Phi_{b_0}^\ast u'_0\|_k \\
&\leq (C^k(b) + C^k(b_0))(\|u - u_0\|_k + \|u_0 - u'_0\|_k) + \|\Phi_b^\ast u'_0 - \Phi_{b_0}^\ast u'_0\|_k\text{.}
\end{align*}
Now let any $\varepsilon > 0$ be given.
Let $K \subseteq B$ be any compact convex neighbourhood of $b_0$ and let $\delta' \definedas \frac{\varepsilon}{12C^k(K)}$.
Pick $u'_0 \in \Gamma(F)$ \st $\|u_0 - u'_0\|_k < \delta'$.
Then for any $u \in \Gamma^k(F)$ with $\|u - u_0\|_k < \delta'$ and any $b\in K$,
\begin{align*}
\|\Phi_b^\ast u - \Phi_{b_0}^\ast u_0\|_k &\leq 2C^k(K)(2\delta') + \|\Phi_b^\ast u'_0 - \Phi_{b_0}^\ast u'_0\|_k \\
&\leq \varepsilon/3 + \|\Phi_b^\ast u'_0 - \Phi_{b_0}^\ast u'_0\|_k\text{.}
\end{align*}
Now
\begin{align*}
\|\Phi_b^\ast u'_0 - \Phi_{b_0}^\ast u'_0\|_k &= \left\|\int_0^1 \frac{\d}{\d t} \Phi_{b_0 + t(b-b_0)}^\ast u'_0\,\d t\right\|_k \\
&\leq \int_0^1 \left\|\frac{\d}{\d t} \Phi_{b_0 + t(b-b_0)}^\ast u'_0\right\|_k \,\d t \\
&\leq D^k\|u'_0\|_{k+1}(b-b_0)\text{,}
\end{align*}
where $D^k > 0$ is a universal expression in the supremum over the $\mathcal{C}^{k+1}$-norms of $\Phi_b$, for $b \in K$. \\
Choose $\delta'' > 0$ so small that $B^B(b_0, \delta'') \subseteq K$ and set
\[
\delta \definedas \min\left\{\delta'', \frac{\varepsilon}{6 D^k\|u'_0\|_{k+1}}, \frac{\varepsilon}{2}, \delta'\right\}\text{.}
\]
The above then shows that for all $(b,u) \in \Gamma^k_B(F_1)$ with $\|(b,u) - (b_0,u_0)\| < \delta$, $\|\Psi^k(b,u) - \Psi^k(b_0,u_0)\| < \varepsilon$.
  \item\label{Proposition_Reparametrisation_action_cts_Proof_2} Completely analogous to \labelcref{Proposition_Reparametrisation_action_cts_Proof_1}.
  \item\label{Proposition_Reparametrisation_action_cts_Proof_3} That $\Psi$ is continuous follows immediately from \labelcref{Proposition_Reparametrisation_action_cts_1}~because for $i = 1,2$, $\Gamma_B(F_i)$ is the inverse limit of the spaces $\Gamma^k_B(F_i)$ and $\Psi$ is the inverse limit of the maps $\Psi^k$. \\
That $\Psi$ in general is not locally uniformly continuous is shown in \cref{Example_Counterexample_reparametrisation_action} below. \\
To show that $\Psi$ in general does not have locally bounded image, assume that $F_1 = F_2 = F$ and that $\Psi_b$ is a diffeomorphism for all $b \in B$, \eg as in \cref{Example_Counterexample_reparametrisation_action} below.
Let $U \subseteq \Gamma_B(F)$ be any neighbourhood of $(b_0,0)$ for some $b_0 \in B$.
Then $U$ contains a subset of the form $K\times (B^{\Gamma^k(F)}(0,\rho)\cap \Gamma(F))$ for a compact neighbourhood $K \subseteq B$ of $b_0$ and some $\rho > 0$.
The claim above then shows that $\Psi(U)$ contains a subset of the form $K\times (B^{\Gamma^k(F)}(0,c^k(K)\rho)\cap \Gamma(F))$, which is not bounded in $\Gamma_B(F)$, \cf \cref{Remark_Bounded_maps}, \labelcref{Remark_Bounded_maps_3}. \\
To show that $\Psi$ is locally bounded, let $K \subseteq B$ be any compact subset.
Then $\Psi|_{\Gamma_K(F)} : \Gamma_K(F) \to \Gamma_K(F)$ is bounded: \\
For if $C \subseteq \Gamma(F)$ is a bounded subset, then $K\times C \subseteq \Gamma_K(F)$ is bounded and if $A \subseteq \Gamma_K(F)$ is a bounded subset, then $C \definedas \pr_2(A) \subseteq \Gamma(F)$ is bounded (for $\pr_2 : \R^r\times \Gamma(F) \to \Gamma(F)$ is a continuous linear map) and $A \subseteq K\times C$.
It hence suffices to show that for any bounded subset $C \subseteq \Gamma(F)$, $\Psi(K\times C) \subseteq \R^r\times \Gamma(F)$ is bounded.
Now $\Psi(K\times C) = \bigcup_{b \in K} \bigcup_{u\in C} (b, \Phi_b^\ast u) \subseteq K\times \bigcup_{b\in K} \Phi_b^\ast C$.
It hence suffices to show that $\bigcup_{b\in K} \Phi_b^\ast C \subseteq \Gamma(F)$ is bounded for any bounded subset $C \subseteq \Gamma(F)$, \ie it has to be shown that for any neighbourhood $U\subseteq \Gamma(F)$ of $0$ there exists $c \in \mathds{k}$ \st $\bigcup_{b\in K} \Phi_b^\ast C \subseteq cU$.
So let such a neighbourhood $U\subseteq \Gamma(F)$ be given.
As has been used several times before, there then exists a $k \in \N_0$ and a $\rho > 0$ \st $B^{\Gamma^k(F)}(0, \rho) \cap \Gamma(F) \subseteq U$.
Because $C$ is bounded, there exists $c' \in \mathds{k}$ \st $C \subseteq c'(B^{\Gamma^k(F)}(0, \rho) \cap \Gamma(F)) = B^{\Gamma^k(F)}(0, |c'|\rho) \cap \Gamma(F)$.
By the claim above, for $b \in B$, then $\Phi_b^\ast C \subseteq B^{\Gamma^k(F)}(0, C^k(b)|c'|\rho) \cap \Gamma(F)$.
Set $c \definedas C^k(K)|c'|$, where $C^k(K)$ is as in \labelcref{Proposition_Reparametrisation_action_cts_Proof_1}~above.
Then $\bigcup_{b\in K} \Phi_b^\ast C \subseteq B^{\Gamma^k(F)}(0, C^k(K)|c'|\rho) \cap \Gamma(F) = c B^{\Gamma^k(F)}(0, \rho) \cap \Gamma(F) \subseteq cU$.
\end{enumerate}
\end{proof}

\begin{example}\label{Example_Counterexample_reparametrisation_action}
Let $B \definedas \R$, $\Sigma \definedas S^1$ with the standard metric, $F \definedas \Sigma\times \R$ with the trivial connection and standard fibre metric and let
\begin{align*}
\phi : B\times \Sigma &\to B\times \Sigma \\
(s, e^{it}) &\mapsto (s, e^{i(t+s)}) \\
\Phi : B\times F &\to B\times F \\
(s, (e^{it}, \xi)) &\mapsto (s, (e^{i(t-s)}, \xi))\text{.}
\end{align*}
It will now be shown that the corresponding maps $\Psi$, $\Psi^k$ and $\Psi^{k,p}$ as above are not locally uniformly continuous.
To do so, in the following I will construct, for any $k\in\N_0$ and $1 < p < \infty$, sequences of smooth sections $(u^k_n)_{n\in\N} \subseteq \Gamma^k(F)\cap \Gamma(F)$ and $(u^{k,p}_n)_{n\in\N} \subseteq W^{k,p}(F)\cap \Gamma(F)$ \st the following hold:
\begin{enumerate}[label=\arabic*.,ref=\arabic*.]
  \item There exist constants $c_k > 0$ \st $\|u^k_n\|_k \leq c_k$ and $\|u^{k,p}_n\|_{k,p} \leq c_k$ for all $n\in\N$.
  \item For any $t > 0$,
\begin{align*}
\liminf_{n\to \infty} \|\Phi_t^\ast u^k_n - u^k_n\|_k &\geq 2
\intertext{and}
\liminf_{n\to \infty} \|\Phi_t^\ast u^{k,p}_n - u^{k,p}_n\|_{k,p} &\geq 2\text{.}
\end{align*}
\end{enumerate}
From this it follows immediately that $\Psi$, $\Psi^k$ and $\Psi^{k,p}$ are not uniformly continuous in any neighbourhood of $(0,0) \in \Gamma_B(F)$: \\
First, restrict to $\Psi^k$, for the proof for $\Psi^{k,p}$ is the same almost ad verbatim, just replacing $u^k_n$ by $u^{k,p}_n$, etc. \\
Second, given any neighbourhood $U\subseteq \Gamma^k_B(F)$ of $(0,0)$, there exists $\delta > 0$ \st $(-\delta,\delta)\times B^{\Gamma^k(F)}(0, \delta) \subseteq U$, where $B^{\Gamma^k(F)}(0,\delta)$ is the ball of radius $\delta$ around $0$ in $\Gamma^k(F)$.
Then the sequence of sections $(\tilde{u}^k_n)_{n\in\N} \subseteq \Gamma^k(F)$, $\tilde{u}^k_n \definedas \frac{\delta}{2c_k}u^k_n$ is contained in $B^{\Gamma^k(F)}(0,\delta)$ any hence $(t,\tilde{u}^k_n) \in U$ for all $t \in (-\delta, \delta)$.
By the above, for any $t \in (-\delta, \delta)$ with $t \neq 0$, there exists $n_0(t) \in \N$ \st
\begin{align*}
\|\Psi^k(t, \tilde{u}^k_n) - \Psi^k(0, \tilde{u}^k_n)\| &= |t| + \|\Phi_t^\ast \tilde{u}^k_n - \tilde{u}^k_n\|_k \\
&= |t| + \frac{\delta}{2c_k}\|\Phi_t^\ast u^k_n - u^k_n\|_k \\
&\geq \frac{\delta}{c_k}
\end{align*}
for all $n \geq n_0(t)$.
On the other hand, for any neighbourhood $V \subseteq \Gamma^k_B(F)$ of $0$, $(t,\tilde{u}^k_n) - (0,\tilde{u}^k_n) = (t,0) \in V$ for all $t \neq 0$ small enough, independent of $\tilde{u}^k_n$.
Now let $W \definedas B^{\Gamma^k_B(F)}(0,\frac{\delta}{2c_k}) \subseteq \Gamma^k_B(F)$.
Then by the above, for any neighbourhood $V \subseteq \Gamma^k_B(F)$ of $0$ we can find points $x,y \in U$ of the form $x = (t, \tilde{u}^k_n)$ and $y = (0, \tilde{u}^k_n)$ for some $t\neq 0$ small enough and some $n\in \N$ large enough with $x-y \in V$, but $\Psi^k(x) - \Psi^k(y) \not\in W$, which contradicts uniform continuity. \\
Third, to show that $\Psi$ is not uniformly continuous in any neighbourhood of $(0,0)$, observe that if $U \subseteq \Gamma_B(F)$ is any neighbourhood of $(0,0)$, because $\Gamma_B(F)$ as a topological space is the inverse limit of the spaces $\Gamma^k_B(F)$, there exists $k \in \N_0$ and $\delta > 0$ \st $(-\delta, \delta)\times (B^{\Gamma^k(F)}(0,\delta) \cap \Gamma(F)) \subseteq U$.
Then the points $(t, \tilde{u}^k_n)$, where $\tilde{u}^k_n$ is defined exactly as before, are contained in $U$.
Define $W \definedas B^{\Gamma^k_B(F)}(0, \frac{\delta}{2c_k}) \cap \Gamma(F)$.
Then by the same calculation as before, for any neighbourhood $V \subseteq \Gamma_B(F)$, $(t, \tilde{u}^k_n) - (0,\tilde{u}^k_n) = (t,0) \in V$, for $t \neq 0$ small enough, but $\Psi(t, \tilde{u}^k_n) - \Psi(0, \tilde{u}^k_n) \not\in W$ for $n \in \N$ large enough.

It remains to show the existence of the sections $u^k_n, u^{k,p}_n$ as above. \\
It follows from a standard exercise done in Calculus I courses that for every $1 < p < \infty$ and $n \in \N$ there exist smooth functions $f^{0}_n, f^{0,p}_n : \R \to \R$ \st
\begin{align*}
\chi_{[1/2 - 1/4n, 1/2 + 1/4n]}(s) &\leq f^{0}_n(s) \leq 2 \chi_{[1/2 - 1/2n, 1/2 + 1/2n]}(s)\text{,} \\
(2n)^{1/p} \chi_{[1/2 - 1/4n, 1/2 + 1/4n]}(s) &\leq f^{0,p}_n(s) \leq 2n^{1/p} \chi_{[1/2 - 1/2n, 1/2 + 1/2n]}(s)\text{,}
\end{align*}
where $\chi_A : \R \to \R$ denotes the characteristic function of a subset $A \subseteq \R$.
Define inductively for $k \in \N_0$, $f^{k+1}_n(s) \definedas \int_{0}^s f^{k}_n(r)\,\d r$ and $f^{k+1,p}_n(s) \definedas \int_{0}^s f^{k,p}_n(r)\,\d r$.
\begin{claim}
The functions $f^k_n, f^{k,p}_n : \R \to \R$, for all $k\in\N_0$, $n\in\N$ and $1 \leq p < \infty$, have the following properties:
\begin{enumerate}[label=\arabic*.,ref=\arabic*.]
  \item\label{Example_Counterexample_reparametrisation_action_Claim_1} $1 \leq \|f^0_n\|_{\mathcal{C}^0}, \|f^{0,p}_n\|_{W^{0,p}} \leq 2$.
  \item\label{Example_Counterexample_reparametrisation_action_Claim_2} For all $k \in \N$, $n\in\N$ and $1 \leq p < \infty$, $f^k_n$ and $f^{k,p}_n$ are nonnegative and monotone increasing with $f^k_n(s) = f^{k,p}_n(s) = 0$ for all $s \leq 0$.
  \item\label{Example_Counterexample_reparametrisation_action_Claim_3} $\frac{\d^\ell}{\d s^\ell}f^{k}_n = f^{k-\ell}_n$ and $\frac{\d^\ell}{\d s^\ell}f^{k,p}_n = f^{k-\ell,p}_n$ for all $0 \leq \ell \leq k$.
  \item\label{Example_Counterexample_reparametrisation_action_Claim_4} For all $k\in\N_0$ there exist polynomials $P^k \in \R[s]$ of degree $k-1$ \st
\[
f^k_n(1 + s), f^k_n(1+s) \leq P^k(s)
\] for all $s \geq 0$, $n\in\N$ and $1 < p < \infty$.
\end{enumerate}
\end{claim}
\begin{proof}
\labelcref{Example_Counterexample_reparametrisation_action_Claim_1,Example_Counterexample_reparametrisation_action_Claim_2,Example_Counterexample_reparametrisation_action_Claim_3}~are immediate from the definition. \\
For \labelcref{Example_Counterexample_reparametrisation_action_Claim_4}~note that for $k = 0$ the statement obviously holds with $P^0 \definedas 0$.
For $k = 1$ and $s \geq 0$,
\begin{align*}
f^{1,p}_n(1+s) &= \int_0^1 f^{0,p}_n(r)\, \d r + \int_1^{1+s} \underbrace{f^{0,p}_n(r)}_{=\;0}\,\d r \\
&= \int_0^1 f^{0,p}_n(r)\,\d r \leq 2n^{1/p}\cdot 1/n = 2/n^{1-1/p} \\
&\leq 2 \defines P^{1}
\end{align*}
and similarly for $f^1_n(1+s) \leq 2 = P^1$.
Now for $k \geq 1$ and $s \geq 0$,
\begin{align*}
f^{k+1,p}_n(1+s) &= \int_0^1 f^{k,p}_n(r)\, \d r + \int_1^{1+s} f^{k,p}_n(r)\, \d r \\
&\leq \int_0^1 P^{k}(0)\, \d r + \int_1^{1+s} P^{k}(r)\, \d r \\
&\defines P^{k+1}(s)\text{.}
\end{align*}
\end{proof}
Pick a smooth cutoff function $\rho : \R \to \R$ with $\supp \rho \subseteq [-1,2]$ and $\rho|_{[0,1]} \equiv 1$.
Define $g^k_n, g^{k,p}_n : \R \to \R$, for $k\in\N_0$, $1 < p < \infty$ and $n\in\N$ by $g^k_n \definedas \rho f^k_n$ and $g^{k,p}_n \definedas \rho f^{k,p}_n$.
It follows from the above claim that there are constants $c_k > 0$ \st $\|g^k_n\|_{\mathcal{C}^k}, \|g^{k,p}_n\|_{W^{k,p}} \leq c_k$. \\
Furthermore, for any $t \in \R$, $g^0_n(t + \cdot)$ and $g^0_n$ have disjoint supports for all $n \geq 1/t$ and likewise for $g^{0,p}_n(t + \cdot)$ and $g^{0,p}_n$.
It follows easily that for all $n\in\N$ with $n \geq 1/t$,
\begin{align*}
\|g^k_n(t+\cdot) - g^k_n\|_{\mathcal{C}^k} &\geq \|g^0_n(t + \cdot) - g^0_n\|_{\mathcal{C}^0} \\
&\geq \|g^0_n(t + \cdot)\|_{\mathcal{C}^0} + \|g^0_n\|_{\mathcal{C}^0} \\
&\geq 2\|g^0_n\|_{\mathcal{C}^0} \\
&\geq 2
\intertext{and likewise}
\|g^{k,p}_n(t+\cdot) - g^{k,p}_n\|_{W^{k,p}} &\geq 2\|g^{0,p}_n\|_{W^{0,p}} \\
&\geq 2\text{.}
\end{align*}
Now define
\begin{align*}
u^k_n : \Sigma &\to F \\
e^{it} &\mapsto \begin{cases} (e^{it}, g^k(t)) & t \in [-1, 2] \\ (e^{it}, 0) & \text{else} \end{cases}
\intertext{and}
u^{k,p}_n : \Sigma &\to F \\
e^{it} &\mapsto \begin{cases} (e^{it}, g^{k,p}(t)) & t \in [-1, 2] \\ (e^{it}, 0) & \text{else} \end{cases}\text{.}
\end{align*}
This finishes the counterexample.
\end{example}

In the following proposition, since $\Psi$ is not uniformly continuous and does not have locally bounded image, asking for Fr{\'e}chet differentiability of $\Psi$ does not make much sense anyway, at least as defined in \cref{Definition_Differentiability_II}.
But the proof easily shows that also the weaker definition of Fr{\'e}chet differentiability from \cref{Remark_Rest_continuity_versus_continuity_at_0_2} is not satisfied either.
The assumption of locally bounded image is used to show that $\mathcal{C}^{\mathrm{bi}}_{\mathrm{ao}}(U,Y)$ is a Hausdorff locally convex topological vector space, \cf \cref{Definition_Function_space_topologies}.
Without it one still gets a topological space $\mathcal{C}_{\mathrm{ao}}(U,Y)$.

In preparation for stating the next proposition about differentiability of $\Psi$, $\Psi^k$ and $\Psi^{k,p}$, it is helpful for expressing the differential to introduce a further bit of notation: \\
For $(b,u) \in \Gamma_B(F_1)$ and $(e,v) \in \Gamma_{\R^r}(F_1)$ let
\begin{align*}
&\nabla_e\phi_b : \Sigma \to T\Sigma & &\text{and} & &\nabla_e\Phi_b : F_1 \to F_2 \\
&\nabla_e\phi_b(z) \definedas \left.\frac{\d}{\d t}\right|_{t=0} \phi_{b + te}(z) & & & &\nabla_e\Phi_b \definedas \left.\frac{\d}{\d t}\right|_{t=0} \left(\Phi_{b + te}\circ P_{\gamma^t}\right)\text{.}
\end{align*}
Here, $\gamma^t$ is the family of paths
\begin{align*}
\gamma^t : [0,1]\times \Sigma &\to \Sigma \\
(s, z) &\mapsto \gamma^t_z \definedas \phi_{b + ste}(z)
\end{align*}
and $P_{\gamma^t}$ denotes the family of parallel transports along $\gamma^t$, \ie for $z \in \Sigma$, one has parallel transport $P_{\gamma^t_z} : (F_1)_{\phi_b(z)} \to (F_1)_{\phi_{b+te}(z)}$ along $\gamma^t_z$. \\
Define
\begin{align*}
\tilde{B} &\definedas TB & \overline{\Phi} : \tilde{B}\times \tilde{F} &\to \tilde{B}\times F_2 \\
\tilde{F} &\definedas \Hom(T\Sigma, F_1) & ((b,e), h) &\mapsto ((b,e), \Phi_b\circ h \circ \nabla_e\phi_b) \\
\tilde{\phi} : \tilde{B}\times \Sigma &\to \tilde{B}\times \Sigma & \tilde{\Phi} : \tilde{B}\times F_1 &\to \tilde{B}\times F_2 \\
\tilde{\phi}_{(b,e)} &\definedas \phi_b & ((b,e), u) &\mapsto \left((b,e), \left(\nabla_e\Phi_b\right)(u)\right)\text{.}
\end{align*}
$\overline{\Phi}_{(b,e)}$ and $\tilde{\Phi}_{(b,e)}$ cover $\tilde{\phi}_{(b,e)}\inv$.

\begin{proposition}\label{Proposition_Reparametrisation_action_diffble}
In the above notation,
\begin{enumerate}[label=\arabic*.,ref=\arabic*.]
  \item\label{Proposition_Reparametrisation_action_diffble_1} $\Psi : \Gamma_B(F_1) \to \Gamma_B(F_2)$ is weakly continuously weakly Fr{\'e}chet differentiable with locally bounded derivative given by
\begin{align*}
D\Psi(b,u)(e,v) &= (e, D\Psi_2(b,u)(e,v)) \\
D\Psi_2(b,u)(e,v) &= \Phi_b\circ \nabla_{\nabla_e\phi_b}u + (\nabla_e\Phi_b)\circ u\circ \phi_b + \Phi_b^\ast v \\
&= \overline{\Phi}_{(b,e)}^\ast \nabla u + \tilde{\Phi}_{(b,e)}^\ast u + \Phi_b^\ast v\text{,}
\end{align*}
for $(b,u) \in \Gamma_B(F_1)$ and $(e,v) \in \Gamma_{\R^r}(F_1)$.
  \item\label{Proposition_Reparametrisation_action_diffble_2} $\Psi : \Gamma_B(F_1) \to \Gamma_B(F_2)$ is in general neither strongly continuously weakly Fr{\'e}chet differentiable nor Fr{\'e}chet differentiable.
  \item\label{Proposition_Reparametrisation_action_diffble_3} $\Psi^k : \Gamma^k_B(F_1) \to \Gamma^k_B(F_2)$ is weakly continuously weakly Fr{\'e}chet differentiable along the canonical inclusion
\[
\iota_k : \Gamma^{k+1}(F_1) \hookrightarrow \Gamma^k(F_1)\text{,}
\]
with locally bounded derivative along $\iota_k$, given by the same formula as for $\Psi$.
  \item\label{Proposition_Reparametrisation_action_diffble_4} $\Psi^k : \Gamma^k_B(F_1) \to \Gamma^k_B(F_2)$ is in general neither weakly Fr{\'e}chet differentiable everywhere, nor strongly continuously weakly Fr{\'e}chet differentiable along $\iota_k$, nor Fr{\'e}chet differentiable along $\iota_k$.
  \item\label{Proposition_Reparametrisation_action_diffble_5} $\Psi^{k,p} : W^{k,p}_B(F_1) \to W^{k,p}_B(F_2)$ is weakly continuously weakly Fr{\'e}chet differentiable along the canonical inclusion
\[
\iota_{k,p} : W^{k+1,p}(F_1) \hookrightarrow W^{k,p}(F_1)\text{,}
\]
with locally bounded derivative along $\iota_{k,p}$, given by the same formula as for $\Psi$.
  \item\label{Proposition_Reparametrisation_action_diffble_6} $\Psi^{k,p} : W^{k,p}_B(F_1) \to W^{k,p}_B(F_2)$ is in general neither weakly Fr{\'e}chet differentiable everywhere, nor strongly continuously weakly Fr{\'e}chet differentiable along $\iota_{k,p}$, nor Fr{\'e}chet differentiable along $\iota_{k,p}$.
\end{enumerate}
\end{proposition}
\begin{proof}
The following calculation holds equally for all three cases, so let $f = \Psi_2$, $f = \Psi^k_2$ or $f = \Psi^{k,p}_2$, $(b,u) \in \Gamma_B(F_1)$, $(b,u) \in \Gamma^{k+1}_B(F_1)$ or $(b,u) \in W^{k+1,p}(F_1)$ and $(e,v) \in \Gamma_{V}(F_1)$, $(e,v) \in \Gamma^k_{V}(F_1)$ or $(e,v) \in W^{k,p}_{V}(F_1)$, respectively, where $V \subseteq \R^r$ is a convex balanced neighbourhood of $0$ \st $b + e \in B$ for all $e \in V$. \\
One then calculates for $t \in (0,1]$, that
\begin{align*}
f(b + te, u + tv) &= f(b,u) \;+ \\
&\quad\; +\; t\Bigl[ \Phi_b\circ \nabla_{\nabla_e\phi_b}u + \left(\nabla_e\Phi_b\right)\circ u\circ \phi_b + \Phi_b^\ast v \;+ \\
&\;\quad \qquad\; +\; \Phi_b\left(\frac{1}{t}\left(P_{\gamma^t}\inv\circ u\circ \phi_{b+te} - u\circ \phi_b\right) - \nabla_{\nabla_e\phi_b}u\right) \;+ \\
&\;\quad \qquad\; +\; \left(\frac{1}{t}\left(\Phi_{b+te}\inv\circ P_{\gamma^t} - \Phi_{b}\inv\right) - \left(\nabla_e\Phi_b\right)\right)\circ P_{\gamma^t}\inv \circ u\circ \phi_{b+te} \;+ \\
&\;\quad \qquad\; +\; \left(\nabla_e\Phi_b\right)\circ \left(P_{\gamma^t}\inv \circ u\circ \phi_{b+te} - u\circ \phi_b\right) \;+ \\
&\;\quad \qquad\; +\; \Phi_{b+te}^\ast v - \Phi_b^\ast v\Bigr]
\intertext{and hence}
r^{f}_{(b,u)}((e,v), t) &= \underbrace{\Phi_b}_{\mathclap{\defines\; \mathrm{(I_1)}}} \circ\underbrace{\left(\frac{1}{t}\left(P_{\gamma^t}\inv\circ u\circ \phi_{b+te} - u\circ \phi_b\right) - \nabla_{\nabla_e\phi_b}u\right)}_{\defines\; \mathrm{(I_2)}(t,e)} \;+ \\
&\quad\; +\; \underbrace{\left(\frac{1}{t}\left(\Phi_{b+te}\inv\circ P_{\gamma^t} - \Phi_{b}\inv\right) - \left(\nabla_e\Phi_b\right)\right)}_{\defines\; \mathrm{(II_1)}(t,e)}\circ \underbrace{P_{\gamma^t}\inv \circ u\circ \phi_{b+te}}_{\defines\; \mathrm{(II_2)}(t,e)} \;+ \\
&\quad\; +\; \underbrace{\left(\nabla_e\Phi_b\right)}_{\defines\; \mathrm{(III_1)}(e)}\circ \underbrace{\left(P_{\gamma^t}\inv \circ u\circ \phi_{b+te} - u\circ \phi_b\right)}_{\defines\; \mathrm{(III_2)}(t,e)} \\
&\quad\; +\; \underbrace{\Phi_{b+te}^\ast v - \Phi_b^\ast v}_{\defines\; \mathrm{(IV)}(t,e,v)}\text{.}
\end{align*}

Note that (at least once \labelcref{Proposition_Reparametrisation_action_diffble_1}, \labelcref{Proposition_Reparametrisation_action_diffble_3}~and \labelcref{Proposition_Reparametrisation_action_diffble_5}~have been shown) for $(b,u) = (b,0)$,
\begin{align*}
D\Psi(b,0)(e,v) &= (e, \Phi_b^\ast v)
\intertext{and}
r^\Psi_{(b,0)}((e,v), t) &= (0, \Phi_{b+te}^\ast v - \Phi_b^\ast v)
\end{align*}
and similarly for $\Psi^k$ and $\Psi^{k,p}$.
Using this and the above formula for $r^f$, \labelcref{Proposition_Reparametrisation_action_diffble_2}, \labelcref{Proposition_Reparametrisation_action_diffble_4}~and \labelcref{Proposition_Reparametrisation_action_diffble_6}~follow immediately from \cref{Proposition_Reparametrisation_action_cts}.

We now examine the terms $\mathrm{(I_1)}$--$\mathrm{(IV)}$ first for $\Psi^k$ and $\Psi^{k,p}$.
$\mathrm{(I_1)}$ is obviously independent of $e$, $v$ and $t$ so there is nothing to show about it. \\
For $\mathrm{(I_2)}$, define
\begin{align*}
\tilde{B} &\definedas (-1,1)\times V \\
\tilde{F} &\definedas \Hom(T\Sigma, F_1) \\
\tilde{\Phi} : \tilde{B}\times \tilde{F} &\to \tilde{B}\times \tilde{F} \\
((s,e), h) &\mapsto ((s,e), P_{\gamma^s}\inv\circ h\circ P^\Sigma_{\gamma^s}) \\
\tilde{\phi} : \tilde{B}\times \Sigma &\to \tilde{B}\times \Sigma \\
((s,e),z) &\mapsto ((s,e), \phi_{b + se}(z))\text{.}
\end{align*}
$P^\Sigma_{\gamma^s}$ here denotes parallel transport in $T\Sigma$.
Furthermore, define a smooth family $\left(X_{(s,e)}\right)_{(s,e)\in \tilde{B}}$ of vector fields $X_{(s,e)} \in \mathfrak{X}(\Sigma)$ by $X_{(s,e)}(z) \definedas \left(P^\Sigma_{\gamma^s}\right)\inv \frac{\d}{\d s}\phi_{b+se}(z)$.
Then
\[
\frac{\d}{\d s} P_{\gamma^s}\inv\circ u \circ \phi_{b + se} = \left(\tilde{\Phi}_{(s,e)}^\ast \nabla u\right)\left(X_{(s,e)}\right)
\]
and
\begin{align*}
\mathrm{(I_2)}(t,e) &= \frac{1}{t}\int_0^t \frac{\d}{\d s} P_{\gamma^s}\inv\circ u \circ \phi_{b + se}\, \d s - \left.\frac{\d}{\d s}\right|_{s=0} P_{\gamma^s}\inv\circ u \circ \phi_{b + se} \\
&= \frac{1}{t}\int_0^t \left[\left(\tilde{\Phi}_{(s,e)}^\ast \nabla u\right)\left(X_{(s,e)}\right) - \left(\tilde{\Phi}_{(0,e)}^\ast \nabla u\right)\left(X_{(0,e)}\right)\right]\d s \\
&= \frac{1}{t}\int_0^t \underbrace{\left(\tilde{\Phi}_{(s,e)}^\ast \nabla u - \tilde{\Phi}_{(0,e)}^\ast \nabla u\right)\left(X_{(0,e)}\right)}_{\defines\; \mathrm{(I_{21})}(s,e)}\, \d s \;+ \\
&\quad\; +\; \frac{1}{t}\int_0^t \underbrace{\left(\tilde{\Phi}_{(s,e)}^\ast \nabla u\right)\left(X_{(s,e)} - X_{(0,e)}\right)}_{\defines\; \mathrm{(I_{22})}(s,e)}\, \d s\text{.}
\end{align*}
This equality holds pointwise (over $\Sigma$) if $u \in \Gamma^{k+1}(F_1)$ or $u \in W^{k+1,p}(F_1)$.
It is used here that the maps
\begin{align*}
\Gamma^{k+1}(F_1) &\to \Gamma^k(\tilde{F}) & W^{k+1,p}(F_1) &\to W^{k,p}(\tilde{F}) \\
u &\mapsto \nabla u & u &\mapsto \nabla u
\end{align*}
are continuous and that $kp > \dim_\R\Sigma$, so the Sobolev embedding theorem applies.
Also using that the maps ($\mathfrak{X}^k(\Sigma) \definedas \Gamma^k(T\Sigma)$)
\begin{align*}
\Gamma^k(\tilde{F})\times \mathfrak{X}^k(\Sigma) &\to \Gamma^k(F_1) & W^{k,p}(\tilde{F})\times \mathfrak{X}^k(\Sigma) &\to W^{k,p}(F_1) \\
(h, X) &\mapsto h(X) & (h, X) &\mapsto h(X)
\end{align*}
are continuous, one sees that if $K\subseteq V$ is a compact subset and $\|\cdot\|$ denotes either $\|\cdot\|_k$ or $\|\cdot\|_{k,p}$ as appropriate, then by \cref{Proposition_Reparametrisation_action_cts} and \cref{Lemma_Topological_Lemma}:
\begin{itemize}
  \item There exist continuous functions (also depending on $(b,u)$, which are fixed) $c^k(\cdot, K), c^{k,p}(\cdot, K) : [0,1] \to [0,\infty)$ with $c^k(0, K) = c^{k,p}(0,K) = 0$ \st
\begin{align*}
\|\mathrm{(I_{21})}(s,e)\|_k &\leq \left\|\tilde{\Phi}_{(s,e)}^\ast \nabla u - \tilde{\Phi}_{(0,e)}^\ast \nabla u\right\|_k\left\|X_{(0,e)}\right\|_k \leq c^k(s,K) \\
\|\mathrm{(I_{21})}(s,e)\|_{k,p} &\leq \left\|\tilde{\Phi}_{(s,e)}^\ast \nabla u - \tilde{\Phi}_{(0,e)}^\ast \nabla u\right\|_{k,p}\left\|X_{(0,e)}\right\|_k \leq c^{k,p}(s,K)
\end{align*}
for all $e \in K$.
  \item There exist constants $C^k(K), C^{k,p}(K) \in \R$ and a continuous function $\tilde{c}^k(\cdot, K) : [0,1] \to [0,\infty)$ with $\tilde{c}^k(0,K) = 0$ \st
\begin{align*}
\|\mathrm{(I_{22})}(s,e)\|_k &\leq \left\|\tilde{\Phi}_{(s,e)}^\ast \nabla u\right\|_k\left\|X_{(s,e)} - X_{(0,e)}\right\|_k \leq C^{k}(K)\tilde{c}^k(s,K) \\
\|\mathrm{(I_{22})}(s,e)\|_{k,p} &\leq \left\|\tilde{\Phi}_{(s,e)}^\ast \nabla u\right\|_{k,p}\left\|X_{(s,e)} - X_{(0,e)}\right\|_k \leq C^{k,p}(K)\tilde{c}^{k}(s,K)
\end{align*}
\end{itemize}
In conclusion, for all $e \in K$,
\begin{align*}
\|\mathrm{(I_2)}(t,e)\|_k &\leq \frac{1}{t}\int_0^1 \|\mathrm{(I_{21})}(s,e)\|_k\, \d s + \frac{1}{t}\int_0^1 \|\mathrm{(I_{22})}(s,e)\|_k\, \d s \\
&\leq \frac{1}{t}\int_0^1 c^k(s,K)\, \d s + \frac{1}{t}\int_0^1 C^k(K)\tilde{c}^k(s,K)\, \d s \\
&\leq \max\{ c^k(s,K) + C^k(K)\tilde{c}^k(s,K) \;|\; s \in [0,t]\} \\
&\to 0 \quad\text{for $t \to 0$.}
\intertext{and similarly}
\|\mathrm{(I_2)}(t,e)\|_{k,p} &\leq \max\{ c^{k,p}(s,K) + C^{k,p}(K)\tilde{c}^k(s,K) \;|\; s \in [0,t]\} \\
&\to 0 \quad\text{for $t \to 0$.}
\end{align*}
Using that $\R^r$ is locally compact, the above shows that $\mathrm{(I_2)}: [0,1]\times V \to \Gamma^k(F_1)$ (or  $\mathrm{(I_2)}: [0,1]\times V \to W^{k,p}(F_1)$) is a continuous function. \\
Completely analogous (somewhat simpler, in fact) arguments using \cref{Proposition_Reparametrisation_action_cts} for the terms $\mathrm{(II_2)}(t,e)$ and $\mathrm{(III_2)}(t,e)$ (applied to a new reparametrisation action $\tilde{\Phi}_{(t,e)}(u) \definedas P_{\gamma^t}\inv\circ u \circ \phi_{b + te}$) show that the terms $\mathrm{(II_1)}(t,e)$--$\mathrm{(III_2)}(t,e)$ are continuous functions in $(t,e)$ as well. \\
Finally, $\mathrm{(IV)}(t,e,v)$ is continuous in $(t,e,v)$ directly by \cref{Proposition_Reparametrisation_action_cts}. \\
This shows that $\Psi^k$ and $\Psi^{k,p}$ are weakly Fr{\'e}chet differentiable on $\Gamma^{k+1}_B(F_1)$ and $W^{k+1,p}_B(F_1)$, respectively. \\
That they are also weakly continuously weakly Fr{\'e}chet differentiable along the inclusions of $\Gamma^{k+1}_B(F_1)$ and $W^{k+1,p}_B(F_1)$ with locally bounded derivative, respectively, follows immediately from the explicit formula for their differential and \cref{Proposition_Reparametrisation_action_cts}. \\
\labelcref{Proposition_Reparametrisation_action_diffble_1}~follows from the corresponding statements for $\Psi^k$ by virtue of the definition of the topologies on $\Gamma(F_1)$ and $\Gamma(F_2)$.
For example, for $(b,u) \in \Gamma_B(F_1)$, $r^{\Psi_2}_{(b,u)} : \Gamma_{V}(F_1)\times [0,1] \to \Gamma_{\R^r}(F_2)$ is continuous because it is the inverse limit of the maps $r^{\Psi^k_2}_{(b,u)} : \Gamma^k_{V}(F_1)\times [0,1] \to \Gamma^k_{\R^r}(F_2)$, etc.
\end{proof}

\clearpage
\section{The linear theory of $\overline{\text{sc}}$-Fr{\'e}chet spaces}\label{Section_Linear_theory}

\Needspace{25\baselineskip}
\subsection{Chains of Banach spaces}

There are various names for the following objects, here I will follow mainly \cite{MR1421572} and \cite{MR2341834}.
The main reason to not use ``sc-Banach space'' is to avoid putting emphasis on the fixed Banach space $E = E_0$.
In fact, most of the setup presented here is motivated by the goal of shifting focus from the Banach space $E_0$ in an sc-Banach space in the sense of \cite{MR2341834}, or the more recent \cite{MR2644764}, to the Fr{\'e}chet space $E_\infty$, which contains the objects one is usually more interested in, and treat the remaining structures as variable additional choices more akin to charts in a manifold. \\
The nonspecific ``Banach space'' also means that one has decided to either work with real or complex Banach spaces and real or complex linear maps, once and for all.
The chosen ground field will be denoted by $\mathds{k}$.

\begin{definition}[\cite{MR1421572}, Definition 1.1 and \cite{MR2341834}, Definition 2.1]\label{Definition_Chain}
\leavevmode
\begin{enumerate}[label=\arabic*.,ref=\arabic*.]
  \item\label{Definition_Chain_1} An \emph{ILB-chain} $\mathbb{E}$ is a sequence $((E_k, \|\cdot\|_k), \iota_k)_{k\in\N_0}$, where
\begin{enumerate}[label=(\alph*),ref=(\alph*)]
  \item\label{Definition_Chain_1a} $(E_k,\|\cdot\|_k)$ is a Banach space and
  \item\label{Definition_Chain_1b} $\iota_k : E_{k+1} \hookrightarrow E_k$ is a continuous embedding with dense image.
  \item\label{Definition_Chain_1c} $\|\iota_k\|_{L_{\mathrm{c}}(E_{k+1}, E_k)} \leq 1$.
\end{enumerate}
One regards the $E_k$ as a nested sequence of linear subspaces of $E_0$ and denotes the (dense) linear subspace $E_\infty \definedas \bigcap_{k\in\N_0} E_k \subseteq E_0$ as a topological vector space with the weakest topology \st all inclusions $E_\infty \hookrightarrow E_k$ are continuous. \\
For $k,\ell \in \N_0$ with $\ell > k$ one denotes
\[
\iota^\ell_k \definedas \iota_{\ell-1}\circ \iota_{\ell-2}\circ \cdots \circ \iota_{k+1}\circ \iota_k : E_\ell \hookrightarrow E_k
\]
and $\iota^k_k \definedas \id_{E_k}$.
$\iota^\infty_k : E_\infty \hookrightarrow E_k$ is defined as the canonical inclusion.
  \item\label{Definition_Chain_2} An \emph{sc-chain} $\mathbb{E}$ is an ILB-chain $\mathbb{E}$ \st all the inclusions $\iota_k$ are compact.
\end{enumerate}
\end{definition}

\begin{remark}
\begin{enumerate}[label=\arabic*.,ref=\arabic*.]
  \item ``Continuous embedding'' here means that $\iota_k : E_{k+1} \to E_k$ is an injective continuous linear map.
It does \emph{not} mean that $\iota_k$ has closed image or that $\iota_k$ is an embedding of topological spaces (\ie that $E_{k+1}$ carries the subspace topology inherited from $\iota_k(E_{k+1}) \subseteq E_k$).
  \item Condition \labelcref{Definition_Chain_1c} is not as restrictive as it might look.
For first of all in the main examples such as \cref{Example_Sobolev_chain_I} and \cref{Example_Ck_chain_I} it is automatically satisfied.
And second if condition \labelcref{Definition_Chain_1c} is dropped, then one can define new norms
\begin{align*}
\interleave\cdot\interleave_k : E_k &\to [0,\infty) \\
e &\mapsto \sum_{i=0}^k \|\iota^k_i(e)\|_i\text{,}
\end{align*}
\ie $\interleave\cdot\interleave_{k+1} = \interleave\iota_k(\cdot)\interleave_k + \|\cdot\|_{k+1}$.
From this it follows immediately that for $e \in E_{k+1}$
\begin{align*}
\interleave\iota_k(e)\interleave_k &\leq \interleave\iota_k(e)\interleave_k + \|e\|_{k+1} \\
&= \interleave e\interleave_{k+1}\text{,}
\intertext{so $\interleave\iota_k\interleave_{L_{\mathrm{c}}(E_{k+1},E_k)} \leq 1$, and}
\|e\|_{k} &\leq \|e\|_{k} + \interleave\iota_{k-1}(e)\interleave_{k-1} \\
&= \interleave e\interleave_k\text{.}
\end{align*}
On the other hand $\iota^k_i : L_{\mathrm{c}}(E_k, E_i)$ is continuous as a composition of continuous maps, hence $\|\iota^k_i\|_{L_{\mathrm{c}}(E_k, E_i)} \defines C^k_i \in [0,\infty)$.
So
\begin{align*}
\interleave e\interleave_k &= \sum_{i=0}^k \|\iota^k_i(e)\|_i \\
&\leq \underbrace{\sum_{i=0}^k C^k_i}_{\defines\; c_k} \|e\|_k
\end{align*}
and hence $\|\cdot\|_k \leq \interleave\cdot\interleave_k \leq c_k\|\cdot\|_k$, so $\interleave\cdot\interleave_k$ and $\|\cdot\|_k$ are equivalent.
  \item $E_\infty$ with the given topology is a Fr{\'e}chet space (in the sense of \cite{MR3154940}, Definition 3.37), with distance function given \eg by
\begin{align*}
d : E_\infty \times E_\infty &\to [0,\infty) \\
(x,y) &\mapsto \sum_{k=0}^\infty c_k \min\left\{r, \|\iota^\infty_k(x - y)\|_k\right\}\text{,}
\intertext{or alternatively}
d' : E_\infty \times E_\infty &\to [0,\infty) \\
(x,y) &\mapsto \sum_{k=0}^\infty c_k \frac{\|\iota^\infty_k(x-y)\|_k}{1 + \|\iota^\infty_k(x-y)\|_k}\text{,}
\end{align*}
where $r > 0$ and $(c_k)_{k\in\N_0} \subseteq (0,\infty)$ is any sequence \st $\sum_{k=0}^\infty c_k < \infty$.
\end{enumerate}
\end{remark}

\begin{example}\label{Example_Sobolev_chain_I}
Let $(\Sigma,g)$ be a closed $n$-dimensional Riemannian manifold and let $\pi : F \to \Sigma$ be a real (or complex) vector bundle equipped with a euclidean (or hermitian) metric and metric connection.
Also let $\mathbf{l} = (l_k)_{k \in \N_0} \subseteq \N_0$ be a strictly monotone increasing sequence and let $1 < p < \infty$ be \st $l_0 p > n$. \\
Define an ILB-chain
\[
\mathbb{W}^{\mathbf{l},p}(F) \definedas \left(\left(W^{l_k, p}(F), \|\cdot\|_{{l_k,p}}\right), \iota_k\right)_{k\in\N_0}\text{,}
\]
where $\left(W^{l_k,p}(F), \|\cdot\|_{{l_k,p}}\right)$ is the Sobolev space of sections of $F$ of class $W^{l_k, p}$, realised as the completion of $\Gamma(F)$ \wrt the norm $\|\cdot\|_{{l_k,p}}$, as in \cref{Subsection_Reparametrisation_action}.
By the Sobolev embedding theorem, $W^{l_k,p}(F)$ will be regarded as a subset of $\Gamma^0(F)$.
$\iota_{k} : W^{l_{k+1},p}(F) \hookrightarrow W^{l_k,p}(F)$ is the canonical inclusion. \\
By the Rellich-Kondrachov theorem this ILB-chain is an sc-chain called the \emph{Sobolev chain of class $(\mathbf{l},p)$}. \\
Furthermore, by the Sobolev embedding theorem, $W^{\mathbf{l},p}(F)_\infty = \Gamma(F)$ equipped with the $\mathcal{C}^\infty$-topology.
\end{example}

\begin{example}\label{Example_Ck_chain_I}
Let $(\Sigma,g)$ be a closed $n$-dimensional Riemannian manifold and let $\pi : F \to \Sigma$ be a real (or complex) vector bundle equipped with a euclidean (or hermitian) metric and metric connection.
Also let $\mathbf{l} = (l_k)_{k \in \N_0} \subseteq \N_0$ be a strictly monotone increasing sequence. \\
Define an ILB-chain
\[
\mathbbl{\Gamma}^{\mathbf{l}}(F) \definedas \left(\left(\Gamma^{l_k}(F), \|\cdot\|_{{l_k}}\right), \iota_k\right)_{k\in\N_0}\text{,}
\]
where $\left(\Gamma^{l_k}(F), \|\cdot\|_{{l_k}}\right)$ is the space of $l_k$-times continuously differentiable sections of $F$ equipped with the $\mathcal{C}^k$-norm $\|\cdot\|_{l_k}$, as in \cref{Subsection_Reparametrisation_action}.
$\iota_k : \Gamma^{l_{k+1}}(F) \hookrightarrow \Gamma^{l_k}(F)$ is the canonical inclusion. \\
By the Theorem of Arzela-Ascoli this ILB-chain is an sc-chain called the \emph{chain of continuously differentiable sections of class $\mathbf{l}$}. \\
By definition, $\Gamma^{\mathbf{l}}(F)_\infty = \Gamma(F)$ equipped with the $\mathcal{C}^\infty$-topology.
\end{example}

\begin{definition}[\cite{MR2341834}, Definition 2.6]
Let $\mathbb{E}$, $\mathbb{E}'$ and $\mathbb{E}''$ be ILB- or sc-chains.
\begin{enumerate}[label=\arabic*.,ref=\arabic*.]
  \item A \emph{continuous linear operator} $\mathbb{T} : \mathbb{E} \to \mathbb{E}'$ is a sequence
\[
\left(T_k : (E_k, \|\cdot\|_k) \to (E'_k, \|\cdot\|'_k)_{k \in \N_0}\right)_{k\in\N_0}
\]
of continuous linear operators \st
\[
T_{k}\circ \iota_{k} = \iota'_{k}\circ T_{k+1} : E_{k+1} \to E'_k \quad\forall\, k\in\N_0\text{.}
\]
It induces a well defined continuous map $T_\infty : E_\infty \to E'_\infty$ via $\iota'^\infty_k(T_\infty(u)) \definedas T_k(\iota^\infty_k(u))$ for any $k\in\N_0$.
  \item For continuous linear operators $\mathbb{T} : \mathbb{E} \to \mathbb{E}'$ and $\mathbb{T}' : \mathbb{E}' \to \mathbb{E}''$, their \emph{composition $\mathbb{T}'\circ \mathbb{T} : \mathbb{E} \to \mathbb{E}''$} is the continuous linear operator $\mathbb{T}''$ with $T''_k \definedas T'_k\circ T_k$.
  \item For continuous linear operators $\mathbb{T}, \mathbb{T}' : \mathbb{E} \to \mathbb{E}'$ and $\lambda,\mu \in \mathds{k}$, one defines $\lambda\mathbb{T} + \mu\mathbb{T}' : \mathbb{E} \to \mathbb{E}'$ as the continuous linear operator $\mathbb{T}''$ with $T''_k \definedas \lambda T_k + \mu T'_k$.
  \item A continuous linear operator $\mathbb{T} : \mathbb{E} \to \mathbb{E}'$ is called an \emph{embedding} if $T_k : E_k \to E'_k$ is a continuous embedding of Banach spaces for all $k \in \N_0$.
\end{enumerate}
\end{definition}

\begin{lemma}\label{Lemma_Relation_mathbb_T_to_T_infty}
Let $\mathbb{E}$, $\mathbb{E}'$ and $\mathbb{E}''$ be ILB- or sc-chains.
\begin{enumerate}[label=\arabic*.,ref=\arabic*.]
  \item For continuous linear operators $\mathbb{T} : \mathbb{E} \to \mathbb{E}'$ and $\mathbb{T}' : \mathbb{E}' \to \mathbb{E}''$, if $\mathbb{T}'' \definedas \mathbb{T}' \circ \mathbb{T}$, then $T''_\infty = T'_\infty \circ T_\infty : E_\infty \to E''_\infty$.
  \item There is a $1$---$1$ correspondence between continuous linear operators $\mathbb{T} : \mathbb{E} \to \mathbb{E}'$ and continuous linear operators $T : E_\infty \to E'_\infty$ \st $\iota'^\infty_k\circ T : E_\infty \to E'_k$ is a bounded operator on $(E_\infty, \|\cdot\|_k|_{E_\infty})$.
\end{enumerate}
\end{lemma}
\begin{proof}
\begin{enumerate}[label=\arabic*.,ref=\arabic*.]
  \item Trivial.
  \item In one direction, one maps $\mathbb{T} : \mathbb{E} \to \mathbb{E}'$ to $T_\infty : E_\infty \to E'_\infty$. \\
In the other direction, to $T : E_\infty \to E'_\infty$ one assigns the continuous linear operator $\mathbb{T} : \mathbb{E} \to \mathbb{E}'$ with $T_k : E_k \to E'_k$ the unique completion of $\iota'^\infty_k\circ T : E_\infty \to E'_k$.
Because $E_\infty$ is dense in $E_k$ and because $\iota'^\infty_k \circ T$ by assumption is bounded \wrt $\|\cdot\|_k$, this completion exists and is unique. \\
To verify that with this definition $T_k\circ \iota_k = \iota'_{k+1}\circ T_{k+1}$, by density of $E_\infty$ in $E_{k+1}$ it suffices to verify this on $E_\infty$, where $T_k\circ\iota_k = \iota'^\infty_k\circ T = \iota'_{k+1}\circ \iota'^\infty_{k+1}\circ T = \iota'_{k+1}\circ T_{k+1}$.
\end{enumerate}
\end{proof}

\begin{example}\label{Example_PDO_I}
Let $(\Sigma,g)$ be a closed $n$-dimensional Riemannian manifold and let $\pi : F \to \Sigma$ and $\pi' : F' \to \Sigma$ be real (or complex) vector bundles equipped with euclidean (or hermitian) metrics and metric connections. \\
Let $k_0 \in \N_0$ and $1 < p < \infty$ be \st $k_0p > n$.
Given $m \in \N_0$, define $\mathbf{k} \definedas (k_0 + k)_{k\in\N_0}$ and $\mathbf{k}+m \definedas (k_0 + m + k)_{k\in\N_0}$. \\
Let $P : \Gamma(F) \to \Gamma(F')$ be a partial differential operator (with smooth coefficients) of class $m$.
Then for all $k\in\N_0$, the completion $P_{k_0+k,p} : W^{k_0 + k + m, p}(F) \to W^{k_0 + k, p}(F')$ of $P$ defines a continuous linear operator and hence
\[
\mathbb{P}^{\mathbf{k},p} \definedas (P_{k_0+k,p})_{k\in\N_0} : \mathbb{W}^{\mathbf{k} + m, p}(F) \to \mathbb{W}^{\mathbf{k}, p}(F')
\]
defines a continuous linear operator with $P_\infty = P : \Gamma(F) \to \Gamma(F')$ between Sobolev chains from \cref{Example_Sobolev_chain_I}. \\
Similarly, one can also look at the completions $P_{k} : \Gamma^{k+m}(F) \to \Gamma^{k}(F')$ of $P$ to give a continuous linear operator ($\mathbf{id} = (j)_{j\in\N_0}$)
\[
\mathbb{P}^{\mathbf{id}} \definedas (P_{k})_{k\in\N_0} : \mathbbl{\Gamma}^{\mathbf{id} + m}(F) \to \mathbbl{\Gamma}^{\mathbf{id}}(F')
\]
between chains of continuously differentiable sections from \cref{Example_Ck_chain_I}.
\end{example}

\begin{definition}
$\mathbf{scChains}^{\mathds{k}}$ is the (preadditive) category with objects the sc-chains of $\mathds{k}$-Banach spaces and morphisms the continuous linear operators. \\
It comes with faithful underlying functors ($\mathbf{Banach}^{\mathds{k}}$ and $\mathbf{Frechet}^{\mathds{k}}$ are the categories of Banach- and Fr{\'e}chet spaces over $\mathds{k}$, respectively)
\begin{align*}
\mathbf{U}_k : \mathbf{scChains}^{\mathds{k}} &\to \mathbf{Banach}^{\mathds{k}} \\
\mathbb{E} &\mapsto E_k \\
(\mathbb{T} : \mathbb{E} \to \mathbb{E}') &\mapsto T_k \\
\mathbf{U}_\infty : \mathbf{scChains}^{\mathds{k}} &\to \mathbf{Frechet}^{\mathds{k}} \\
\mathbb{E} &\mapsto E_\infty \\
(\mathbb{T} : \mathbb{E} \to \mathbb{E}') &\mapsto T_\infty\text{.}
\end{align*}
and the maps $\iota^k_\ell$ for $0\leq \ell \leq k \leq \infty$ provide natural transformations between these ($\mathbf{Banach}^{\mathds{k}}$ comes with a natural forgetful functor to $\mathbf{Frechet}^{\mathds{k}}$).
\end{definition}

\Needspace{25\baselineskip}
\subsection{Rescaling, weak morphisms and equivalence}

\begin{definition}
Let $\mathbb{E}$ and $\mathbb{E}'$ be ILB- or sc-chains and let $\mathbf{k} = (k_j)_{j\in\N_0} \subseteq \N_0$ be a strictly monotone increasing sequence.
\begin{enumerate}[label=\arabic*.,ref=\arabic*.]
  \item The ILB- or sc-chain
\[
\mathbb{E}^{\mathbf{k}} = \left((E^{\mathbf{k}}_j, \|\cdot\|^{\mathbf{k}}_j), \iota^{\mathbf{k}}_j\right)_{j\in\N_0} \definedas \left((E_{k_j}, \|\cdot\|_{k_j}), \iota^{k_{j+1}}_{k_j}\right)_{j\in\N_0}
\]
is called a \emph{rescaling of $\mathbb{E}$ (by $\mathbf{k}$)}. \\
It satisfies $E^{\mathbf{k}}_{\infty} = E_\infty$ as topological vector spaces.
  \item A strictly monotone increasing sequence of the form $\mathbf{k} = (k_0 + j)_{j\in\N_0}$ for some $k_0 \in \N_0$ is called a \emph{shift}. \\
The rescaling of $\mathbb{E}^{\mathbf{k}}$ of $\mathbb{E}$ by $\mathbf{k}$ is called a \emph{shifted} ILB- or sc-chain.
  \item Let $\mathbf{l} = (l_j)_{j\in\N_0} \subseteq \N_0$ be another strictly monotone increasing sequence with $\mathbf{k} \geq \mathbf{l}$, \ie $k_j \geq l_j \;\forall\, j\in\N_0$. \\
Then there is a canonical continuous linear operator $\mathbb{I}^{\mathbf{k}}_{\mathbf{l}} : \mathbb{E}^{\mathbf{k}}  \to \mathbb{E}^{\mathbf{l}}$ given by $\left(I^{\mathbf{k}}_{\mathbf{l}}\right)_j \definedas \iota^{k_j}_{l_j} : E^{\mathbf{k}}_j = E_{k_j} \to E_{l_j} = E^{\mathbf{l}}_j$. \\
One abbreviates $\mathbb{I}^{\mathbf{k}} \definedas \mathbb{I}^{\mathbf{k}}_{\mathbf{id}}$, where $\mathbf{id} \definedas (j)_{j\in\N_0}$. \\
The $\mathbb{I}^{\mathbf{k}}_{\mathbf{l}}$ satisfy $\left(I^{\mathbf{k}}_{\mathbf{l}}\right)_\infty = \id_{E_{\infty}} : E^{\mathbf{k}}_\infty = E_\infty \to E_\infty = E^{\mathbf{l}}_\infty$.
  \item If $\mathbb{T} : \mathbb{E} \to \mathbb{E}'$ is a continuous linear operator, then there is an induced continuous linear operator $\mathbb{T}^{\mathbf{k}} : \mathbb{E}^{\mathbf{k}} \to \mathbb{E}'^{\mathbf{k}}$ defined by
\[
T^{\mathbf{k}}_j \definedas T_{k_j} : E^{\mathbf{k}}_j = E_{k_j} \to E'_{k_j} = E^{\mathbf{k}}_j\text{.}
\]
It satisfies $T^{\mathbf{k}}_\infty = T_\infty : E^{\mathbf{k}}_\infty = E_\infty \to E'_\infty = E^{\mathbf{k}}_\infty$. \\
Conversely, $\mathbb{T}^{\mathbf{k}}$ is uniquely determined by $\mathbf{k}$ and $T_\infty$.
  \item A \emph{weak morphism} between $\mathbb{E}$ and $\mathbb{E}'$ is a continuous linear operator $T : E_\infty \to E'_\infty$ \st there exist strictly monotone increasing sequences $\mathbf{k},\mathbf{l} \subseteq \N_0$ and a continuous linear operator $\mathbb{T} : \mathbb{E}^{\mathbf{k}} \to \mathbb{E}'^{\mathbf{l}}$ with $T = T_\infty$. \\
In this case $\mathbb{T} : \mathbb{E}^{\mathbf{k}} \to \mathbb{E}'^{\mathbf{l}}$ is said to \emph{extend} (or be an \emph{extension of}) $T : E_\infty \to E'_\infty$.
  \item A weak morphism $T : E_\infty \to E'_\infty$ is called a \emph{weak embedding} if there exist strictly monotone increasing sequences $\mathbf{k},\mathbf{l} \subseteq \N_0$ and an embedding $\mathbb{T} : \mathbb{E}^{\mathbf{k}} \to \mathbb{E}'^{\mathbf{l}}$ \st $T = T_\infty$.
\end{enumerate}
\end{definition}

\begin{remark}
Note that by \cref{Lemma_Relation_mathbb_T_to_T_infty}, in the definition of ``weak morphism'' for given choices of $\mathbf{k},\mathbf{l} \subseteq \N_0$ the extension $\mathbb{T} : \mathbb{E}^{\mathbf{k}} \to \mathbb{E}'^{\mathbf{l}}$, if it exists, is uniquely determined by $T$.
\end{remark}

\begin{lemma}\label{Lemma_Rescaling_and_weak_morphisms}
Let $\mathbb{E}$, $\mathbb{E}'$ and $\mathbb{E}''$ be ILB- or sc-chains.
\begin{enumerate}[label=\arabic*.,ref=\arabic*.]
  \item\label{Lemma_Rescaling_and_weak_morphisms_1} Let $\mathbf{k},\mathbf{l},\mathbf{m} \subseteq \N_0$ be strictly monotone increasing sequences. \\
$\left(\mathbb{E}^{\mathbf{k}}\right)^{\mathbf{l}} = \mathbb{E}^{\mathbf{k}\circ \mathbf{l}}$ is a rescaling of $\mathbb{E}$ by $\mathbf{k}\circ \mathbf{l} \definedas (k_{l_j})_{j\in\N_0}$. \\
$\mathbb{I}^{\mathbf{k}\circ\mathbf{l}} = \mathbb{I}^{\mathbf{l}}\circ \left(\mathbb{I}^{\mathbf{k}}\right)^{\mathbf{l}}$. \\
If $\mathbf{k} \geq \mathbf{l} \geq \mathbf{m}$, then $\mathbb{I}^{\mathbf{k}}_{\mathbf{m}} = \mathbb{I}^{\mathbf{l}}_{\mathbf{m}}\circ \mathbb{I}^{\mathbf{k}}_{\mathbf{l}}$. \\
If $\mathbb{T} : \mathbb{E} \to \mathbb{E}'$ is a continuous linear operator, then
\[
\mathbb{T}^{\mathbf{l}}\circ \mathbb{I}^{\mathbf{k}}_{\mathbf{l}} = \mathbb{I}'^{\mathbf{k}}_{\mathbf{l}}\circ \mathbb{T}^{\mathbf{k}}\text{.}
\]
  \item\label{Lemma_Rescaling_and_weak_morphisms_2} Let $\mathbf{k},\mathbf{l} \subseteq \N_0$ be strictly monotone increasing sequences with $\mathbf{k} \geq \mathbf{l}$. \\
$\mathbb{I}^{\mathbf{k}}_{\mathbf{l}} : \mathbb{E}^{\mathbf{k}} \to \mathbb{E}^{\mathbf{l}}$ is an embedding. \\
If $\mathbb{T} : \mathbb{E} \to \mathbb{E}'$ is an embedding, then so is $\mathbb{T}^{\mathbf{k}} : \mathbb{E}^{\mathbf{k}} \to \mathbb{E}'^{\mathbf{k}}$.
  \item\label{Lemma_Rescaling_and_weak_morphisms_3} A continuous linear operator $T : E_\infty \to E'_\infty$ defines a weak morphism \iff there exists a strictly monotone increasing sequence $\mathbf{k} \subseteq \N_0$ and a continuous linear operator $\mathbb{T} : \mathbb{E}^{\mathbf{k}} \to \mathbb{E}'$ \st $T = T_\infty$.
  \item\label{Lemma_Rescaling_and_weak_morphisms_4} A weak morphism $T : E_\infty \to E'_\infty$ is an embedding $\mathbb{E} \to \mathbb{E}'$ \iff there exists a strictly monotone increasing sequence $\mathbf{k} \subseteq \N_0$ and an embedding $\mathbb{T} : \mathbb{E}^{\mathbf{k}} \to \mathbb{E}'$ \st $T = T_\infty$.
  \item\label{Lemma_Rescaling_and_weak_morphisms_5} Let $T : E_\infty \to E'_\infty$ be a continuous linear operator and let $\mathbb{T} : \mathbb{E}^{\mathbf{k}} \to \mathbb{E}'$ and $\mathbb{S} : \mathbb{E}^{\mathbf{l}} \to \mathbb{E}'$ be extensions of $T$, for strictly monotone increasing sequences $\mathbf{k}, \mathbf{l} \subseteq \N_0$.
Then there exists a strictly monotone increasing sequence $\mathbf{m} \subseteq \N_0$ \st $\mathbf{m}\geq \mathbf{k}, \mathbf{l}$ and for any such sequence $\mathbf{m}$,
\[
\mathbb{T}\circ \mathbb{I}^{\mathbf{m}}_{\mathbf{k}} = \mathbb{S}\circ \mathbb{I}^{\mathbf{m}}_{\mathbf{l}} : \mathbb{E}^{\mathbf{m}} \to \mathbb{E}'\text{.}
\]
  \item\label{Lemma_Rescaling_and_weak_morphisms_6} $\id_{E_\infty} : E_\infty \to E_\infty$ is a weak embedding.
  \item\label{Lemma_Rescaling_and_weak_morphisms_7} Let $T : E_\infty \to E'_\infty$ and $T' : E'_\infty \to E''_\infty$ be weak morphisms.
Then their composition $T'\circ T : E_\infty \to E''_\infty$ is a weak morphism as well. \\
If $T$ and $T'$ are weak embeddings then so is $T'\circ T$.
  \item\label{Lemma_Rescaling_and_weak_morphisms_8} Let $T,T' : E_\infty \to E'_\infty$ be weak morphisms and let $\lambda, \mu \in \mathds{k}$.
Then $\lambda T + \mu T' : E_\infty \to E'_\infty$ is a weak morphism as well.
\end{enumerate}
\end{lemma}
\begin{proof}
\begin{enumerate}[label=\arabic*.,ref=\arabic*.]
  \item Obvious.
  \item Obvious.
  \item One direction is trivial, taking $\mathbf{l} \definedas (j)_{j\in\N_0}$. \\
For the other direction, assume that $T : E_\infty \to E'_\infty$ is a weak morphism and let $\mathbb{T} : \mathbb{E}^{\mathbf{k}} \to \mathbb{E}'^{\mathbf{l}}$ be an extension of $T_\infty$, where $\mathbf{k},\mathbf{l} \subseteq \N_0$ are strictly monotone increasing sequences.
Then $\mathbb{T}' \definedas \mathbb{I}'^{\mathbf{l}}\circ \mathbb{T} : \mathbb{E}^{\mathbf{k}} \to \mathbb{E}'$ is a continuous linear operator with $T'_\infty = I'^{\mathbf{l}}_\infty \circ T_\infty = \id_{E'_\infty}\circ T = T$.
  \item This follows from the proof of \labelcref{Lemma_Rescaling_and_weak_morphisms_3}~together with \labelcref{Lemma_Rescaling_and_weak_morphisms_2}.
  \item For existence of such a sequence $\mathbf{m}$ set $\mathbf{m} \definedas \left(\max\{k_j,l_j\}\right)_{j\in\N_0}$.
Now if $\mathbb{T}' \definedas \mathbb{T}\circ \mathbb{I}^{\mathbf{m}}_{\mathbf{k}}$ and $\mathbb{S}' \definedas \mathbb{S}\circ \mathbb{I}^{\mathbf{m}}_{\mathbf{l}}$, then $T'_\infty = S'_\infty$ and the claim follows from \cref{Lemma_Relation_mathbb_T_to_T_infty}.
  \item Obvious.
  \item Using \labelcref{Lemma_Rescaling_and_weak_morphisms_1}, let $\mathbb{T} : \mathbb{E}^{\mathbf{k}} \to \mathbb{E}'$ and $\mathbb{T}' : \mathbb{E}'^{\mathbf{l}} \to \mathbb{E}''$ be extensions of $T_\infty$ and $T'_\infty$, respectively, for strictly monotone increasing sequences $\mathbf{k}, \mathbf{l} \subseteq \N_0$.
Then $\mathbb{T}'' \definedas \mathbb{T}' \circ \mathbb{T}^{\mathbf{l}} : \left(\mathbb{E}^{\mathbf{k}}\right)^{\mathbf{l}} \to \mathbb{E}''$ is a continuous linear operator with $T''_\infty = T'\circ T$ \ie an extension of $T'\circ T$. \\
This formula also shows the last statement about embeddings.
  \item Let $\mathbb{T} : \mathbb{E}^{\mathbf{k}} \to \mathbb{E}'$ and $\mathbb{T}' : \mathbb{E}^{\mathbf{l}} \to \mathbb{E}'$ be extensions of $T$ and $T'$, respectively, for strictly monotone increasing sequences $\mathbf{k},\mathbf{l}\subseteq \N_0$.
Define $\mathbf{m} \definedas \left(\max\{k_j,l_j\}\right)_{j\in\N_0}$.
Then $\lambda\mathbb{T}\circ \mathbb{I}^{\mathbf{m}}_{\mathbf{k}} + \mu\mathbb{T}'^{\mathbf{m}}_{\mathbf{l}} : E^{\mathbf{m}} \to \mathbb{E}'$ is an extension of $\lambda T + \mu T' : E_\infty \to E'_\infty$.
\end{enumerate}
\end{proof}

\begin{example}\label{Example_Sobolev_chain_II}
Let $(\Sigma,g)$ be a closed $n$-dimensional Riemannian manifold and let $\pi : F \to \Sigma$ be a real (or complex) vector bundle equipped with a euclidean (or hermitian) metric and metric connection. \\
Let $k_0 \in \N_0$ and $1 < p < \infty$ be \st $k_0p > n$.
Let $\mathbf{k} \definedas (k_0 + k)_{k\in\N_0}$ and let $\mathbf{l} \subseteq \N_0$ be a strictly monotone increasing sequence with $l_0 \geq k_0$.
Define $\mathbf{l}' \definedas (l_j - k_0)_{j\in\N_0}$
Then the Sobolev chains from \cref{Example_Sobolev_chain_I} satisfy $\mathbb{W}^{\mathbf{l},p}(F) = \left(\mathbb{W}^{\mathbf{k},p}(F)\right)^{\mathbf{l}'}$. \\
Given $p \leq q < \infty$, for every $j \in \N_0$, the Sobolev embedding theorem there is a canonical continuous embedding $J_j : W^{l_j,q}(F) \hookrightarrow W^{l_j,p}(F)$ that is the identity on $\Gamma(F)$.
These fit together to an embedding $\mathbb{J} : \mathbb{W}^{\mathbf{l},q}(F) \to \mathbb{W}^{\mathbf{l},p}(F)$ with $J_\infty = \id_{\Gamma(F)}$.
\end{example}

\begin{example}\label{Example_PDO_II}
Let $(\Sigma,g)$ be a closed $n$-dimensional Riemannian manifold and let $\pi : F \to \Sigma$ and $\pi' : F' \to \Sigma$ be real (or complex) vector bundles equipped with euclidean (or hermitian) metrics and metric connections. \\
Let $k_0 \in \N_0$ and $1 < p < \infty$ be \st $k_0p > n$ and define $\mathbf{k} \definedas (k_0 + k)_{k\in\N_0}$. \\
Let $P : \Gamma(F) \to \Gamma(F')$ be a partial differential operator (with smooth coefficients) of class $m$.
Then $P : W^{\mathbf{k},p}(F)_\infty = \Gamma(F) \to \Gamma(F') = W^{\mathbf{k},p}(F')_\infty$ defines a weak morphism between $\mathbb{W}^{\mathbf{k},p}(F)$ and $\mathbb{W}^{\mathbf{k},p}(F')$.
For putting $\mathbf{l} \definedas (m + j)_{j\in\N_0}$, the continuous linear operator $\mathbb{P}^{\mathbf{k},p} : \mathbb{W}^{\mathbf{k} + m, p}(F) = \left(\mathbb{W}^{\mathbf{k}, p}(F)\right)^{\mathbf{l}} \to \mathbb{W}^{\mathbf{k},p}(F')$ from \cref{Example_PDO_I}, extends $P$. \\
In just the same way, the operator $\mathbb{P}^{\mathbf{k}} : \mathbbl{\Gamma}^{\mathbf{id} + m}(F) = \left(\mathbbl{\Gamma}^{\mathbf{id}}(F)\right)^{\mathbf{l}} \to \mathbbl{\Gamma}^{\mathbf{id}}(F')$ extends $P : \Gamma^{\mathbf{id}}(F)_\infty = \Gamma(F) \to \Gamma(F') = \Gamma^{\mathbf{id}}(F')_\infty$, showing that $P$ defines a weak morphism between $\mathbbl{\Gamma}^{\mathbf{id}}(F)$ and $\mathbbl{\Gamma}^{\mathbf{id}}(F')$ as well.
\end{example}

\begin{definition}
Let $\mathbb{E}$ and $\tilde{\mathbb{E}}$ be ILB- or sc-chains. \\
A weak embedding $J : E_\infty \to \tilde{E}_\infty$ is called an \emph{equivalence} if there exists another weak embedding $K : \tilde{E}_\infty \to E_\infty$ \st $J\circ K = \id_{\tilde{E}_\infty}$ and $K\circ J = \id_{E_\infty}$. \\
$\mathbb{E}$ and $\tilde{\mathbb{E}}$ are called \emph{equivalent} if there exists an equivalence between them.
\end{definition}

\begin{remark}\label{Remark_Extensions_of_weak_morphisms}
Note that the above definition of equivalence of ILB- or sc-chains does \emph{not} mean that one can find continuous linear operators $\mathbb{J} : \mathbb{E}^{\mathbf{k}} \to \tilde{\mathbb{E}}^{\mathbf{l}}$ and $\mathbb{K} : \tilde{\mathbb{E}}^{\mathbf{l}} \to \mathbb{E}^{\mathbf{k}}$ between rescalings \st $\mathbb{J}\circ \mathbb{K} : \tilde{\mathbb{E}}^{\mathbf{l}} \to \tilde{\mathbb{E}}^{\mathbf{l}}$ and $\mathbb{K}\circ \mathbb{T} : \mathbb{E}^{\mathbf{k}} \to \mathbb{E}^{\mathbf{k}}$ are the identity. \\
Instead, at least for sc-chains, by \cref{Lemma_Operators_covering_id} below and (the proof of) \cref{Lemma_Rescaling_and_weak_morphisms}, one can find $\mathbb{J} : \mathbb{E}^{\mathbf{k}} \to \tilde{\mathbb{E}}$ and $\mathbb{K} : \tilde{\mathbb{E}}^{\mathbf{l}} \to \mathbb{E}$ \st
\begin{align*}
\mathbb{K} \circ \mathbb{J}^{\mathbf{l}} &= \mathbb{I}^{\mathbf{k}\circ\mathbf{l}} : \mathbb{E}^{\mathbf{k}\circ\mathbf{l}} \to \mathbb{E}
\intertext{and}
\mathbb{J} \circ \mathbb{K}^{\mathbf{k}} &= \tilde{\mathbb{I}}^{\mathbf{l}\circ\mathbf{k}} : \tilde{\mathbb{E}}^{\mathbf{l}\circ\mathbf{k}} \to \tilde{\mathbb{E}}\text{.}
\end{align*}
See \eg \cref{Example_Sobolev_chain_III} below.
\end{remark}

Because of the previous remark and the following lemma, in the future I will restrict almost exclusively to sc-chains and there is no such thing as an $\overline{\text{ILB}}$-Fr{\'e}chet space appearing in this text.

\begin{lemma}\label{Lemma_Operators_covering_id}
Let $\mathbb{E}$ be an sc-chain, let $\mathbf{k}, \mathbf{l} \subseteq \N_0$ be strictly monotone increasing sequences and let $\mathbb{J} : \mathbb{E}^{\mathbf{k}} \to \mathbb{E}^{\mathbf{l}}$ be a morphism with $J_\infty = \id_{E_\infty}$.
\begin{enumerate}[label=\arabic*.,ref=\arabic*.]
  \item If $\dim E_\infty < \infty$, then $E_k = E_\infty$, $\iota_k = \id_{E_\infty}$ and $J_k = \id_{E_\infty}$ for all $k\in\N_0$.
  \item If $\dim E_\infty = \infty$, then $\mathbf{k} \geq \mathbf{l}$ and $\mathbb{J} = \mathbb{I}^{\mathbf{k}}_{\mathbf{l}}$.
\end{enumerate}
\end{lemma}
\begin{proof}
\begin{enumerate}[label=\arabic*.,ref=\arabic*.]
  \item If $\dim E_\infty < \infty$, because $E_\infty \subseteq E_k$ is dense for all $k\in\N_0$, $\dim E_k < \infty$ for all $k \in \N_0$ and $\iota^\infty_k : E_\infty \to E_k$ is the identity.
Hence also $E^{\mathbf{k}}_j = E_\infty = E^{\mathbf{l}}_j$ for all $j \in \N_0$ and it follows that $J_k = \id_{E_\infty}$ for all $k\in\N_0$ directly from the axioms for a continuous linear operator between sc-chains.
  \item $J_j : E^{\mathbf{k}}_j = E_{k_j} \to E_{l_j} = E^{\mathbf{l}}_j$ is a continuous linear operator between Banach spaces with $J_j|_{E_\infty} = \iota^\infty_{l_j}$.
Since $E_\infty \subseteq E_{k_j}$ is dense, this uniquely determines $J_j$.
If $k_j \geq l_j$, then $\left(\mathbb{I}^{\mathbf{k}}_{\mathbf{l}}\right)_j = \iota^{k_j}_{l_j} : E_{k_j} \to E_{l_j}$ is a continuous linear operator with $\iota^{k_j}_{l_j}|_{E_\infty} = \iota^\infty_{l_j}$, hence $J_j = \left(\mathbb{I}^{\mathbf{k}}_{\mathbf{l}}\right)_j$. \\
Let $k_j < l_j$ and assume that there exists a continuous linear operator $J_j : E_{k_j} \to E_{l_j}$ with $J_j|_{E_\infty} = \iota^\infty_{l_j}$.
Then $K \definedas \iota^{l_j}_{k_j}\circ J_j : E_{k_j} \to E_{k_j}$ is a compact (because $\iota^{l_j}_{k_j}$ is compact) linear operator with $K|_{E_\infty} = \iota^\infty_{k_j}$.
By uniqueness it follows that $K = \id_{E_{k_j}}$ is compact, hence $\dim E_{k_j} < \infty$ and also $\dim E_\infty < \infty$, a contradiction to the assumption $\dim E_\infty = \infty$.
\end{enumerate}
\end{proof}

\begin{lemma}\label{Lemma_Equivalence_of_chains}
Let $\mathbb{E}$ and $\mathbb{E}'$ be ILB- or sc-chains.
\begin{enumerate}[label=\arabic*.,ref=\arabic*.]
  \item Equivalence of ILB- or sc-chains is an equivalence relation.
  \item If $\mathbf{k} \subseteq \N_0$ is a strictly monotone increasing sequence, then $\mathbb{E}^{\mathbf{k}}$ and $\mathbb{E}$ are equivalent.
  \item Let $\tilde{\mathbb{E}}$ and $\tilde{\mathbb{E}}'$ be ILB- or sc-chains that are equivalent to $\mathbb{E}$ and $\mathbb{E}'$, respectively.
Then there is a $1$---$1$ correspondence between weak morphisms $T : E_\infty \to E'_\infty$ and weak morphisms $\tilde{T} : \tilde{E}_\infty \to \tilde{E}'_\infty$, mapping the identity to the identity (for $\mathbb{E} = \mathbb{E}'$ and $\tilde{\mathbb{E}} = \tilde{\mathbb{E}}'$), weak embeddings to weak embeddings and that is compatible with composition of weak morphisms.
\end{enumerate}
\end{lemma}
\begin{proof}
\begin{enumerate}[label=\arabic*.,ref=\arabic*.]
  \item $\mathbb{E}$ is equivalent to $\mathbb{E}$ via the identity (as a weak morphism) and equivalence is a symmetric relation by definition.
Transitivity follows because compositions of weak embeddings are weak embeddings by \cref{Lemma_Rescaling_and_weak_morphisms}.
  \item The weak embeddings in both directions are just tautologically given by the identity $\mathbb{E}^{\mathbf{k}} \to \mathbb{E}^{\mathbf{k}}$.
  \item Let $J : E_\infty \to \tilde{E}_\infty$, $K : \tilde{E}_\infty \to E_\infty$ and $J' : E'_\infty \to \tilde{E}'_\infty$, $K' : \tilde{E}'_\infty \to E'_\infty$ be as in the definition of equivalence for $\mathbb{E}, \tilde{\mathbb{E}}$ and $\mathbb{E}', \tilde{\mathbb{E}}'$, respectively. \\
Then for a weak morphism $T : E_\infty \to E'_\infty$ define $\tilde{T} : \tilde{E}_\infty \to \tilde{E}'_\infty$ by $\tilde{T} \definedas J'\circ T\circ K$. \\
Conversely, for a weak morphism $\tilde{T} : \tilde{E}_\infty \to \tilde{E}'_\infty$ define $T : E_\infty \to E'_\infty$ by $T \definedas K'\circ T\circ J$. \\
This has the desired properties by \cref{Lemma_Rescaling_and_weak_morphisms}.
\end{enumerate}
\end{proof}

\begin{example}\label{Example_Sobolev_chain_III}
Let $(\Sigma,g)$ be a closed $n$-dimensional Riemannian manifold and let $\pi : F \to \Sigma$ be a real (or complex) vector bundle equipped with a euclidean (or hermitian) metric and metric connection. \\
Let $k_0 \in \N_0$ and $1 < p < \infty$ be \st $k_0p > n$.
Let $\mathbf{k} \definedas (k_0 + k)_{k\in\N_0}$ and let $\mathbf{l} \subseteq \N_0$ be a strictly monotone increasing sequence with $l_0 \geq k_0$. \\
By \cref{Example_Sobolev_chain_II}, $\mathbb{W}^{\mathbf{k},p}(F)$ and $\mathbb{W}^{\mathbf{l},p}(F)$ are equivalent. \\
We want to show that for $1 < p \leq q < \infty$, $\mathbb{W}^{\mathbf{k},p}(F)$ and $\mathbb{W}^{\mathbf{k},q}(F)$ are equivalent as well.
From \cref{Example_Sobolev_chain_II} we already have the embedding $\mathbb{J} : \mathbb{W}^{\mathbf{l},q}(F) \to \mathbb{W}^{\mathbf{l},p}(F)$ with $J_\infty = \id_{\Gamma(F)}$.
So (using \cref{Example_Sobolev_chain_II}) one is left to construct another embedding $\mathbb{K} : \mathbb{W}^{\mathbf{l},p}(F) \to \mathbb{W}^{\mathbf{k},q}(F)$, for some strictly monotone increasing sequence $\mathbf{l} \subseteq \N_0$, satsifying $K_\infty = \id_{\Gamma(F)}$. \\
To do so, pick any $m\in\N_0$ with $m > n\frac{q-p}{pq}$ and define $\mathbf{l} \definedas (k_0 + m + k)_{k\in\N_0}$.
Then by the Sobolev embedding theorem, for any $j\in\N_0$ there is a canonical continuous embedding $K_j : W^{l_j,p}(F) \hookrightarrow W^{k_j,q}(F)$ that is the identity on $\Gamma(F)$.
These fit together to an embedding $\mathbb{K} : \mathbb{W}^{\mathbf{l},p}(F) \to \mathbb{W}^{\mathbf{k},q}(F)$ with $K_\infty = \id_{\Gamma(F)}$. \\
This, together with transitivity of equivalence of sc-chains, shows that the $\mathbb{W}^{\mathbf{k},p}(F)$ for different choices of $\mathbf{k}$ and $p$ are all equivalent.
\end{example}

\begin{example}\label{Example_Sobolev_and_Ck_chain}
Let $(\Sigma,g)$ be a closed $n$-dimensional Riemannian manifold and let $\pi : F \to \Sigma$ be a real (or complex) vector bundle equipped with a euclidean (or hermitian) metric and metric connection. \\
Let $k_0 \in \N_0$ and $1 < p < \infty$ be \st $k_0p > n$ and let $\mathbf{k} \definedas (k_0 + k)_{k\in\N_0}$.
Then $\mathbbl{\Gamma}^{\mathbf{id}}(F)$ and $\mathbb{W}^{\mathbf{k},p}(F)$ are equivalent: \\
Because $\Sigma$ was assumed closed, there is a canonical inclusion $\Gamma^{\mathbf{id}}(F)^{\mathbf{k}}_j = \Gamma^{k_0 + j}(F) \to W^{\mathbf{k},p}(F)_j = W^{k_0 + j,p}(F)$.
In the other direction, the Sobolev embedding theorem gives a canonical inclusion $W^{\mathbf{k},p}(F)_j = W^{k_0 + j,p}(F) \to \Gamma^j(F) = \Gamma^{\mathbf{id}}(F)_j$.
Both of these inclusions are the identity on $\Gamma(F) = \Gamma^{\mathbf{id}}(F)^{\mathbf{k}}_\infty = W^{\mathbf{k},p}(F)_\infty = \Gamma^{\mathbf{id}}(F)_\infty$. \\
Since for any strictly monotone increasing sequence $\mathbf{l} \subseteq \N_0$, $\mathbbl{\Gamma}^{\mathbf{l}}(F) = \left(\mathbbl{\Gamma}^{\mathbf{id}}\right)^{\mathbf{l}}$, this, together with \cref{Example_Sobolev_chain_III}, shows that the chains $\mathbb{W}^{\mathbf{l},p}(F)$ and $\mathbbl{\Gamma}^{\mathbf{l}}(F)$ are all equivalent, for any choice of strictly monotone increasing sequence $\mathbf{l} \subseteq \N_0$ and any $1 < p < \infty$ (\st $l_0p > n$).
\end{example}

\Needspace{25\baselineskip}
\subsection{$\overline{\text{sc}}$-structures}

\begin{definition}\label{Definition_sc_Frechet_space}
Let $E$ be a topological vector space.
\begin{enumerate}[label=\arabic*.,ref=\arabic*.]
  \item An \emph{$\overline{\text{sc}}$-structure} on $E$ is a pair $(\mathbb{E}, \phi)$, where $\mathbb{E}$ is an sc-chain and $\phi : E_\infty \to E$ is an isomorphism of topological vector spaces.
  \item Two $\overline{\text{sc}}$-structures $(\mathbb{E}, \phi)$ and $(\tilde{\mathbb{E}}, \tilde{\phi})$ on $E$ are called \emph{equivalent} if there exists an equivalence (of sc-chains) $J : E_\infty \to \tilde{E}_\infty$ with $\phi = \tilde{\phi}\circ J$. \\
Equivalently, if there exists an equivalence (of sc-chains) $K : \tilde{E}_\infty \to E_\infty$ with $\tilde{\phi} = \phi\circ K$.
\[
\xymatrix{
E_\infty \ar[r]^-{J} \ar[d]_-{\phi} \ar@{}[rd]|-{\circlearrowleft} & \tilde{E}_\infty \ar[d]^{\tilde{\phi}} \\
E \ar@{=}[r] & E
}
\qquad\qquad
\xymatrix{
\tilde{E}_\infty \ar[r]^-{K} \ar[d]_-{\tilde{\phi}} \ar@{}[rd]|-{\circlearrowleft} & E_\infty \ar[d]^{\phi} \\
E \ar@{=}[r] & E
}
\]
Note that in this case $J = \tilde{\phi}\inv\circ\phi$ and $K = J\inv = \phi\inv\circ\tilde{\phi}$ as continuous linear operators between topological vector spaces.
  \item An \emph{$\overline{\text{sc}}$-Fr{\'e}chet space} is a topological vector space $E$ together with an equivalence class of $\overline{\text{sc}}$-structures on it.
The sc-chains in this equivalence class are then called \emph{compatible}.
  \item A \emph{morphism} between $\overline{\text{sc}}$-Fr{\'e}chet spaces $E$ and $E'$ is a continuous linear operator $T : E \to E'$ between $E$ and $E'$ as topological vector spaces \st there exist compatible $\overline{\text{sc}}$-structures $(\mathbb{E}, \phi)$ and $(\mathbb{E}', \phi')$ on $E$ and $E'$, respectively, \st $T_\infty \definedas \phi'^{-1}\circ T\circ \phi : E_\infty \to E'_\infty$ defines a weak morphism.
\[
\xymatrix{
E_\infty \ar[r]^-{T_\infty} \ar[d]_-{\phi} \ar@{}[rd]|-{\circlearrowleft} & E'_\infty \ar[d]^{\phi'} \\
E \ar[r]^-{T} & E'
}
\]
  \item A morphism $T : E \to E'$ between $\overline{\text{sc}}$-Fr{\'e}chet spaces is called an \emph{embedding} if there exist compatible $\overline{\text{sc}}$-structures $(\mathbb{E}, \phi)$ and $(\mathbb{E}', \phi')$ on $E$ and $E'$, respectively, \st $\phi'^{-1}\circ T\circ \phi : E_\infty \to E'_\infty$ defines a weak embedding.
\end{enumerate}
\end{definition}

\begin{remark}
Note that the underlying topological vector space of an $\overline{\text{sc}}$-Fr{\'e}chet space is automatically a Fr{\'e}chet space.
\end{remark}

\begin{lemma}\label{Lemma_Morphisms_of_sc_Frechet_spaces}
Let $E$, $E'$ and $E''$ be $\overline{\text{sc}}$-Fr{\'e}chet spaces.
\begin{enumerate}[label=\arabic*.,ref=\arabic*.]
  \item\label{Lemma_Morphisms_of_sc_Frechet_spaces_1} A continuous linear operator $T : E \to E'$ defines a morphism of $\overline{\text{sc}}$-Fr{\'e}chet spaces \iff there exists a pair of compatible $\overline{\text{sc}}$-structures $(\tilde{\mathbb{E}}, \tilde{\phi})$ and $(\tilde{\mathbb{E}}', \tilde{\phi}')$ on $E$ and $E'$, respectively, and a continuous linear operator $\mathbb{T} : \tilde{\mathbb{E}} \to \tilde{\mathbb{E}}'$ \st $\tilde{T}_\infty = \tilde{\phi}'^{-1}\circ T\circ \tilde{\phi}$.
  \item\label{Lemma_Morphisms_of_sc_Frechet_spaces_2} A continuous linear operator $T : E \to E'$ defines a morphism of $\overline{\text{sc}}$-Fr{\'e}chet spaces \iff for any pair of compatible $\overline{\text{sc}}$-structures $(\tilde{\mathbb{E}}, \tilde{\phi})$ and $(\tilde{\mathbb{E}}', \tilde{\phi}')$ on $E$ and $E'$, respectively, $\tilde{T}_\infty \definedas \tilde{\phi}'^{-1}\circ T\circ \tilde{\phi}$ defines a weak morphism.
  \item\label{Lemma_Morphisms_of_sc_Frechet_spaces_3} A continuous linear operator $T : E \to E'$ defines an embedding of $\overline{\text{sc}}$-Fr{\'e}chet spaces \iff for any pair of compatible $\overline{\text{sc}}$-structures $(\tilde{\mathbb{E}}, \tilde{\phi})$ and $(\tilde{\mathbb{E}}', \tilde{\phi}')$ on $E$ and $E'$, respectively, $\tilde{T}_\infty \definedas \tilde{\phi}'^{-1}\circ T\circ \tilde{\phi}$ defines an embedding.
  \item\label{Lemma_Morphisms_of_sc_Frechet_spaces_4} $\id_E : E \to E$ is a morphism and embedding.
  \item\label{Lemma_Morphisms_of_sc_Frechet_spaces_5} Let $T, T' : E \to E'$ be morphisms and let $\lambda,\mu \in \mathds{k}$.
Then $\lambda T + \mu T' : E\to E'$ is a morphism as well.
  \item\label{Lemma_Morphisms_of_sc_Frechet_spaces_6} Let $T : E \to E'$ and $T' : E' \to E''$ be morphisms.
Then $T'\circ T : E \to E''$ is a morphism as well.
  \item\label{Lemma_Morphisms_of_sc_Frechet_spaces_7} Let $T : E \to E'$ and $T' : E' \to E''$ be embeddings.
Then $T'\circ T : E \to E''$ is an embedding as well.
\end{enumerate}
\end{lemma}
\begin{proof}
\begin{enumerate}[label=\arabic*.,ref=\arabic*.]
  \item\label{Lemma_Morphisms_of_sc_Frechet_spaces_Proof_1} Let $(\mathbb{E},\phi)$ and $(\mathbb{E}',\phi')$ be as in the definition of a morphism.
That $\tilde{T}_\infty \definedas \tilde{\phi}'^{-1}\circ T\circ \tilde{\phi}$ defines a weak morphism means by definition that there exist strictly monotone increasing sequences $\mathbf{k},\mathbf{l} \subseteq \N_0$ an a continuous linear operator $\mathbb{T} : \mathbb{E}^{\mathbf{k}} \to \mathbb{E}'^{\mathbf{l}}$ extending $T_\infty$.
But if $(\mathbb{E},\phi)$ and $(\mathbb{E}',\phi')$ are compatible $\overline{\text{sc}}$-structures on $E$ and $E'$, respectively, then so are $(\tilde{\mathbb{E}}, \tilde{\phi}) \definedas (\mathbb{E}^{\mathbf{k}},\phi)$ and $(\tilde{\mathbb{E}}', \tilde{\phi}') \definedas (\mathbb{E}'^{\mathbf{l}},\phi')$ by \cref{Lemma_Equivalence_of_chains}.
  \item\label{Lemma_Morphisms_of_sc_Frechet_spaces_Proof_2} Let $(\mathbb{E},\phi)$ and $(\mathbb{E}',\phi')$ be as in the definition of a morphism and let $K : \tilde{E}_\infty \to E_\infty$ and $J' : E'_\infty \to \tilde{E}'_\infty$ be as in the definition of equivalence for $(\mathbb{E},\phi)$ with $(\tilde{\mathbb{E}}, \tilde{\phi})$ and $(\mathbb{E}',\phi')$ with $(\tilde{\mathbb{E}}', \tilde{\phi}')$, respectively.
Then the claim follows from the following commutative diagram, where the composition of the upper row is given by $\tilde{T}_\infty$.
\[
\xymatrix{
\tilde{E}_\infty \ar[r]^-{K} \ar[d]_-{\tilde{\phi}} \ar@{}[rd]|-{\circlearrowleft} & E_\infty \ar[r]^-{T_\infty} \ar[d]_-{\phi} \ar@{}[rd]|-{\circlearrowleft} & E'_\infty \ar[r]^-{J'} \ar[d]^-{\phi'} \ar@{}[rd]|-{\circlearrowleft} & \tilde{E}'_\infty \ar[d]^-{\tilde{\phi}'} \\
E \ar@{=}[r] & E \ar[r]^-{T} & E' \ar@{=}[r] & E'
}
\]
  \item\label{Lemma_Morphisms_of_sc_Frechet_spaces_Proof_3} Follows as in \labelcref{Lemma_Morphisms_of_sc_Frechet_spaces_Proof_2}~because in the above diagram $K$ and $J'$ are weak embeddings.
  \item\label{Lemma_Morphisms_of_sc_Frechet_spaces_Proof_4} Trivial.
  \item\label{Lemma_Morphisms_of_sc_Frechet_spaces_Proof_5} Immediate from \labelcref{Lemma_Morphisms_of_sc_Frechet_spaces_2}~and \cref{Lemma_Rescaling_and_weak_morphisms}.
  \item\label{Lemma_Morphisms_of_sc_Frechet_spaces_Proof_6} Because $T$ and $T'$ are morphisms, by using \labelcref{Lemma_Morphisms_of_sc_Frechet_spaces_2}~for any compatible $\overline{\text{sc}}$-structures $(\mathbb{E},\phi)$, $(\mathbb{E}',\phi')$ and $(\mathbb{E}'',\phi'')$ on $E$, $E'$ and $E''$, respectively,
$T_\infty \definedas \phi'^{-1}\circ T\circ \phi : E_\infty \to E'_\infty$ and $T'_\infty \definedas \phi''^{-1}\circ T'\circ \phi' : E'_\infty \to E''_\infty$ define weak morphisms.
Then so does $T''_\infty \definedas T'_\infty\circ T_\infty = \phi''^{-1}\circ (T'\circ T)\circ \phi : E_\infty \to E''_\infty$ by \cref{Lemma_Rescaling_and_weak_morphisms}.
Hence $T'\circ T$ is a morphism.
  \item\label{Lemma_Morphisms_of_sc_Frechet_spaces_Proof_7} Follows from the proof of \labelcref{Lemma_Morphisms_of_sc_Frechet_spaces_6}~and \cref{Lemma_Rescaling_and_weak_morphisms}.
\end{enumerate}
\end{proof}

\begin{example}\label{Example_Finite_dimensional_sc_Frechet_space}
Let $E$ be a finite dimensional (real or complex) vector space.
By \cref{Lemma_Operators_covering_id}, $E$ defines an $\overline{\text{sc}}$-Fr{\'e}chet space in a unique way, where a compatible $\overline{\text{sc}}$-structure on $E$ is given by $(\mathbb{E}, \id_E)$, where $E_j \definedas E$, $\iota_j \definedas \id_E$ and $\|\cdot\|_j$ is an arbitrary norm on $E$.
\end{example}

\begin{example}\label{Example_Sobolev_chain_IV}
Let $(\Sigma,g)$ be a closed $n$-dimensional Riemannian manifold and let $\pi : F \to \Sigma$ be a real (or complex) vector bundles equipped with a euclidean (or hermitian) metric and metric connection. \\
Let $\Gamma(F)$ be the space of smooth sections of $F$ equipped with the $\mathcal{C}^\infty$-topology. \\
By \crefrange{Example_Sobolev_chain_I}{Example_Sobolev_chain_III}, for any $1 < p < \infty$ and any strictly monotone increasing sequence $\mathbf{k} \subseteq \N_0$, $\left(\mathbb{W}^{\mathbf{k},p}(F), \id_{\Gamma(F)}\right)$ defines an $\overline{\text{sc}}$-structure on $\Gamma(F)$ and these $\overline{\text{sc}}$-structures are all equivalent. \\
$\Gamma(F)$ together with the $\overline{\text{sc}}$-Fr{\'e}chet space structure given by the equivalence class these $\overline{\text{sc}}$-structures define will be denoted by $W(F)$ and called the \emph{Sobolev $\overline{\text{sc}}$-space of sections of $F$}.
\end{example}

\begin{example}\label{Example_Ck_chain_II}
Let $(\Sigma,g)$ be a closed $n$-dimensional Riemannian manifold and let $\pi : F \to \Sigma$ be a real (or complex) vector bundles equipped with a euclidean (or hermitian) metric and metric connection. \\
Let $\Gamma(F)$ be the space of smooth sections of $F$ equipped with the $\mathcal{C}^\infty$-topology. \\
By \cref{Example_Ck_chain_I}, for any strictly monotone increasing sequence $\mathbf{k} \subseteq \N_0$, $\left(\mathbbl{\Gamma}^{\mathbf{k}}(F), \id_{\Gamma(F)}\right)$ defines an $\overline{\text{sc}}$-structure on $\Gamma(F)$ and these $\overline{\text{sc}}$-structures are all equivalent.
By abuse of notation, $\Gamma(F)$ together with the $\overline{\text{sc}}$-Fr{\'e}chet space structure given by the equivalence class these $\overline{\text{sc}}$-structures define will be denoted by $\Gamma(F)$ and called the \emph{$\overline{\text{sc}}$-space of sections of $F$}.
\end{example}

\begin{example}\label{Example_Sobolev_and_Ck_sc_space}
By \cref{Example_Sobolev_and_Ck_chain},
\[
W(F) = \Gamma(F)
\]
as $\overline{\text{sc}}$-spaces and this is an actual identity, not just an isomorphism.
\end{example}

\begin{example}\label{Example_PDO_III}
Let $(\Sigma,g)$ be a closed $n$-dimensional Riemannian manifold and let $\pi : F \to \Sigma$ and $\pi' : F' \to \Sigma$ be real (or complex) vector bundles equipped with euclidean (or hermitian) metrics and metric connections. \\
Let $P : \Gamma(F) \to \Gamma(F')$ be a partial differential operator (with smooth coefficients).
Then by \cref{Example_PDO_I,Example_PDO_II}, $P$ defines a morphism $P : W(F) \to W(F')$ between the Sobolev $\overline{\text{sc}}$-spaces of sections of $F$ and $F'$. \\
Which is the same as saying that $P : \Gamma(F) \to \Gamma(F')$ defines a morphism between the $\overline{\text{sc}}$-spaces of sections of $F$ and $F'$.
\end{example}

\begin{definition}
$\overline{\mathbf{sc}}\mathbf{Frechet}^{\mathds{k}}$ is the (preadditive) category with objects the $\overline{\text{sc}}$-Fr{\'e}chet spaces and morphisms the morphisms of $\overline{\text{sc}}$-Fr{\'e}chet spaces.
It comes with a faithful underlying functor $\mathbf{U} : \overline{\mathbf{sc}}\mathbf{Frechet}^{\mathds{k}} \to \mathbf{Frechet}^{\mathds{k}}$.
\end{definition}

\Needspace{25\baselineskip}
\subsection{Subspaces and direct sums}\label{Subsection_Subspaces_direct_sums}

\begin{definition}[\cite{MR2341834}, Definition 2.5]\label{Definition_Subchains_and_sums}
Let $\mathbb{E}$, $\mathbb{E}'$ and $\mathbb{E}^1, \dots, \mathbb{E}^k$, for some $k\in\N_0$, be ILB- or sc-chains.
\begin{enumerate}[label=\arabic*.,ref=\arabic*.]
  \item $\mathbb{E}'$ is called a \emph{subchain of $\mathbb{E}$} if for all $j\in\N_0$, $E'_j \subseteq E_j$ is a closed linear subspace, $\|\cdot\|'_j = \|\cdot\|_j|_{E'_j}$ and $\iota'_j = \iota_j|_{E'_{j+1}}$.
  \item The \emph{direct sum} (also: \emph{biproduct}, \emph{direct product}) of the $\mathbb{E}^i$, $i=1,\dots,k$, is the ILB- or sc-chain $\mathbb{E}^1 \oplus \cdots \oplus \mathbb{E}^k \definedas \mathbb{E}''$ with
\begin{align*}
E''_j &\definedas E^1_j\oplus\cdots\oplus E^k_j\text{,} \\
\|\cdot\|''_j &\definedas \|\cdot\|^1_j + \cdots + \|\cdot\|^k_j\text{ and} \\
\iota''_j &\definedas \iota^1_j \oplus \cdots\oplus \iota^k_j\text{.}
\end{align*}
It satisfies $E''_\infty = E^1_\infty \times\cdots\times E^k_\infty$ (with the product topology) and comes with the continuous linear operators
\begin{align*}
\mathbb{P}^i : \mathbb{E}^1 \oplus \cdots \oplus \mathbb{E}^k &\to \mathbb{E}^i \\
P^i_j(e^1, \dots, e^k) &\definedas e^i\quad\forall\, (e^1,\dots,e^k)\in E^1_j\oplus\cdots\oplus E^k_j \\
\mathbb{J}^i : \mathbb{E}^i &\to \mathbb{E}^1 \oplus \cdots \oplus \mathbb{E}^k \\
J^i_j(e) &\definedas (0, \dots, 0, e, 0, \dots, 0) \in E^1_j\oplus\cdots\oplus E^k_j \quad\forall\, e \in E^i_j
\end{align*}
for $i = 1, \dots, k$, called the projection and injection operators, respectively.
  \item A subchain $\mathbb{E}'$ of $\mathbb{E}$ is said to \emph{split} if there exists another subchain $\mathbb{E}''$ of $\mathbb{E}$ \st the canonical continuous linear operator
\begin{align*}
\mathbb{J} : \mathbb{E}' \oplus \mathbb{E}'' &\to \mathbb{E} \\
J_j(e,e') &\definedas e + e'\quad \text{for $(e,e')\in E'_j\oplus E''_j$}
\end{align*}
is an isomorphism.
\end{enumerate}
\end{definition}

\begin{remark}\label{Remark_Subchain_determined_by_Ezero}
Note that a subchain $\mathbb{E}'$ of an ILB- or sc-chain $\mathbb{E}$ is uniquely determined by the closed linear subspace $E'_0 \subseteq E_0$, for $E'_j = \bigl(\iota^j_0\bigr)\inv(E'_0)$ for all $j\in\N_0$ and $\|\cdot\|'_j$ and $\iota'_j$ are simply defined via restriction.
\end{remark}

\begin{example}\label{Example_Finite_dimensional_split_subchain}
Let $\mathbb{E}$ be an ILB- or sc-chain and let $C \subseteq E_\infty$ be a finite dimensional subspace (as vector spaces).
Then $\mathbb{E}' \definedas \left((C, \|\cdot\|_j|_{C}), \id_C\right)_{j\in\N_0}$ is a split subchain of $\mathbb{E}$.
For a proof, see \cite{MR2341834}, Proposition 2.7.
\end{example}

\begin{example}\label{Example_Rescaling_of_subchain_is_subchain}
Let $\mathbb{E}$ be an ILB- or sc-chain and let $\mathbb{E}'$ be a subchain of $\mathbb{E}$.
If $\mathbf{k} \subseteq \N_0$ is a strictly monotone increasing sequence, then $\mathbb{E}'^{\mathbf{k}}$ is a subchain of $\mathbb{E}^{\mathbf{k}}$.
\end{example}

\begin{example}\label{Example_Inverse_image_subchain}
Let $\mathbb{E}$ and $\mathbb{E}'$ be ILB- or sc-chains and let $\mathbb{T} : \mathbb{E} \to \mathbb{E}'$ be a continuous linear operator.
Let furthermore $\tilde{\mathbb{E}}'$ be a subchain of $\mathbb{E}'$.
Then for all $j\in\N_0$, $\tilde{E}_j \definedas T_j\inv(\tilde{E}'_j)$ is a closed linear subspace of $E_j$ and from $\iota'_j\circ T_{j+1} = T_j\circ \iota_j : E_{j+1} \to E'_j$ it follows that $\iota_j\inv\left(T_j\inv(\tilde{E}'_j)\right) = T_{j+1}\inv\left((\iota'_j)^{-1}(\tilde{E}'_j)\right) = T_{j+1}\inv\left(\tilde{E}'_{j+1}\right)$.
Unfortunately, one cannot a priori guarantee that $\tilde{E}_\infty$ lies dense in $\tilde{E}_j$ for all $j\in\N_0$, so the $\left(\tilde{E}_j\right)_{j\in\N_0}$ need \emph{not} form a subchain of $\mathbb{E}$.
\end{example}

\begin{example}\label{Example_Splitting_along_submanifold_I}
Let $(\Sigma,g)$ be a closed $n$-dimensional Riemannian manifold and let $\pi : F \to \Sigma$ be a real (or complex) vector bundle equipped with a euclidean (or hermitian) metric and metric connection.
Also let $\mathbf{l} = (l_k)_{k \in \N_0} \subseteq \N_0$ be a strictly monotone increasing sequence and let $m \in \N_0$ and $1 < p < \infty$ be \st $l_0p > n$ and $mp > n$. \\
For a closed submanifold $C \subseteq \Sigma$, by the Sobolev trace theorem, there exists a continuous linear operator
\begin{align*}
\mathbb{W}^{\mathbf{l} + m,p}(F) &\to \mathbb{W}^{\mathbf{l},p}(F|_{C}) \\
u &\mapsto u|_C
\end{align*}
between sc-chains from \cref{Example_Sobolev_chain_I}. \\
There is then a well-defined subchain $\mathbb{W}^{\mathbf{l} + m,p}(F;C)$, defined via the closed linear subspaces
\[
W^{l_j + m,p}(F;C) \definedas \{u \in W^{l_j + m, p}(F) \;|\; u|_C = 0\}\text{.}
\]
Analogously, for an arbitrary strictly monotone increasing sequence $\mathbf{l} = (l_k)_{k \in \N_0} \subseteq \N_0$, there exists a continuous linear operator
\begin{align*}
\mathbbl{\Gamma}^{\mathbf{l}}(F) &\to \mathbbl{\Gamma}^{\mathbf{l}}(F|_{C}) \\
u &\mapsto u|_C
\end{align*}
between sc-chains from \cref{Example_Ck_chain_I}. \\
There is then a well-defined subchain $\mathbbl{\Gamma}^{\mathbf{l}}(F;C)$, defined via the closed linear subspaces
\[
\Gamma^{l_j}(F;C) \definedas \{u \in \Gamma^{l_j}(F) \;|\; u|_C = 0\}\text{.}
\]
\end{example}

\begin{lemma}
$\mathbf{scChains}^{\mathds{k}}$ is an additive category.
\end{lemma}
\begin{proof}
This is a straightforward verification of the axioms, showing that the biproduct from \cref{Definition_Subchains_and_sums} is indeed a biproduct.
\end{proof}

\begin{definition}\label{Definition_sc_subspace}
Let $E$ be an $\overline{\text{sc}}$-Fr{\'e}chet space.
A closed linear subspace $E' \subseteq E$ is called an \emph{$\overline{\text{sc}}$-subspace} if there exists a compatible $\overline{\text{sc}}$-structure $(\mathbb{E}, \phi)$ on $E$ and a subchain $\mathbb{E}'$ of $\mathbb{E}$ \st $\phi(E'_\infty) = E'$.
\end{definition}

\begin{lemma}\label{Lemma_Characterisation_sc_subspace}
Let $E$ be an $\overline{\text{sc}}$-Fr{\'e}chet space and let $E' \subseteq E$ be a closed linear subspace.
The following are equivalent:
\begin{enumerate}[label=\arabic*.,ref=\arabic*.]
  \item\label{Lemma_Characterisation_sc_subspace_1} $E'$ is an $\overline{\text{sc}}$-subspace.
  \item\label{Lemma_Characterisation_sc_subspace_2} For every compatible $\overline{\text{sc}}$-structure $(\mathbb{E}, \phi)$ on $E$ there exists a shift $\mathbf{k} \subseteq \N_0$ \st for the compatible $\overline{\text{sc}}$-structure $(\mathbb{E}^{\mathbf{k}}, \phi)$ on $E$, the shifted sc-chain $\mathbb{E}^{\mathbf{k}}$ has a subchain $\mathbb{E}'$ with $\phi(E'_\infty) = E'$.
\end{enumerate}
\end{lemma}
\begin{proof}
The direction ``\labelcref{Lemma_Characterisation_sc_subspace_2}$\;\Rightarrow\;$\labelcref{Lemma_Characterisation_sc_subspace_1}'' is trivial. \\
In the other direction, let $(\mathbb{E},\phi)$ be a compatible $\overline{\text{sc}}$-structure on $E$.
By \cref{Remark_Subchain_determined_by_Ezero}, it suffices to find $k_0\in\N_0$ and a closed linear subspace $E'_{k_0} \subseteq E_{k_0}$ \st $\phi\left(\bigl(\iota^\infty_{k_0}\bigr)\inv(E'_{k_0})\right) = E'$. \\
Because $E'$ is an $\overline{\text{sc}}$-subspace, there exists a compatible $\overline{\text{sc}}$-structure $(\tilde{\mathbb{E}}, \tilde{\phi})$ on $E$ and a subchain $\tilde{\mathbb{E}}'$ of $\tilde{\mathbb{E}}$ \st $\tilde{\phi}(\tilde{E}'_\infty) = E'$.
By compatibility of $(\mathbb{E},\phi)$ and $(\tilde{\mathbb{E}}, \tilde{\phi})$ there exists a strictly monotone increasing sequence $\mathbf{k} \subseteq \N_0$ and a continuous linear operator $\mathbb{J} : \mathbb{E}^{\mathbf{k}} \to \tilde{\mathbb{E}}$ \st $J_\infty\inv = \phi\inv\circ \tilde{\phi}$.
Define $E'_{k_0} \definedas J_0\inv(\tilde{E}'_0) \subseteq E^{\mathbf{k}}_0 = E_{k_0}$.
Then
\begin{align*}
\phi\left(\bigl(\iota^\infty_{k_0}\bigr)\inv(E'_{k_0})\right) &= \phi\left(\bigl(\iota^\infty_{k_0}\bigr)\inv(J_0\inv(\tilde{E}'_{0}))\right) \\
&= \phi\left(\bigl(J_0\circ\iota^\infty_{k_0}\bigr)\inv(\tilde{E}'_{0})\right) \\
&= \phi\left(J_\infty\inv(\tilde{E}'_{\infty})\right) \\
&= \phi\circ\phi\inv\circ\tilde{\phi}\left(\tilde{E}'_\infty\right) \\
&= E'\text{.}
\end{align*}
\end{proof}

\begin{definition}\label{Definition_Sum_of_scFrechet_spaces}
Let $E^1, \dots, E^k$, for some $k\in\N_0$, be $\overline{\text{sc}}$-Fr{\'e}chet spaces.
The \emph{direct sum} (also: \emph{biproduct}, \emph{direct product}) of the $E^i$, $i=1,\dots,k$, is the $\overline{\text{sc}}$-Fr{\'e}chet space $E^1 \oplus \cdots \oplus E^k$ where the underlying Hausdorff locally convex topological vector space is $E^1\times\cdots\times E^k$ with the product topology and the equivalence class of compatible $\overline{\text{sc}}$-structures on $E^1\times\cdots\times E^k$ is generated by the following $\overline{\text{sc}}$-structures on $E^1\times \cdots\times E^k$: \\
For every $i = 1,\dots,k$ pick a compatible $\overline{\text{sc}}$-structure $(\mathbb{E}^i,\phi^i)$ on $E^i$.
Then $\mathbb{E} \definedas \mathbb{E}^1\oplus \cdots\oplus \mathbb{E}^k$ is an sc-chain and $\phi \definedas \phi^1 \times \cdots\times \phi^k : \mathbb{E}_\infty = E^1_\infty \times \cdots \times E^k_\infty \to E^1 \times \cdots\times E^k$ is a homeomorphism.
$(\mathbb{E},\phi)$ defines an $\overline{\text{sc}}$-structure on $E^1\times \cdots\times E^k$ and for any two such choices the resulting $\overline{\text{sc}}$-structures are compatible by a straightforward verification.
The resulting equivalence class of $\overline{\text{sc}}$-structures on $E^1\times \cdots\times E^k$ then defines the $\overline{\text{sc}}$-Fr{\'e}chet space $E^1\oplus\cdots\oplus E^k$. \\
It comes with the morphisms
\begin{align*}
P^i : E^1 \oplus \cdots \oplus E^k &\to E^i \\
P^i(e^1, \dots, e^k) &\definedas e^i\quad\forall\, (e^1,\dots,e^k)\in E^1\times\cdots\times E^k \\
J^i : E^i &\to E^1 \oplus \cdots \oplus E^k \\
J^i(e) &\definedas (0, \dots, 0, e, 0, \dots, 0) \in E^1\oplus\cdots\oplus E^k \quad\forall\, e \in E^i
\end{align*}
for $i = 1, \dots, k$, called the projection and injection morphisms, respectively.
\end{definition}

\begin{example}\label{Example_Splitting_along_submanifold_II}
Let $(\Sigma,g)$ be a closed $n$-dimensional Riemannian manifold, let $C \subseteq \Sigma$ be a closed submanifold and let $\pi : F \to \Sigma$ be a real (or complex) vector bundle equipped with a euclidean (or hermitian) metric and metric connection. \\
By \cref{Example_Splitting_along_submanifold_I}, there is a well-defined morphism
\begin{align*}
W(F) &\to W(F|_C) \\
\Gamma(F) \ni u &\mapsto u|_C \in \Gamma(F|_C)
\end{align*}
between $\overline{\text{sc}}$-Fr{\'e}chet spaces from \cref{Example_Sobolev_chain_IV} and also a well-defined $\overline{\text{sc}}$-subspace $W(F;C) \subseteq W(F)$ with underlying topological vector space $\Gamma(F;C) \definedas \{u \in \Gamma(F) \;|\; u|_C = 0\}$. \\
Now let $T_C$ be a metric tubular neighbourhood of $C$, \ie $T_C$ is the diffeomorphic image of an $\varepsilon$-neighbourhood, for some $\varepsilon > 0$, of the zero section in
\[
T^\perp_C \Sigma \definedas \{\xi \in T_z\Sigma \;|\; z \in C, \langle \xi, v \rangle = 0 \;\forall\, v \in T_zC\}
\]
under the map
\begin{align*}
T^\perp_C \Sigma &\to \Sigma \\
\xi &\mapsto \exp(\xi)\text{.}
\end{align*}
It comes with a canonical projection $\pi^{T_C}_C : T_C \to C$ which corresponds to the vector bundle projection $T^\perp_C \Sigma \to C$ under the above diffeomorphism and a map $\rho : T_C \to [0,\varepsilon)$, $\rho(\exp_z(\xi)) \definedas |\xi|$.
Furthermore, there is a canonical identification
\begin{align*}
\left(\pi^{T_C}_C\right)^\ast F|_C &\to F|_{T_C} \\
(e_z, \exp_z(\xi)) &\mapsto P_\xi e_z\text{,}
\end{align*}
where $P_\xi : F_z \to F_{\exp_z(\xi)}$ denotes parallel transport along the geodesic $[0,1] \to \Sigma$, $t \mapsto \exp_z(t\xi)$. \\
Given a smooth map $\phi : \R \to \R$ \st $\phi(t) \equiv 1$ for $t \leq \varepsilon/3$ and $\phi(t) \equiv 0$ for $t \geq 2\varepsilon/3$, using the above one can define 
\begin{align*}
\Phi : \Gamma(F) &\to \Gamma(F|_C) \times \Gamma(F;C) \\
u &\mapsto \left(u|_C, u - (\phi\circ\rho)\cdot \left(\pi^{T_C}_C\right)^\ast u|_C\right)
\intertext{and}
\Psi : \Gamma(F|_C) \times \Gamma(F;C) &\to \Gamma(F) \\
(u_0, u_1) &\mapsto u_1 + (\phi\circ\rho)\cdot \left(\pi^{T_C}_C\right)^\ast u_0\text{.}
\end{align*}
These maps are inverse to each other and have extensions
\begin{align*}
\mathbb{W}^{\mathbf{l} + m,p}(F) &\to \mathbb{W}^{\mathbf{l},p}(F|_C) \oplus \mathbb{W}^{\mathbf{l},p}(F;C)
\intertext{and}
\mathbb{W}^{\mathbf{l},p}(F|_C)\oplus \mathbb{W}^{\mathbf{l},p}(F; C) &\to \mathbb{W}^{\mathbf{l},p}(F)\text{,}
\end{align*}
where $\mathbf{l} \subseteq \N_0$ is a strictly monotone increasing sequence, $m \in \N_0$ and $1 < p < \infty$ are \st $l_0p > n$ and $mp > n$. \\
This shows that there is an $\overline{\text{sc}}$-splitting
\[
W(F) \cong W(F|_C)\oplus W(F;C)\text{.}
\]
In just the same way as above, $\Phi$ and $\Psi$ also have extensions
\begin{align*}
\mathbbl{\Gamma}^{\mathbf{l}}(F) &\to \mathbbl{\Gamma}^{\mathbf{l}}(F|_C) \oplus \mathbbl{\Gamma}^{\mathbf{l}}(F;C)
\intertext{and}
\mathbbl{\Gamma}^{\mathbf{l}}(F|_C)\oplus \mathbbl{\Gamma}^{\mathbf{l}}(F; C) &\to \mathbbl{\Gamma}^{\mathbf{l}}(F)\text{,}
\end{align*}
where $\mathbf{l} \subseteq \N_0$ is any strictly monotone increasing sequence. \\
This shows that there is an $\overline{\text{sc}}$-splitting
\[
\Gamma(F) \cong \Gamma(F|_C)\oplus \Gamma(F;C)
\]
using the $\overline{\text{sc}}$-structures on $\Gamma(F) = W(F)$, etc., defined via chains of continuously differentiable sections.
\end{example}

\begin{lemma}
$\mathbf{scFrechet}^{\mathds{k}}$ is an additive category.
\end{lemma}
\begin{proof}
This is a straightforward verification of the axioms, showing that the biproduct from \cref{Definition_Sum_of_scFrechet_spaces} is indeed a biproduct.
\end{proof}

\begin{definition}
Let $E$ be an $\overline{\text{sc}}$-Fr{\'e}chet space.
An $\overline{\text{sc}}$-subspace $E'$ of $E$ is said to \emph{split} if there exists another $\overline{\text{sc}}$-subspace $E''$ of $E$ \st the canonical morphism
\begin{align*}
J : E' \oplus E'' &\to E \\
(e,e') &\mapsto e + e'
\end{align*}
is an isomorphism.
\end{definition}

\begin{example}\label{Example_Finite_dimensional_split_sc_subspace}
Let $E$ be an $\overline{\text{sc}}$-Fr{\'e}chet space and let $C \subseteq E$ be a finite dimensional subspace (as vector spaces).
Then $C$ is a split $\overline{\text{sc}}$-subspace.
This follows immediately from \cref{Example_Finite_dimensional_split_subchain}.
\end{example}

\Needspace{25\baselineskip}
\subsection{(Strongly) smoothing operators}

\begin{definition}\label{Definition_Compact_operator}
Let $E$ and $E'$ be Hausdorff locally convex topological vector spaces and let $K : E \to E'$ be a continuous linear operator.
$K$ is called \emph{compact} if there exists a neighbourhood $U\subseteq E$ of $0$ \st the closure $\overline{K(U)}$ of $K(U)$ in $E'$ is compact.
\end{definition}

\begin{remark}
If $K : E \to E'$ is a compact linear operator and $A \subseteq E$ is a bounded subset, then $\overline{K(A)} \subseteq E'$ is compact. \\
For by definition of $K$ compact there exists a neighbourhood $U \subseteq E$ of $0$ \st $\overline{K(U)} \subseteq E'$ is compact and by definition of $A$ bounded there exists $c \in \mathds{k}$ \st $A \subseteq cU$.
Hence $K(A) \subseteq K(cU) = cK(U)$ and $\overline{cK(U)} = c\overline{K(U)}$ is compact.
So $\overline{K(A)} \subseteq c\overline{K(U)}$ is compact as a closed subset of a compact set.
\end{remark}

\begin{lemma}\label{Lemma_Compact_operators}
Let $E$, $E'$ and $E''$ be Fr{\'e}chet spaces and let $K,K' : E \to E'$, $S : E'' \to E$ and $T : E \to E''$ be continuous linear operators with $K$ and $K'$ compact.
Let furthermore $\lambda,\mu \in \mathds{k}$.
Then the following operators are compact:
\begin{enumerate}[label=\arabic*.,ref=\arabic*.]
  \item $K\circ S : E'' \to E'$.
  \item $T \circ K : E \to E''$.
  \item $\lambda K + \mu K' : E \to E'$.
\end{enumerate}
\end{lemma}
\begin{proof}
See \cite{MR3154940}, Chapter 5, Exercises 26 and 27.
\end{proof}

\begin{definition}
Let $\mathbb{E}$ and $\mathbb{E}'$ be sc-chains. \\
A continuous linear operator $\mathbb{K} : \mathbb{E} \to \mathbb{E}'$ is called
\begin{enumerate}[label=\arabic*.,ref=\arabic*.]
  \item \emph{smoothing} if there exists a strictly monotone increasing sequence $\mathbf{k} \subseteq \N_0$ and a continuous linear operator $\tilde{\mathbb{K}} : \mathbb{E} \to \mathbb{E}'^{\mathbf{k}}$ \st $\mathbf{k} > \mathbf{id}$ and $\mathbb{K} = \mathbb{I}'^{\mathbf{k}} \circ \tilde{\mathbb{K}}$.
  \item \emph{strongly smoothing} if for all strictly monotone increasing sequences $\mathbf{k} \subseteq \N_0$ there exists a continuous linear operator $\tilde{\mathbb{K}}_{\mathbf{k}} : \mathbb{E} \to \mathbb{E}'^{\mathbf{k}}$ \st $\mathbb{K} = \mathbb{I}'^{\mathbf{k}}\circ \tilde{\mathbb{K}}_{\mathbf{k}}$.
\end{enumerate}
\end{definition}

\begin{example}
Let $\mathbb{E}$ be an sc-chain and let $\mathbf{k} \subseteq \N_0$ be strictly monotone increasing sequences.
If $\mathbf{k} > \mathbf{id}$, then $\mathbb{I}^{\mathbf{k}} : \mathbb{E}^{\mathbf{k}} \to \mathbb{E}$ is smoothing but not strongly smoothing.
\end{example}

\begin{remark}\label{Remark_Smoothing_factors_through_Eone}
Note that a continuous linear operator $\mathbb{K} : \mathbb{E} \to \mathbb{E}'$ is smoothing \iff there exists a continuous linear operator $\tilde{\mathbb{K}} : \mathbb{E} \to \mathbb{E}'^{\mathbf{1}}$, where $\mathbf{1} \definedas (k+1)_{k\in\N_0}$.
Or in other words, $K_j = \iota'_j\circ \overline{K}_j$, for a continuous linear operator $\overline{K}_j : E_j \to E'_{j+1}$, for all $j\in\N_0$. \\
For one direction is trivial and in the other direction, $\mathbf{k} > \mathbf{id}$ means precisely that $\mathbf{k} \geq \mathbf{1}$.
So given $\tilde{\mathbb{K}} : \mathbb{E} \to \mathbb{E}'^{\mathbf{k}}$, define $\overline{\mathbb{K}} \definedas \mathbb{I}'^{\mathbf{k}}_{\mathbf{1}} : \mathbb{E} \to \mathbb{E}'^{\mathbf{1}}$.
\end{remark}

\begin{lemma}\label{Lemma_Characterisation_strongly_smoothing}
Let $\mathbb{E}$ and $\mathbb{E}'$ be sc-chains and let $\mathbb{K} : \mathbb{E} \to \mathbb{E}'$ be a continuous linear operator.
The following are equivalent:
\begin{enumerate}[label=\arabic*.,ref=\arabic*.]
  \item\label{Lemma_Characterisation_strongly_smoothing_1} $\mathbb{K}$ is strongly smoothing.
  \item\label{Lemma_Characterisation_strongly_smoothing_1b} For every shift $\mathbf{k} \subseteq \N_0$ there exists a continuous linear operator $\tilde{\mathbb{K}}_{\mathbf{k}} : \mathbb{E} \to \mathbb{E}'^{\mathbf{k}}$ \st $\mathbb{K} = \mathbb{I}'^{\mathbf{k}}\circ \tilde{\mathbb{K}}_{\mathbf{k}}$.
  \item\label{Lemma_Characterisation_strongly_smoothing_2} There exists a continuous linear operator $\overline{K}_0 : E_0 \to E'_\infty$ \st
\[
K_j = \iota'^\infty_j\circ \overline{K}_0\circ \iota^j_0\quad\forall\, j \in \N_0\cup\{\infty\}\text{.}
\]
  \item\label{Lemma_Characterisation_strongly_smoothing_3} $K_\infty : \left(E_\infty, \|\cdot\|_0|_{E_\infty}\right) \to E'_\infty$ is continuous, \ie $K_\infty : \left(E_\infty, \|\cdot\|_0|_{E_\infty}\right) \to \left(E'_\infty, \|\cdot\|'_j|_{E'_\infty}\right)$ is continuous for all $j\in\N_0$.
\end{enumerate}
In particular there is a $1$---$1$ correspondence between strongly smoothing operators $\mathbb{K} : \mathbb{E} \to \mathbb{E}'$ and continuous linear operators $\overline{K}_0 : E_0 \to E'_\infty$.
\end{lemma}
\begin{proof}
I will show \labelcref{Lemma_Characterisation_strongly_smoothing_1}$\;\Rightarrow\;$\labelcref{Lemma_Characterisation_strongly_smoothing_1b}$\;\Rightarrow\;$\labelcref{Lemma_Characterisation_strongly_smoothing_3}$\;\Rightarrow\;$\labelcref{Lemma_Characterisation_strongly_smoothing_2}$\;\Rightarrow\;$\labelcref{Lemma_Characterisation_strongly_smoothing_1}.
\begin{enumerate}[label=\arabic*.,ref=\arabic*.]
  \item[\labelcref{Lemma_Characterisation_strongly_smoothing_1} $\Rightarrow$ \labelcref{Lemma_Characterisation_strongly_smoothing_1b}:] Clear.
  \item[\labelcref{Lemma_Characterisation_strongly_smoothing_1b} $\Rightarrow$ \labelcref{Lemma_Characterisation_strongly_smoothing_3}:] Given $j\in\N_0$, let $\mathbf{j} \definedas (k + j)_{k\in\N_0} \subseteq \N_0$.
By assumption there exists a continuous linear operator $\tilde{\mathbb{K}}_{\mathbf{j}} : \mathbb{E} \to \mathbb{E}'^{\mathbf{j}}$ \st $\mathbb{K} = \mathbb{I}'^{\mathbf{j}}\circ \tilde{\mathbb{K}}_{\mathbf{j}}$.
Then $K_\infty = (\tilde{K}_{\mathbf{j}})_\infty$ and $(\tilde{K}_{\mathbf{j}})_\infty = (\tilde{K}_{\mathbf{j}})_0|_{E_\infty}$.
Now $(\tilde{K}_{\mathbf{j}})_0 : E_0 \to E'^{\mathbf{j}}_0 = E'_j$ is continuous and hence $K_\infty : \left(E_\infty, \|\cdot\|_0|_{E_\infty}\right) \to \left(E'_\infty, \|\cdot\|'_j|_{E'_\infty}\right)$ is continuous.
  \item[\labelcref{Lemma_Characterisation_strongly_smoothing_3} $\Rightarrow$ \labelcref{Lemma_Characterisation_strongly_smoothing_2}:] For each $j\in\N_0$ let $D_j \in (0,\infty)$ be the operator norm of $K_\infty : \left(E_\infty, \|\cdot\|_0|_{E_\infty}\right) \to \left(E'_\infty, \|\cdot\|'_j|_{E'_\infty}\right)$ and set $C_j \definedas \max\{1,D_j\} \in [1,\infty)$.
Then
\begin{align*}
d : E'_\infty \times E'_\infty &\to [0,\infty) \\
(x,y) &\mapsto \sum\limits_{j=0}^\infty \frac{1}{C_j2^j}\min\{1, \|x-y\|'_j\}
\end{align*}
is a complete metric on $E'_\infty$ inducing the given topology on $E'_\infty$.
And for $x, y \in E_\infty$,
\begin{align*}
d(K_\infty x, K_\infty y) &= \sum_{j=0}^\infty \frac{1}{C_j2^j}\min\{1, \|K_\infty(x-y)\|'_j\} \\
&\leq \sum_{j=0}^\infty \frac{1}{C_j2^j}\min\{1, D_j\|x-y\|_0\} \\
&\leq \sum_{j=0}^\infty \frac{D_j}{C_j2^j}\|x-y\|_0 \\
&\leq 2\|x-y\|_0\text{.}
\end{align*}
So $K_\infty : \left(E_\infty, \|\cdot\|_0|_{E_\infty}\right) \to \left(E'_\infty, d\right)$ is Lipschitz and hence has a unique continuous completion $\overline{K}_0 : E_0 \to E'_\infty$.
It is straightforward to see (by the usual density argument) that this $\overline{K}_0$ satisfies $2$.
  \item[\labelcref{Lemma_Characterisation_strongly_smoothing_2} $\Rightarrow$ \labelcref{Lemma_Characterisation_strongly_smoothing_1}:] Given a strictly monotone increasing sequence $\mathbf{k} \subseteq \N_0$, define $\tilde{K}_j \definedas \iota^\infty_{k_j}\circ \overline{K}_0\circ \iota^j_0 : E_j \to E'_{k_j} = E^{\mathbf{k}}_j$.
The operator $\tilde{\mathbb{K}}_{\mathbf{k}} : \mathbb{E} \to \mathbb{E}'^{\mathbf{k}}$ defined by the $\tilde{K}_j$ then has the required properties.
\end{enumerate}
\end{proof}

\begin{corollary}\label{Corollary_Characterisation_strongly_smoothing_I}
Let $\mathbb{E}$ and $\mathbb{E}'$ be sc-chains and let $\mathbf{k} \subseteq \N_0$ be strictly monotone increasing sequences with $k_0 = 0$.
If $\mathbb{K} : \mathbb{E}^{\mathbf{k}} \to \mathbb{E}'$ is a strongly smoothing continuous linear operator, then there exists a strongly smoothing continuous linear operator $\tilde{\mathbb{K}} : \mathbb{E} \to \mathbb{E}'$ \st $\mathbb{K} = \mathbb{I}'^{\mathbf{k}}\circ \tilde{\mathbb{K}}^{\mathbf{k}}$.
\end{corollary}

\begin{corollary}\label{Corollary_Strongly_smoothing_implies_compact}
Let $\mathbb{E}$ and $\mathbb{E}'$ be sc-chains and let $\mathbb{K} : \mathbb{E} \to \mathbb{E}'$ be a continuous linear operator.
If $\mathbb{K}$ is strongly smoothing, then $K_\infty : E_\infty \to E'_\infty$ is a compact operator.
\end{corollary}

\begin{example}
Let $\mathbb{E}$ be an sc-chain.
$\id_{\mathbb{E}} : \mathbb{E} \to \mathbb{E}$ is strongly smoothing \iff $\mathbb{E}$ is the constant sc-chain on a finite dimensional vector space $E_\infty$.
\end{example}

\begin{remark}
The analogous statement of the previous corollary for smoothing instead of strongly smoothing operators is evidently false, for the operator $\mathbb{I}^{\mathbf{1}} : \mathbb{E}^{\mathbf{1}} \to \mathbb{E}$ is smoothing, but $I^{\mathbf{1}}_\infty : E^{\mathbf{1}}_\infty = E_\infty \to E_\infty$ is the identity, which is compact only if $\dim E_\infty < \infty$. \\
This is one of the reasons the notion of a smoothing operator does not transfer well to weak morphisms and $\overline{\text{sc}}$-Fr{\'e}chet spaces (the operators $\mathbb{I}^{\mathbf{k}}$ all induce the identity on the $\overline{\text{sc}}$-Fr{\'e}chet space $E_\infty$). \\
Another reason is that if $\mathbb{K} : \mathbb{E} \to \mathbb{E}'$ is smoothing and $\mathbf{k} \subseteq \N_0$ is a strictly monotone increasing sequence, then $\mathbb{K}^{\mathbf{k}}$ need not be smoothing.
The analogous statement for strongly smoothing on the other hand does hold, as is shown in the following lemma.
\end{remark}

\begin{lemma}\label{Lemma_Composition_smoothing_operators}
Let $\mathbb{E}$, $\mathbb{E}'$ and $\mathbb{E}''$ be sc-chains and let $\mathbb{K}, \mathbb{K}' : \mathbb{E} \to \mathbb{E}'$, $\mathbb{S} : \mathbb{E}'' \to \mathbb{E}$ and $\mathbb{T} : \mathbb{E}' \to \mathbb{E}''$ be continuous linear operators.
Let furthermore $\lambda,\mu \in \mathds{k}$ and let $\mathbf{l} \subseteq \N_0$ be a strictly monotone increasing sequence.
\begin{enumerate}[label=\arabic*.,ref=\arabic*.]
  \item If $\mathbb{K}$ and $\mathbb{K}'$ are smoothing, then so are
\begin{enumerate}[label=(\alph*),ref=(\alph*)]
  \item $\lambda\mathbb{K} + \mu\mathbb{K}' : \mathbb{E} \to \mathbb{E}'$,
  \item $\mathbb{K}\circ \mathbb{S} : \mathbb{E}'' \to \mathbb{E}'$ and
  \item $\mathbb{T}\circ \mathbb{K} : \mathbb{E} \to \mathbb{E}''$.
\end{enumerate}
  \item If $\mathbb{K}$ and $\mathbb{K}'$ are strongly smoothing, then so are
\begin{enumerate}[label=(\alph*),ref=(\alph*)]
  \item $\lambda\mathbb{K} + \mu\mathbb{K}' : \mathbb{E} \to \mathbb{E}'$,
  \item $\mathbb{K}\circ \mathbb{S} : \mathbb{E}'' \to \mathbb{E}'$,
  \item $\mathbb{T}\circ \mathbb{K} : \mathbb{E} \to \mathbb{E}''$ and
  \item $\mathbb{K}^{\mathbf{l}} : \mathbb{E}^{\mathbf{l}} \to \mathbb{E}'^{\mathbf{l}}$.
 \end{enumerate}
\end{enumerate}
\end{lemma}
\begin{proof}
\begin{enumerate}[label=\arabic*.,ref=\arabic*.]
  \item Because $\mathbb{K}$ is smoothing, there exists a strictly monotone increasing sequence $\mathbf{k} \subseteq \N_0$ with $\mathbf{k} > \mathbf{id}$ and a continuous linear operator $\tilde{\mathbb{K}} : \mathbb{E} \to \mathbb{E}'^{\mathbf{k}}$ with $\mathbb{K} = \mathbb{I}'^{\mathbf{k}}\circ \tilde{\mathbb{K}}$.
\begin{enumerate}[label=(\alph*),ref=(\alph*)]
  \item Straightforward.
  \item $\mathbb{K}\circ \mathbb{S} = \mathbb{I}'^{\mathbf{k}}\circ \tilde{\mathbb{K}}\circ \mathbb{S} = \mathbb{I}'^{\mathbf{k}}\circ \overline{\mathbb{K}}$, where $\overline{\mathbb{K}} \definedas \tilde{\mathbb{K}}\circ \mathbb{S} : \mathbb{E}'' \to \mathbb{E}'^{\mathbf{k}}$.
So $\mathbb{K}\circ \mathbb{S}$ is smoothing.
  \item $\mathbb{T}\circ \mathbb{K} = \mathbb{T}\circ \mathbb{I}'^{\mathbf{k}}\circ \tilde{\mathbb{K}} = \mathbb{I}''^{\mathbf{k}}\circ \mathbb{T}^{\mathbf{k}}\circ \tilde{\mathbb{K}} = \mathbb{I}''^{\mathbf{k}}\circ \overline{\mathbb{K}}$, where $\overline{\mathbb{K}} \definedas \mathbb{T}^{\mathbf{k}}\circ \tilde{\mathbb{K}} : \mathbb{E} \to \mathbb{E}''^{\mathbf{k}}$.
So $\mathbb{T}\circ \mathbb{K}$ is smoothing.
\end{enumerate}
  \item Let $\mathbf{k} \subseteq \N_0$ be any strictly monotone increasing sequence.
Because $\mathbb{K}$ is strongly smoothing, there exists a continuous linear operator $\tilde{\mathbb{K}}_{\mathbf{k}} : \mathbb{E} \to \mathbb{E}'^{\mathbf{k}}$ with $\mathbb{K} = \mathbb{I}'^{\mathbf{k}}\circ \tilde{\mathbb{K}}_{\mathbf{k}}$.
\begin{enumerate}[label=(\alph*),ref=(\alph*)]
  \item Straightforward.
  \item $\mathbb{K}\circ \mathbb{S} = \mathbb{I}'^{\mathbf{k}}\circ \tilde{\mathbb{K}}_{\mathbf{k}}\circ \mathbb{S} = \mathbb{I}'^{\mathbf{k}}\circ \overline{\mathbb{K}}_{\mathbf{k}}$, where $\overline{\mathbb{K}}_{\mathbf{k}} \definedas \tilde{\mathbb{K}}\circ \mathbb{S} : \mathbb{E}'' \to \mathbb{E}'^{\mathbf{k}}$.
So $\mathbb{K}\circ \mathbb{S}$ is strongly smoothing.
  \item $\mathbb{T}\circ \mathbb{K} = \mathbb{T}\circ \mathbb{I}'^{\mathbf{k}}\circ \tilde{\mathbb{K}}_{\mathbf{k}} = \mathbb{I}''^{\mathbf{k}}\circ \mathbb{T}^{\mathbf{k}}\circ \tilde{\mathbb{K}}_{\mathbf{k}} = \mathbb{I}''^{\mathbf{k}}\circ \overline{\mathbb{K}}_{\mathbf{k}}$, where $\overline{\mathbb{K}}_{\mathbf{k}} \definedas \mathbb{T}^{\mathbf{k}}\circ \tilde{\mathbb{K}}_{\mathbf{k}} : \mathbb{E} \to \mathbb{E}''^{\mathbf{k}}$.
So $\mathbb{T}\circ \mathbb{K}$ is strongly smoothing.
  \item Because $\mathbb{K}$ is strongly smoothing, there exists a continuous linear operator $\tilde{\mathbb{K}}_{\mathbf{l}\circ \mathbf{k}} : \mathbb{E} \to \mathbb{E}'^{\mathbf{l}\circ \mathbf{k}} = \left(\mathbb{E}'^{\mathbf{l}}\right)^{\mathbf{k}}$ with $\mathbb{K} = \mathbb{I}'^{\mathbf{l}\circ \mathbf{k}}\circ \tilde{\mathbb{K}}_{\mathbf{l}\circ \mathbf{k}}$.
Define $\tilde{\mathbb{K}}^{\mathbf{l}}_{\mathbf{k}} \definedas \tilde{\mathbb{K}}_{\mathbf{l}\circ \mathbf{k}} \circ \mathbb{I}^{\mathbf{l}} : \mathbb{E}^{\mathbf{l}} \to \left(\mathbb{E}'^{\mathbf{l}}\right)^{\mathbf{k}}$.
Then $\mathbb{K}^{\mathbf{l}} = \mathbb{I}'^{\mathbf{l}\circ \mathbf{k}}_{\mathbf{l}}\circ \tilde{\mathbb{K}}^{\mathbf{l}}_{\mathbf{k}} : \mathbb{E}^{\mathbf{l}} \to \mathbb{E}'^{\mathbf{l}}$.
For if $\overline{\mathbb{K}} \definedas \mathbb{I}'^{\mathbf{l}\circ \mathbf{k}}_{\mathbf{l}}\circ \tilde{\mathbb{K}}^{\mathbf{l}}_{\mathbf{k}}$, then $\overline{K}_\infty = K^{\mathbf{l}}_\infty$.
\end{enumerate}
\end{enumerate}
\end{proof}

\begin{definition}
Let $\mathbb{E}$ and $\mathbb{E}'$ be sc-chains. \\
A weak morphism $K : E_\infty \to E'_\infty$ is called \emph{strongly smoothing} if it has a strongly smoothing extension $\mathbb{K} : \mathbb{E}^{\mathbf{k}} \to \mathbb{E}'^{\mathbf{l}}$ for some strictly monotone increasing sequences $\mathbf{k},\mathbf{l} \subseteq \N_0$.
\end{definition}

\begin{remark}\label{Remark_Strongly_smoothing_implies_compact_weak}
Note that by \cref{Corollary_Strongly_smoothing_implies_compact} a strongly smoothing weak morphism $K : E_\infty \to E'_\infty$ between sc-chains $\mathbb{E}$ and $\mathbb{E}'$ is a compact operator between the Fr{\'e}chet spaces $E_\infty$ and $E'_\infty$.
\end{remark}

\begin{lemma}\label{Lemma_Characterisation_strongly_smoothing_weak_morphisms}
Let $\mathbb{E}$ and $\mathbb{E}'$ be sc-chains and let $K : E_\infty \to E'_\infty$ be a weak morphism.
Then the following are equivalent:
\begin{enumerate}[label=\arabic*.,ref=\arabic*.]
  \item\label{Lemma_Characterisation_strongly_smoothing_weak_morphisms_1} $K$ is strongly smoothing.
  \item\label{Lemma_Characterisation_strongly_smoothing_weak_morphisms_2} $K$ has a strongly smoothing extension $\mathbb{K} : \mathbb{E}^{\mathbf{k}} \to \mathbb{E}'$ for some strictly monotone increasing sequence $\mathbf{k} \subseteq \N_0$.
  \item\label{Lemma_Characterisation_strongly_smoothing_weak_morphisms_3} For every extension $\mathbb{K} : \mathbb{E}^{\mathbf{k}} \to \mathbb{E}'^{\mathbf{l}}$ of $K$ there exists a shift $\mathbf{m} \subseteq \N_0$ \st the shifted extension $\mathbb{K}^{\mathbf{m}} : \mathbb{E}^{\mathbf{k}\circ \mathbf{m}} \to \mathbb{E}'^{\mathbf{l}\circ \mathbf{m}}$ of $K$ is strongly smoothing.
\end{enumerate}
\end{lemma}
\begin{proof}
I will show \labelcref{Lemma_Characterisation_strongly_smoothing_weak_morphisms_1}$\;\Rightarrow\;$\labelcref{Lemma_Characterisation_strongly_smoothing_weak_morphisms_2}, \labelcref{Lemma_Characterisation_strongly_smoothing_weak_morphisms_2}$\;\Rightarrow\;$\labelcref{Lemma_Characterisation_strongly_smoothing_weak_morphisms_1}, \labelcref{Lemma_Characterisation_strongly_smoothing_weak_morphisms_1}$\;\Rightarrow\;$\labelcref{Lemma_Characterisation_strongly_smoothing_weak_morphisms_3}~and \labelcref{Lemma_Characterisation_strongly_smoothing_weak_morphisms_3}$\;\Rightarrow\;$\labelcref{Lemma_Characterisation_strongly_smoothing_weak_morphisms_1}.
\begin{enumerate}
  \item[\labelcref{Lemma_Characterisation_strongly_smoothing_weak_morphisms_1} $\Rightarrow$ \labelcref{Lemma_Characterisation_strongly_smoothing_weak_morphisms_2}:] If $\mathbb{K} : \mathbb{E}^{\mathbf{k}} \to \mathbb{E}'^{\mathbf{l}}$ is a strongly smoothing extension of $K$, then by \cref{Lemma_Composition_smoothing_operators} so is $\mathbb{I}'^{\mathbf{l}}\circ \mathbb{K} : \mathbb{E}^{\mathbf{k}} \to \mathbb{E}'$.
  \item[\labelcref{Lemma_Characterisation_strongly_smoothing_weak_morphisms_2} $\Rightarrow$ \labelcref{Lemma_Characterisation_strongly_smoothing_weak_morphisms_1}:] Trivial.
  \item[\labelcref{Lemma_Characterisation_strongly_smoothing_weak_morphisms_1} $\Rightarrow$ \labelcref{Lemma_Characterisation_strongly_smoothing_weak_morphisms_3}:] Because $K$ is strongly smoothing, one can find a strongly smoothing extension $\overline{\mathbb{K}} : \mathbb{E}^{\mathbf{n}} \to \mathbb{E}'$ of $K$.
By \cref{Lemma_Characterisation_strongly_smoothing} this is equivalent to
\[
K : \left(E^{\mathbf{n}}_\infty, \|\cdot\|^{\mathbf{n}}_{0}|_{E^{\mathbf{n}}_\infty}\right) = \left(E_\infty, \|\cdot\|_{n_0}|_{E_\infty}\right) \to \left(E'_\infty, \|\cdot\|'_j|_{E'_\infty}\right)
\]
being continuous for all $j \in \N_0$.
Let $m_0 \definedas \min\{j\in\N_0 \;|\; k_j \geq n_0\}$ and $\mathbf{m} \definedas (m_0 + j)_{j\in\N_0}$.
Then
\begin{multline*}
K : \left(E^{\mathbf{k}\circ \mathbf{m}}_\infty, \|\cdot\|^{\mathbf{k}\circ \mathbf{m}}_{0}|_{E^{\mathbf{k}\circ \mathbf{m}}_\infty}\right) = \left(E_\infty, \|\cdot\|_{k_{m_0}}|_{E_\infty}\right) \to \\
\left(E'_\infty, \|\cdot\|'_{l_{m_j}}|_{E'_\infty}\right) = \left(E'^{\mathbf{l}\circ \mathbf{m}}_\infty, \|\cdot\|'^{\mathbf{l}\circ \mathbf{m}}_j|_{E'^{\mathbf{l}\circ \mathbf{m}}_\infty}\right)
\end{multline*}
is continuous for all $j \in \N_0$, because since $k_{m_0} \geq n_0$ by definition,
\[
\id_{E_\infty} : \left(E_\infty, \|\cdot\|_{k_{m_0}}|_{E_\infty}\right) \to \left(E_\infty, \|\cdot\|_{n_0}|_{E_\infty}\right)
\]
is continuous.
And by \cref{Lemma_Characterisation_strongly_smoothing} again,
$\mathbb{K}^{\mathbf{m}} : \mathbb{E}^{\mathbf{k}\circ \mathbf{m}} \to \mathbb{E}'^{\mathbf{l}\circ \mathbf{m}}$ is strongly smoothing.
  \item[\labelcref{Lemma_Characterisation_strongly_smoothing_weak_morphisms_3} $\Rightarrow$ \labelcref{Lemma_Characterisation_strongly_smoothing_weak_morphisms_1}:] If $\mathbb{K} : \mathbb{E}^{\mathbf{k}} \to \mathbb{E}^{\mathbf{l}}$ is an extension of $K$, then so is $\mathbb{K}^{\mathbf{m}} : \mathbb{E}^{\mathbf{k}\circ \mathbf{m}} \to \mathbb{E}'^{\mathbf{l}\circ \mathbf{m}}$.
\end{enumerate}
\end{proof}

\begin{lemma}\label{Lemma_Strongly_smoothing_weak_morphisms}
Let $\mathbb{E}$, $\mathbb{E}'$ and $\mathbb{E}''$ be sc-chains, let $K,K' : E_\infty \to E'_\infty$, $S : E''_\infty \to E_\infty$ and $T : E'_\infty \to E''_\infty$ be weak morphisms and let $\lambda, \mu \in \mathds{k}$.
If $K$ and $K'$ are strongly smoothing, then so are
\begin{enumerate}[label=\arabic*.,ref=\arabic*.]
  \item $\lambda K + \mu K' : E_\infty \to E'_\infty$,
  \item $K\circ S : E''_\infty \to E_\infty$ and
  \item $T\circ K : E_\infty \to E''_\infty$.
\end{enumerate}
\end{lemma}
\begin{proof}
This is a corollary to \cref{Lemma_Composition_smoothing_operators}.
For if $\mathbb{K} : \mathbb{E}^{\mathbf{k}} \to \mathbb{E}'$, $\mathbb{S} : \mathbb{E}''^{\mathbf{l}} \to \mathbb{E}$ and $\mathbb{T} : \mathbb{E}'^{\mathbf{m}} \to \mathbb{E}''$ are extensions of $K$, $S$ and $T$, respectively, then $\mathbb{K}\circ \mathbb{S}^{\mathbf{k}}$ and $\mathbb{T}\circ \mathbb{K}^{\mathbf{m}}$ are extensions of $K\circ S$ and $T\circ K$, respectively. \\
And similarly for $\lambda K + \mu K'$: If $\mathbb{K} : \mathbb{E}^{\mathbf{k}} \to \mathbb{E}'$ and $\mathbb{K}' : \mathbb{E}^{\mathbf{l}} \to \mathbb{E}'$ are extensions of $K$ and $K'$, respectively, then $\lambda \mathbb{K}\circ \mathbb{I}^{\mathbf{m}}_{\mathbf{k}} + \mu \mathbb{K}'\circ \mathbb{I}^{\mathbf{m}}_{\mathbf{l}}$ is an extension of $\lambda K + \mu K'$, where $\mathbf{m} \definedas \left(\max\{k_j,l_j\}\right)_{j\in\N_0}$.
\end{proof}

\begin{definition}\label{Definition_Strongly_smoothing_morphism}
Let $E$ and $E'$ be $\overline{\text{sc}}$-Fr{\'e}chet spaces.
A morphism $K: E \to E'$ is called \emph{strongly smoothing} if there exist compatible $\overline{\text{sc}}$-structures $(\mathbb{E},\phi)$ and $(\mathbb{E}', \phi')$ on $E$ and $E'$, respectively, \st $K_\infty \definedas \phi'^{-1}\circ K\circ \phi : E_\infty \to E'_\infty$ defines a strongly smoothing weak morphism
\end{definition}

\begin{remark}\label{Remark_Strongly_smoothing_implies_compact_scFrechet}
By \cref{Remark_Strongly_smoothing_implies_compact_weak} and \cref{Lemma_Compact_operators}, a strongly smoothing morphism $K : E \to E'$ between $\overline{\text{sc}}$-Fr{\'e}chet spaces is a compact operator between the underlying Fr{\'e}chet spaces.
\end{remark}

\begin{corollary}\label{Corollary_Basic_properties_strongly_smoothing_morphisms}
Let $E$, $E'$ and $E''$ be $\overline{\text{sc}}$-Fr{\'e}chet spaces.
\begin{enumerate}[label=\arabic*.,ref=\arabic*.]
  \item A continuous linear operator $K : E \to E'$ defines a strongly smoothing morphism of $\overline{\text{sc}}$-Fr{\'e}chet spaces \iff for any pair of compatible $\overline{\text{sc}}$-structures $(\tilde{\mathbb{E}}, \tilde{\phi})$ and $(\tilde{\mathbb{E}}', \tilde{\phi}')$ on $E$ and $E'$, respectively, $\tilde{K}_\infty \definedas \tilde{\phi}'^{-1}\circ K\circ \tilde{\phi}$ defines a strongly smoothing weak morphism.
  \item Let $K,K' : E \to E'$ be strongly smoothing morphisms and let $\lambda, \mu \in \mathds{k}$.
Then $\lambda K + \mu K' : E \to E'$ is a strongly smoothing morphism as well.
  \item Let $K : E \to E'$, $S : E'' \to E$ and $T : E' \to E''$ be morphisms with $K$ strongly smoothing.
Then $K\circ S : E'' \to E$ and $T\circ K : E \to E''$ are strongly smoothing morphisms as well.
\end{enumerate}
\end{corollary}
\begin{proof}
This is a corollary to \cref{Lemma_Strongly_smoothing_weak_morphisms}.
\end{proof}

\begin{corollary}\label{Corollary_Strongly_smoothing_operators}
Let $E$ and $E'$ be $\overline{\text{sc}}$-Fr{\'e}chet spaces and let $(\mathbb{E},\phi)$ be a compatible $\overline{\text{sc}}$-structure on $E$.
Then there is a $1$---$1$ correspondence between the set of strongly smoothing morphisms $K : E \to E'$ and
\[
\varinjlim_{i} L_{\mathrm{c}}(E_i, E')\text{,}
\]
where the limit is taken over the direct system $(L_{\mathrm{c}}(E_i, E'))_{i \in \N_0}$ with structure maps given by
\begin{align*}
L_{\mathrm{c}}(E_i, E') &\to L_{\mathrm{c}}(E_j, E') \\
K &\mapsto K\circ \iota^j_i\text{,}
\end{align*}
for $i,j\in\N_0$, $i \leq j$. \\
The correspondence is given by the direct system of maps
\begin{align*}
L_{\mathrm{c}}(E_i,E') &\to L_{\mathrm{c}}(E,E') \\
K &\mapsto K\circ \iota^\infty_i\circ \phi\inv\text{.}
\end{align*}
\end{corollary}
\begin{proof}
This is a corollary to \cref{Lemma_Characterisation_strongly_smoothing,Lemma_Characterisation_strongly_smoothing_weak_morphisms}.
\end{proof}

\begin{remark}
Above and below, when writing $L_{\mathrm{c}}(E_i, E')$, $E'$ is considered as just its underlying Fr{\'e}chet space.
\end{remark}

\begin{example}\label{Example_Strongly_smoothing_morphisms}
\begin{enumerate}[label=\arabic*.,ref=\arabic*.]
  \item Let $E$ be an $\overline{\text{sc}}$-Fr{\'e}chet space.
$\id_E : E \to E$ is strongly smoothing \iff $E$ is finite dimensional.
  \item Let $E$ and $E'$ be $\overline{\text{sc}}$-Fr{\'e}chet spaces and let $F : E \to E'$ be a morphism.
If there exists a finite dimensional $\overline{\text{sc}}$-Fr{\'e}chet space $C$ and morphisms $F_1 : E \to C$, $F_2 : C \to E'$ \st $F = F_2\circ F_1$, then $F$ is strongly smoothing.
\end{enumerate}
\end{example}

\begin{definition}\label{Definition_Space_of_strongly_smoothing_morphisms}
Let $E$ and $E'$ be $\overline{\text{sc}}$-Fr{\'e}chet spaces.
Then $\operatorname{Sm}(E,E')$ denotes the \emph{space of strongly smoothing morphisms from $E$ to $E'$} together with the topology defined by the direct limit of topological spaces
\[
\operatorname{Sm}(E,E') \cong \varinjlim_i L_{\mathrm{c}}(E_i, E')\text{,}
\]
for some compatible $\overline{\text{sc}}$-structure $(\mathbb{E}, \phi)$ on $E$, as in \cref{Corollary_Strongly_smoothing_operators}.
\end{definition}

\begin{remark}
By \cref{Corollary_Strongly_smoothing_operators} and the definition of the direct limit topology, $\operatorname{Sm}(E,E')$ is a Hausdorff locally convex topological vector space that comes with a canonical continuous inclusion $\operatorname{Sm}(E,E') \hookrightarrow L_{\mathrm{c}}(E,E')$.
But a priori the topology on $\operatorname{Sm}(E,E')$ is stronger than the subspace topology induced by the bounded-open topology on $L_{\mathrm{c}}(E,E')$.
\end{remark}

\begin{example}\label{Example_Strongly_smoothing_operators_on_finite_dimensional_vs}
Let $E$ and $E'$ be $\overline{\text{sc}}$-Fr{\'e}chet spaces with $E$ finite dimensional.
Then $\operatorname{Sm}(E,E') \cong L_{\mathrm{c}}(E,E')$ as topological spaces.
\end{example}

\begin{proposition}\label{Proposition_Perturbation_of_invertible_by_strongly_smoothing_invertible}
Let $E$ and $E'$ be $\overline{\text{sc}}$-Fr{\'e}chet spaces and let $F : E \to E'$ be an isomorphism.
Given any compatible $\overline{\text{sc}}$-structure $(\mathbb{E}, \phi)$ on $E$, there exists $i_0 \in \N_0$ and for every $i \geq i_0$ a neighbourhood $U_i \subseteq L_{\mathrm{c}}(E_i, E') \subseteq \operatorname{Sm}(E,E')$ of $0$ \st $F + K : E \to E'$ is an isomorphism for all $K \in U_i$. \\
Furthermore, $(F + K)\inv = F\inv + K'$ for a unique $K' \in \operatorname{Sm}(E', E)$.
\end{proposition}
\begin{proof}
By the same arguments used before, one can assume that $(\mathbb{E},\phi)$ and $(\mathbb{E}',\phi')$ are compatible $\overline{\text{sc}}$-structures on $E$ and $E'$, respectively, \st there are continuous linear operators $\mathbb{F} : \mathbb{E} \to \mathbb{E}'$, $\mathbb{F}' : \mathbb{E}'^{\mathbf{l}} \to \mathbb{E}$, for some strictly monotone increasing sequence $\mathbf{l} \subseteq \N_0$, with $F = \phi'\circ F_\infty\circ \phi\inv$, $F\inv = \phi\circ F'_\infty\circ \phi'^{-1}$, $\mathbb{F}\circ \mathbb{F}' = \mathbb{I}'^{\mathbf{l}}$ and $\mathbb{F}'\circ \mathbb{F}^{\mathbf{l}} = \mathbb{I}^{\mathbf{l}}$.
To simplify notation, identify $E$ with $E_\infty$ and $E'$ with $E'_\infty$.
For $i \in \N_0$, let
\[
U_i \definedas \{K \in L_{\mathrm{c}}(E_i, E'_\infty) \;|\; \|\iota^\infty_{i}\circ F'_\infty\circ K\|_{L_{\mathrm{c}}(E_{i}, E_{i})} < 1\}\text{.}
\]
Then $U_i \subseteq L_{\mathrm{c}}(E_i, E'_\infty)$ is an open subset.
Given $K \in U_i$, define (via geometric series in the Banach space $L_{\mathrm{c}}(E_i, E_i)$)
\[
K' \definedas -F'_\infty\circ K\circ \left(\sum_{k=0}^\infty(-\iota^\infty_i\circ F'_\infty\circ K)^k\right)\circ F'_i : E'^{\mathbf{l}}_i \to E_\infty\text{.}
\]
$\overline{K} \definedas K\circ \iota^\infty_i\circ \phi\inv : E \to E'$ and $\overline{K}' \definedas K'\circ (\iota'^{\mathbf{l}})^\infty_i\circ \phi'^{-1} : E' \to E$ define strongly smoothing operators.
Two completely straightforward calculations, done pointwise over $E$ and $E'$, show that $(F' + \overline{K}')\circ (F + \overline{K}) = \id_{E}$ and $(F + \overline{K})\circ (F' + \overline{K}') = \id_{E'}$, respectively. \\
That $\overline{K}'$ is unique is immediate from $\overline{K}' = (F + \overline{K})\inv - F\inv$.
\end{proof}

\Needspace{25\baselineskip}
\subsection{$\overline{\text{sc}}$-Fredholm operators}\label{Subsection_Sc_Fredholm_operators}

\begin{proposition}\label{Proposition_Fredholm_I}
Let $E$ and $E'$ be Fr{\'e}chet spaces and let $F : E \to E'$ be a continuous linear operator.
Then the following are equivalent:
\begin{enumerate}[label=\arabic*.,ref=\arabic*.]
  \item\label{Proposition_Fredholm_I_1} $F$ is invertible modulo compact operators, \ie there exist continuous linear operators $F' : E' \to E$, $K : E\to E$ and $K' : E' \to E'$ with $K$ and $K'$ compact \st
\begin{align*}
F' \circ F &= \id_E + K \\
F \circ F' &= \id_{E'} + K'\text{.}
\end{align*}
  \item\label{Proposition_Fredholm_I_2} $\dim\ker F < \infty$, $\im F \subseteq E'$ is closed and $\dim\coker F < \infty$.
  \item\label{Proposition_Fredholm_I_3} $\dim\ker F < \infty$ and $\dim\coker F < \infty$.
\end{enumerate}
\end{proposition}
\begin{proof}
I will show \labelcref{Proposition_Fredholm_I_1}$\;\Rightarrow\;$\labelcref{Proposition_Fredholm_I_3}$\;\Rightarrow\;$\labelcref{Proposition_Fredholm_I_2}$\;\Rightarrow\;$\labelcref{Proposition_Fredholm_I_1}.
\begin{enumerate}[label=\arabic*.,ref=\arabic*.]
  \item[\labelcref{Proposition_Fredholm_I_1} $\Rightarrow$ \labelcref{Proposition_Fredholm_I_3}:] $F' \circ F = \id_E + K$ implies $F' \circ F|_{\ker F} = 0 = \id_E|_{\ker F} + K|_{\ker F} = \id_{\ker F} + K|_{\ker F}$, \ie $\id_{\ker F} = -K|_{\ker F}$ is a compact operator.
In other words $\ker F$ is a locally compact topological vector space, hence finite dimensional (\cf \cite{MR3154940}, Corollary 2.11).
$F\circ F' = \id_{E'} + K'$, \ie $-K' = \id_{E'} - F\circ F'$ implies that $-K'$ maps $\im F$ to $\im F$ and hence the well-defined map $\tilde{K} : \coker F \to \coker F$ which $-K'$ induces is compact and equal to the identity.
Hence as before $\coker F$ is finite dimensional.
  \item[\labelcref{Proposition_Fredholm_I_3} $\Rightarrow$ \labelcref{Proposition_Fredholm_I_2}:] $\ker F \subseteq E$ is closed, and hence by replacing $E$ with $E/_{\ker F}$ (which is a Fr{\'e}chet space by \cite{MR3154940} Corollary 1.36 and Corollary 3.36) and $F$ by the induced map $E/_{\ker F} \to E'$ one can assume that $F$ is injective.
Let $C \subseteq E'$ be a finite dimensional, hence in particular closed, complement of $\im F$.
The continuous linear map
\begin{align*}
T : E\oplus C &\to E' \\
(e,c) &\mapsto F(e) + c
\end{align*}
is an isomorphism of vector spaces hence by the open mapping theorem (\cite{MR3154940} Theorem 4.35) it is a homeomorphism.
Thus $\im F = T(E \oplus \{0\})$ is closed.
  \item[\labelcref{Proposition_Fredholm_I_2} $\Rightarrow$ \labelcref{Proposition_Fredholm_I_1}:] $\ker F$ is finite dimensional hence it has a closed complement $X \subseteq E$.
Since $\coker F$ is finite dimensional, $\im F$ has a finite dimensional, hence closed (\cf \cite{MR3154940}, Corollary 2.10), complement $C \subseteq E'$.
$X$, $\ker F$, $\im F$ and $C$ are Fr{\'e}chet spaces (\cf \cite{MR3154940}, Corollary 1.36 and Corollary 3.36) and
$E = X \oplus \ker F$, $E' = \im F \oplus C$.
$F|_X : X \to \im F$ is a continuous linear map that is an isomorphism of vector spaces, hence by the open mapping theorem (\cite{MR3154940} Theorem 4.35) it is a homeomorphism.
Let $R \definedas \left(F|_X\right)\inv$ and define $F' : E' \to E$ by $F' \definedas \iota^X_E\circ R\circ \pr^{E'}_{\im F}$, where $\iota^X_E : X \hookrightarrow E$ is the inclusion and $\pr^{E'}_{\im F} : E' \to \im F$ is the projection along $C$.
Then $F' \circ F = 1 - \pr^E_{\ker F}$, where $\pr^E_{\ker F} : E \to \ker F$ is the projection along $X$, and $F \circ F' = 1 - \pr^{E'}_C$, where $\pr^{E'}_C : E' \to C$ is the projection along $\im F$.
$K\definedas -\pr^E_{\ker F}$ and $K' \definedas -\pr^{E'}_C$ are finite rank, hence in particular compact, operators.
\end{enumerate}
\end{proof}

\begin{definition}\label{Definition_Fredholm_operator}
Let $X, Y$ be Fr{\'e}chet spaces
\begin{enumerate}
  \item If a continuous linear operator $F : X \to Y$ satisfies any (hence all) of the conditions of \cref{Proposition_Fredholm_I} then $F$ is called a \emph{Fredholm operator}. \\
A continuous linear operator $F' : Y \to X$ as in \cref{Proposition_Fredholm_I}, \labelcref{Proposition_Fredholm_I_1}, is called a \emph{Fredholm inverse} to $F$.
  \item The number
\[
\ind F \definedas \dim \ker F - \dim \coker F
\]
is called the \emph{(Fredholm) index} of $F$.
\end{enumerate}
\end{definition}

\begin{proposition}\label{Lemma_Composition_of_Fredholm_operators}
Let $E$, $E'$ and $E''$ be Fr{\'e}chet spaces, let $F : E \to E'$ and $F' : E' \to E''$ be Fredholm operators and let $K : E \to E'$ be a compact operator.
\begin{enumerate}[label=\arabic*.,ref=\arabic*.]
  \item $F' \circ F : E \to E''$ is Fredholm with $\ind (F'\circ F) = \ind F' + \ind F$.
  \item $F + K : E \to E'$ is Fredholm with $\ind (F + K) = \ind F$.
\end{enumerate}
\end{proposition}
\begin{proof}
\begin{enumerate}[label=\arabic*.,ref=\arabic*.]
  \item That $F'\circ F$ is Fredholm is immediate from the characterisation as invertible modulo compact operators and \cref{Lemma_Compact_operators}. \\
For the index formula, one uses the following result from linear algebra:
\begin{claim}
Let
\[
0 = V_0 \overset{f_0}{\longrightarrow} V_1 \overset{f_1}{\longrightarrow} \cdots \overset{f_{n-2}}{\longrightarrow} V_{n-1} \overset{f_{n-1}}{\longrightarrow} V_n = 0
\]
be an exact sequence of finite dimensional vector spaces.
Then
\[
\sum\limits_{j=0}^n (-1)^j\dim V_j = 0\text{.}
\]
\end{claim}
\begin{proof}
Exercise.
\end{proof}
The index formula then follows easily by applying this claim to the exact sequence
\begin{multline*}
0 \longrightarrow \ker F \overset{\iota}{\longrightarrow} \ker(F'\circ F) \overset{F}{\longrightarrow} \ker F' \overset{q}{\longrightarrow} \\
\overset{q}{\longrightarrow} \coker F \overset{F'}{\longrightarrow} \coker(F'\circ F) \overset{q'}{\longrightarrow} \coker F' \longrightarrow 0\text{,}
\end{multline*}
where $\iota : \ker F \hookrightarrow \ker(F'\circ F)$ is the inclusion and $q : E' \supseteq \ker F' \to E'/_{\im F} = \coker F$ and $q' : \coker(F'\circ F) = E''/_{\im(F'\circ F)} \to E''/_{\im F'} = \left(E''/_{\im(F'\circ F)}\right)/_{\left(\im F'/_{\im(F'\circ F)}\right)}$ are the quotient maps.
  \item See \cite{MR3154940}, Chapter 5, Exercises 28--31.
\end{enumerate}
\end{proof}

\begin{remark}
Because obviously for any Fr{\'e}chet space $\id_E : E \to E$ is Fredholm, by the above Lemma any operator of the form $\id_E + K$, where $K : E \to E$ is compact, is Fredholm of index $0$. \\
Also, if $F : E \to E'$ and $F' : E' \to E$ are Fredholm inverses to each other, then $\ind F = -\ind F'$.
\end{remark}

\begin{definition}
Let $\mathbb{E}$ and $\mathbb{E}'$ be sc-chains.
\begin{enumerate}
  \item A continuous linear operator $\mathbb{F} : \mathbb{E} \to \mathbb{E}'$ is called \emph{Fredholm} \iff $\mathbb{F}$ is invertible modulo strongly smoothing operators, \ie \iff there exist continuous linear operators $\mathbb{F}' : \mathbb{E}' \to \mathbb{E}$, $\mathbb{K} : \mathbb{E} \to \mathbb{E}$ and $\mathbb{K}' : \mathbb{E}' \to \mathbb{E}'$ with $\mathbb{K}$ and $\mathbb{K}'$ strongly smoothing \st
\begin{align*}
\mathbb{F}' \circ \mathbb{F} &= \id_{\mathbb{E}} + \mathbb{K} \\
\mathbb{F} \circ \mathbb{F}' &= \id_{\mathbb{E}'} + \mathbb{K}'\text{.}
\end{align*}
Such a continuous linear operator $\mathbb{F}' : \mathbb{E}' \to \mathbb{E}$ is then called a \emph{Fredholm inverse} of $\mathbb{F}$.
  \item The \emph{Fredholm index} $\ind \mathbb{F}$ of a Fredholm morphism $\mathbb{F} : \mathbb{E} \to \mathbb{E}'$ is defined as the Fredholm index of $F_\infty$ as a Fredholm operator between the Fr{\'e}chet spaces $E_\infty$ and $E'_\infty$.
\end{enumerate}
\end{definition}

\begin{remark}
Note that if $\mathbb{F} : \mathbb{E} \to \mathbb{E}'$ is Fredholm and $\mathbb{F}' : \mathbb{E}' \to \mathbb{E}$ is a Fredholm inverse to $\mathbb{F}$, then given any strongly smoothing continuous linear operator $\mathbb{K} : \mathbb{E}' \to \mathbb{E}$, $\mathbb{F}'' \definedas \mathbb{F}' + \mathbb{K}$ is a Fredholm inverse to $\mathbb{F}$ as well, by \cref{Lemma_Composition_smoothing_operators}.
\end{remark}

\begin{definition}\label{Definition_Regularising}
Let $\mathbb{E}$ and $\mathbb{E}'$ be sc-chains and let $\mathbb{F} : \mathbb{E} \to \mathbb{E}'$ be a continuous linear operator.
$\mathbb{F}$ is called
\begin{enumerate}[label=\arabic*.,ref=\arabic*.]
  \item \emph{regularising} \iff for all $e \in E_0$ and $m\in\N_0$, if $F_0(e) \in E'_m$ then $e \in E_m$.
  \item \emph{weakly regularising} \iff there exists an $m_0\in\N_0$ \st for all $e \in E_{m_0}$, if $F_{m_0}(e) \in E'_\infty$ then $e \in E_\infty$. \\
Equivalently, \iff there exists an $m_0 \in \N_0$ \st for all $m \geq m_0$ and $e \in E_m$, if $F_m(e) \in E'_\infty$ then $e \in E_\infty$.
\end{enumerate}
\end{definition}

\begin{remark}\label{Remark_Weakly_regularising}
If $\mathbb{F} : \mathbb{E} \to \mathbb{E}'$ is weakly regularising, then there exists $m_0 \in \N_0$ \st $\ker F_j = \ker F_k \subseteq E_\infty$ for all $j, k \in \N_0$ with $j,k \geq m_0$.
\end{remark}

\begin{proposition}\label{Proposition_Fredholm_operator_between_sc_chains}
Let $\mathbb{E}$ and $\mathbb{E}'$ be sc-chains and let $\mathbb{F} : \mathbb{E} \to \mathbb{E}'$ be a continuous linear operator.
Then the following are equivalent:
\begin{enumerate}[label=\arabic*.,ref=\arabic*.]
  \item\label{Proposition_Fredholm_operator_between_sc_chains_1} $\mathbb{F}$ is Fredholm.
  \item\label{Proposition_Fredholm_operator_between_sc_chains_2} $\mathbb{F}$ is invertible modulo smoothing operators, \ie there exist continuous linear operators $\mathbb{F}' : \mathbb{E}' \to \mathbb{E}$, $\mathbb{K} : \mathbb{E} \to \mathbb{E}$ and $\mathbb{K}' : \mathbb{E}' \to \mathbb{E}'$ with $\mathbb{K}$ and $\mathbb{K}'$ smoothing \st
\begin{align*}
\mathbb{F}' \circ \mathbb{F} &= \id_{\mathbb{E}} + \mathbb{K} \\
\mathbb{F} \circ \mathbb{F}' &= \id_{\mathbb{E}'} + \mathbb{K}'\text{.}
\end{align*}
  \item\label{Proposition_Fredholm_operator_between_sc_chains_3} $\mathbb{F}$ is sc-Fredholm in the sense of \cite{1209.4040}, Definition 3.1, \ie
\begin{enumerate}[label=(\alph*),ref=(\alpha*)]
  \item\label{Proposition_Fredholm_operator_between_sc_chains_3a} $\mathbb{F}$ is regularising.
  \item\label{Proposition_Fredholm_operator_between_sc_chains_3b} $F_0 : E_0 \to E'_0$ is a Fredholm operator (between Banach spaces).
\end{enumerate}
  \item\label{Proposition_Fredholm_operator_between_sc_chains_4} $\mathbb{F}$ is HWZ-Fredholm in the sense of \cite{1209.4040}, Definition 3.4, \ie
\begin{enumerate}[label=(\alph*),ref=(\alph*)]
  \item\label{Proposition_Fredholm_operator_between_sc_chains_4a} $\dim\ker F_0 < \infty$, $\ker F_0 \subseteq E_\infty$ and there exists a subchain $\mathbb{X} \subseteq \mathbb{E}$ \st $\mathbb{E} = \ker F_0 \oplus \mathbb{X}$.
  \item\label{Proposition_Fredholm_operator_between_sc_chains_4b} $\im \mathbb{F} \definedas \left(\im F_j, \iota_j|_{\im F_{j+1}}\right)_{j\in\N_0}$ is a well-defined sc-chain and there exists a finite dimensional subspace $C \subseteq E'_\infty$ with $\mathbb{E}' = \im\mathbb{F} \oplus C$.
  \item\label{Proposition_Fredholm_operator_between_sc_chains_4c} For any subchain $\mathbb{X} \subseteq \mathbb{E}$ as in \labelcref{Proposition_Fredholm_operator_between_sc_chains_4a},
\[
\pr^{\mathbb{E}'}_{\im\mathbb{F}}\circ \mathbb{F}\circ \iota^{\mathbb{X}}_{\mathbb{E}} : \mathbb{X} \to \im\mathbb{F}
\]
is an isomorphism of sc-chains.
\end{enumerate}
\end{enumerate}
\end{proposition}
\begin{proof}
I will show \labelcref{Proposition_Fredholm_operator_between_sc_chains_1}$\;\Rightarrow\;$\labelcref{Proposition_Fredholm_operator_between_sc_chains_2}$\;\Rightarrow\;$\labelcref{Proposition_Fredholm_operator_between_sc_chains_3}$\;\Rightarrow\;$\labelcref{Proposition_Fredholm_operator_between_sc_chains_4}$\;\Rightarrow\;$\labelcref{Proposition_Fredholm_operator_between_sc_chains_1}.
\begin{enumerate}
  \item[\labelcref{Proposition_Fredholm_operator_between_sc_chains_1} $\Rightarrow$ \labelcref{Proposition_Fredholm_operator_between_sc_chains_2}:] Trivial.
  \item[\labelcref{Proposition_Fredholm_operator_between_sc_chains_2} $\Rightarrow$ \labelcref{Proposition_Fredholm_operator_between_sc_chains_3}:]
\begin{enumerate}[label=(\alph*),ref=(\alph*)]
  \item Let $\mathbb{F}'$, $\mathbb{K}$ and $\mathbb{K}'$ be as in \labelcref{Proposition_Fredholm_operator_between_sc_chains_2}.
One can assume (\cf \cref{Remark_Smoothing_factors_through_Eone}) that $\mathbb{K} = \mathbb{I}^{\mathbf{1}}\circ \tilde{\mathbb{K}}$ for some $\tilde{\mathbb{K}} : \mathbb{E} \to \mathbb{E}^{\mathbf{1}}$.
Let $e \in E_0$ with $F_0(e) \in E'_m$ for some $m\in\N_0$.
Then $E_m \ni F'_0\circ F_0(e) = e + K_0e = e + \underbrace{\iota_1\tilde{K}_0(e)}_{\in E_1}$, hence $e = F'_0\circ F_0(e) - \iota_1\tilde{K}_0(e) \in E_1$.
Hence $\iota_1\tilde{K}_0(e) \in E_2$, so $e \in E_2$.
Proceeding inductively one shows that $e \in E_m$, hence $\mathbb{F}$ is regularising.
  \item Written out, one has $F'_0F_0 = \id_{E_0} + \underbrace{K_0}_{\mathclap{=\;\iota_1\circ \tilde{K}_0}}$ and $F_0\circ F'_0 = \id_{E'_0} + \underbrace{K'_0}_{\mathclap{=\; \iota'_1\circ \tilde{K}'_0}}$.
Since $\iota_1$ and $\iota'_1$ are compact, so are $K_0$ and $K'_0$, so $F_0$ is invertible modulo compact operators, hence Fredholm (\cf \cref{Proposition_Fredholm_I} and \cref{Definition_Fredholm_operator}).
\end{enumerate}
  \item[\labelcref{Proposition_Fredholm_operator_between_sc_chains_3} $\Rightarrow$ \labelcref{Proposition_Fredholm_operator_between_sc_chains_4}:] See \cite{1209.4040}, Lemmas 3.5 and 3.6.
  \item[\labelcref{Proposition_Fredholm_operator_between_sc_chains_4} $\Rightarrow$ \labelcref{Proposition_Fredholm_operator_between_sc_chains_1}:] This follows formally in precisely the same way as the direction \labelcref{Proposition_Fredholm_I_2}$\;\Rightarrow\;$\labelcref{Proposition_Fredholm_I_1}~in the proof of \cref{Proposition_Fredholm_I}, noting that the projections $\mathbb{E} = \ker F_0\oplus \mathbb{X} \to \ker F_0 \subseteq E_\infty$ and $\mathbb{E}' = \im \mathbb{F} \oplus C \to C \subseteq E'_\infty$ are strongly smoothing.
\end{enumerate}
\end{proof}

\begin{remark}
The condition \labelcref{Proposition_Fredholm_operator_between_sc_chains_4c} in \cref{Proposition_Fredholm_operator_between_sc_chains}, \labelcref{Proposition_Fredholm_operator_between_sc_chains_4}, is slightly more general than that in \cite{1209.4040}, Definition 3.4.
But a look at the proof of Lemma 3.6 in \cite{1209.4040} immediately shows that one can assume this to hold as well.
\end{remark}

\begin{remark}
In particular, by \cref{Proposition_Fredholm_operator_between_sc_chains}, if $\mathbb{F} : \mathbb{E} \to \mathbb{E}'$ is Fredholm, then for every $k \in \N_0$, $\iota^\infty_k|_{\ker F_\infty} : \ker F_\infty \to \ker F_k$ is an isomorphism and if $C \subseteq E'_\infty$ is a complement to $\im F_\infty$, \ie if the quotient projection $C \to \coker F_\infty = E'_\infty/_{\im F_\infty}$ is an isomorphism, then $\iota^\infty_k|_C : C \to E'_k$ induces an isomorphism $C \cong \coker F_k = E'_k/_{\im F_k}$ for every $k\in\N_0\cup\{\infty\}$. \\
In particular, $F_k : E_k \to E'_k$ is Fredholm and $\ind F_k = \ind F_\ell$ for all $k, \ell \in \N_0 \cup \{\infty\}$.
\end{remark}

\begin{example}\label{Example_Elliptic_operator_I}
Let $(\Sigma,g)$ be a closed $n$-dimensional Riemannian manifold and let $\pi : F \to \Sigma$ and $\pi' : F' \to \Sigma$ be real (or complex) vector bundles equipped with euclidean (or hermitian) metrics and metric connections. \\
Let $k_0 \in \N_0$ and $1 < p < \infty$ be \st $k_0p > n$ and define $\mathbf{k} \definedas (k_0 + j)_{j\in\N_0}$, $\mathbf{k} + m \definedas (k_0 + m + j)_{j\in\N_0}$.
Define $\mathbb{E} \definedas \mathbb{W}^{\mathbf{k} + m,p}(F)$ and $\mathbb{E}' \definedas \mathbb{W}^{\mathbf{k},p}(F')$. \\
Let $P : \Gamma(F) \to \Gamma(F')$ be an elliptic partial differential operator (with smooth coefficients) of class $m$.
Then the continuous linear operator $\mathbb{P} : \mathbb{W}^{\mathbf{k} + m, p}(F) = \mathbb{E} \to \mathbb{E}' = \mathbb{W}^{\mathbf{k},p}(F')$ from \cref{Example_PDO_II} is Fredholm: \\
By standard elliptic theory there exist pseudodifferential operators
\begin{align*}
Q : \Gamma(F') &\to \Gamma(F)\text{,} \\
R : \Gamma(F) &\to \Gamma(F)
\intertext{and}
R' : \Gamma(F') &\to \Gamma(F')
\end{align*}
of orders $-m$, $-1$ and $-1$, respectively, \st
\begin{align*}
Q\circ P &= \id + R \\
P\circ Q &= \id + R'\text{.}
\end{align*}
Furthermore, $Q$, $R$ and $R'$ induce continuous linear operators
\begin{align*}
Q_j : W^{k_0 + j, p}(F') &\to W^{k_0 + m + j, p}(F)\text{,} \\
R_j : W^{k_0 + m + j, p}(F) &\to W^{k_0 + m + 1 + j, p}(F)
\intertext{and}
R'_j : W^{k_0 + j, p}(F') &\to W^{k_0 + 1 + j, p}(F')
\end{align*}
for all $j\in\N_0$.
These define continuous linear operators
\begin{align*}
\mathbb{Q} : \mathbb{W}^{\mathbf{k}, p}(F') &\to \mathbb{W}^{\mathbf{k} + m, p}(F)\text{,} \\
\mathbb{R} : \mathbb{W}^{\mathbf{k} + m, p}(F) &\to \mathbb{W}^{\mathbf{k} + m + 1, p}(F)
\intertext{and}
\mathbb{R}' : \mathbb{W}^{\mathbf{k}, p}(F') &\to \mathbb{W}^{\mathbf{k} + 1, p}(F')
\end{align*}
\st
\begin{align*}
\mathbb{Q}\circ \mathbb{P} &= \id + \mathbb{R} \\
\mathbb{P}\circ \mathbb{Q} &= \id + \mathbb{R}'
\end{align*}
and $\mathbb{R}$ and $\mathbb{R}'$ are smoothing.
\end{example}

\begin{definition}\label{Definition_sc_Fredholm}
Let $E$ and $E'$ be $\overline{\text{sc}}$-Fr{\'e}chet spaces.
\begin{enumerate}[label=\arabic*.,ref=\arabic*.]
  \item A morphism $F : E \to E'$ is called \emph{Fredholm} \iff $F$ is invertible modulo strongly smoothing operators \ie \iff there exists a morphism $F' : E' \to E$ and strongly smoothing morphisms $K : E \to E$ and $K' : E' \to E'$ \st
\begin{align*}
F' \circ F &= \id_E + K \\
F \circ F' &= \id_{E'} + K'\text{.}
\end{align*}
Such a morphism $F' : E' \to E$ is then called a \emph{Fredholm inverse} of $F$.
  \item The \emph{Fredholm index} $\ind F$ of a Fredholm morphism $F : E \to E'$ is defined as the Fredholm index of $F$ as a Fredholm operator between the underlying Fr{\'e}chet spaces.
\end{enumerate}
\end{definition}

\begin{remark}\label{Remark_sc_Fredholm_implies_Fredholm}
That a Fredholm morphism between $\overline{\text{sc}}$-Fr{\'e}chet spaces defines a Fredholm operator between the underlying Fr{\'e}chet spaces is immediate from \cref{Remark_Strongly_smoothing_implies_compact_scFrechet} and \cref{Definition_Fredholm_operator}.
\end{remark}

\begin{remark}
Note that if $F : E \to E'$ is Fredholm, $F' : E' \to E$ is a Fredholm inverse to $F$ and $K : E' \to E$ is strongly smoothing, then $F'' \definedas F' + K$ is a Fredholm inverse to $F$ as well, by \cref{Corollary_Basic_properties_strongly_smoothing_morphisms}.
\end{remark}

\begin{proposition}\label{Lemma_Basic_properties_sc_Fredholm_operators}
Let $E$, $E'$ and $E''$ be $\overline{\text{sc}}$-Fr{\'e}chet spaces, let $F : E \to E'$ and $F' : E' \to E''$ be Fredholm morphisms and let $K : E \to E'$ be a strongly smoothing morphism.
\begin{enumerate}[label=\arabic*.,ref=\arabic*.]
  \item $F' \circ F : E \to E''$ is Fredholm with $\ind (F'\circ F) = \ind F' + \ind F$.
  \item $F + K : E \to E'$ is Fredholm with $\ind (F + K) = \ind F$.
\end{enumerate}
\end{proposition}
\begin{proof}
Immediate from the definition using \cref{Corollary_Basic_properties_strongly_smoothing_morphisms} and \cref{Remark_Strongly_smoothing_implies_compact_scFrechet}.
\end{proof}

For the following theorem remember that a Fredholm morphism between $\overline{\text{sc}}$-Fr{\'e}chet spaces defines a Fredholm operator between the underlying Fr{\'e}chet spaces (\cf \cref{Remark_sc_Fredholm_implies_Fredholm}) and hence its kernel is finite dimensional (\cf \cref{Proposition_Fredholm_I}). \\
Consequently (\cf \cref{Example_Finite_dimensional_split_sc_subspace}) its kernel is a split subspace. \\
Maps of the form $\pr^Z_W$ denote a projection \wrt some splitting $Z = W \oplus W'$ and maps of the form $\iota^W_Z$ denote inclusions of a subspace $W \subseteq Z$. \\
The theorem is stated in a fairly complicated manner because it then provides several technial results which will be used later on.
For the gist of the statement it may be advisable to read its \cref{Proposition_sc_Fredholm_morphism} first.
\begin{theorem}\label{Theorem_sc_Fredholm_morphism}
Let $E$ and $E'$ be $\overline{\text{sc}}$-Fr{\'e}chet spaces and let $F : E \to E'$ be a Fredholm morphism. \\
Then $F$ is a Fredholm operator between the Fr{\'e}chet spaces underlying $E$ and $E'$, \ie $\dim \ker F < \infty$ and $\dim\coker F < \infty$, and the following holds: \\
Given any Fredholm inverse $F' : E' \to E$ to $F$ and any splitting $E = \ker F \oplus X$, for some $\overline{\text{sc}}$-subspace $X \subseteq E$, there exist the following:
\begin{enumerate}[label=\arabic*.,ref=\arabic*.]
  \item\label{Theorem_sc_Fredholm_morphism_1} finite dimensional subspaces $A, B \subseteq X$ and $C' \subseteq E'$ with images
\[
A' \definedas F(A) \subseteq \ker F'\text{,}\quad B' \definedas F_\infty(B) \subseteq C' \quad\text{and}\quad C \definedas F'(C') \subseteq \ker F\text{,}
\]
  \item\label{Theorem_sc_Fredholm_morphism_2} finite dimensional subspaces $\overline{A}' \subseteq \ker F'$, $\overline{B}' \subseteq C'$ and $\overline{B} \subseteq \ker F$,
  \item\label{Theorem_sc_Fredholm_morphism_3} $\overline{\text{sc}}$-subspaces $Y \subseteq E$ and $Y' \subseteq E'$,
  \item\label{Theorem_sc_Fredholm_morphism_4} and splittings
\begin{align*}
E &= \underbrace{C \oplus \overline{C}}_{=\; \ker F} \oplus \underbrace{A \oplus B \oplus Y}_{=\; X} \\
E' &= \underbrace{A' \oplus \overline{A}'}_{=\; \ker F'} \oplus \underbrace{B' \oplus \overline{B}'}_{=\; C'} \oplus Y'\text{.}
\end{align*}
\end{enumerate}
\begin{align*}
\pr^{E'}_{A'}\circ F\circ \iota^{A}_{E} : A &\to A'
\intertext{and}
\pr^{E'}_{B'}\circ F\circ \iota^{B}_{E} : B &\to B'
\end{align*}
are isomorphisms of finite dimensional vector spaces.
Consequently one can define
\begin{align*}
F'' &\definedas F' - F'\circ \iota^{C'}_{E'}\circ \pr^{E'}_{C'} \;+ \\
&\quad\; +\; \iota^{A}_{E}\circ \left(\pr^{E'}_{A'}\circ F\circ \iota^{A}_{E}\right)\inv \circ \pr^{E'}_{A'} \;+ \\ 
&\quad\; +\; \iota^{B}_{E}\circ \left(\pr^{E'}_{B'}\circ F\circ \iota^{B}_{E}\right)\inv\circ \pr^{E'}_{B'}
\end{align*}
and with
\[
H \definedas F''\circ F \quad\text{and}\quad H' \definedas F\circ F''
\]
the following hold:
\begin{enumerate}[label=\arabic*.,ref=\arabic*.]\setcounter{enumi}{4}
  \item\label{Theorem_sc_Fredholm_morphism_5} $F'' : E' \to E$ defines a Fredholm inverse to $F$.
  \item\label{Theorem_sc_Fredholm_morphism_6} $\im F \subseteq E'$ and $\im F'' \subseteq E$ are split $\overline{\text{sc}}$-subspaces with
\begin{align*}
\ker F &= \ker H & \ker F'' &= \ker H' \\
\im F &= \im H' & \im F'' &= \im H \\
& & &= X \\
E &= \ker F \oplus \im F'' & E' &= \ker F'' \oplus \im F \\
&= \ker H \oplus \im H & &= \ker H' \oplus \im H' \\
&= \ker F \oplus X\text{.} & &
\end{align*}
  \item\label{Theorem_sc_Fredholm_morphism_7} $F|_X : X \to \im F$, $F'' : \im F \to X$, $H|_X : X \to X$ and $H'|_{\im F} : \im F \to \im F$ are isomorphisms with
\begin{align*}
(F|_X)\inv &= (H|_X)\inv \circ F''|_{\im F} = F'' \circ (H'|_{\im F})\inv \\
(F''|_X)\inv &= (H'|_{\im F})\inv \circ F|_{X} = F \circ (H|_{X})\inv\text{.}
\end{align*}
\end{enumerate}
\end{theorem}
\begin{proof}
That $F : E \to E'$ defines a Fredholm operator between Fr{\'e}chet spaces has already been remarked in \cref{Remark_sc_Fredholm_implies_Fredholm}.

Let $K : E \to E$ and $K' : E' \to E'$ be strongly smoothing with $F'\circ F = \id_E + K$, $F\circ F' = \id_{E'} + K'$. \\
After choosing $\overline{\text{sc}}$-structures $(\mathbb{E},\phi)$ and $(\mathbb{E}',\phi')$ there exist (\cf \cref{Lemma_Morphisms_of_sc_Frechet_spaces}) strictly monotone increasing sequences $\mathbf{k},\mathbf{l} \subseteq \N_0$ and continuous linear operators $\mathbb{F} : \mathbb{E}^{\mathbf{k}} \to \mathbb{E}'$ and $\mathbb{F}' : \mathbb{E}'^{\mathbf{l}} \to \mathbb{E}$ with $F_\infty = \phi'^{-1}\circ F \circ \phi$ and $F'_\infty = \phi\inv\circ F'\circ\phi'$. \\
Furthermore, one can assume that there is a splitting $\mathbb{E} = \ker F_\infty \oplus \mathbb{X}$ \st $\phi(X_\infty) = X$ and analogously that there is a splitting $\mathbb{E}' = \ker F'_\infty \oplus \mathbb{X}'$, by \cref{Example_Finite_dimensional_split_subchain}. \\
Replacing $\mathbb{E}$ by $\mathbb{E}^{\mathbf{k}}$, $\mathbb{F}'$ by $\mathbb{F}'^{\mathbf{k}}$ and $\mathbb{X}$ by $\mathbb{X}^{\mathbf{k}}$, one can assume that $\mathbf{k} = (j)_{j\in\N_0}$, so $\mathbb{F} : \mathbb{E} \to \mathbb{E}'$ and $\mathbb{F}' : \mathbb{E}'^{\mathbf{l}} \to \mathbb{E}$. \\
Using \cref{Lemma_Operators_covering_id}, $\mathbb{K} \definedas \mathbb{F}'\circ \mathbb{F}^{\mathbf{l}} - \mathbb{I}^{\mathbf{l}} : \mathbb{E}^{\mathbf{l}} \to \mathbb{E}$ is an extension of $K_\infty = \phi^{-1}\circ K\circ \phi = F'_\infty\circ F_\infty - \id_{E_\infty}$ and using \cref{Lemma_Composition_smoothing_operators,Lemma_Characterisation_strongly_smoothing_weak_morphisms} to modify $\mathbb{F}$ and $\mathbb{F}'$, one can assume that $\mathbb{K}$ is strongly smoothing.
And similarly for an extension $\mathbb{K}' = \mathbb{F} \circ \mathbb{F}' - \mathbb{I}'^{\mathbf{l}} : \mathbb{E}'^{\mathbf{l}} \to \mathbb{E}'$ of $K'_\infty = \phi'^{-1}\circ K'\circ \phi' = F_\infty\circ F'_\infty - \id_{E'_\infty}$.
\begin{claim}
$\mathbb{F}$ is weakly regularising.
\end{claim}
\begin{proof}
From $\mathbb{F}'\circ \mathbb{F}^{\mathbf{l}} = \mathbb{I}^{\mathbf{k}\circ\mathbf{l}} + \mathbb{K}$ one has that $F'_0\circ F_{l_0} = \iota^{k_{l_0}}_0 + K_0 : E^{\mathbf{k}}_{l_0} = E_{k_{l_0}} \to E_0$.
So if $e \in E^{\mathbf{k}}_{l_0}$ with $F_{l_0}(e) \in E'_\infty$, then $F'_0(F_{l_0}(e)) \in E_\infty$ and $K_0(e) \in E_\infty$ because $\mathbb{K}$ is strongly smoothing.
Hence $\iota^{k_{l_0}}_0(e) = F'_0(F_{l_0}(e))) - K_0(e) \in E_\infty$, so $e \in E_\infty$.
\end{proof}
Now replace $\mathbb{E}$ by $\mathbb{E}^{\mathbf{m}}$, $\mathbb{F}$ by $\mathbb{I}'^{\mathbf{m}}\circ \mathbb{F}^{\mathbf{m}}$ and $\mathbb{F}'$ by $\mathbb{F}'^{\mathbf{m}}$.
Then (\cf \cref{Remark_Weakly_regularising}) for all $j\in\N_0$, $\ker F_j = \ker F_\infty \subseteq E_\infty$ which is finite dimensional since $F : E \to E'$ and hence $F_\infty : E_\infty \to E'_\infty$ is Fredholm. \\
Applying the same reasoning to $\mathbb{F}'$, one can also assume that $\ker F'_j = \ker F'_\infty \subseteq E'_\infty$ for all $j \in \N_0$. \\
Because $\mathbb{K}$ and $\mathbb{K}'$ are strongly smoothing, there are continuous linear operators $\tilde{\mathbb{K}} : \mathbb{E}^{\mathbf{l}} \to \mathbb{E}^{\mathbf{l}}$ and $\tilde{\mathbb{K}}' : \mathbb{E}'^{\mathbf{l}} \to \mathbb{E}'^{\mathbf{l}}$ \st $\mathbb{K} = \mathbb{I}^{\mathbf{l}}\circ \tilde{\mathbb{K}}$ and $\mathbb{K}' = \mathbb{I}'^{\mathbf{l}}\circ \tilde{\mathbb{K}}'$.
Furthermore, by \cref{Corollary_Characterisation_strongly_smoothing_I}, one can assume that $\tilde{\mathbb{K}} = \overline{\mathbb{K}}^{\mathbf{l}}$ for a strongly smoothing continuous linear operator $\overline{\mathbb{K}} : \mathbb{E} \to \mathbb{E}$ and similarly $\tilde{\mathbb{K}}' = \overline{\mathbb{K}}'^{\mathbf{l}}$ for a strongly smoothing continuous linear operator $\overline{\mathbb{K}}' : \mathbb{E}' \to \mathbb{E}'$. \\
In conclusion,
\begin{align*}
\mathbb{F}'\circ \mathbb{F}^{\mathbf{l}} &= \mathbb{I}^{\mathbf{l}}\circ (\underbrace{\id_{\mathbb{E}} + \overline{\mathbb{K}}}_{\defines\; \mathbb{G}})^{\mathbf{l}} \\
\mathbb{F}\circ \mathbb{F}' &= \mathbb{I}'^{\mathbf{l}}\circ (\underbrace{\id_{\mathbb{E}'} + \overline{\mathbb{K}}'}_{\defines\; \mathbb{G}'})^{\mathbf{l}}
\end{align*}
Because $G_\infty = F'_\infty\circ F_\infty$, one has $\ker F_\infty \subseteq \ker G_\infty$ and denoting 
\[
\tilde{A} \definedas (F_\infty|_{X_\infty})\inv(\ker F'_\infty)\text{,}
\]
$\ker G_\infty = \ker F_\infty \oplus \tilde{A}$, which is finite dimensional as well, because $G_\infty$ is Fredholm as an operator between Fr{\'e}chet spaces.
By \cref{Example_Finite_dimensional_split_subchain} again, there hence exists a splitting $\mathbb{X} = \tilde{A} \oplus \mathbb{A}$ and hence $\mathbb{E} = \ker F_\infty \oplus \tilde{A} \oplus \mathbb{A}$, where $\ker G_\infty = \ker F_\infty \oplus \tilde{A}$ and $F_\infty(A_\infty) \cap \ker F'_\infty = \{0\}$.
Applying the same reasoning to $\mathbb{G}'$, one obtains a similar splitting $\mathbb{E}' = \ker F'_\infty \oplus \tilde{A}' \oplus \mathbb{A}'$, where $\ker G'_\infty = \ker F'_\infty \oplus \tilde{A}'$ and $F'_\infty(A'_\infty) \cap \ker F_\infty = \{0\}$.
Note that $F_\infty$ maps $\tilde{A}$ injectively into $\ker F'_\infty$ and $F'_\infty$ map $\tilde{A}'$ injectively into $\ker F_\infty$.
Hence, one can further decompose
\begin{align}
\mathbb{E} &= \underbrace{\overbrace{F'_\infty(\tilde{A}') \oplus \overline{A}}^{=\; \ker F_\infty} \oplus \tilde{A}}_{=\; \ker G_\infty} \oplus \mathbb{A}\label{Equation_Decomposition_Fredholm_operator_1} \\
\mathbb{E}' &= \underbrace{\overbrace{F_\infty(\tilde{A}) \oplus \overline{A}'}^{=\; \ker F'_\infty} \oplus \tilde{A}'}_{=\; \ker G'_\infty} \oplus \mathbb{A}'\text{.}\label{Equation_Decomposition_Fredholm_operator_2}
\end{align}
W.\,r.\,t.~this decomposition, one can then write in matrix form
{\allowdisplaybreaks
\begin{align*}
F_\infty &= \begin{blockarray}{ccccc}
F'_\infty(\tilde{A}') & \overline{A} & \tilde{A} & A_\infty \\
\begin{block}{(cccc)c}
0 & 0 & \alpha & 0 & F_\infty(\tilde{A}) \\
0 & 0 & 0 & 0 & \overline{A}' \\
0 & 0 & 0 & a_1 & \tilde{A}' \\
0 & 0 & 0 & a_2 & A'_\infty \\
\end{block}
\end{blockarray} \\
F'_\infty &= \begin{blockarray}{ccccc}
F_\infty(\tilde{A}) & \overline{A}' & \tilde{A}' & A'_\infty \\
\begin{block}{(cccc)c}
0 & 0 & \beta & 0 & F'_\infty(\tilde{A}') \\
0 & 0 & 0 & 0 & \overline{A} \\
0 & 0 & 0 & b_1 & \tilde{A} \\
0 & 0 & 0 & b_2 & A_\infty \\
\end{block}
\end{blockarray} \\
G_\infty &= \begin{blockarray}{ccccc}
F'_\infty(\tilde{A}') & \overline{A} & \tilde{A} & A_\infty \\
\begin{block}{(cccc)c}
0 & 0 & 0 & \beta a_1 & F'_\infty(\tilde{A}') \\
0 & 0 & 0 & 0 & \overline{A} \\
0 & 0 & 0 & b_1a_2 & \tilde{A} \\
0 & 0 & 0 & b_2a_2 & A_\infty \\
\end{block}
\end{blockarray} \\
G'_\infty &= \begin{blockarray}{ccccc}
F_\infty(\tilde{A}) & \overline{A}' & \tilde{A}' & A'_\infty \\
\begin{block}{(cccc)c}
0 & 0 & 0 & \alpha b_1 & F_\infty(\tilde{A}) \\
0 & 0 & 0 & 0 & \overline{A}' \\
0 & 0 & 0 & a_1b_2 & \tilde{A}' \\
0 & 0 & 0 & a_2b_2 & A'_\infty \\
\end{block}
\end{blockarray}\text{,}
\end{align*}
}
where by the choices of the above splittings, $\alpha = \pr^{E'_\infty}_{F_\infty(\tilde{A})}\circ F_\infty\circ \iota^{\tilde{A}}_{E_\infty}$ is invertible as a linear map between finite dimensional vector spaces.
Now define $\mathbb{F}'' : \mathbb{E}'^{\mathbf{l}} \to \mathbb{E}$ by
\[
\mathbb{F}'' \definedas \mathbb{F}' - \mathbb{F}'\circ \iota^{\tilde{A}'}_{\mathbb{E}'}\circ \pr^{\mathbb{E}'}_{\tilde{A}'} + \iota^{\tilde{A}}_{\mathbb{E}}\circ \left(\pr^{E'_\infty}_{F_\infty(\tilde{A})}\circ F_\infty\circ \iota^{\tilde{A}}_{E_\infty}\right)\inv \circ \pr^{\mathbb{E}'}_{F_\infty(\tilde{A})}\text{,}
\]
where $\iota^{\tilde{A}}_{\mathbb{E}}$ is the inclusion and the maps of the form $\pr^X_Y$ are the projections \wrt the splittings above.
Note that the two summands on the right are strongly smoothing operators, by \cref{Example_Strongly_smoothing_morphisms}, so $\mathbb{F}''$ defines a Fredholm inverse $F'' : E' \to E$ to $F$.
Written in matrix form,
\[
F''_\infty = \begin{blockarray}{ccccc}
F_\infty(\tilde{A}) & \overline{A}' & \tilde{A}' & A'_\infty \\
\begin{block}{(cccc)c}
0 & 0 & 0 & 0 & F'_\infty(\tilde{A}') \\
0 & 0 & 0 & 0 & \overline{A} \\
\alpha\inv & 0 & 0 & b_1 & \tilde{A} \\
0 & 0 & 0 & b_2 & A_\infty \\
\end{block}
\end{blockarray}\text{.}
\]
Now again $\mathbb{F}''\circ \mathbb{F}^{\mathbf{l}} = \mathbb{I}^{\mathbf{l}}\circ \mathbb{H}^{\mathbf{l}}$ and $\mathbb{F}\circ \mathbb{F}'' = \mathbb{I}'^{\mathbf{l}}\circ \mathbb{H}'^{\mathbf{l}}$, where
\begin{align*}
H_\infty &= \begin{blockarray}{ccccc}
F'_\infty(\tilde{A}') & \overline{A} & \tilde{A} & A_\infty \\
\begin{block}{(cccc)c}
0 & 0 & 0 & 0 & F'_\infty(\tilde{A}') \\
0 & 0 & 0 & 0 & \overline{A} \\
0 & 0 & \id_{\tilde{A}} & b_1a_2 & \tilde{A} \\
0 & 0 & 0 & b_2a_2 & A_\infty \\
\end{block}
\end{blockarray} \\
H'_\infty &= \begin{blockarray}{ccccc}
F_\infty(\tilde{A}) & \overline{A}' & \tilde{A}' & A'_\infty \\
\begin{block}{(cccc)c}
\id_{F_\infty(\tilde{A})} & 0 & 0 & \alpha b_1 & F_\infty(\tilde{A}) \\
0 & 0 & 0 & 0 & \overline{A}' \\
0 & 0 & 0 & a_1b_2 & \tilde{A}' \\
0 & 0 & 0 & a_2b_2 & A'_\infty \\
\end{block}
\end{blockarray}\text{.}
\end{align*}
Note that $H'_\infty|_{A'_\infty} = G'_\infty|_{A'_\infty}$ and hence $\ker H'_\infty|_{A'_\infty} = \ker G'_\infty|_{A'_\infty} = \{0\}$, by choice of $\mathbb{A}'$.
Just the same, $F''_\infty|_{A'_\infty} = F'_\infty|_{A'_\infty}$ and hence $\ker F''_\infty|_{A'_\infty} = \ker F'_\infty|_{A'_\infty} = \{0\}$, by choice of $\mathbb{A}'$.
It follows that $\ker F''_\infty = \overline{A}' \oplus \tilde{A}' = \ker H'_\infty$. \\
Also, from the above matrix representations one sees that $\ker H_\infty = F'_\infty(\tilde{A}')\oplus \overline{A} \oplus \ker H_\infty|_{A_\infty}$ and $H_\infty|_{A_\infty}=  F''_\infty|_{A'_\infty}\circ \underbrace{\pr^{E'_\infty}_{A'_\infty}\circ F_\infty|_{A_\infty}}_{=\; a_2} = F'_\infty|_{A'_\infty}\circ a_2$, hence $\ker H_\infty|_{A_\infty} = \ker a_2$, because $F'_\infty|_{A'_\infty}$ is injective by choice of $\mathbb{A}'$.
Define
\[
\tilde{A}'' \definedas \ker \pr^{E'_\infty}_{A'_\infty}\circ F_\infty|_{A_\infty} \subseteq A_\infty\text{.}
\]
Then $F_\infty(\tilde{A}'') = a_1(\tilde{A}'') \subseteq \tilde{A}'$ and $\pr^{E'_\infty}_{F_\infty(\tilde{A}'')}\circ F_\infty\circ \iota^{\tilde{A}''}_{E_\infty} : \tilde{A}'' \to F_\infty(\tilde{A}'')$ is a bijective linear map between finite dimensional vector spaces.
Hence one can define $\mathbb{F}''' : \mathbb{E}'^{\mathbf{l}} \to \mathbb{E}$ by
\begin{align*}
\mathbb{F}''' &\definedas \mathbb{F}'' + \iota^{F_\infty(\tilde{A}'')}_{\mathbb{E}'}\circ \left(\pr^{E'_\infty}_{F_\infty(\tilde{A}'')}\circ F_\infty\circ \iota^{\tilde{A}''}_{E_\infty}\right)\inv\circ \pr^{\mathbb{E}}_{\tilde{A}''} \\
&= \mathbb{F}' - \mathbb{F}'\circ \iota^{\tilde{A}'}_{\mathbb{E}'}\circ \pr^{\mathbb{E}'}_{\tilde{A}'} \;+ \\
&\quad\; +\; \iota^{\tilde{A}}_{\mathbb{E}}\circ \left(\pr^{E'_\infty}_{F_\infty(\tilde{A})}\circ F_\infty\circ \iota^{\tilde{A}}_{E_\infty}\right)\inv \circ \pr^{\mathbb{E}'}_{F_\infty(\tilde{A})} \;+ \\ 
&\quad\; +\; \iota^{\tilde{A}''}_{\mathbb{E}}\circ \left(\pr^{E'_\infty}_{F_\infty(\tilde{A}'')}\circ F_\infty\circ \iota^{\tilde{A}''}_{E_\infty}\right)\inv\circ \pr^{\mathbb{E}'}_{F_\infty(\tilde{A}'')}
\end{align*}
If one now applies the same reasoning as above, replacing $\mathbb{F}'$ by $\mathbb{F}''$, then in the decompositions \labelcref{Equation_Decomposition_Fredholm_operator_1,Equation_Decomposition_Fredholm_operator_2} one gets after this replacement, $\tilde{A}' = \{0\}$.
A simple calculation with the matrix representations of $F_\infty$ and $F'''_\infty$ \wrt this new decomposition then shows that
\begin{align*}
\ker F_\infty &= \ker(F'''_\infty\circ F_\infty)
\intertext{and}
\ker F'''_\infty &= \ker (F_\infty\circ F'''_\infty)\text{.}
\end{align*}
Using the above, replacing $\mathbb{F}'$ by $\mathbb{F}''$, one can hence assume that $\ker G_\infty = \ker F_\infty$ and $\ker G'_\infty = \ker F'_\infty$.
Also note that $\mathbb{G} : \mathbb{E} \to \mathbb{E}$ and $\mathbb{G}' : \mathbb{E}' \to \mathbb{E}'$ are Fredholm operators of index $0$.
Also, from $G_\infty = F'_\infty\circ F_\infty$ and $G'_\infty = F_\infty\circ F'_\infty$, it follows that $\im G_\infty \subseteq \im F'_\infty$ and $\im G'_\infty \subseteq \im F_\infty$, so $\dim \coker G_\infty \geq \dim \coker F'_\infty$ and $\dim\coker G'_\infty \geq \dim \coker F_\infty$.
Consequently,
\begin{align*}
\dim \ker F_\infty &= \dim \ker G_\infty \\
&= \dim \coker G_\infty \\
&\geq \dim \coker F'_\infty
\intertext{and}
\dim \ker F'_\infty &= \dim \ker G'_\infty \\
&= \dim \coker G'_\infty \\
&\geq \dim \coker F_\infty\text{.}
\end{align*}
Using that furthermore, because $F_\infty$ and $F'_\infty$ are Fredholm inverses to each other, that by \cref{Lemma_Composition_of_Fredholm_operators} $\ind F_\infty = -\ind F'_\infty$, it follows that
\[
0 \leq \dim \ker F_\infty - \dim \coker F'_\infty = \dim \coker F_\infty - \dim \ker F'_\infty \leq 0\text{,}
\]
hence
\begin{align*}
\dim \ker F_\infty &= \dim \coker F'_\infty
\intertext{and}
\dim \coker F_\infty &= \dim \ker F'_\infty\text{.}
\end{align*}
By the above then also $\dim \coker G_\infty = \dim \coker F'_\infty$ and $\dim \coker G'_\infty = \dim \coker F_\infty$.
Because $\mathbb{G} : \mathbb{E} \to \mathbb{E}$ and $\mathbb{G}' : \mathbb{E}' \to \mathbb{E}'$ are Fredholm operators, by \cref{Proposition_Fredholm_operator_between_sc_chains}, $\im \mathbb{G} \subseteq \mathbb{E}$ and $\im \mathbb{G}' \subseteq \mathbb{E}'$ are well defined subchains.
Furthermore, if $e' \in \im F_\infty \cap \ker F'_\infty$, $e' = F_\infty(e)$ for some $e \in E_\infty$, then $0 = F'_\infty\circ F_\infty(e) = G_\infty(e)$, hence $e \in \ker G_\infty = \ker F_\infty$ and so $e' = 0$.
It follows that $\im G_\infty \cap \ker F'_\infty = \im F_\infty \cap \ker F'_\infty = \{0\}$ and similarly $\im G'_\infty \cap \ker F_\infty = \im F'_\infty \cap \ker F_\infty = \{0\}$.
So, summarising,
\begin{align*}
\ker G_\infty &= \ker F_\infty & \ker G'_\infty &= \ker F'_\infty \\
\im G_\infty &= \im F'_\infty & \im G'_\infty &= \im F_\infty \\
E_\infty &= \ker F_\infty \oplus \im F'_\infty & E'_\infty &= \ker F'_\infty \oplus \im F_\infty \\
&= \ker G_\infty \oplus \im G_\infty & &= \ker G'_\infty \oplus \im G'_\infty\text{.}
\end{align*}
Using the above, define $\mathbb{X}'' \definedas \im \mathbb{H}'$.
The above shows that $X = \phi(X_\infty)$ and $\im F = \phi(X''_\infty)$ are split $\overline{\text{sc}}$-subspaces of $E$ and $E'$ with $\overline{\text{sc}}$-complements $\ker F$ and $\ker F''$, respectively. \\
Define
\begin{align*}
A &\definedas \phi(\tilde{A}) \\
B &\definedas \phi(\tilde{A}'')
\intertext{and}
C' &\definedas \phi'(\tilde{A}')\text{.}
\end{align*}
From the above,
\begin{align*}
\tilde{\mathbb{G}} &\definedas \mathbb{G}|_{\mathbb{X}} : \mathbb{X} \to \mathbb{X} \quad\text{and} \\
\tilde{\mathbb{G}}' &\definedas \mathbb{G}'|_{\mathbb{X}''} : \mathbb{X}'' \to \mathbb{X}''
\intertext{are well-defined isomorphisms, and}
\tilde{\mathbb{F}} &\definedas \mathbb{F}|_{\mathbb{X}} : \mathbb{X} \to \mathbb{X}'' \quad\text{and} \\
\tilde{\mathbb{F}}' &\definedas \mathbb{F}'|_{\mathbb{X}''^{\mathbf{l}}} : \mathbb{X}''^{\mathbf{l}} \to \mathbb{X}
\intertext{are well-defined and satisfy}
\tilde{\mathbb{F}}' \circ \tilde{\mathbb{F}}^{\mathbf{l}} &= \mathbb{I}^{\mathbf{l}}\circ \tilde{\mathbb{G}}^{\mathbf{l}} \\
\tilde{\mathbb{F}} \circ \tilde{\mathbb{F}}' &= \mathbb{I}'^{\mathbf{l}}\circ \tilde{\mathbb{G}}'^{\mathbf{l}}\text{.}
\end{align*}
$\mathbb{H} \definedas \tilde{\mathbb{F}}'\circ \left(\tilde{\mathbb{G}}'^{-1}\right)^{\mathbf{l}} : \mathbb{X}''^{\mathbf{l}} \to \mathbb{X}$ is a well-defined continuous linear operator which extends $H_\infty = \tilde{F}'_\infty \circ \left(\tilde{G}'_\infty\right)^{-1} = \tilde{F}'_\infty\circ \left(\tilde{F}_\infty \circ \tilde{F}'_\infty\right)\inv = \tilde{F}_\infty\inv$.
So $\tilde{F}_\infty\inv : X''_\infty \to X_\infty$ defines a weak morphism that is inverse to $\tilde{F}_\infty : X_\infty \to X''_\infty$. \\
Hence $F|_X : X \to X'' = \im F$ is an isomorphism of $\overline{\text{sc}}$-Fr{\'e}chet spaces.
\end{proof}

\begin{corollary}\label{Proposition_sc_Fredholm_morphism}
Let $E$ and $E'$ be $\overline{\text{sc}}$-Fr{\'e}chet spaces and let $F : E \to E'$ be a morphism.
Then the following are equivalent:
\begin{enumerate}[label=\arabic*.,ref=\arabic*.]
  \item\label{Proposition_sc_Fredholm_morphism_1} $F$ is Fredholm.
  \item\label{Proposition_sc_Fredholm_morphism_2}
\begin{enumerate}[label=(\alph*),ref=(\alph*)]
  \item\label{Proposition_sc_Fredholm_morphism_2a} $F$ is a Fredholm operator between the Fr{\'e}chet spaces underlying $E$ and $E'$, \ie $\dim \ker F < \infty$ and $\dim\coker F < \infty$.
  \item\label{Proposition_sc_Fredholm_morphism_2b} There exist splittings of $\overline{\text{sc}}$-Fr{\'e}chet spaces
\[
E = \ker F \oplus X \qquad\text{and}\qquad E' = \im F \oplus C\text{.}
\]
In particular $\im F$ defines a split $\overline{\text{sc}}$-subspace of $E'$ and the quotient projection $E' \to \coker F = E'/_{\im F}$ is a well defined morphism between $\overline{\text{sc}}$-Fr{\'e}chet spaces.
  \item\label{Proposition_sc_Fredholm_morphism_2c} For any pair of splittings as in \labelcref{Proposition_sc_Fredholm_morphism_2b},
\[
\pr^{E'}_{\im F}\circ F\circ \iota^X_E : X \to \im F\]
is an isomorphism of $\overline{\text{sc}}$-Fr{\'e}chet spaces.
\end{enumerate}
\end{enumerate}
\end{corollary}
\begin{proof}
\leavevmode
\begin{enumerate}[label=\arabic*.,ref=\arabic*.]
  \item[\labelcref{Proposition_sc_Fredholm_morphism_1} $\Rightarrow$ \labelcref{Proposition_sc_Fredholm_morphism_2}]
Immediate from \cref{Theorem_sc_Fredholm_morphism}.
  \item[\labelcref{Proposition_sc_Fredholm_morphism_2} $\Rightarrow$ \labelcref{Proposition_sc_Fredholm_morphism_1}] This again follows formally in the same way as the direction \labelcref{Proposition_Fredholm_I_2}$\;\Rightarrow\;$\labelcref{Proposition_Fredholm_I_1}~in the proof of \cref{Proposition_Fredholm_I}, noting that the projections $E = \ker F\oplus X \to \ker F$ and $E' = \im F \oplus C \to C$ are strongly smoothing.
\end{enumerate}
\end{proof}

\begin{example}\label{Example_Elliptic_operator_II}
If $E$ and $E'$ are $\overline{\text{sc}}$-Fr{\'e}chet spaces and $(\mathbb{E}, \phi)$ and $(\mathbb{E}', \phi')$ are compatible $\overline{\text{sc}}$-structures on $E$ and $E'$, respectively, then any Fredholm operator $\mathbb{F} : \mathbb{E} \to \mathbb{E}'$ induces a Fredholm operator $F : E \to E'$. \\
In particular, let $(\Sigma,g)$ be a closed $n$-dimensional Riemannian manifold and let $\pi : F \to \Sigma$ and $\pi' : F' \to \Sigma$ be real (or complex) vector bundles equipped with euclidean (or hermitian) metrics and metric connections.
Then by \cref{Example_Elliptic_operator_I}, any elliptic differential operator $P : \Gamma(F) \to \Gamma(F')$ defines a Fredholm operator $P : W(F) \to W(F')$.
\end{example}

\clearpage
\section{Nonlinear maps between $\overline{\text{sc}}$-Fr{\'e}chet spaces}\label{Section_Nonlinear_maps}

\Needspace{25\baselineskip}
\subsection{Envelopes and $\overline{\text{sc}}$-continuity}\label{Subsection_Envelopes}

\begin{definition}\label{Definition_Envelope}
Let $\mathbb{E}$ and $\mathbb{E}'$ be ILB- or sc-chains, let $A \subseteq E_\infty$ be a subset, and let $f : A \to E'_\infty$ be a map.
\begin{enumerate}[label=\arabic*.,ref=\arabic*.]
  \item\label{Definition_Envelope_1} An \emph{envelope of $A$ (in $\mathbb{E}$)} is given by a sequence $\mathcal{U} = (U_k)_{k\in\N_0}$, where
\begin{enumerate}[label=(\alph*),ref=(\alph*)]
  \item\label{Definition_Envelope_1a} $U_k \subseteq E_k$ is an open subset,
  \item\label{Definition_Envelope_1b} $\iota_k(U_{k+1}) \subseteq U_k$ for all $k\in\N_0$, and
  \item\label{Definition_Envelope_1c} $A = \bigcap_{k\in\N_0} (\iota^\infty_k)\inv(U_k)$.
\end{enumerate}
It is called \emph{strict} if $U_{k+1} = \left(\iota_k\right)\inv(U_k)$ for all $k \in \N_0$.
  \item\label{Definition_Envelope_2} A \emph{(continuous) envelope of $f$ (in $\mathbb{E}$ and $\mathbb{E}'$)} is given by a sequence $(f_k : U_k \to E'_k)_{k\in \N_0}$, where
\begin{enumerate}[label=(\alph*),ref=(\alph*)]
  \item\label{Definition_Envelope_2a} $\mathcal{U} = (U_k)_{k\in\N_0}$ is an envelope of $A$ (in $\mathbb{E}$) and
  \item\label{Definition_Envelope_2b} the $f_k : U_k \to E'_k$ are continuous maps \st
\begin{enumerate}[label=\roman*.,ref=\roman*.]
  \item\label{Definition_Envelope_2bi} $f_k \circ \iota^\infty_k|_A = \iota'^\infty_k\circ f$ for all $k \in \N_0$.
  \item\label{Definition_Envelope_2bii} $f_k\circ \iota_k|_{U_{k+1}} = \iota'_k\circ f_{k+1}$ for all $k\in\N_0$.
\end{enumerate}
\end{enumerate}
It will be denoted by $\mathcal{F} : \mathcal{U} \to \mathbb{E}'$, or just $\mathcal{F}$ if the choice of $\mathcal{U}$ and $\mathbb{E}'$ is clear, and it will be called \emph{strict} if $\mathcal{U}$ is a strict envelope of $A$.
  \item Given two envelopes $\mathcal{U} = (U_k)_{k\in\N_0}$ and $\tilde{\mathcal{U}} = (\tilde{U}_k)_{k\in\N_0}$ of $A$, $\tilde{\mathcal{U}}$ is said to be a \emph{refinement} of $\mathcal{U}$ if $\tilde{U}_k \subseteq U_k$ for all $k \in\N_0$.
  \item Given an envelope $\mathcal{F} = (f_k : U_k \to E'_k)_{k\in\N_0} : \mathcal{U} \to \mathbb{E}'$ and a refinement $\tilde{\mathcal{U}}$ of $\mathcal{U}$, the \emph{restriction} $\mathcal{F}|_{\tilde{\mathcal{U}}} : \tilde{\mathcal{U}} \to \mathbb{E}'$ is the envelope $(f_k|_{\tilde{U}_k} : \tilde{U}_k \to E'_k)_{k\in\N_0}$ of $f$.
  \item Given two envelopes $\mathcal{F}$ and $\tilde{\mathcal{F}}$ of $f$, $\tilde{\mathcal{F}}$ is said to be a \emph{refinement} of $\mathcal{F}$, if $\tilde{\mathcal{U}}$ is a refinement of $\mathcal{U}$ and $\tilde{\mathcal{F}} = \mathcal{F}|_{\tilde{\mathcal{U}}}$.
  \item Two envelopes of $f$ are called \emph{equivalent} \iff they have a common refinement.
\end{enumerate}
\end{definition}

\begin{lemma}\label{Lemma_Envelope_determines_map}
Let $\mathbb{E}$ and $\mathbb{E}'$ be ILB- or sc-chains, let $A \subseteq E_\infty$ be a subset, and let $f : A \to E'_\infty$ be a map. \\
$f$ is continuous and uniquely determined by the maps $f_k$ for $k\in\N_0$ in the following sense: \\
Let $\mathcal{U} = (U_k)_{k\in\N_0}$ be an envelope of a subset $A\subseteq E_\infty$ and let $\mathcal{F} = (f_k : U_k \to E'_k)_{k\in\N_0}$ be a sequence of continuous maps \st $f_k\circ \iota_k = \iota'_k\circ f_{k+1}$ for all $k\in\N_0$.
Then there exists a unique continuous map $f : A \to E'_\infty$ \st $\mathcal{F}$ is an envelope of $f$.
\end{lemma}
\begin{proof}
That $f$ is continuous follows immediately from \cref{Definition_Envelope}, \labelcref{Definition_Envelope_2}~\labelcref{Definition_Envelope_2b}~\labelcref{Definition_Envelope_2bi,Definition_Envelope_2bii}~because $f$ is the inverse limit of the (continuous) maps $f_k$. \\
Given $a\in A$, $f_k(\iota^\infty_k(a)) = f_k(\iota_k\circ \iota^\infty_{k+1}(a)) = \iota'_k f_1(\iota^\infty_{k+1}(a)) = \cdots = \iota'^\ell_k f_\ell(\iota^\infty_\ell(a))$ for all $k,\ell \in \N_0$ with $\ell > k$.
So $f_k(\iota^\infty_k(a)) \in \im\iota'^\ell_k$ for all $k,\ell\in\N_0$ with $\ell > k$ and hence $f_k(\iota^\infty_k(a)) \in \im\iota'^\infty_k$. \\
Defining $f(a) \definedas \left(\iota'^\infty_k\right)\inv(f_k(\iota^\infty_k(a)))$, the above shows that this is well defined, independent of the choice of $k\in \N_0$ and thus defines map $f : A \to E'_\infty$.
Furthermore, by definition, $f_k\circ \iota^\infty_k|_A = \iota'^\infty_k\circ f$ for all $k\in\N_0$.
It remains to show that $f$ is continuous.
By definition of the topology on $E'_\infty$, $f$ is continuous \iff the maps $\iota'^\infty_k\circ f$ are continuous for al $k \in \N_0$.
But the formula $\iota'^\infty_k\circ f = f_k\circ \iota^\infty_k|_A$ shows that they are compositions of continuous functions, hence continuous.
\end{proof}

\begin{remark}
Note that if $\mathcal{U}$ is a strict envelope of $A$, then $A = (\iota^\infty_0)\inv(U_0)\subseteq E_\infty$ is open and $U_\ell = (\iota^\ell_k)\inv(U_k)$ for all $\ell \geq k$.
\end{remark}

\begin{lemma}\label{Lemma_Envelopes}
Let $\mathbb{E}$ and $\mathbb{E}'$ be ILB- or sc-chains, let $A \subseteq E_\infty$ be a subset, and let $f : A \to E'_\infty$ be a map.
Then:
\begin{enumerate}[label=\arabic*.,ref=\arabic*.]
  \item\label{Lemma_Envelopes_1} Any two envelopes of $A$ have a common refinement.
  \item\label{Lemma_Envelopes_2} A strict envelope of $A$ is a refinement of any other envelope of $A$.
In particular, $A$ has at most one strict envelope.
  \item\label{Lemma_Envelopes_3} A strict envelope of $f$ is a refinement of any other envelope of $f$.
In particular, $f$ has at most one strict envelope.
\end{enumerate}
\end{lemma}
\begin{proof}
\begin{enumerate}[label=\arabic*.,ref=\arabic*.]
  \item Let $\mathcal{U} = (U_k)_{k\in\N_0}$ and $\tilde{\mathcal{U}} = (\tilde{U}_k)_{k\in\N_0}$ be two envelopes of $A$.
Then $\mathcal{U} \cap \tilde{\mathcal{U}} \definedas (U_k \cap \tilde{U}_k)_{k\in\N_0}$ is a common refinement of $\mathcal{U}$ and $\tilde{\mathcal{U}}$.
  \item If $(U_k)_{k\in\N_0}$ is a strict envelope of $A$, then $A = \left(\iota^\infty_k\right)\inv(U_k)$ for any $k\in\N_0$.
Now if $e \in U_k$, because $\iota^\infty_k(E_\infty)$ is dense in $E_k$, there exists a sequence $(e_j)_{j\in\N_0} \subseteq E_\infty$ \st $\iota^\infty_k(e_j) \to e$ in $E_k$.
Since $U_k$ is open, for $j$ large enough, $\iota^\infty_k(e_j) \in U_k$ and hence $e_j \in \left(\iota^\infty_k\right)\inv(U_k) = A$.
So one can assume that $(e_j)_{j\in\N_0} \subseteq A$ and hence it follows that $\iota^\infty_k(A)$ is dense in $U_k$.
Now if $(\tilde{U}_k)_{k\in\N_0}$ is another envelope of $A$, then for each $k\in\N_0$, both $U_k$ and $\tilde{U}_k$ contain the same subset $\iota^\infty_k(A)$, and since $\iota^\infty_k(A)$ is dense in $U_k$ hence $U_k \subseteq \tilde{U}_k$.
  \item If $\mathcal{F} = (f_k : U_k \to E'_k)_{k\in\N_0}$ and $\tilde{\mathcal{F}} = (\tilde{f}_k : \tilde{U}_k \to E'_k)_{k\in\N_0}$ are two envelopes of $f$ with $\mathcal{F}$ strict, then by \labelcref{Lemma_Envelopes_2}, for each $k \in \N_0$, $U_k \subseteq \tilde{U}_k$ and $f_k|_{\iota^\infty_k(A)} = \iota'^\infty_k\circ f\circ \left(\iota^\infty_k\right)\inv|_{\iota^\infty_k(A)} = \tilde{f}_k|_{\iota^\infty_k(A)}$, so by continuity $f_k = \tilde{f}_k|_{U_k}$.
\end{enumerate}
\end{proof}

\begin{definition}
Let $\mathbb{E}$, $\tilde{\mathbb{E}}$, $\mathbb{E}'$ and $\tilde{\mathbb{E}}'$ be ILB- or sc-chains, let $A \subseteq E_\infty$ be a subset and let $f : A \to E'_\infty$ be a continuous map.
Let furthermore $\mathbf{k} \subseteq \N_0$ be a strictly monotone increasing sequence, and let $\mathbb{T} : \tilde{\mathbb{E}} \to \mathbb{E}$ and $\mathbb{S} : \mathbb{E}' \to \tilde{\mathbb{E}}'$ be continuous linear operators.
\begin{enumerate}[label=\arabic*.,ref=\arabic*.]
  \item Given an envelope $\mathcal{U} = (U_k)_{k\in\N_0}$ of $A$,
\begin{enumerate}[label=(\alph*),ref=(\alph*)]
  \item the \emph{rescaling of $\mathcal{U}$ (by $\mathbf{k}$)} is the envelope
\[
\mathcal{U}^{\mathbf{k}} \definedas (U_{k_j})_{j\in \N_0}
\]
of $A$ in $\mathbb{E}^{\mathbf{k}}$
  \item The \emph{pullback of $\mathcal{U}$ by $\mathbb{T}$} is the envelope
\[
\mathbb{T}^\ast \mathcal{U} \definedas (T_k\inv(U_{k}))_{k\in\N_0}
\]
of $T_\infty^\ast A \definedas T_\infty\inv(A)\subseteq \tilde{E}_\infty$.
\end{enumerate}
  \item Given an envelope $\mathcal{F} = (f_k : U_k \to E'_k)_{k\in\N_0}$ of $f$,
\begin{enumerate}[label=(\alph*),ref=(\alph*)]
  \item the \emph{rescaling of $\mathcal{F}$ (by $\mathbf{k}$)} is the envelope
\[
\mathcal{F}^{\mathbf{k}} \definedas (f_{k_j} : U_{k_j} \to E'_{k_j})_{j\in\N_0} : \mathcal{U}^{\mathbf{k}} \to \mathbb{E}'
\]
of $f$ in $\mathbb{E}^{\mathbf{k}}$ and $\mathbb{E}'^{\mathbf{k}}$
  \item The \emph{pullback of $\mathcal{F}$ by $\mathbb{T}$} is the envelope
\[
\mathbb{T}^\ast \mathcal{F} \definedas (f_{k}\circ T_k : T_k\inv(U_{k}) \to E'_{k})_{k\in\N_0} : \mathbb{T}^\ast \mathcal{U} \to \mathbb{E}'
\]
of $T_\infty^\ast f \definedas f\circ T_\infty : T_\infty^\ast A \to E'_\infty$.
The \emph{composition} of $\mathcal{F}$ with $\mathbb{S}$ is the envelope
\[
\mathbb{S}\circ \mathcal{F} \definedas (S_k\circ f_k : U_k \to \tilde{E}'_k)_{k\in\N_0} : \mathcal{U} \to \tilde{\mathbb{E}}'
\]
of $S_\infty \circ f : A \to \tilde{E}'_\infty$.
\end{enumerate}
\end{enumerate}
\end{definition}

\begin{lemma}\label{Lemma_Rescaling_of_envelopes}
Let $\mathbb{E}$, $\tilde{\mathbb{E}}$, $\mathbb{E}'$ and $\tilde{\mathbb{E}}'$ be ILB- or sc-chains, let $\mathbf{k}, \mathbf{l}, \mathbf{l}', \mathbf{m} \subseteq \N_0$ be strictly monotone increasing sequences with $\mathbf{k} \geq \mathbf{l} \geq \mathbf{l}' \geq \mathbf{m}$, and let $\mathbb{T} : \tilde{\mathbb{E}} \to \mathbb{E}$ and $\mathbb{S} : \mathbb{E}' \to \tilde{\mathbb{E}}'$ be continuous linear operators.
Let furthermore $A \subseteq E_\infty$ be a subset and let $f : A \to E'_\infty$ be a map together with envelopes $\mathcal{U} = (U_k)_{k\in\N_0}$ and $\mathcal{F} = (f_k : U_k \to E'_k)_{k\in\N_0}$ of $A$ and $f$, respectively.
Then the following holds:
\begin{enumerate}[label=\arabic*.,ref=\arabic*.]
  \item\label{Lemma_Rescaling_of_envelopes_1} The envelope $\left(\mathbb{I}^{\mathbf{k}}_{\mathbf{l}}\right)^\ast \mathcal{U}^{\mathbf{l}}$ of $A$ in $\mathbb{E}^\mathbf{k}$ is a refinement of $\left(\mathbb{I}^{\mathbf{k}}_{\mathbf{l}'}\right)^\ast \mathcal{U}^{\mathbf{l}'}$ and the envelope $\mathbb{I}'^{\mathbf{l}}_{\mathbf{m}}\circ \left(\mathbb{I}^{\mathbf{k}}_{\mathbf{l}}\right)^\ast \mathcal{F}^{\mathbf{l}}$ of $f$ in $\mathbb{E}^{\mathbf{k}}$ and $\mathbb{E}'^{\mathbf{m}}$ is a refinement of $\mathbb{I}'^{\mathbf{l}'}_{\mathbf{m}}\circ \left(\mathbb{I}^{\mathbf{k}}_{\mathbf{l}'}\right)^\ast \mathcal{F}^{\mathbf{l}'}$. \\
In particular, up to equivalence, the envelopes $\left(\mathbb{I}^{\mathbf{k}}_{\mathbf{l}}\right)^\ast \mathcal{U}^{\mathbf{l}}$ and $\mathbb{I}'^{\mathbf{l}}_{\mathbf{m}}\circ \left(\mathbb{I}^{\mathbf{k}}_{\mathbf{l}}\right)^\ast \mathcal{F}^{\mathbf{l}}$ are independent of $\mathbf{l}$.
  \item\label{Lemma_Rescaling_of_envelopes_2} Given another envelope $\tilde{\mathcal{F}} : \tilde{\mathcal{U}} \to \mathbb{E}'$, $\mathbb{I}'^{\mathbf{k}}_{\mathbf{m}}\circ \mathcal{F}^{\mathbf{k}}$ and $\mathbb{I}'^{\mathbf{k}}_{\mathbf{m}}\circ \tilde{\mathcal{F}}^{\mathbf{k}}$ are equivalent \iff $\mathcal{F}^{\mathbf{k}}$ and $\tilde{\mathcal{F}}^{\mathbf{k}}$ are equivalent.
  \item\label{Lemma_Rescaling_of_envelopes_3} If $\mathcal{U}$ is strict, then so are $\mathcal{U}^{\mathbf{k}}$ and $\mathbb{T}^\ast \mathcal{U}$.
  \item\label{Lemma_Rescaling_of_envelopes_4} If $\mathcal{F}$ is strict, then so are $\mathcal{F}^{\mathbf{k}}$, $\mathbb{T}^\ast\mathcal{F}$ and $\mathbb{S}\circ \mathcal{F}$.
\end{enumerate}
\end{lemma}
\begin{proof}
\begin{enumerate}[label=\arabic*.,ref=\arabic*.]
  \item Denote $\tilde{\mathcal{U}} \definedas \left(\mathbb{I}^{\mathbf{k}}_{\mathbf{l}}\right)^\ast \mathcal{U}^{\mathbf{l}}$ and $\tilde{\mathcal{U}}' \definedas \left(\mathbb{I}^{\mathbf{k}}_{\mathbf{l}'}\right)^\ast \mathcal{U}^{\mathbf{l}'}$.
Similarly, denote $\tilde{\mathcal{F}} \definedas \mathbb{I}'^{\mathbf{l}}_{\mathbf{m}}\circ \left(\mathbb{I}^{\mathbf{k}}_{\mathbf{l}}\right)^\ast \mathcal{F}^{\mathbf{l}}$ and $\tilde{\mathcal{F}}' \definedas \mathbb{I}'^{\mathbf{l}'}_{\mathbf{m}}\circ \left(\mathbb{I}^{\mathbf{k}}_{\mathbf{l}'}\right)^\ast \mathcal{F}^{\mathbf{l}'}$.
Then for any $j \in \N_0$,
\begin{align*}
\tilde{U}'_j &= \left(\iota^{k_j}_{l'_j}\right)\inv U_{l'_j} = \left(\iota^{l_j}_{l'_j}\circ \iota^{k_j}_{l_j}\right)\inv U_{l'_j} \\
&= \left(\iota^{k_j}_{l_j}\right)\inv \underbrace{\left(\iota^{l_j}_{l'_j}\right)\inv U_{l'_j}}_{\supseteq\; U_{l_j}} \\
&\supseteq \left(\iota^{k_j}_{l_j}\right)\inv U_{l_j} \\
&= \tilde{U}_j
\intertext{and}
\tilde{f}'_j &= \iota^{l'_j}_{m_j}\circ f_{l'_j}\circ \iota^{k_j}_{l'_j}|_{\tilde{U}'_j} \\
&= \iota^{l'_j}_{m_j}\circ f_{l'_j}\circ \iota^{l_j}_{l'_j}\circ \iota^{k_j}_{l_j}|_{\tilde{U}'_j} \\
&= \iota^{l'_j}_{m_j}\circ \iota'^{l_j}_{l'_j}\circ f_{l_j}\circ \iota^{k_j}_{l_j}|_{\tilde{U}'_j} \\
&= \iota^{l_j}_{m_j}\circ f_{l_j}\circ \iota^{k_j}_{l_j}|_{\tilde{U}'_j}\text{,}
\intertext{so}
\tilde{f}'_j|_{\tilde{U}_j} &= \iota^{l_j}_{m_j}\circ f_{l_j}\circ \iota^{k_j}_{l_j}|_{\tilde{U}_j} \\
&= \tilde{f}_j\text{.}
\end{align*}
  \item Straightforward because all the $\iota'^{k_j}_{m_j}$ for $j\in\N_0$ are injective.
  \item That $\mathcal{U}^{\mathbf{k}}$ is strict is almost tautological and if $\tilde{\mathcal{U}} = \mathbb{T}^\ast \mathcal{U}$, then for any $k\in\N_0$, $\tilde{U}_{k+1} = T_{k+1}\inv(U_{k+1}) = T_{k+1}\inv(\iota_k\inv(U_k)) = (\iota_k\circ T_{k+1})\inv(U_k) = (T_k\circ \tilde{\iota}_k)\inv(U_k) = \tilde{\iota}_k\inv(T_k\inv(U_k)) = \tilde{\iota}_k\inv(\tilde{U}_k)$.
  \item Immediate from \labelcref{Lemma_Rescaling_of_envelopes_3}~and the definition of ``strict''.
\end{enumerate}
\end{proof}

\begin{definition}
Let $\mathbb{E}$ and $\mathbb{E}'$ be ILB- or sc-chains, let $A \subseteq E_\infty$ be a subset, and let $f : A \to E'_\infty$ be a map. \\
Two envelopes $\mathcal{F}$ and $\tilde{\mathcal{F}}$ of $f$ are said to be \emph{weakly equivalent} if there exists a strictly monotone increasing sequence $\mathbf{k} \subseteq \N_0$ \st $\mathcal{F}^{\mathbf{k}}$ and $\tilde{\mathcal{F}}^{\mathbf{k}}$ are equivalent.
\end{definition}

\begin{lemma}
Weak equivalence of envelopes is an equivalence relation.
\end{lemma}
\begin{proof}
Symmetry and reflexivity are clear. \\
For transitivity, let $\mathbb{E}$ and $\mathbb{E}'$ be ILB- or sc-chains, let $A \subseteq E_\infty$ be a subset and let $f : A \to E'_\infty$ be a continuous map. \\
Let envelopes $\mathcal{F}$, $\tilde{\mathcal{F}}$ and $\overline{\mathcal{F}}$ of $f$ be given \st $\mathcal{F}$ is weakly equivalent to $\tilde{\mathcal{F}}$ and $\tilde{\mathcal{F}}$ is weakly equivalent to $\overline{\mathcal{F}}$.
Then there exist strictly monotone increasing sequences $\mathbf{k}, \mathbf{k}' \subseteq \N_0$ \st $\mathcal{F}^{\mathbf{k}}$ is equivalent to $\tilde{\mathcal{F}}^{\mathbf{k}}$ and $\tilde{\mathcal{F}}^{\mathbf{k}'}$ is equivalent to $\overline{\mathcal{F}}{}^{\mathbf{k}'}$.
Then $\mathbb{I}'^{\mathbf{k}}\circ \mathcal{F}^{\mathbf{k}}$ and $\mathbb{I}'^{\mathbf{k}}\circ \tilde{\mathcal{F}}^{\mathbf{k}}$ are equivalent and by \cref{Lemma_Rescaling_of_envelopes} hence so are $\left(\mathbb{I}^{\mathbf{k}}\right)^\ast \mathcal{F}$ and $\left(\mathbb{I}^{\mathbf{k}}\right)^\ast \tilde{\mathcal{F}}$.
And similarly $\left(\mathbb{I}^{\mathbf{k}'}\right)^\ast \tilde{\mathcal{F}}$ and $\left(\mathbb{I}^{\mathbf{k}'}\right)^\ast \overline{\mathcal{F}}$ are equivalent.
Define $\mathbf{m} = (\max\{k_j,k'_j\})_{j\in \N_0} \subseteq \N_0$.
Then $\left(\mathbb{I}^{\mathbf{m}}\right)^\ast \mathcal{F} = \left(\mathbb{I}^{\mathbf{m}}_{\mathbf{k}}\right)^\ast \left(\mathbb{I}^{\mathbf{k}}\right)^\ast \mathcal{F}$ is equivalent to $\left(\mathbb{I}^{\mathbf{m}}\right)^\ast \tilde{\mathcal{F}}  = \left(\mathbb{I}^{\mathbf{m}}_{\mathbf{k}}\right)^\ast \left(\mathbb{I}^{\mathbf{k}}\right)^\ast \tilde{\mathcal{F}} = \left(\mathbb{I}^{\mathbf{m}}_{\mathbf{l}}\right)^\ast \left(\mathbb{I}^{\mathbf{l}}\right)^\ast \tilde{\mathcal{F}}$ which in turn is equivalent to $\left(\mathbb{I}^{\mathbf{m}}\right)^\ast \overline{\mathcal{F}} = \left(\mathbb{I}^{\mathbf{m}}_{\mathbf{l}}\right)^\ast \left(\mathbb{I}^{\mathbf{l}}\right)^\ast \overline{\mathcal{F}}$.
So $\left(\mathbb{I}^{\mathbf{m}}\right)^\ast \mathcal{F}$ is equivalent to $\left(\mathbb{I}^{\mathbf{m}}\right)^\ast \overline{\mathcal{F}}$ and by \cref{Lemma_Rescaling_of_envelopes}, $\mathbb{I}'^{\mathbf{m}}\circ \mathcal{F}^{\mathbf{m}}$ is equivalent to $\mathbb{I}'^{\mathbf{m}}\circ \overline{\mathcal{F}}{}^{\mathbf{m}}$.
Again by \cref{Lemma_Rescaling_of_envelopes}, $\mathcal{F}^{\mathbf{m}}$ is equivalent to $\overline{\mathcal{F}}{}^{\mathbf{m}}$.
\end{proof}

\begin{definition}\label{Definition_Envelope_in_sc_Frechet_space}
Let $E$ and $E'$ be $\overline{\text{sc}}$-Fr{\'e}chet spaces, let $A \subseteq E$ be a subset, and let $f : A \to E'$ be a map.
\begin{enumerate}[label=\arabic*.,ref=\arabic*.]
  \item An \emph{envelope of $A$} is given by a pair $((\mathbb{E},\phi), \mathcal{U})$, where $(\mathbb{E}, \phi)$ is a compatible $\overline{\text{sc}}$-structure on $E$ and $\mathcal{U}$ is an envelope of $\phi\inv(A) \subseteq E_\infty$ in $\mathbb{E}$. \\
It is called \emph{strict} if $\mathcal{U}$ is strict.
  \item An \emph{envelope of $f$} is given by a triple $((\mathbb{E},\phi), (\mathbb{E}',\phi'), \mathcal{F} : \mathcal{U} \to \mathbb{E}')$, where $(\mathbb{E}, \phi)$ and $(\mathbb{E}',\phi')$ are compatible $\overline{\text{sc}}$-structures on $E$ and $E'$, respectively, and $\mathcal{F} : \mathcal{U} \to \mathbb{E}'$ is an envelope of $\phi'^{-1}\circ f\circ \phi : \phi\inv(A) \to E'_\infty$ in $\mathbb{E}$ and $\mathbb{E}'$. \\
It is called \emph{strict} if $\mathcal{F}$ is strict.
  \item Let $\mathfrak{F} = ((\mathbb{E},\phi), (\mathbb{E}',\phi'), \mathcal{F})$ and $\tilde{\mathfrak{F}} = ((\tilde{\mathbb{E}},\tilde{\phi}), (\tilde{\mathbb{E}}',\tilde{\phi}'), \tilde{\mathcal{F}})$ be two envelopes of $f$.
By compatibility of $(\mathbb{E}, \phi)$ with $(\tilde{E}, \tilde{\phi})$ and of $(\mathbb{E}', \phi')$ with $(\tilde{\mathbb{E}}', \tilde{\phi}')$, $J = \tilde{\phi}\inv\circ \phi : E_\infty \to \tilde{E}_\infty$ and $K' = \phi'^{-1}\circ \tilde{\phi}' : \tilde{E}'_\infty \to E'_\infty$ are weak equivalences.
$\mathfrak{F}$ and $\tilde{\mathfrak{F}}$ are said to be \emph{equivalent} \iff for any extensions $\mathbb{J} : \mathbb{E}^{\mathbf{k}} \to \tilde{\mathbb{E}}$ and $\mathbb{K}' : \tilde{\mathbb{E}}'^{\mathbf{l}} \to \mathbb{E}'$ of $J$ and $K$, respectively, $\mathbb{I}'^{\mathbf{k}\circ \mathbf{l}}\circ \mathcal{F}^{\mathbf{k}\circ \mathbf{l}}$ and $\mathbb{K}'\circ \bigl(\mathbb{J}^\ast \tilde{\mathcal{F}}\bigr)^{\mathbf{l}}$ are weakly equivalent.
\end{enumerate}
\end{definition}

\begin{lemma}
Equivalence of envelopes is a well defined (\ie independent of choices) equivalence relation.
\end{lemma}
\begin{proof}
Straightforward.
\end{proof}

\begin{proposition}\label{Proposition_Strict_envelopes}
Let $E$ and $E'$ be $\overline{\text{sc}}$-Fr{\'e}chet spaces, let $A \subseteq E$ be a subset, and let $f : A \to E'$ be a map. \\
If $A \subseteq E$ is open, then for any $a \in A$ there exists a neighbourhood $B \subseteq A$ of $a$ in $E$ \st the following holds:
\begin{enumerate}[label=\arabic*.,ref=\arabic*.]
  \item\label{Proposition_Strict_envelopes_1} If $((\mathbb{E},\phi), \mathcal{U})$ is any envelope of $B$, then there exists a shift $\mathbf{k} \subseteq \N_0$ \st $\mathcal{U}^{\mathbf{k}}$ has a refinement by a strict envelope.
  \item\label{Proposition_Strict_envelopes_2} If $((\mathbb{E},\phi), (\mathbb{E}',\phi'), \mathcal{F})$ is any envelope of $f|_B$, then there exists a shift $\mathbf{k} \subseteq \N_0$ \st $\mathcal{F}^{\mathbf{k}}$ has a refinement by a strict envelope.
Furthermore, $\mathbf{k}$ only depends on $B$ and $(\mathbb{E},\phi)$.
  \item\label{Proposition_Strict_envelopes_3} Every envelope $((\mathbb{E},\phi), (\mathbb{E}',\phi'), \mathcal{F})$ of $f|_B$ is equivalent to a strict envelope.
  \item\label{Proposition_Strict_envelopes_4} Any two envelopes of $f|_B$ are equivalent.
\end{enumerate}
\end{proposition}
\begin{proof}
Let $(\mathbb{E}, \phi)$ be any compatible $\overline{\text{sc}}$-structure on $E$ and denote $\tilde{A} \definedas \phi\inv(A) \subseteq E_\infty$ and $\tilde{a} \definedas \phi\inv(a) \in \tilde{A}$.
Since $\tilde{A}$ is open, by definition of the Fr{\'e}chet topology on $E_\infty$, there exists a $k_0 \in \N_0$ and an open neighbourhood $\tilde{B}_0 \subseteq E_{k_0}$ of $\iota^\infty_{k_0}(\tilde{a})$ in $E_{k_0}$ \st $\tilde{B} \definedas \left(\iota^\infty_{k_0}\right)\inv(\tilde{B}_0) \subseteq \tilde{A}$ is an open neighbourhood of $\tilde{a}$ in $E_\infty$.
Set $V_j \definedas \left(\iota^{k_0+j}_{k_0}\right)\inv(\tilde{B}_0) \subseteq E_{k_0 + j}$ and $\mathcal{V} \definedas (V_j)_{j\in\N_0}$.
$\mathcal{V}$ is a strict envelope of $\tilde{B}$ in $\mathbb{E}^{\mathbf{k}}$, where $\mathbf{k} \definedas (k_0 + j)_{j\in\N_0}$. \\
Set $B \definedas \phi(\tilde{B})$.
\begin{enumerate}[label=\arabic*.,ref=\arabic*.]
  \item If $((\mathbb{E},\phi), \mathcal{U})$ is any envelope of $B$ (with $(\mathbb{E},\phi)$ still fixed), \ie $\mathcal{U}$ is an envelope of $\tilde{B}$, then $\mathcal{V}$ is a refinement of $\mathcal{U}^{\mathbf{k}}$ by \cref{Lemma_Envelopes}. \\
Now if $(\tilde{\mathbb{E}}, \tilde{\phi})$ is any other compatible $\overline{\text{sc}}$-structure on $E$, then there exists an equivalence $K : \tilde{E}_\infty \to E_\infty$ with $\tilde{\phi} = \phi\circ K$.
Using \cref{Lemma_Rescaling_and_weak_morphisms} there exists a continuous linear operator $\mathbb{K} : \tilde{\mathbb{E}}^{\mathbf{l}} \to \mathbb{E}^{\mathbf{k}}$ for some strictly monotone increasing sequence $\mathbf{l} \subseteq \N_0$, with $K_\infty = K$.
Then $\mathbb{K}^\ast \mathcal{V}$ is a strict envelope of $\tilde{\phi}\inv(B)$ in $\tilde{\mathbb{E}}^{\mathbf{l}}$ by \cref{Lemma_Rescaling_of_envelopes} and by \cref{Lemma_Envelopes}, $\mathbb{K}^\ast \mathcal{V}$ is a refinement of $\tilde{\mathcal{U}}^{\mathbf{l}}$ for any envelope $((\tilde{\mathbb{E}}, \tilde{\phi}), \tilde{\mathcal{U}})$ of $B$.
  \item Immediate from \labelcref{Proposition_Strict_envelopes_1}~by restriction of $\mathcal{F}^{\mathbf{k}}$.
  \item Immediate from \labelcref{Proposition_Strict_envelopes_2}~and the definition of equivalence.
  \item  Let $\mathfrak{F} = ((\mathbb{E},\phi), (\mathbb{E}',\phi'), \mathcal{F})$ and $\tilde{\mathfrak{F}} = ((\tilde{\mathbb{E}},\tilde{\phi}), (\tilde{\mathbb{E}}',\tilde{\phi}'), \tilde{\mathcal{F}})$ be two envelopes of $f|_B$.
By compatibility of $(\mathbb{E}, \phi)$ with $(\tilde{E}, \tilde{\phi})$ and of $(\mathbb{E}', \phi')$ with $(\tilde{\mathbb{E}}', \tilde{\phi}')$, $J = \tilde{\phi}\inv\circ \phi : E_\infty \to \tilde{E}_\infty$ and $K' = \phi'^{-1}\circ \tilde{\phi}' : \tilde{E}'_\infty \to E'_\infty$ are weak equivalences, hence there exist extensions $\mathbb{J} : \mathbb{E}^{\mathbf{k}} \to \tilde{\mathbb{E}}$ and $\mathbb{K}' : \tilde{\mathbb{E}}'^{\mathbf{l}} \to \mathbb{E}'$ of $J$ and $K$, respectively.
$\mathfrak{F}$ is equivalent to $\mathfrak{F}' \definedas ((\mathbb{E}^{\mathbf{k}\circ \mathbf{l}},\phi), (\mathbb{E}',\phi'), \mathcal{F}')$, where $\mathcal{F}' \definedas \mathbb{I}'^{\mathbf{k}\circ \mathbf{l}}\circ \mathcal{F}^{\mathbf{k}\circ \mathbf{l}}$ and $\tilde{\mathfrak{F}}$ is equivalent to $\tilde{\mathfrak{F}}' \definedas ((\mathbb{E}^{\mathbf{k}\circ \mathbf{l}},\phi), (\mathbb{E}',\phi'), \tilde{\mathcal{F}}')$, where $\tilde{\mathcal{F}}' \definedas \mathbb{K}'\circ \bigl(\mathbb{J}^\ast \tilde{\mathcal{F}}\bigr)^{\mathbf{l}}$. \\
By \labelcref{Proposition_Strict_envelopes_2}~and \cref{Lemma_Envelopes}, after a further rescaling, $\mathfrak{F}'$ and $\tilde{\mathfrak{F}}'$ have a common refinement.
\end{enumerate}
\end{proof}

\begin{definition}\label{Definition_sc_continuous}
Let $E$ and $E'$ be $\overline{\text{sc}}$-Fr{\'e}chet spaces, let $U \subseteq E$ be an open subset and let $f : U \to E'$ be a map.
\begin{enumerate}[label=\arabic*.,ref=\arabic*.]
  \item\label{Definition_sc_continuous_1} $f$ is called \emph{$\overline{\text{sc}}$-continuous} or $\overline{\text{sc}}^0$ \iff for every point $x \in U$ there exists a neighbourhood $V \subseteq U$ of $x$ \st $f|_V : V \to E'$ has an envelope.
  \item\label{Definition_sc_continuous_2} $f$ is called \emph{$\underline{\text{sc}}$-continuous} or $\underline{\text{sc}}^0$ \iff for every point $x \in U$ there exists a neighbourhood $V \subseteq U$ of $x$ and compatible $\overline{\text{sc}}$-structures $(\mathbb{E},\phi)$ on $E$ and $(\mathbb{E}', \phi')$ on $E'$ \st $\tilde{f} \definedas \phi'^{-1}\circ f|_V\circ \phi : E_\infty \supseteq \phi\inv(V) \to E'_\infty$ has the following properties:
\begin{enumerate}[label=(\alph*),ref=(\alph*)]
  \item $\phi\inv(V) \subseteq (E_\infty, (\iota^\infty_j)^\ast\|\cdot\|_j)$ is open for all $j\in\N_0$ and
  \item
\[
\tilde{f} : (E_\infty, (\iota^\infty_j)^\ast\|\cdot\|_j) \supseteq \phi\inv(V) \to (E'_\infty, (\iota'^\infty_j)^\ast\|\cdot\|'_j)
\]
is continuous for all $j\in\N_0$.
\end{enumerate}
\end{enumerate}
\end{definition}

\begin{lemma}\label{Lemma_Morphism_equals_sc0}
Let $E$ and $E'$ be $\overline{\text{sc}}$-Fr{\'e}chet spaces and let $F : E \to E'$ a linear map.
Then $F$ is a morphism of $\overline{\text{sc}}$-Fr{\'e}chet spaces \iff $F$ is $\underline{\text{sc}}^0$.
\end{lemma}
\begin{proof}
The implication that a morphism of $\overline{\text{sc}}$-Fr{\'e}chet spaces is $\underline{\text{sc}}^0$ is trivial.
So let $F$ be linear and $\underline{\text{sc}}^0$.
Then by definition there exists a neighbourhood $V \subseteq E$ of $0$ and compatible $\overline{\text{sc}}$-structures $(\mathbb{E}, \phi)$ on $E$ and $(\mathbb{E}', \phi')$ on $E'$ \st $\tilde{F} \definedas \phi'^{-1}\circ F\circ \phi : E_\infty \to E'_\infty$ has the property that the linear map $\iota'^\infty_j \circ \tilde{F} : (E_\infty, (\iota^\infty_j)^\ast\|\cdot\|_j) \to (E'_j, \|\cdot\|'_j)$ is continuous on the open neighbourhood $\phi\inv(V)$ of $0$ in the normed space $(E_\infty, (\iota^\infty_j)^\ast\|\cdot\|_j)$.
But $\tilde{F}$ being linear, this implies that $\iota'^\infty_j \circ \tilde{F} : (E_\infty, (\iota^\infty_j)^\ast\|\cdot\|_j) \to (E'_j, \|\cdot\|'_j)$ is continuous everywhere and hence has a unique extension to a continuous linear map $\tilde{F}_j : E_j \to E'_j$.
It is easy to see that the maps $(\tilde{F}_j)_{j\in\N_0}$ define a continuous linear operator $\tilde{\mathbb{F}} : \mathbb{E} \to \mathbb{E}'$, which shows that $F : E \to E'$ defines a morphism of $\overline{\text{sc}}$-Fr{\'e}chet spaces.
\end{proof}

\begin{lemma}\label{Lemma_sc0_well_defined}
Let $E$ and $E'$ be $\overline{\text{sc}}$-Fr{\'e}chet spaces, let $U \subseteq E$ be an open subset and let $f : U \to E'$ be a map.
\begin{enumerate}[label=\arabic*.,ref=\arabic*.]
  \item If $f$ is $\overline{\text{sc}}$-continuous, then given $x \in U$ and compatible $\overline{\text{sc}}$-structures $(\mathbb{E}, \phi)$ and $(\mathbb{E}',\phi')$ on $E$ and $E'$, respectively, there exists a neighbourhood $V \subseteq U$ of $x$ and a strictly monotone increasing sequence $\mathbf{k} \subseteq \N_0$ \st $f|_V$ has an envelope of the form $((\mathbb{E}^{\mathbf{k}}, \phi), (\mathbb{E}', \phi'), \mathcal{F} : \mathcal{V} \to \mathbb{E}')$.
  \item If $f$ is $\underline{\text{sc}}$-continuous, then given $x \in U$ and compatible $\overline{\text{sc}}$-structures $(\mathbb{E}, \phi)$ and $(\mathbb{E}',\phi')$ on $E$ and $E'$, respectively, there exists a neighbourhood $V \subseteq U$ of $x$ and a strictly monotone increasing sequence $\mathbf{k} \subseteq \N_0$ \st $\tilde{f} \definedas \phi'^{-1}\circ f|_V\circ \phi : E^{\mathbf{k}}_\infty \supseteq \phi\inv(V) \to E'_\infty$ has the following properties:
\begin{enumerate}[label=(\alph*),ref=(\alph*)]
  \item $\phi\inv(V) \subseteq (E^{\mathbf{k}}_\infty, ((\iota^{\mathbf{k}})^\infty_j)^\ast\|\cdot\|^{\mathbf{k}}_j)$ is open for all $j\in\N_0$ and
  \item 
\[
\tilde{f} : (E^{\mathbf{k}}_\infty, ((\iota^{\mathbf{k}})^\infty_j)^\ast\|\cdot\|^{\mathbf{k}}_j) \supseteq \phi\inv(V) \to (E'_\infty, (\iota'^\infty_j)^\ast\|\cdot\|'_j)
\]
is continuous for all $j\in\N_0$.
\end{enumerate}
\end{enumerate}
\end{lemma}
\begin{proof}
Straightforward, using \cref{Proposition_Strict_envelopes}.
\end{proof}

\begin{proposition}\label{Proposition_Composition_sc_continuous}
Let $E$, $E'$ and $E''$ be $\overline{\text{sc}}$-Fr{\'e}chet spaces, let $U \subseteq E$ and $U' \subseteq E'$ be open subsets and let $f : U \to E'$ and $f' : U' \to E''$ be maps with $f(U) \subseteq U'$.
\begin{enumerate}[label=\arabic*.,ref=\arabic*.]
  \item\label{Proposition_Composition_sc_continuous_1} If $f$ is $\underline{\text{sc}}$-continuous, then $f$ is continuous (as a map between Fr{\'e}chet spaces).
  \item\label{Proposition_Composition_sc_continuous_2} If $f$ is $\overline{\text{sc}}$-continuous, then $f$ is $\underline{\text{sc}}$-continuous.
  \item\label{Proposition_Composition_sc_continuous_3} If $f$ and $f'$ are $\overline{\text{sc}}$-continuous, then so is $f'\circ f$.
  \item\label{Proposition_Composition_sc_continuous_4} If $f$ and $f'$ are $\underline{\text{sc}}$-continuous, then so is $f'\circ f$.
\end{enumerate}
\end{proposition}
\begin{proof}
\begin{enumerate}[label=\arabic*.,ref=\arabic*.]
  \item Clear from the definitions.
  \item Dito.
  \item Let $x \in U$ and $x' \definedas f(x) \in U'$.
By assumption there exist neighbourhoods $V \subseteq U$ and $V' \subseteq U'$ of $x$ and $x'$, respectively together with envelopes $\mathfrak{F} = ((\mathbb{E},\phi), (\mathbb{E}',\phi'), \mathcal{F} : \mathcal{U} \to \mathbb{E}')$ and $\mathfrak{F}' = ((\tilde{\mathbb{E}}',\tilde{\phi}'), (\mathbb{E}'',\phi''), \mathcal{F}' : \mathcal{U}' \to \mathbb{E}'')$ of $f|_V$ and $f|_{V'}$, respectively.
Since $(\mathbb{E}',\phi')$ and $(\tilde{\mathbb{E}}', \tilde{\phi}')$ are equivalent $\overline{\text{sc}}$-structures, there exists a weak equivalence $K : E'_\infty \to \tilde{E}'_\infty$ with extension $\mathbb{K} : \mathbb{E}'^{\mathbf{k}} \to \tilde{\mathbb{E}}'$ for some strictly monotone increasing sequence $\mathbf{k} \subseteq \N_0$.
Replacing $\mathfrak{F}$ by $((\mathbb{E}^{\mathbf{k}},\phi), (\mathbb{E}'^{\mathbf{k}},\phi'), \mathcal{F}^{\mathbf{k}})$ and $\mathfrak{F}'$ by $((\mathbb{E}'^{\mathbf{k}},\phi'), (\mathbb{E}'',\phi''), \mathcal{F}'\circ \mathbb{K})$, one can assume that $(\mathbb{E}',\phi') = (\tilde{\mathbb{E}}', \tilde{\phi}')$.
Using \cref{Proposition_Strict_envelopes}, after shrinking $V$ and $V'$ and rescaling, one can furthermore assume that $\mathcal{F}$ and $\mathcal{F}'$ are strict.
Now $f_0 : U_0 \to E'_0$ is continuous, $U'_0 \subseteq E'_0$ is open and $f_0(\iota^\infty_0(\phi(x))) \in U'_0$.
Shrink $U_0$ to a smaller neighbourhood $\tilde{U}_0$ of $\iota^\infty_0(x)$ \st $f_0(\tilde{U}_0) \subseteq U'_0$.
Let $\tilde{U}_j \definedas \bigl(\iota^j_0\bigr)\inv(\tilde{U}_0)$ for $j \in \N_0\cup \{\infty\}$ and $\tilde{V} \definedas \phi\inv(\tilde{U}_\infty)$
Then $\tilde{\mathcal{U}} \definedas (\tilde{U}_j)_{j\in\N_0}$ is a strict envelope for the open neighbourhood $\tilde{U}_\infty$ of $\phi\inv(x)$ and $\mathcal{F}|_{\tilde{\mathcal{U}}}$ is a strict envelope for $f_\infty|_{\tilde{U}_\infty}$, where $f_\infty \definedas \phi'^{-1}\circ f \circ \phi$.
Furthermore, because $\tilde{\mathcal{U}}$ and $\mathcal{U}'$ are strict, $f_j(\tilde{U}_j) \subseteq U'_j$ for all $j\in\N_0\cup\{\infty\}$ and $f(\tilde{V}) \subseteq V'$.
Hence there is a well defined envelope $\mathcal{F}'' : \tilde{U} \to \mathbb{E}''$ of $\phi''^{-1}\circ (f'\circ f|_{\tilde{V}})\circ \phi\inv$ defined by $f''_j \definedas f'_j\circ f_j|_{\tilde{U}_j}$ for all $j\in\N_0$.
  \item If $(\mathbb{E}, \phi)$ and $(\mathbb{E}', \phi')$ are compatible $\overline{\text{sc}}$-structures on $E$ and $E'$, respectively, \st the condition on $f$ in \cref{Definition_sc_continuous}, \labelcref{Definition_sc_continuous_2}, is satisfied, then the same holds for $(\mathbb{E}^{\mathbf{k}}, \phi)$ and $(\mathbb{E}'^{\mathbf{l}}, \phi')$, where $\mathbf{k}, \mathbf{l} \subseteq \N_0$ are strictly monotone increasing sequences with $\mathbf{k} \geq \mathbf{l}$. \\
Also, if $f$ is $\underline{\text{sc}}$-continuous, then given arbitrary compatible $\overline{\text{sc}}$-structures $(\mathbb{E}, \phi)$ and $(\mathbb{E}', \phi')$ on $E$ and $E'$, respectively, there exists a strictly monotone increasing sequence $\mathbf{k} \subseteq \N_0$ \st \cref{Definition_sc_continuous}, \labelcref{Definition_sc_continuous_2}, is satisfied, when $(\mathbb{E}, \phi)$ is replaced by $(\mathbb{E}^{\mathbf{k}}, \phi)$. \\
These two facts together easily imply the claim.
\end{enumerate}
\end{proof}

\begin{example}\label{Example_Reparametrisation_Action_I}
I continue \cref{Subsection_Reparametrisation_action}. \\
Then in the notation of that Subsection, by \cref{Proposition_Reparametrisation_action_cts}, the reparametrisation action $\Psi : \Gamma_B(F_1) \to \Gamma_B(F_2)$ is $\overline{\text{sc}}$-continuous, where an envelope of $\Psi$ is provided by either the maps $\Psi^k : \Gamma^k_B(F_1) \to \Gamma^k_B(F_2)$, $k \in \N_0$, or the maps $\Psi^{k,p} : W^{k,p}_B(F_1) \to W^{k,p}_B(F_2)$, for some fixed $1 < p < \infty$ and $k \in \N_0$ \st $kp > n$.
\end{example}

\Needspace{25\baselineskip}
\subsection{$\overline{\text{sc}}$-differentiability}

\begin{definition}\label{Definition_sc1}
Let $E$ and $E'$ be $\overline{\text{sc}}$-Fr{\'e}chet spaces, let $U \subseteq E$ be an open subset, and let $f : U \to E'$ be a continuous map.
\begin{enumerate}[label=\arabic*.,ref=\arabic*.]
  \item If $f$ is weakly Fr{\'e}chet differentiable as a continuous map between open subsets of the Fr{\'e}chet spaces $E$ and $E'$ and both $f$ and the \emph{differential of $f$},
\begin{align*}
Df : E\oplus E \supseteq U\times E &\to E' \\
(x,u) &\mapsto Df(x)u\text{,}
\end{align*}
are
\begin{enumerate}[label=(\alph*),ref=(\alph*)]
  \item $\overline{\text{sc}}$-continuous, then $f$ is called \emph{continuously $\overline{\text{sc}}$-differentiable} or $\overline{\text{sc}}^1$.
  \item $\underline{\text{sc}}$-continuous, then $f$ is called \emph{continuously $\underline{\text{sc}}$-differentiable} or $\underline{\text{sc}}^1$.
\end{enumerate}
  \item If $f$ is pointwise weakly Fr{\'e}chet differentiable as a continuous map between open subsets of the Fr{\'e}chet spaces $E$ and $E'$ and both $f$ and the \emph{differential of $f$},
\begin{align*}
Df : E\oplus E \supseteq U\times E &\to E' \\
(x,u) &\mapsto Df(x)u\text{,}
\end{align*}
are
\begin{enumerate}[label=(\alph*),ref=(\alph*)]
  \item $\overline{\text{sc}}$-continuous, then $f$ is called \emph{continuously pointwise $\overline{\text{sc}}$-differentiable} or p-$\overline{\text{sc}}^1$.
  \item $\underline{\text{sc}}$-continuous, then $f$ is called \emph{continuously pointwise $\underline{\text{sc}}$-differentiable} or p-$\underline{\text{sc}}^1$.
\end{enumerate}
\end{enumerate}
\end{definition}

\begin{definition}
Let $E$ and $E'$ be $\overline{\text{sc}}$-Fr{\'e}chet spaces, let $U \subseteq E$ be an open subset, and let $f : U \to E'$ be a map.
\begin{enumerate}[label=\arabic*.,ref=\arabic*.]
  \item An envelope $((\mathbb{E},\phi), (\mathbb{E}',\phi'), \mathcal{F} : \mathcal{U} \to \mathbb{E}')$ of $f$ is called \emph{continuously $\overline{\text{sc}}$-differentiable} or $\overline{\text{sc}}^1$ \iff $f_k : U_k \to E'_k$ is weakly continuously weakly Fr{\'e}chet differentiable along $\iota_k|_{U_{k+1}} : U_{k+1}\to U_k$ for every $k\in\N_0$.
  \item $f$ is called \emph{continuously $\widehat{\text{sc}}$-differentiable} or $\widehat{\text{sc}}^1$ \iff $f$ is $\overline{\text{sc}}$-continuous and for any $x \in U$ there exists a neighbourhood $V \subseteq U$ of $x$ \st $f|_V : V \to E'$ has a continuously $\overline{\text{sc}}$-differentiable envelope.
\end{enumerate}
\end{definition}

\begin{remark}
There is no ``pointwise'' version of $\widehat{\text{sc}}$-differentiability due to the lack of a chain rule:
\cref{Theorem_Chain_rule_I} does not hold for pointwise weak Fr{\'e}chet differentiability.
There is however a chain rule for maps that are, in the present terminology, weakly continuously pointwise weakly Fr{\'e}chet differentiable (along the identity), \cf \cite{MR656198}, Section I.3.3 (esp.~Theorem 3.3.4).
This suffices for a chain rule for continuously pointwise $\overline{\text{sc}}$-differentiable and $\underline{\text{sc}}$-differentiable maps.
\end{remark}

\begin{remark}
Note that in the definition of a continuously $\overline{\text{sc}}$-differentiable envelope it is not required that $f$ itself is weakly continuously weakly Fr{\'e}chet differentiable as a map between Fr{\'e}chet spaces.
\end{remark}

\begin{proposition}\label{Lemma_sc_diffble_implies_weakly_sc_diffble}
Let $E$ and $E'$ be $\overline{\text{sc}}$-Fr{\'e}chet spaces, let $U \subseteq E$ be an open subset, and let $f : U \to E'$ be a map.
Then the following implications hold:
\[
\xymatrix{
f \;\, \widehat{\text{sc}}^1 \;\; \ar@{=>}[d] & \\
f \;\, \overline{\text{sc}}^1 \;\; \ar@{=>}[r] \ar@{=>}[d] & \;\; f \;\, \text{p-}\overline{\text{sc}}^1 \ar@{=>}[d] \\
f \;\, \underline{\text{sc}}^1 \;\; \ar@{=>}[r] & \;\; f \;\, \text{p-}\underline{\text{sc}}^1
}
\]
\end{proposition}
\begin{proof}
I will only show the implication $f$ $\widehat{\text{sc}}^1$ $\Rightarrow$ $f$ $\overline{\text{sc}}^1$.
The remaining implications are either shown analogously or are immediate from the definitions, \cref{Proposition_Composition_sc_continuous} and \cref{Theorem_Relations_between_notions_of_differentiability}. \\
So let $f$ be $\widehat{\text{sc}}^1$.
Then for any $x \in U$ there exists a neighbourhood $V \subseteq U$ of $x$ and a continuously differentiable envelope $\mathfrak{F} = ((\mathbb{E}, \phi), (\mathbb{E}', \phi'), \mathcal{F} : \mathcal{U} \to \mathbb{E}')$ of $f|_V$.
Using \cref{Proposition_Strict_envelopes}, after possibly making $V$ smaller, rescaling and restriction, one can assume that $\mathfrak{F}$ is strict. \\
Define $\mathbf{1} \definedas (k+1)_{k\in\N_0}$.
Then $\mathcal{U}^{\mathbf{1}}\oplus \mathbb{E} \definedas (U_{k+1}\times E_k)_{k\in\N_0} \subseteq \mathbb{E}^{\mathbf{1}}\oplus \mathbb{E}$ is an envelope of $U_\infty \times E_\infty = (\phi\oplus\phi)\inv(V\times E)$, hence $((\mathbb{E}^{\mathbf{1}}\oplus \mathbb{E}, \phi\oplus \phi), \mathcal{U}^{\mathbf{1}}\oplus \mathbb{E})$ a strict envelope of $V\times E\subseteq E\oplus E$.
Now define $\mathcal{DF} : \mathcal{U}^{\mathbf{1}}\times \mathbb{E} \to \mathbb{E}'$ by the sequence of maps
\begin{align*}
Df_k : U_{k+1}\times E_k &\to E'_k \\
(y,u) &\mapsto Df_k(\iota_k(y))u\text{.}
\end{align*}
Each $Df_k$ is a well defined continuous map because $\mathfrak{F}$ is differentiable by assumption, \ie each of the maps $f_k : U_k \to E'_k$ in the envelope $\mathcal{F}$ of $f_\infty\definedas \phi'^{-1}\circ f|_V\circ \phi : U_\infty \to E'_\infty$ is weakly continuously weakly Fr{\'e}chet differentiable along $\iota_k$.
It will now be proved that $f_\infty$ is weakly continuously weakly Fr{\'e}chet differentiable, and hence so is $f|_V$, and that $\mathcal{DF}$ is an envelope of $Df_\infty : U_\infty\times E_\infty \to E'_\infty$.
This finishes the proof. \\
First, note that from $f_k\circ \iota_k|_{U_{k+1}} = \iota'_k\circ f_{k+1}$, by the chain rule, \cref{Theorem_Chain_rule_I}, it follows that $Df_k\circ (\iota_{k+1}|_{U_{k+1}}\times \iota_k) = \iota'_k\circ Df_{k+1} : U_{k+2}\times E_{k+1} \to E'_{k}$.
\cref{Lemma_Envelope_determines_map} shows that there exists a unique continuous map $\Phi : U_\infty\times E_\infty \to E'_\infty$ \st $\mathcal{DF}$ is an envelope of $\Phi$, \ie $Df_k\circ (\iota^\infty_{k+1}|_{U_\infty}\times \iota^\infty_k) = \iota'^\infty_k\circ \Phi$.
It remains to show that $f_\infty$ is weakly Fr{\'e}chet differentiable with differential $Df_\infty = \Phi$.
For this, one computes for $(y,u) \in U_\infty \times E_\infty$, $t \in (0,1]$ and
\begin{align*}
r^{f_\infty}_{y}(u, t) &= \frac{1}{t}\left(f_\infty(y + tu) - f_\infty(u)\right) - \Phi(y, u)
\intertext{that}
\iota'^\infty_k\circ r^{f_\infty}_y(u, t) &= \frac{1}{t}\left(\iota'^\infty_k\circ f_\infty(y + tu) - \iota'^\infty_k\circ f_\infty(y)\right) - \iota'^\infty_k\circ \Phi(y, u) \\
&= \frac{1}{t}\left(f_k(\iota^\infty_k(y) + t\iota^\infty_k(u)) - f_k(\iota^\infty_k(y))\right) - Df_k((\iota^\infty_{k+1}\times \iota^\infty_k)(y,u)) \\
&= r^{f_k}_{\iota^\infty_{k}(y)}(\iota^\infty_k(u), t)
\end{align*}
and hence (note that $\iota^\infty_k(y) = \iota_k(\iota^\infty_{k+1}(y)) \in \im \iota_k$)
\[
\iota'^\infty_k\circ r^{f_\infty}_{y} = r^{f_k}_{\iota^\infty_{k}(y)}\circ (\iota^\infty_k\times \id_{[0,1]})\text{.}
\]
By the definition of the topology on $E'_\infty$, $r^{f_\infty}_y$ is continuous \iff $\iota'^\infty_k\circ r^{f_\infty}_y$ is continuous for all $k \in \N_0$.
And by the above equality, these maps are continuous as compositions of continuous functions.
So by definition, $f_\infty$ is weakly Fr{\'e}chet differentiable with differential given by $\Phi$.
\end{proof}

\begin{definition}
Let $\mathbb{E}$ and $\mathbb{E}'$ be sc-chains, let $U_\infty \subseteq E_\infty$ be an open subset, and let $\mathcal{F} : \mathcal{U} \to \mathbb{E}'$ be a continuously $\overline{\text{sc}}$-differentiable envelope of a continuous map $f_\infty : U_\infty \to E'_\infty$. \\
The envelope $\mathcal{DF} : \mathcal{U}^{\mathbf{1}}\oplus \mathbb{E} \to \mathbb{E}'$ of $Df_\infty : U_\infty\times E_\infty \to E'_\infty$ given by the sequence of maps
\begin{align*}
Df_k : U_{k+1}\times E_k &\to E'_k \\
(y,u) &\mapsto Df_k(\iota_k(y))u
\end{align*}
is called the \emph{differential of $\mathcal{F}$}.
\end{definition}

\begin{definition}\label{Definition_sck}
Let $E$ and $E'$ be $\overline{\text{sc}}$-Fr{\'e}chet spaces and let $U \subseteq E$ be an open subset.
Let furthermore $k \in \N_0$, $k \geq 2$.
\begin{enumerate}[label=\arabic*.,ref=\arabic*.]
  \item An
\begin{enumerate}[label=\roman*.,ref=\roman*.]
  \item $\overline{\text{sc}}$-continuous
  \item $\underline{\text{sc}}$-continuous
  \item $\overline{\text{sc}}$-continuous
  \item $\underline{\text{sc}}$-continuous
\end{enumerate}
map $f : U \to E'$ is called
\begin{enumerate}[label=\roman*.,ref=\roman*.]
  \item \emph{$k$-times continuously $\overline{\text{sc}}$-differentiable} or $\overline{\text{sc}}^k$
  \item \emph{$k$-times continuously $\underline{\text{sc}}$-differentiable} or $\underline{\text{sc}}^k$
  \item \emph{$k$-times continuously pointwise $\overline{\text{sc}}$-differentiable} or p-$\overline{\text{sc}}^k$
  \item \emph{$k$-times continuously pointwise $\underline{\text{sc}}$-differentiable} or p-$\underline{\text{sc}}^k$
\end{enumerate}
\iff
\begin{enumerate}[label=\roman*.,ref=\roman*.]
  \item $f$ is $\overline{\text{sc}}^k$ and $Df$ is $\overline{\text{sc}}^{k-1}$.
It is called $\overline{\text{sc}}^\infty$ \iff it is $\overline{\text{sc}}^k$ for all $k \in \N_0$.
  \item $f$ is $\underline{\text{sc}}^k$ and $Df$ is $\underline{\text{sc}}^{k-1}$.
It is called $\underline{\text{sc}}^\infty$ \iff it is $\underline{\text{sc}}^k$ for all $k \in \N_0$.
  \item $f$ is p-$\overline{\text{sc}}^k$ and $Df$ is p-$\overline{\text{sc}}^{k-1}$.
It is called p-$\overline{\text{sc}}^\infty$ \iff it is p-$\overline{\text{sc}}^k$ for all $k \in \N_0$.
  \item $f$ is p-$\underline{\text{sc}}^k$ and $Df$ is p-$\underline{\text{sc}}^{k-1}$.
It is called p-$\underline{\text{sc}}^\infty$ \iff it is p-$\underline{\text{sc}}^k$ for all $k \in \N_0$.
\end{enumerate}
  \item An envelope $((\mathbb{E},\phi), (\mathbb{E}',\phi'), \mathcal{F} : \mathcal{U} \to \mathbb{E}')$ of $f$ is called \emph{$k$-times continuously $\overline{\text{sc}}$-differentiable} or $\overline{\text{sc}}^k$ \iff $\mathcal{F}$ is $\overline{\text{sc}}^1$ and $\mathcal{DF}$ is $\overline{\text{sc}}^{k-1}$.
It is called $\overline{\text{sc}}^\infty$ \iff it is $\overline{\text{sc}}^k$ for all $k \in \N_0$.
  \item An $\overline{\text{sc}}$-continuous map $f : U \to E'$ is called \emph{$k$-times continuously $\widehat{\text{sc}}$-differentiable} or $\widehat{\text{sc}}^k$, for $k \in \N_0 \cup \{\infty\}$, \iff for any $x \in U$ there exists a neighbourhood $V \subseteq U$ of $x$ \st $f|_V : V \to E'$ has a $k$-times continuously differentiable envelope.
\end{enumerate}
\end{definition}

\begin{remark}
As usual, if $f : U \to E'$ is $k$-times continuously differentiable in one of the versions above, then the repeated differentials provide maps
\begin{align*}
D^k f : E\oplus E^{\oplus k} \supseteq U \times E^k &\to E' \\
(x, u_1, \dots, u_k) &\mapsto D^kf(x)(u_1, \dots, u_k)
\end{align*}
and one can express $D(Df)$ in terms of $D^1f$ and $D^2f$, etc.
\end{remark}

\begin{theorem}[Chain rule]\label{Theorem_Chain_rule_sc_differentiable}
Let $E$, $E'$ and $E''$ be $\overline{\text{sc}}$-Fr{\'e}chet spaces, let $U\subseteq E$ and $U'\subseteq E'$ be open subsets and let $f : U\to E'$ and $f' : U' \to E''$ be maps with $f(U) \subseteq U'$.
For any $k\in \N_0 \cup \{\infty\}$, if $f$ and $f'$ are $\overline{\text{sc}}^k$, $\underline{\text{sc}}^k$, p-$\overline{\text{sc}}^k$, p-$\underline{\text{sc}}^k$ or $\widehat{\text{sc}}^k$, then so is $f'\circ f : U \to E''$.
In all cases, for $k\geq 1$,
\[
D(f'\circ f)(x) = Df'(f(x))\circ Df(x) \quad\forall\, x\in U\text{.}
\]
\end{theorem}
\begin{proof}
This is a corollary to \cref{Theorem_Chain_rule_I}, \cite{MR656198}, Part I, Theorem 3.3.4, and \cref{Proposition_Composition_sc_continuous}.
\end{proof}

\begin{proposition}
Let $E$ and $E'$ be $\overline{\text{sc}}$-Fr{\'e}chet spaces, let $U \subseteq E$ be an open subset, let $f : U \to E'$ be a map, and let $((\mathbb{E},\phi), (\mathbb{E}',\phi'), \mathcal{F} : \mathcal{U} \to \mathbb{E}')$ be an envelope of $f$. \\
If $\mathcal{F} : \mathcal{U} \to \mathbb{E}'$ is strict and $\overline{\text{sc}}^k$ for $k \in \N_0\cup \{\infty\}$, then $f_0 : U_0 \to E'_0$ defines an $\text{sc}^k$-map in the sense of \cite{MR2644764}, Definition 1.8.
\end{proposition}
\begin{proof}
This is a corollary to \cref{Lemma_Strong_Gateaux_implies_Frechet,Lemma_Strong_Gateaux_implies_Frechet_II}.
\end{proof}

\begin{example}\label{Example_Reparametrisation_Action_II}
I continue \cref{Subsection_Reparametrisation_action} and \cref{Example_Reparametrisation_Action_I}. \\
Then in the notation of that Subsection, by \cref{Proposition_Reparametrisation_action_diffble}, the reparametrisation action $\Psi : \Gamma_B(F_1) \to \Gamma_B(F_2)$ is $\widehat{\text{sc}}^\infty$, where an envelope of $\Psi$ is provided by either the maps $\Psi^k : \Gamma^k_B(F_1) \to \Gamma^k_B(F_2)$, $k \in \N_0$, or the maps $\Psi^{k,p} : W^{k,p}_B(F_1) \to W^{k,p}_B(F_2)$, for some fixed $1 < p < \infty$ and $k \in \N_0$ \st $kp > n$. \\
Also, by \cref{Proposition_Reparametrisation_action_diffble}, the differential of $\Psi$ is composed of maps of the same type as $\Psi$:
\begin{align*}
\Psi(b,u) &= (b, \Psi_2(b,u)) \\
\Psi_2(b,u) &= \Phi_b^\ast u \\
D\Psi(b,u)(e,v) &= (e, D\Psi_2(b,u)(e,v)) \\
D\Psi_2(b,u)(e,v) &= \overline{\Phi}_{(b,e)}^\ast \nabla u + \tilde{\Phi}_{(b,e)}^\ast u + \Phi_b^\ast v\text{.}
\end{align*}
Showing that $\Psi$ is $\overline{\text{sc}}^k$ for all $k\in\N_0\cup\{\infty\}$ is now a simple matter of induction, repeatedly using \cref{Proposition_Reparametrisation_action_diffble}.
\end{example}

\clearpage
\section{The Nash-Moser inverse function theorem}\label{Section_Nash_Moser}

In this section, I first give a quick overview of the necessary terminology, adapted to the present notation, from \cite{MR656198} (mainly Chapter II), before stating the famous Nash-Moser inverse function theorem.
This is intended to set up notation and does not aim to give a general pedagogic introduction to the Nash-Moser inverse function theorem.
It assumes knowledge of at least the main points of the article \cite{MR656198}, or Section 51 of the textbook \cite{MR1471480}.
Subsequently, a version that is better adapted to the present notation and needed to prove the constant rank theorem, finite dimensional reduction, and the Sard-Smale theorem for nonlinear Fredholm maps in the next section is stated and proved.

\Needspace{25\baselineskip}
\subsection{Tame $\overline{\text{sc}}$-structures and morphisms}

\begin{lemma}\label{Lemma_Tame_sequence}
Let $\mathbf{k} = (k_j)_{j\in\N_0} \subseteq \N_0$ be a strictly monotone increasing sequence.
Then for all $i, j \in \N_0$ with $i \leq j$,
\[
k_j - k_i \geq j - i
\]
and the following are equivalent:
\begin{enumerate}[label=\arabic*.,ref=\arabic*.]
  \item\label{Lemma_Tame_sequence_1} There exists $c \in \N_0$ \st $k_j - k_i \leq j - i + c$ for all $i,j\in\N_0$ with $i \leq j$.
  \item\label{Lemma_Tame_sequence_2} There exists $c \in \N_0$ \st $k_j \leq j + c$ for all $j\in\N_0$, \ie there exists a shift $\mathbf{l} \subseteq \N_0$ with $\mathbf{k} \leq \mathbf{l}$.
  \item\label{Lemma_Tame_sequence_3} There exists $j_0 \in \N_0$ \st $k_j - k_i = j - i$ for all $j \geq i \geq j_0$.
  \item\label{Lemma_Tame_sequence_4} There exists $j_0 \in \N_0$ and $r \in \N_0$ \st $k_j = j + r$ for all $j \geq j_0$.
  \item\label{Lemma_Tame_sequence_5} There exists a shift $\mathbf{l}$ \st $\mathbf{k}\circ \mathbf{l}$ is a shift again.
\end{enumerate}
\end{lemma}
\begin{proof}
$\mathbf{k}$ strictly monotone increasing means that $k_l - k_{l-1} \geq 1$ for all $l \in \N_0$.
Hence
\[
k_j - k_i = \sum_{l = i+1}^j (k_l - k_{l-1}) \geq \sum_{l=i+1}^j 1 = j - i\text{.}
\]
I will show \labelcref{Lemma_Tame_sequence_1}$\;\Rightarrow\;$\labelcref{Lemma_Tame_sequence_2}$\;\Rightarrow\;$\labelcref{Lemma_Tame_sequence_4}$\;\Rightarrow\;$\labelcref{Lemma_Tame_sequence_1}, \labelcref{Lemma_Tame_sequence_3}$\;\Rightarrow\;$\labelcref{Lemma_Tame_sequence_4}$\;\Rightarrow\;$\labelcref{Lemma_Tame_sequence_3}~and \labelcref{Lemma_Tame_sequence_4}$\;\Rightarrow\;$\labelcref{Lemma_Tame_sequence_5}$\;\Rightarrow\;$\labelcref{Lemma_Tame_sequence_2}.
\begin{enumerate}
  \item[\labelcref{Lemma_Tame_sequence_1} $\Rightarrow$ \labelcref{Lemma_Tame_sequence_2}:] Set $i = 0$.
  \item[\labelcref{Lemma_Tame_sequence_2} $\Rightarrow$ \labelcref{Lemma_Tame_sequence_4}:] It is $k_j - j = k_0 + \sum_{i=1}^j (k_j - k_{j-1} - 1)$ and $k_j - k_{j-1} - 1 \geq 0$ because $\mathbf{k}$ is strictly monotone increasing.
So $k_j -j \leq c$ for all $j \in \N_0$ implies that $k_j - k_{j-1} - 1 \neq 0$ only for finitely many $j \in \N_0$ and hence $k_j - j$ is eventually constant.
Set $r \definedas \lim_{j\to\infty} (k_j - j)$.
  \item[\labelcref{Lemma_Tame_sequence_4} $\Rightarrow$ \labelcref{Lemma_Tame_sequence_1}:] By assumption, if $j_0 \leq i \leq j$, then $k_j - k_i = j + r - (i + r) = j-i$.
If $j \geq j_0$ and $i \leq j_0$, then $k_j - k_i = j + r - k_i \leq j + r \leq j - i + j_0 + r$.
If $0 \leq i \leq j \leq j_0$, then $k_j - k_i \leq k_{j_0} \leq j - i + k_{j_0}$.
So for arbitrary $i,j\in\N_0$ with $i \leq j$, $k_j - k_i \leq j - i + \max\{k_{j_0}, j_0 + r\}$.
  \item[\labelcref{Lemma_Tame_sequence_3} $\Rightarrow$ \labelcref{Lemma_Tame_sequence_4}:] Set $i = j_0$. Then $k_j = j + \underbrace{(k_i - i)}_{\defines \; r}$.
  \item[\labelcref{Lemma_Tame_sequence_4} $\Rightarrow$ \labelcref{Lemma_Tame_sequence_3}:] For $j \geq i \geq j_0$, $k_j - k_i = j + r - (i + r) = j - i$.
  \item[\labelcref{Lemma_Tame_sequence_4} $\Rightarrow$ \labelcref{Lemma_Tame_sequence_5}:] Define $\mathbf{l} \definedas (j_0 + j)_{j\in\N_0}$.
Then $\mathbf{m} \definedas \mathbf{k}\circ \mathbf{l}$ satisfies $m_j = k_{l_j} = k_{j_0 + j} = j_0 + j + r = j + (j_0 + r)$.
  \item[\labelcref{Lemma_Tame_sequence_5} $\Rightarrow$ \labelcref{Lemma_Tame_sequence_2}:] It is $\mathbf{k}\circ \mathbf{l} \geq \mathbf{k}$.
\end{enumerate}
\end{proof}

\begin{definition}\label{Definition_Tame_ILB_chain}
Let $\mathbf{k} = (k_j)_{j\in\N_0} \subseteq \N_0$ be a strictly monotone increasing sequence and let $\mathbb{E}$ and $\mathbb{E}'$ be ILB- or sc-chains.
\begin{enumerate}[label=\arabic*.,ref=\arabic*.]
  \item $\mathbf{k}$ is called \emph{tame} \iff any/all of the conditions in \cref{Lemma_Tame_sequence} are satisfied.
  \item The rescaling $\mathbb{E}^{\mathbf{k}}$ of $\mathbb{E}$ is called \emph{tame} \iff $\mathbf{k}$ is tame.
  \item A weak morphism $T : E_\infty \to E'_\infty$ is called \emph{tame} \iff it has an extension $\mathbb{T} : \mathbb{E}^{\mathbf{k}} \to \mathbb{E}'$ \st $\mathbb{E}^{\mathbf{k}}$ is a tame rescaling.
  \item An equivalence $J : E_\infty \to E'_\infty$ is called \emph{tame} \iff $J$ is tame and it has a tame inverse, \ie if there exists a tame weak embedding $K : E'_\infty \to E_\infty$ \st $J\circ K = \id_{E'_\infty}$ and $K\circ J = \id_{E_\infty}$.
  \item $\mathbb{E}$ and $\mathbb{E}'$ are called \emph{tamely equivalent} \iff there exists a tame equivalence between them.
\end{enumerate}
\end{definition}

\begin{example}
Every shift is tame.
The only rescalings explicitely chosen in any lemma/proposition/theorem so far, such as \cref{Lemma_Characterisation_sc_subspace}, \cref{Lemma_Characterisation_strongly_smoothing_weak_morphisms} or \cref{Proposition_Strict_envelopes}, have been shifts, in particular tame.
\end{example}

\begin{example}
Let $\mathbb{E}$ be a tame ILB- or sc-chain and let $\mathbf{k}, \mathbf{l} \subseteq \N_0$ be tame strictly monotone increasing sequences with $\mathbf{k}\geq \mathbf{l}$.
Then $\mathbb{I}^{\mathbf{k}}_{\mathbf{l}} : \mathbb{E}^{\mathbf{k}} \to \mathbb{E}^{\mathbf{l}}$ is tame.
\end{example}

\begin{lemma}\label{Lemma_Tame_sequence}
Let $\mathbf{k}, \mathbf{l} \subseteq \N_0$ be strictly monotone increasing sequences.
If $\mathbf{k}$ and $\mathbf{l}$ are tame, then so is $\mathbf{k}\circ \mathbf{l}$.
\end{lemma}
\begin{proof}
Obvious.
\end{proof}

\begin{definition}\label{Definition_Pretame_sc_Frechet_space}
Let $E$ be a topological vector space.
\begin{enumerate}[label=\arabic*.,ref=\arabic*.]
  \item Two $\overline{\text{sc}}$-structures $(\mathbb{E}, \phi)$ and $(\tilde{\mathbb{E}}, \tilde{\phi})$ on $E$ are called \emph{tamely equivalent} if there exists a tame equivalence $J : E_\infty \to \tilde{E}_\infty$ with $\phi = \tilde{\phi}\circ J$.
  \item A \emph{pre-tame $\overline{\text{sc}}$-Fr{\'e}chet space} is a topological vector space $E$ together with a tame equivalence class of $\overline{\text{sc}}$-structures on it.
The sc-chains in this equivalence class are then called \emph{tamely compatible}.
  \item A \emph{tame morphism} between pre-tame $\overline{\text{sc}}$-Fr{\'e}chet spaces $E$ and $E'$ is a continuous linear operator $T : E \to E'$ between $E$ and $E'$ as topological vector spaces \st there exist tamely compatible $\overline{\text{sc}}$-structures $(\mathbb{E}, \phi)$ and $(\mathbb{E}', \phi')$ on $E$ and $E'$, respectively, \st $T_\infty \definedas \phi'^{-1}\circ T\circ \phi : E_\infty \to E'_\infty$ defines a tame weak morphism.
\[
\xymatrix{
E_\infty \ar[r]^-{T_\infty} \ar[d]_-{\phi} \ar@{}[rd]|-{\circlearrowleft} & E'_\infty \ar[d]^{\phi'} \\
E \ar[r]^-{T} & E'
}
\]
\end{enumerate}
\end{definition}

\begin{remark}
So the relation between $\overline{\text{sc}}$-Fr{\'e}chet spaces and pre-tame $\overline{\text{sc}}$-Fr{\'e}chet spaces is as follows: \\
Every $\overline{\text{sc}}$-Fr{\'e}chet space has its associated class of compatible $\overline{\text{sc}}$-structures.
This class of $\overline{\text{sc}}$-structures then is further decomposed into the equivalence classes of tame equivalence.
Picking one such equivalence class produces a pre-tame $\overline{\text{sc}}$-Fr{\'e}chet space.
\end{remark}

\begin{remark}\label{Remark_Subspaces_sums_pretame_context}
One also has the corresponding notions of (split) subspaces, direct sums, (strongly) smoothing morphisms and Fredholm operators for pre-tame $\overline{\text{sc}}$-Fr{\'e}chet spaces just as in \crefrange{Subsection_Subspaces_direct_sums}{Subsection_Sc_Fredholm_operators}.
This generalisation is completely straightforward and would repeat the definitions and results from \crefrange{Subsection_Subspaces_direct_sums}{Subsection_Sc_Fredholm_operators} almost word by word, just inserting the word ``(pre-)tame'' over and over again, so it will be skipped.
\end{remark}

\begin{definition}
Let $(X, \|\cdot\|_X)$ be a Banach space.
The \emph{chain of exponentially decreasing sequences in $X$} is the following ILB-chain $\mathbbl{\Sigma}(X) = ((\Sigma_k(X), \|\cdot\|_k), \iota_k)_{k\in\N_0}$:
\begin{align*}
\Sigma_k(X) &\definedas \left\{\mathbf{x} = (x_j)_{j\in\N_0}\in X^{\N_0} \;\left|\; \sum_{j=0}^\infty e^{kj}\|x_j\|_X < \infty\right.\right\} \\
\|\mathbf{x}\|_k &\definedas \sum_{j=0}^\infty e^{kj}\|x_j\|_X \\
\iota_k &\text{ is the canonical inclusion.}
\end{align*}
\end{definition}

\begin{remark}
$\mathbbl{\Sigma}(X)$ is not an sc-chain unless $X$ is finite dimensional.
To see this take a sequence $\mathbf{x}^i \subseteq \Sigma_\infty(X)$, where $x^i_j = 0$ for all $j > 0$ and $(x^i_0)_{\in\N_0}$ is any bounded sequence in $X$.
\end{remark}

\begin{definition}
Let $\mathbb{E}$ be an ILB- or sc-chain. \\
$\mathbb{E}$ is called \emph{tame} if there exist tame weak morphisms $J : E_\infty \to \Sigma_\infty(X)$ and $K : \Sigma_\infty(X) \to E_\infty$ between $\mathbb{E}$ and $\mathbbl{\Sigma}(X)$ for some Banach space $X$ \st $K\circ J = \id_{E_\infty}$.
\end{definition}

\begin{lemma}\label{Lemma_Tame_ILB_chain}
Let $\mathbb{E}$ and $\mathbb{E}'$ be ILB- or sc-chains and let $\mathbf{k}, \mathbf{l} \subseteq \N_0$ be strictly monotone increasing sequences.
\begin{enumerate}[label=\arabic*.,ref=\arabic*.]
  \item\label{Lemma_Tame_ILB_chain_1} Assume that there exist tame weak morphisms $J : E'_\infty \to E_\infty$ and $K : E_\infty \to E'_\infty$ \st $K\circ J = \id_{E'_\infty}$.
If $\mathbb{E}$ is tame, then so is $\mathbb{E}'$.
  \item\label{Lemma_Tame_ILB_chain_2} If $\mathbb{E}$ and $\mathbb{E}'$ are tamely equivalent, then $\mathbb{E}$ is tame \iff $\mathbb{E}'$ is tame.
In particular, if $\mathbb{E}$ is tame, then so is every tame rescaling of $\mathbb{E}$.
\end{enumerate}
\end{lemma}
\begin{proof}
All of these are immediate from the definitions.
\end{proof}

\begin{example}
All the ILB-chains or sc-chains and morphisms and equivalences between them seen so far in the previous sections are tame, \cf \cite{MR656198}, esp.~Chapter II, Corollary 1.3.9.
\end{example}

\begin{definition}\label{Definition_Smoothing_operators}
Let $\mathbb{E}$ be an ILB- or sc-chain.
\begin{enumerate}[label=\arabic*.,ref=\arabic*.]
  \item $\mathbb{E}$ is said to \emph{admit smoothing operators} if there exists a continuous family of strongly smoothing operators $(\mathbb{S}_t : \mathbb{E} \to \mathbb{E})_{t\in [0,\infty)}$, \ie there exists a continuous map
\begin{align*}
\overline{S} : [0,\infty) \times E_0 &\to E_\infty \\
(t, e) &\mapsto \overline{S}_te\text{,}
\end{align*}
with each $\overline{S}_t : E_0 \to E_\infty$ linear, \st the following hold: \\
For $k,l \in \N_0$, denote $S^{k}_{l,t} \definedas \iota^{\infty}_l\circ \overline{S}_t \circ \iota^k_0 : E_k \to E_l$.
Then $\mathbb{S}_t = (S^{k}_{k,t})_{k\in\N_0} : \mathbb{E} \to \mathbb{E}$. \\
Furthermore there exist constants $p\in\N_0$, $C^k_{l} \in [0,\infty)$ for $k,l\in\N_0$ \st
\begin{align*}
\|S^{k}_{l,t}\|_{L_{\mathrm{c}}(E_k,E_l)} &\leq C^k_{l}\bigl(1 + e^{(p + (l-k))t}\bigr) & &\forall\, k,l\in\N_0, t \geq 0 \\
\|\iota^k_l - S^{k}_{l,t}\|_{L_{\mathrm{c}}(E_k,E_l)} &\leq C^k_{l} e^{(p - (k-l))t} & &\forall\, k,l\in\N_0, k - l \geq p, t \geq 0\text{.}
\end{align*}
  \item $\mathbb{E}$ is called \emph{weakly tame} if there exists a tame rescaling $\mathbb{E}^{\mathbf{k}}$ of $\mathbb{E}$ that admits smoothing operators.
\end{enumerate}
\end{definition}

\begin{remark}\label{Remark_Alternative_definition_smoothing_operators}
Note that for $k,l\in\N_0$ with $k-l \geq p$ the first inequality follows from the second (modulo replacing $C^k_l$ by $C^k_l + 1$), because $\|S^{k}_{l,t}\|_{L_{\mathrm{c}}(E_k,E_l)} = \|\iota^k_l - S^{k}_{l,t} - \iota^k_l\|_{L_{\mathrm{c}}(E_k,E_l)} \leq \|\iota^k_l\|_{L_{\mathrm{c}}(E_k,E_l)} + \|\iota^k_l - S^{k}_{l,t}\|_{L_{\mathrm{c}}(E_k,E_l)} \leq 1 + C^k_le^{(p - (k-l))t} \leq 1 + C^k_l$, since $p - (k-l) \leq 0$.
And $C^k_l\bigl(1 + e^{p+(l-k))t}\bigr) \geq C^k_l$. \\
And for $k-l < p$, \ie $p + (l-k) > 0$, one has $e^{(p + (l-k))t} \leq 1 + e^{(p + (l-k))t} \leq 2e^{(p+l-k))t}$.
Hence one could also replace the first inequality by
\begin{align*}
\|S^{k}_{l,t}\|_{L_{\mathrm{c}}(E_k,E_l)} &\leq C^k_{l}e^{(p + (l-k))t} & &\forall\, k,l\in\N_0, k-l < p, t \geq 0
\end{align*}
\end{remark}

\begin{remark}\label{Remark_Alternative_definition_smoothing_operators_II}
Defining $\mathbf{p} \definedas (p + j)_{j\in\N_0}$, $\tilde{S}_t \definedas \overline{S}_t\circ \iota^p_0 : E^{\mathbf{p}}_0 = E_p \to E_\infty$, one can consider the $\tilde{S} : [0,\infty)\times E^{\mathbf{p}}_0\to E_\infty$ as defining a continuous family of strongly smoothing operators $(\tilde{\mathbb{S}}_t : \mathbb{E}^{\mathbf{p}} \to \mathbb{E})_{t\in[0,\infty)}$.
Then the continuous linear operators $\tilde{S}^k_{l,t} \definedas \iota^\infty_l\circ \tilde{S}_t\circ (\iota^{\mathbf{p}})^k_0 : E^{\mathbf{p}}_k \to E_l$ satisfy $\tilde{S}^k_{l,t} = S^{p+k}_{l,t}$ and satisfy the inequalities
\begin{align*}
\|\tilde{S}^k_{l,t}\|_{L_{\mathrm{c}}(E^{\mathbf{p}}_k,E_l)} &\leq \tilde{C}^k_l\bigl(1 + e^{(l-k)t}\bigr) & &\forall\, k,l\in\N_0, t\geq 0 \\
\|\iota^{p+k}_l - \tilde{S}^{k}_{l,t}\|_{L_{\mathrm{c}}(E^{\mathbf{p}}_k,E_l)} &\leq \tilde{C}^k_{l} e^{-(k-l)t} & &\forall\, k,l\in\N_0, k - l \geq 0, t \geq 0\text{,}
\end{align*}
where $\tilde{C}^k_l \definedas C^{p+k}_l$.
\end{remark}

\begin{remark}\label{Remark_Strongly_cts_family_of_smoothing_operators}
Note that by \cref{Lemma_Weakly_continuous_implies_strongly_continuous} the map
\begin{align*}
[0,\infty) &\to L_{\mathrm{c}}(E_1, E_\infty) \\
t &\mapsto \overline{S}_t\circ \iota^1_0
\end{align*}
is continuous.
So if $\mathbb{E}$ is a weakly tame sc-chain, then there exists a tame rescaling $\mathbb{E}^{\mathbf{k}}$ of $\mathbb{E}$ that admits smoothing operators that a fortiori have the property that the map
\begin{align*}
S : [0,\infty) &\to L_{\mathrm{c}}(E_0, E_\infty) \\
t &\mapsto \overline{S}_t
\end{align*}
is continuous.
\end{remark}

\begin{lemma}\label{Lemma_Weakly_tame_ILB_chain}
Let $\mathbb{E}$ and $\mathbb{E}'$ be ILB- or sc-chains and let $\mathbf{k} \subseteq \N_0$ be a strictly monotone increasing sequence.
\begin{enumerate}[label=\arabic*.,ref=\arabic*.]
  \item\label{Lemma_Weakly_tame_ILB_chain_1} If $\mathbf{k}$ is tame and $\mathbb{E}$ admits smoothing operators, then $\mathbb{E}^{\mathbf{k}}$ admits smoothing operators.
  \item\label{Lemma_Weakly_tame_ILB_chain_2} Assume that there exist tame weak morphisms $J : E'_\infty \to E_\infty$ and $K : E_\infty \to E'_\infty$ \st $K\circ J = \id_{E'_\infty}$.
If $\mathbb{E}$ is weakly tame, then so is $\mathbb{E}'$.
  \item\label{Lemma_Weakly_tame_ILB_chain_3} If $\mathbb{E}$ and $\mathbb{E}'$ are tamely equivalent, then $\mathbb{E}$ is weakly tame \iff $\mathbb{E}'$ is weakly tame.
In particular, if $\mathbb{E}$ is weakly tame, then so is every tame rescaling of $\mathbb{E}$.
\end{enumerate}
\end{lemma}
\begin{proof}
\begin{enumerate}[label=\arabic*.,ref=\arabic*.]
  \item In the notation of \cref{Definition_Smoothing_operators}, define $\overline{S}' : [0,\infty)\times E^{\mathbf{k}}_0 \to E^{\mathbf{k}}_\infty$ by $\overline{S}'_t \definedas \overline{S}_t\circ \iota^{k_0}_0 : E^{\mathbf{k}}_0 = E_{k_0} \to E_\infty = E^{\mathbf{k}}_\infty$.
Then
\begin{align*}
S'^m_{n,t} &= (\iota^{\mathbf{k}})^\infty_n\circ \overline{S}'_t\circ (\iota^{\mathbf{k}})^m_0 \\
&= \iota^\infty_{k_n}\circ \overline{S}_t\circ \iota^{k_0}_0\circ \iota^{k_m}_{k_0} \\
&= \iota^\infty_{k_n}\circ \overline{S}_t\circ \iota^{k_m}_0 \\
&= S^{k_m}_{k_n,t}\text{.}
\end{align*}
Using that by \cref{Lemma_Tame_sequence} there exists $c \in\N_0$ \st $j-i \leq k_j - k_i \leq j-i + c$ or $i-j - c \leq k_i - k_j \leq i-j$, for $j \geq i$, the required estimates for $S'^m_{n,t}$ follow from those for the $S^m_{n,t}$ with a constant $p' \definedas p + c$.
  \item As in \cref{Remark_Extensions_of_weak_morphisms}, one can find tame strictly monotone increasing sequences $\mathbf{k}, \mathbf{l} \subseteq \N_0$ and extensions $\mathbb{K} : \mathbb{E}^{\mathbf{l}} \to \mathbb{E}'$ and $\mathbb{J} : \mathbb{E}'^{\mathbf{k}} \to \mathbb{E}^{\mathbf{l}}$ with $\mathbb{K}\circ \mathbb{J} = \mathbb{I}'^{\mathbf{k}} : \mathbb{E}'^{\mathbf{k}} \to \mathbb{E}'$.
Furthermore, using \labelcref{Lemma_Weakly_tame_ILB_chain_1} and \cref{Lemma_Tame_sequence}, \labelcref{Lemma_Tame_sequence_5}, one can assume that $\mathbf{k}$ and $\mathbf{l}$ are shifts, \ie $\mathbf{k} = (k_0 + j)_{j\in\N_0}$ and $\mathbf{l} = (l_0 + j)_{j\in \N_0}$ for some $k_0,l_0\in \N_0$. \\
Then $\mathbb{E}'^{\mathbf{k}}$ admits smoothing operators: \\
Let $\overline{S} : [0,\infty)\times E_0 \to E_\infty$ and $p \in \N_0$, $C^m_n \in [0,\infty)$ be as in \cref{Definition_Smoothing_operators}.
Define
\begin{align*}
p' &\definedas p + k_0\text{,} \\
C'^m_n &\definedas C_{k_0+l_0+n}^{l_0+m}\|K_{k_0+n}\|_{L_{\mathrm{c}}(E^{\mathbf{l}}_{k_0+n}, E'_{k_0+n})}\|J_m\|_{L_{\mathrm{c}}(E'^{\mathbf{k}}_{m}, E^{\mathbf{l}}_{m})}\text{,}
\intertext{and define $\overline{S}' : [0,\infty) \times E'^{\mathbf{k}}_0 \to E'^{\mathbf{k}}_\infty$ by}
\overline{S}'_t &\definedas K_\infty\circ \overline{S}_t\circ \iota^{l_0}_0\circ J_0 : E'^{\mathbf{k}}_0 = E'_{k_0} \to E'_\infty = E'^{\mathbf{k}}_\infty\text{.}
\end{align*}
Now one calculates for $n,m\in\N_0$
\begin{align*}
S'^{m}_{n,t} &= (\iota'^{\mathbf{k}})^\infty_n\circ S'_t\circ (\iota'^{\mathbf{k}})^m_0 \\
&= \iota'^\infty_{k_0+n}\circ K_\infty\circ \overline{S}_t\circ \iota^{l_0}_0\circ J_0\circ (\iota'^{\mathbf{k}})^m_0 \\
&= K_{k_0+n}\circ (\iota^{\mathbf{l}})^\infty_{k_0+n}\circ \overline{S}_t\circ \iota^{l_0}_0\circ (\iota^{\mathbf{l}})^m_0\circ J_m \\
&= K_{k_0+n}\circ \iota^\infty_{k_0+l_0+n}\circ \overline{S}_t\circ \iota^{l_0}_0\circ \iota^{l_0+m}_{l_0}\circ J_m \\
&= K_{k_0+n}\circ \iota^\infty_{k_0 + l_0 + n}\circ \overline{S}_t\circ \iota^{l_0+m}_0\circ J_m \\
&= K_{k_0+n}\circ S_{k_0+l_0+n,t}^{l_0+m}\circ J_m\text{.}
\end{align*}
And if $l_0 + m \geq k_0 + l_0 + n + p$, \ie $m \geq n + k_0 + p = n + p'$, then one can furthermore calculate
\begin{align*}
K_{k_0+n}\circ \iota^{l_0+m}_{k_0+l_0+n}\circ J_m &= K_{k_0+n}\circ (\iota^{\mathbf{l}})^m_{k_0+n}\circ J_m \\
&= K_{k_0+n}\circ J_{k_0+n}\circ (\iota'^{\mathbf{k}})^m_{k_0+n} \\
&= (\mathbb{I}'^{\mathbf{k}})_{k_0+n}\circ (\iota'^{\mathbf{k}})^m_{k_0+n} \\
&= \iota'^{2k_0+n}_{k_0+n}\circ \iota'^{k_0+m}_{2k_0+n} \\
&= \iota'^{k_0+m}_{k_0+n} \\
&= (\iota'^{\mathbf{k}})^m_n\text{,}
\end{align*}
hence
\[
(\iota'^{\mathbf{k}})^m_n - S'^{m}_{n,t} = K_{k_0+n}\circ (\iota^{l_0+m}_{k_0+l_0+n} - S_{k_0+l_0+n,t}^{l_0+m})\circ J_m\text{.}
\]
Applying the estimates for the $S^{k}_{l,t}$ and $\iota^k_l - S^{k}_{l,t}$ to the above formulas for $S'^{m}_{n,t}$ and $(\iota'^{\mathbf{k}})^m_n - S'^{m}_{n,t}$ immediately yields the result.
  \item This is immediate from \labelcref{Lemma_Weakly_tame_ILB_chain_2}.
\end{enumerate}
\end{proof}

\begin{proposition}
Let $\mathbb{E}$ be an ILB- or sc-chain.
If $\mathbb{E}$ is tame, then $\mathbb{E}$ is weakly tame.
\end{proposition}
\begin{proof}
By~\cite{MR656198}, Part II, Section 1.4, $\mathbbl{\Sigma}(X)$, for $(X, \|\cdot\|_X)$ a Banach space, is weakly tame.
The result now follows immediately from the definition of tame and \cref{Lemma_Weakly_tame_ILB_chain}, \labelcref{Lemma_Weakly_tame_ILB_chain_2}.
\end{proof}

\begin{definition}\label{Definition_Weakly_tame_sc_Frechet_space}
Let $E$ be a pre-tame $\overline{\text{sc}}$-Fr{\'e}chet space.
\begin{enumerate}[label=\arabic*.,ref=\arabic*.]
  \item A tamely compatible $\overline{\text{sc}}$-structure $(\mathbb{E}, \phi)$ on $E$ is called \emph{tame} if $\mathbb{E}$ is tame.
  \item $E$ is called \emph{tame} if one/any tamely compatible $\overline{\text{sc}}$-structure on $E$ is tame.
  \item A \emph{tame $\overline{\text{sc}}$-Fr{\'e}chet space} is a tame pre-tame $\overline{\text{sc}}$-Fr{\'e}chet space.
  \item A tamely compatible $\overline{\text{sc}}$-structure $(\mathbb{E}, \phi)$ on $E$ is called \emph{weakly tame} if $\mathbb{E}$ is weakly tame.
  \item $E$ is called \emph{weakly tame} if one/any tamely compatible $\overline{\text{sc}}$-structure on $E$ is weakly tame.
  \item A \emph{weakly tame $\overline{\text{sc}}$-Fr{\'e}chet space} is a weakly tame pre-tame $\overline{\text{sc}}$-Fr{\'e}chet space.
\end{enumerate}
\end{definition}

\begin{example}
All the $\overline{\text{sc}}$-Fr{\'e}chet spaces that appeared so far define tame $\overline{\text{sc}}$-Fr{\'e}chet spaces and all the morphisms between them that appeared so far are tame.
\end{example}

\begin{proposition}\label{Proposition_Tame_subspaces_and_direct_sums}
Let $E$, $E'$ and $E^1, \dots, E^k$, for some $k\in\N_0$, be pre-tame $\overline{\text{sc}}$-Fr{\'e}chet spaces.
\begin{enumerate}[label=\arabic*.,ref=\arabic*.]
  \item If $E$ is (weakly) tame, $E' \subseteq E$ is a pre-tame $\overline{\text{sc}}$-subspace, and if there exists a tame morphism $P : E \to E'$ \st $P\circ J = \id_{E'}$, where $J : E' \hookrightarrow E$ is the canonical inclusion, then $E'$ is (weakly) tame.
  \item If the $E^i$, $i=1,\dots, k$, are (weakly) tame, then so is their direct sum $E^1 \oplus \cdots \oplus E^k$ (as a pre-tame $\overline{\text{sc}}$-Fr{\'e}chet space) and the canonical projections $P^i : E^1 \oplus \cdots \oplus E^k \to E^i$ and canonical inclusions $J^i : E^i \to E^1 \oplus \cdots \oplus E^k$ are tame morphisms.
\end{enumerate}
\end{proposition}
\begin{proof}
\begin{enumerate}[label=\arabic*.,ref=\arabic*.]
  \item This follows immediately from \cref{Lemma_Tame_ILB_chain,Lemma_Weakly_tame_ILB_chain}.
  \item The case that the $E^i$, $i=1,\dots, k$, are tame follows from \cite{MR656198}, Part II, Lemma 1.3.4. \\
For the case that the $E^i$, $i = 1, \dots, k$, are weakly tame it suffices to show that if $\mathbb{E}^i$, $i=1,\dots,k$, are ILB-chains which admit smoothing operators, then their direct sum $\mathbb{E} \definedas \mathbb{E}^1\oplus \cdots \oplus \mathbb{E}^k$ admits smoothing operators as well.
So assume this to be the case and let $(\mathbb{S}^i_t : \mathbb{E}^i \to \mathbb{E}^i)_{t\in[0,\infty)}$ be the corresponding continuous families of strongly smoothing operators, given by continuous maps
\begin{align*}
\overline{S}^i : [0,\infty) \times E^i_0 &\to E^i_\infty \\
(t, e) &\mapsto \overline{S}^i_te\text{,}
\end{align*}
as in \cref{Definition_Smoothing_operators}.
Define $(\mathbb{S}_t : \mathbb{E} \to \mathbb{E})_{t\in[0,\infty)}$ via
\begin{align*}
\overline{S} : [0,\infty) \times E_0 &\to E_\infty \\
(t, (e^1, \dots, e^k)) &\mapsto (\overline{S}^1_te^1, \dots, \overline{S}^k_te^k)\text{.}
\end{align*}
\end{enumerate}
Then for $m,n\in\N_0$ and $(e^1, \dots, e^k) \in E_m$,
\[
S^m_{n,t}(e^1, \dots, e^k) = ((S^1)^m_{n,t}e^1, \dots, (S^k)^m_{n,t}e^k)\text{.}
\]
Hence, if $p^i, C^{i,m}_n \in [0,\infty)$ are the constants from \cref{Definition_Smoothing_operators} for the $\mathbb{E}^i$, then
\begin{align*}
\left\|S^m_{n,t}(e^1, \dots, e^k)\right\|_n &= \left\|(S^1)^m_{n,t}e^1\right\|^1_n + \cdots + \left\|(S^k)^m_{n,t}e^k\right\|^k_n \\
&\leq C^{1,m}_n\bigl(1 + e^{(p^1 + (n-m))t}\bigr)\|e^1\|^1_m + \cdots + C^{k,m}_n\bigl(1 + e^{(p^k + (n-m))t}\bigr)\|e^k\|^k_m \\
&\leq C^{m}_n\bigl(1 + e^{(p + (n-m))t}\bigr)(\|e^1\|^1_m + \cdots + \|e^k\|^k_m) \\
&= C^m_n\bigl(1 + e^{(p + (n-m))t}\bigr)\|(e^1, \dots, e^k)\|_m\text{,}
\end{align*}
where $C^m_n \definedas \max\{C^{i,m}_n \;|\; i=1,\dots, k\}$ and $p \definedas \max\{p^i \;|\; i=1,\dots, k\}$.
Hence $\|S^m_{n,t}\|_{L_{\mathrm{c}}(E_m,E_n)} \leq C^m_n\bigl(1 + e^{(p + (n-m))t}\bigr)$. \\
And analogously for the estimate for $\iota^m_n - S^m_{n,t}$.
\end{proof}

\Needspace{25\baselineskip}
\subsection{Tame nonlinear maps}

First, note that all the notions and results on envelopes from \cref{Subsection_Envelopes} carry over ad verbatim to pre-tame $\overline{\text{sc}}$-Fr{\'e}chet spaces, just adding the word ``tame'' in the appropriate places.
E.\,g.~if $E$ and $E'$ are pre-tame $\overline{\text{sc}}$-Fr{\'e}chet spaces, $A \subseteq E$ is a subset and $f : A \to E'$ is a map, then for an envelope $((\mathbb{E}, \phi), \mathcal{U})$ of $A$ it is assumed that $(\mathbb{E}, \phi)$ is a tamely compatible $\overline{\text{sc}}$-structure.
And likewise for an envelope $((\mathbb{E}, \phi), (\mathbb{E}', \phi'), \mathcal{F} : \mathcal{U} \to \mathbb{E}')$ of $f$ it is assumed that $(\mathbb{E}, \phi)$ and $(\mathbb{E}', \phi')$ are tamely compatible $\overline{\text{sc}}$-structures on $E$ and $E'$, respectively, and all the rescalings appearing in the definition of equivalence of envelopes are assumed to be tame.
To not unnecessarily bloat this text this will not be carried out explicitely but I will rely instead on this being sufficiently obvious.

\begin{definition}\label{Definition_Tame_estimate}
Let $(X, \|\cdot\|_X)$ and $(Y, \|\cdot\|_Y)$ be normed vector spaces, let $A \subseteq X$ be a subset, and let $f : A \to Y$ be a map. \\
$f$ is said to \emph{satisfy a tame estimate} \iff there exists a constant $C \in [0,\infty)$ \st
\[
\|f(x)\|_Y \leq C(1 + \|x\|_X) \quad \forall\, x \in A\text{.}
\]
\end{definition}

\begin{definition}\label{Definition_Tame_nonlinear_map}
Let $\mathbb{E}$ and $\mathbb{E}'$ be ILB- or sc-chains, let $A \subseteq E_\infty$ be a subset, and let $f : A \to E'_\infty$ be a map.
\begin{enumerate}[label=\arabic*.,ref=\arabic*.]
  \item $f$ is called \emph{tame (\wrt $\mathbb{E}$ and $\mathbb{E}'$)} \iff for every $k \in \N_0$,
\[
f : (E_\infty, (\iota^\infty_k)^\ast\|\cdot\|_k) \supseteq U \to (E'_\infty, (\iota'^\infty_k)^\ast \|\cdot\|_k)
\]
satisfies a tame estimate.
  \item An envelope $\mathcal{F} : \mathcal{U} \to \mathbb{E}'$ of $f$ is called \emph{tame} \iff for every $k \in \N_0$,
\[
f_k : (E_k, \|\cdot\|_k)\supseteq U_k \to (E'_k, \|\cdot\|'_k)
\]
satisfies a tame estimate.
\end{enumerate}
\end{definition}

\begin{lemma}\label{Lemma_Tame_maps}
Let $\mathbb{E}$, $\tilde{\mathbb{E}}$, $\mathbb{E}'$ and $\tilde{\mathbb{E}}'$ be ILB- or sc-chains, let $A \subseteq E_\infty$ be a subset and let $f : A \to E'_\infty$ be a map, and let $\mathbb{T} : \tilde{E} \to \mathbb{E}$ and $\mathbb{S} : \mathbb{E}' \to \tilde{\mathbb{E}}'$ be continuous linear operators.
\begin{enumerate}[label=\arabic*.,ref=\arabic*.]
  \item If $f$ is tame as a map from $A \subseteq E_\infty$ to $E'_\infty$, then $f$ is tame as a map from $A\subseteq E^{\mathbf{k}}_\infty = E_\infty$ to $E'^{\mathbf{l}}_\infty = E'_\infty$ for all strictly monotone increasing sequences $\mathbf{k}, \mathbf{l} \subseteq \N_0$ with $\mathbf{k} \geq \mathbf{l}$.
  \item If $f$ is tame, then $f\circ T_\infty : \tilde{E}_\infty \supseteq T_\infty\inv(A) \to E'_\infty$ and $S_\infty\circ f : A \to \tilde{E}'_\infty$ are tame as well.
  \item If $\mathcal{F} : \mathcal{U} \to \mathbb{E}'$ is a tame envelope of $f$, then so is any refinement of $\mathcal{F}$ as well as $\mathbb{I}^{\mathbf{k}}_{\mathbf{l}}\circ \mathcal{F}^{\mathbf{k}} : \mathcal{U}^{\mathbf{k}} \to \mathbb{E}'^{\mathbf{l}}$ for all strictly monotone increasing sequences $\mathbf{k}, \mathbf{l} \subseteq \N_0$ with $\mathbf{k} \geq \mathbf{l}$.
  \item If $\mathcal{F} : \mathcal{U} \to \mathbb{E}'$ is a tame envelope of $f$, then $\mathbb{T}^\ast \mathcal{F}$ and $\mathbb{S}\circ \mathcal{F}$ are tame envelopes of $f\circ T_\infty$ and $S_\infty\circ f$, respectively.
\end{enumerate}
\end{lemma}
\begin{proof}
Straightforward.
\end{proof}

\begin{definition}\label{Definition_Tame_nonlinear_map_scFrechet_spaces}
Let $E$ and $E'$ be pre-tame $\overline{\text{sc}}$-Fr{\'e}chet spaces, let $U \subseteq E$ be an open subset and let $f : U \to E'$ be a map.
\begin{enumerate}[label=\arabic*.,ref=\arabic*.]
  \item If $f$ is $\underline{\text{sc}}^0$, then $f$ is called \emph{tame} \iff for every point $x \in U$ there exists a neighbourhood $V \subseteq U$ of $x$ and tamely compatible $\overline{\text{sc}}$-structures $(\mathbb{E},\phi)$ and $(\mathbb{E}',\phi')$ on $E$ and $E'$, respectively, \st $\phi'^{-1}\circ f\circ \phi : \phi\inv(V) \to E'_\infty$ is tame.
  \item An envelope $((\mathbb{E},\phi),(\mathbb{E}',\phi'), \mathcal{F})$ of $f$ is called \emph{tame} \iff $(\mathbb{E}, \phi)$ and $(\mathbb{E}', \phi')$ are tamely compatible $\overline{\text{sc}}$-structures on $E$ and $E'$, respectively, and $\mathcal{F}$ is tame.
  \item If $f$ is $\overline{\text{sc}}^0$, then $f$ is called \emph{tame} \iff for every point $x \in U$ there exists a neighbourhood $V \subseteq U$ of $x$ \st $f|_V : V \to E'$ has a tame envelope.
\end{enumerate}
\end{definition}

\begin{proposition}\label{Proposition_Tame_maps}
Let $E$ and $E'$ be pre-tame $\overline{\text{sc}}$-Fr{\'e}chet spaces, let $U \subseteq E$ be an open subset and let $f : U \to E'$ be a map.
\begin{enumerate}[label=\arabic*.,ref=\arabic*.]
  \item\label{Proposition_Tame_maps_1} If $f$ is $\underline{\text{sc}}^0$, then $f$ is tame \iff for every point $x \in U$ there exists a neighbourhood $V \subseteq U$ of $x$ \st for every pair of tamely compatible $\overline{\text{sc}}$-structures $(\mathbb{E}, \phi)$ and $(\mathbb{E}', \phi')$ on $E$ and $E'$, respectively, there exist tame strictly monotone increasing sequences $\mathbf{k}, \mathbf{l} \subseteq \N_0$ \st $\phi'^{-1}\circ f\circ \phi : E^{\mathbf{k}}_\infty \supseteq \phi\inv(V) \to E'^{\mathbf{l}}_\infty$ is tame.
  \item\label{Proposition_Tame_maps_2} If $f$ is $\overline{\text{sc}}^0$, then $f$ is tame \iff for every point $x \in U$ there exists a neighbourhood $V \subseteq U$ of $x$ \st for every envelope $((\mathbb{E},\phi),(\mathbb{E}',\phi'), \mathcal{F})$ of $f|_V$ there exists a tame strictly monotone increasing sequence $\mathbf{k} \subseteq \N_0$ \st $((\mathbb{E}^{\mathbf{k}},\phi),(\mathbb{E}',\phi'), \mathbb{I}'^{\mathbf{k}}\circ \mathcal{F}^{\mathbf{k}})$ is tame.
  \item\label{Proposition_Tame_maps_3} If $f$ is $\overline{\text{sc}}^0$, then $f$ is tame as an $\overline{\text{sc}}^0$-map \iff $f$ is tame as an $\underline{\text{sc}}^0$-map.
\end{enumerate}
\end{proposition}
\begin{proof}
\begin{enumerate}[label=\arabic*.,ref=\arabic*.]
  \item This is a straightforward verification.
  \item By \cref{Proposition_Strict_envelopes}, \labelcref{Proposition_Strict_envelopes_2}, and \cref{Lemma_Tame_maps}, for every $x \in U$ there exists a neighbourhood $V \subseteq U$ of $x$ and a tame strict envelope of $f|_V$.
And by \cref{Proposition_Strict_envelopes}, \labelcref{Proposition_Strict_envelopes_4}, any other envelope of $f|_V$ is equivalent to this tame strict envelope.
Now one just needs to unravel the definition of equivalence of envelopes and apply \cref{Lemma_Tame_maps}.
  \item One direction is trivial and if $f$ is tame as an $\underline{\text{sc}}^0$-map, then by the same arguments as above, for every point $x \in U$ there exists a neighbourhood $V \subseteq U$ of $x$ \st $f|_V$ has a strict envelope $((\mathbb{E},\phi),(\mathbb{E}',\phi'), \mathcal{F} : \mathcal{U} \to \mathbb{E}')$ \st $f_\infty \definedas \phi'^{-1}\circ f \circ \phi : E_\infty \supseteq \phi\inv(V) \to E'_\infty$ is tame.
But because $\mathcal{F} : \mathcal{U} \to \mathbb{E}'$ is strict, $U_\infty = \phi\inv(V)$ is dense in $U_k$ for all $k\in\N_0$ and $f_k : U_k \to E'_k$ is the unique continuous extension of $f_\infty$.
So the result follows from a straightforward continuity argument.
\end{enumerate}
\end{proof}

\begin{proposition}\label{Proposition_Standard_properties_tame}
Let $E$, $E'$ and $E''$ be pre-tame $\overline{\text{sc}}$-Fr{\'e}chet spaces.
\begin{enumerate}[label=\arabic*.,ref=\arabic*.]
  \item Let $T : E \to E'$ be a continuous linear map.
Then $T$ defines a tame morphism \iff $T$ defines a tame $\overline{\text{sc}}$-continuous map.
  \item Let $U \subseteq E$ and $U' \subseteq E'$ be open subsets, and let $f : U \to E'$ and $f' : U' \to E''$ be $\overline{\text{sc}}^0$ or $\underline{\text{sc}}^0$ with $f(U) \subseteq U'$. \\
If $f$ and $f'$ are tame, then so is $f' \circ f$.
\end{enumerate}
\end{proposition}
\begin{proof}
\begin{enumerate}[label=\arabic*.,ref=\arabic*.]
  \item \cite{MR656198}, Chapter II, Theorem 2.1.5.
  \item \cite{MR656198}, Chapter II, Theorem 2.1.6.
\end{enumerate}
\end{proof}

\begin{definition}
Let $E$ and $E'$ be pre-tame $\overline{\text{sc}}$-Fr{\'e}chet spaces, let $U \subseteq E$ be an open subset, let $f : U \to E'$ be a map and let $k \in \N_0$.
\begin{enumerate}[label=\arabic*.,ref=\arabic*.]
  \item If $f$ is $\overline{\text{sc}}^k$ or p-$\overline{\text{sc}}^k$, then $f$ is called \emph{tame up to order $k$} \iff for every $0 \leq j \leq k$, $D^j f : U\times E^j \to E'$ is tame as an $\overline{\text{sc}}$-continuous map.
  \item If $f$ is $\underline{\text{sc}}^k$ or p-$\underline{\text{sc}}^k$, then $f$ is called \emph{tame up to order $k$} \iff for every $0 \leq j \leq k$, $D^j f : U\times E^j \to E'$ is tame as an $\underline{\text{sc}}$-continuous map.
  \item A tame envelope $((\mathbb{E},\phi), (\mathbb{E}',\phi'), \mathcal{F} : \mathcal{U} \to \mathbb{E}')$ of $f$ that is $\overline{\text{sc}}^k$ is called \emph{tame up to order $k$} \iff $\mathcal{DF} : \mathcal{U}^{\mathbf{1}}\oplus \mathbb{E} \to \mathbb{E}'$ is tame up to order $k-1$.
  \item If $f$ is $\widehat{\text{sc}}^k$, then $f$ is called \emph{tame up to order $k$} \iff $f$ has an envelope that is $\overline{\text{sc}}^k$ and tame up to order $k$.
\end{enumerate}
\end{definition}

\begin{remark}
Note that the condition of tameness for the differentials $D^kf : U\times E^k$, if $f$ is $\underline{\text{sc}}^k$, coincides with the condition of tameness for functions of more than one variable in \cite{MR656198}, p.~142\,f.
\end{remark}

\begin{example}\label{Example_Reparametrisation_Action_III}
I continue \cref{Subsection_Reparametrisation_action} and \cref{Example_Reparametrisation_Action_I,Example_Reparametrisation_Action_II}. \\
In the notation from this Subsection, $\Psi : \Gamma_B(F_1) \to \Gamma_B(F_2)$ is tame up to arbitrary order:
By \cref{Example_Reparametrisation_Action_II}, $\Psi$ is $\widehat{\text{sc}}^\infty$.
That $\Psi$ is tame follows at once from the claim in the proof of \cref{Proposition_Reparametrisation_action_cts}.
That $\Psi$ is tame to arbitrary order follows by a simple induction, repeatedly applying \cref{Proposition_Reparametrisation_action_diffble} and the claim in the proof of \cref{Proposition_Reparametrisation_action_cts} to the terms on the right hand side of the equations in \cref{Example_Reparametrisation_Action_II} for $\Psi$ and $D\Psi$.
\end{example}

\begin{proposition}\label{Proposition_Underline_sc_implies_overline_sc}
Let $E$ and $E'$ be pre-tame $\overline{\text{sc}}$-Fr{\'e}chet spaces, let $U \subseteq E$ be an open subset, let $f : U \to E'$ be a map and let $k \in \N_0$ with $k \geq 1$.
\begin{enumerate}[label=\arabic*.,ref=\arabic*.]
  \item\label{Proposition_Underline_sc_implies_overline_sc_1} If $f$ is $\underline{\text{sc}}^k$ and tame up to order $k$, then $f$ is $\overline{\text{sc}}^{k-1}$ and tame up to order $k-1$.
  \item\label{Proposition_Underline_sc_implies_overline_sc_2} If $f$ is p-$\underline{\text{sc}}^k$ and tame up to order $k$, then $f$ is p-$\overline{\text{sc}}^{k-1}$ and tame up to order $k-1$.
\end{enumerate}
\end{proposition}
\begin{proof}
By induction and \cref{Lemma_sc_diffble_implies_weakly_sc_diffble}, it suffices to show the implication $f$ p-$\underline{\text{sc}}^1$ and tame up to order $1$ $\Rightarrow$ $f$ $\overline{\text{sc}}^0$ and tame. \\
So let $f$ be p-$\underline{\text{sc}}^1$ and tame up to order $1$ and let $x_0 \in U$.
In the following, as usual, I will drop the inclusions in the sc-chains from the notation. \\
Now an easy argument with rescaling and the definition of tameness shows that one can choose a neighbourhood $V \subseteq U$ of $x_0$ and compatible tame $\overline{\text{sc}}$-structures $(\mathbb{E}, \phi)$ and $(\mathbb{E}', \phi')$ on $E$ and $E'$, respectively, together with a strict envelope $\mathcal{V} \subseteq \mathbb{E}$ of $V_\infty \definedas \phi\inv(V)$ \st $g \definedas \phi'\circ f\circ \phi\inv|_{V_\infty} : V_\infty \to E'_\infty$ is tame and so is $Dg : V_\infty\times E_\infty \to E'_\infty$.
Because $\mathcal{V}$ is strict, one can assume w.\,l.\,o.\,g., after shrinking $V$, that $V_k = \{x \in E_k \;|\; \|x\|_0 \leq c\}$ for some $c > 0$.
In particular, $V_\infty$ is dense in $V_k$ for all $k \in \N_0$. \\
By definition there then exist constants $C_k, D_k > 0$, for $k \in \N_0$, \st
\begin{align*}
\|g(x)\|'_k &\leq C_k(1 + \|x\|_k) \\
\|Dg(x)u\|'_k &\leq D_k(1 + \|x\|_k + \|u\|_k)
\end{align*}
for all $x \in V_\infty$, $u \in E_\infty$ and $k \in \N_0$. \\
Now by \cite{MR656198}, Part II, Lemma 2.1.7, there exist constants $c_k > 0$, for $k \in \N_0$, \st
\begin{align*}
\|Dg(x)u\|'_k &\leq c_k(\|x\|_k\|u\|_0 + \|u\|_k) \\
&\leq c_k(1 + \|x\|_k)\|u\|_k
\end{align*}
for all $x \in V_\infty$, $u \in E_\infty$ and $k \in \N_0$. \\
For $x,y\in V_\infty$, because
\begin{align*}
g(y) - g(x) &= \int_0^1 Dg(x + t(y-x))(y-x)\,\d t\text{,}
\intertext{it hence follows that}
\|g(y) - g(x)\|_k &\leq \int_0^1 \|Dg(x+t(y-x))(y-x)\|'_k\, \d t \\
&\leq \int_0^1 c_k(1 + \|x + t(y-x)\|_k)\|y-x\|_k\, \d t \\
&\leq c_k(1 + \|x\|_k + \|y\|_k)\|y-x\|_k\text{.}
\end{align*}
So $g : (V_\infty, \|\cdot\|_k) \to (E'_\infty, \|\cdot\|'_k)$ is locally uniformly continuous and hence there exists a unique continuous continuation $g_k : V_k \to E'_k$ of $g$. \\
The $g_k$ form the desired envelope $\mathcal{G} : \mathcal{V} \to \mathbb{E}'$ of $g$ and hence $((\mathbb{E}, \phi), (\mathbb{E}', \phi'), \mathcal{G} : \mathcal{V} \to \mathbb{E}')$ forms an envelope of $f|_V$.
Now because $g$ is tame, so is the envelope just constructed (\cf \cref{Proposition_Tame_maps}, \labelcref{Proposition_Tame_maps_3}).
\end{proof}

\Needspace{25\baselineskip}
\subsection{The Nash-Moser inverse function theorem}

\begin{definition}\label{Definition_Family_of_morphisms}
Let $E$, $E'$ and $E^i$, $i=1,\dots,3$, be $\overline{\text{sc}}$-Fr{\'e}chet spaces, and let $U \subseteq E$ and $U' \subseteq E'$ be open subsets.
\begin{enumerate}[label=\arabic*.,ref=\arabic*.]
  \item An \emph{$\overline{\text{sc}}^0$ family (over $U$) of morphisms (from $E^1$ to $E^2$)} is an $\overline{\text{sc}}^0$ map
\begin{align*}
\phi : U\times E^1 &\to E^2 \\
(x,u) &\mapsto \phi(x)u
\end{align*}
that is linear in the second factor. \\
By abuse of notation, given a morphism $F : E^1 \to E^2$, the constant family
\begin{align*}
U\times E^1 &\to E^2 \\
(x,u) &\mapsto F(u)
\end{align*}
will be denoted by $F$ again (and will always be $\overline{\text{sc}}^0$).
  \item Given $\overline{\text{sc}}^0$ families of morphisms $\phi : U\times E^1 \to E^2$ and $\psi : U\times E^2 \to E^3$, their \emph{composition} is the $\overline{\text{sc}}^0$ family of morphisms
\begin{align*}
\psi\circ\phi : U\times E^1 &\to E^3 \\
(x,u) &\mapsto \psi(x)\phi(x)u\text{.}
\end{align*}
  \item An $\overline{\text{sc}}^0$ family of morphisms $\phi : U\times E^1 \to E^2$ is called \emph{invertible} if there exists an $\overline{\text{sc}}^0$ family of morphisms $\psi : U\times E^2 \to E^1$ with $\psi\circ \phi = \id_{E^1}$ and $\phi\circ \psi = \id_{E^2}$.
$\psi$ will then be called the \emph{inverse} to $\phi$ and denoted by $\phi\inv \definedas \psi$.
  \item Given an $\overline{\text{sc}}^0$ family of morphisms $\phi : U\times E^1 \to E^2$ and an $\overline{\text{sc}}^0$ map $f : U' \to U$, the \emph{pullback} of $\phi$ by $f$ is the $\overline{\text{sc}}^0$ family of morphisms $f^\ast \phi \definedas \phi\circ (f\times \id_{E^1})$, \ie
\begin{align*}
f^\ast \phi : U' \times E^1 &\to E^2 \\
(x,u) &\mapsto \phi(f(x))u\text{.}
\end{align*}
\end{enumerate}
And analogously for $\phi$ $\underline{\text{sc}}^0$ and/or tame.
\end{definition}

\begin{remark}
If $\phi : U\times E^1 \to E^2$ is an $\overline{\text{sc}}^0$ family of morphisms, then for every $x \in U$, $\phi(x) : E^1 \to E^2$ is a morphism of $\overline{\text{sc}}$-Fr{\'e}chet spaces.
For by assumption it is a linear map and as the composition of $\phi$ with the inclusion $E^1 \to U\times E^1$, $e \mapsto (x,e)$ it is $\overline{\text{sc}}^0$, hence $\underline{\text{sc}}^0$.
So the claim follows from \cref{Lemma_Morphism_equals_sc0}.
\end{remark}

\begin{theorem}[The Nash-Moser inverse function theorem]\label{Theorem_Nash_Moser}
Let $E$ and $E'$ be pre-tame $\overline{\text{sc}}$-Fr{\'e}chet spaces, let $U \subseteq E$ be an open subset and let $f : U\to E'$ be a map that is (p-)$\underline{\text{sc}}^k$ and tame up to order $k$ for some $k\in \N_0\cup\{\infty\}$ with $k \geq 2$. \\
If $E$ is weakly tame and $Df : U\times E \to E'$ is invertible with an inverse $\psi : U\times E' \to E$ that is (p-)$\underline{\text{sc}}^0$ and tame (up to order $0$), then $f$ is locally invertible and each local inverse is (p-)$\underline{\text{sc}}^k$ and tame up to order $k$. \\
I.\,e.~for each $x \in U$ there exists an open neighbourhood $V \subseteq U$ of $x$ \st $f(V) \subseteq E'$ is an open neighbourhood of $f(x)$ and there exists a map $g : f(V) \to V$ that is \mbox{(p-)$\underline{\text{sc}}^k$} and tame up to order $k$, and that satisfies $g\circ f|_V = \id_V$ and $f|_V\circ g = \id_{f(V)}$.
Also, $Dg = g^\ast\psi$.
\end{theorem}

\Needspace{15\baselineskip}
\subsubsection{The proof of the Nash-Moser inverse function theorem}

There are two main differences between \cref{Theorem_Nash_Moser} and \cite{MR656198}, Part III, Theorem 1.1.1:
In \cref{Theorem_Nash_Moser} it is only assumed that we are dealing with weakly tame $\overline{\text{sc}}$-Fr{\'e}chet spaces instead of tame $\overline{\text{sc}}$-Fr{\'e}chet spaces and furthermore, we are also dealing with maps of class $\underline{\text{sc}}^k$ instead of only maps of class p-$\underline{\text{sc}}^k$. \\
Now the proof of the Nash-Moser inverse function theorem in \cite{MR656198} (or the nearly identical one in \cite{MR1471480}), while also using the existence of smoothing operators, makes explicit use of the tameness condition, especially in the proof of local surjectivity.
While there are alternative sources for a proof using only existence of smoothing operators, such as \cite{MR546504}, these use a definition of smoothing operators which requires the constant $p$ from \cref{Definition_Smoothing_operators} to be zero.
But allowing $p > 0$ is crucial for the notion of an ILB-chain to admit smoothing operators to be well behaved under tame rescalings and equivalences, \cf \cref{Lemma_Weakly_tame_ILB_chain}.
The proof in \cite{MR546504} also does not show tameness of the inverse map, \cf the definition of the quantity $\sigma(n,b)$ in \cite{MR546504}, Lemma 4, which defines a non-tame rescaling.

As a first step, observe that it only needs to be shown that in case $k = 2$, around each point $x \in U$ there exists a neighbourhood $V \subseteq U$ of $x$ and a local inverse $g : f(V) \to V$ with $Dg = g^\ast\psi$ that is $\underline{\text{sc}}^1$ (p-$\underline{\text{sc}}^1$) and tame.
For the cases of $k \geq 2$ then follow from \cref{Theorem_Chain_rule_sc_differentiable,Lemma_Invertible_families,Proposition_Standard_properties_tame}.

\begin{lemma}\label{Lemma_Invertible_families}
Let $E$ and $E^i$, $i=1,2$, be $\overline{\text{sc}}$-Fr{\'e}chet spaces, and let $U \subseteq E$ be an open subset.
Let furthermore $\phi : U\times E^1 \to E^2$ be an invertible ($\overline{\text{sc}}^0$ or $\underline{\text{sc}}^0$) family of morphisms and let $k\in\N_0\cup \{\infty\}$.
\begin{enumerate}[label=\arabic*.,ref=\arabic*.]
  \item\label{Lemma_Invertible_families_1} If $\phi$ is $\overline{\text{sc}}^k$ and $\phi\inv$ is $\overline{\text{sc}}^0$, then $\phi\inv$ is $\overline{\text{sc}}^k$.
  \item\label{Lemma_Invertible_families_2} If $\phi$ is $\underline{\text{sc}}^k$ and $\phi\inv$ is $\underline{\text{sc}}^0$, then $\phi\inv$ is $\underline{\text{sc}}^k$.
  \item\label{Lemma_Invertible_families_3} If $\phi$ is p-$\overline{\text{sc}}^k$ and $\phi\inv$ is $\overline{\text{sc}}^0$, then $\phi\inv$ is p-$\overline{\text{sc}}^k$.
  \item\label{Lemma_Invertible_families_4} If $\phi$ is p-$\underline{\text{sc}}^k$ and $\phi\inv$ is $\underline{\text{sc}}^0$, then $\phi\inv$ is p-$\underline{\text{sc}}^k$.
\end{enumerate}
In all of these cases, if $E$ and $E^i$, $i=1,2$, are pre-tame $\overline{\text{sc}}$-Fr{\'e}chet spaces, $\phi$ is tame up to order $k$ and $\phi\inv$ is tame (up to order $0$), then $\phi\inv$ is tame up to order $k$.
\end{lemma}
\begin{proof}
It will be shown that in all cases,
\begin{align*}
D\phi\inv : (U\times E^2)\times (E\times E^2) &\to E^1 \\
D\phi\inv(x,u)(e,v) &= \phi\inv(x)(v - D\phi(x, \phi\inv(x)u)(e, 0))\text{.}
\end{align*}
\labelcref{Lemma_Invertible_families_3,Lemma_Invertible_families_4}~ would actually follow from \cite{MR656198}, Part I, Theorem 5.3.1, and Part II, Theorem 3.1.1, in conjunction with \cref{Proposition_Composition_sc_continuous}, but will be explicitely shown below. \\
First, observe that once the above formula for $D\phi\inv$ has been shown, it suffices to consider the case $k = 1$, for the other cases then follow by induction, using \cref{Theorem_Chain_rule_sc_differentiable,Proposition_Standard_properties_tame}. \\
So let $((x,u),(e,v)) \in (U\times E^2)\times (E\times E^2)$ and $t \in (0,\infty)$.
Then
\begin{align*}
\phi\inv(x+te)(u+tv) - \phi\inv(x)u &= \phi\inv(x+te)u - \phi\inv(x)u + t\phi\inv(x+te)v \\
&= \phi\inv(x+te)(u - \phi(x+te)\phi\inv(x)u + tv) \\
&= \phi\inv(x+te)(tv - (\phi(x+te)- \phi(x))\phi\inv(x)u) \\
&= \phi\inv(x+te)(tv - tD\phi(x, \phi\inv(x)u)(e,0) \; \\
&\quad\; -\; t\,r^\phi_{(x, \phi\inv(x)u)}((e,0),t)) \\
&= t\phi\inv(x)(v - D\phi(x, \phi\inv(x)u)(e,0)) \;+ \\
&\quad\; +\; tr^{\phi\inv}_{(x,u)}((e,v),t)\text{,}
\end{align*}
where
\begin{align*}
r^{\phi\inv}_{(x,u)}((e,v),t) &= (\phi\inv(x+te) - \phi\inv(x))(v - D\phi(x, \phi\inv(x)u)(e,0)) \;- \\
&\quad\; -\; r^\phi_{(x, \phi\inv(x)u)}((e,0),t)\text{.}
\end{align*}
From this it follows by definition that $\phi\inv$ is (pointwise) weakly Fr{\'e}chet differentiable provided that $\phi$ is and provided that $\phi\inv$ is weakly continuous (which just means that it is a continuous map $U\times E^2 \to E^1$). \\
The above formula for $D\phi\inv$ then also shows (by \cref{Theorem_Chain_rule_sc_differentiable}) that $D\phi\inv$ is $\overline{\text{sc}}^0$ or $\underline{\text{sc}}^0$, provided that $\phi\inv$ and $D\phi$ are, and (by \cref{Proposition_Standard_properties_tame}) that $D\phi\inv$ is tame, provided that $\phi\inv$ and $D\phi$ are.
\end{proof}

Now assume that the assumptions of \cref{Theorem_Nash_Moser} hold.
I.\,e.~let $E$ and $E'$ be pre-tame $\overline{\text{sc}}$-Fr{\'e}chet spaces with $E$ weakly tame, let $U \subseteq E$ be an open subset and let $f : U \to E'$ be a map that is (p-)$\underline{\text{sc}}^2$ and tame up to order $2$.
Furthermore, assume that $Df : U\times E \to E'$ is invertible with inverse $\psi : U\times E' \to E$ that is (p-)$\underline{\text{sc}}^0$ and tame.
Given $x \in U$ one has to show that there exists an open neighbourhood $V \subseteq U$ \st $f(V) \subseteq E'$ is an open neighbourhood of $f(x)$ and there exists a map $g : f(V) \to V$ which is $\underline{\text{sc}}^1$ and tame, and satisfies $g\circ f|_V = \id_V$ and $f|_V\circ g = \id_{f(V)}$. \\
As usual in this type of proof, one can assume w.\,l.\,o.\,g.~that $x = 0$ and $f(x) = 0$.
For if $x \neq 0$ or $f(x) \neq 0$, then replace $f$ by
\begin{align*}
\tilde{f} : E \supseteq U - x &\to E' \\
y &\mapsto f(x + y) - f(x)\text{.}
\end{align*}

Moving to a concrete choice of sc-chain on $E$, directly from the definitions (\cref{Definition_sc_continuous,Definition_sc1,Definition_sck,Definition_Tame_estimate,Definition_Tame_nonlinear_map}), and using (the analogue in the tame context of) \cref{Lemma_Tame_sequence}, \labelcref{Lemma_Tame_sequence_5}, \cref{Lemma_sc0_well_defined,Proposition_Tame_maps}, there exists a neighbourhood $U' \subseteq U$ of $U$ and a compatible $\overline{\text{sc}}$-structure $(\mathbb{E}, \phi)$ on $E$ that admits smoothing operators as well as a compatible $\overline{\text{sc}}$-structure $(\mathbb{E}', \phi')$ on $E'$, \st the following holds: \\
By abuse of notation, for $j \in \N_0$, denote the norm $(\iota^\infty_j)^\ast \|\cdot\|_j : E_\infty \to [0,\infty)$ by $\|\cdot\|_j$ again and similarly for $\|\cdot\|'_j : E'_\infty \to [0,\infty)$.
And similarly if $\mathbf{k} \subseteq \N_0$ is a strictly monotone increasing sequence, then $E^{\mathbf{k}}_\infty = E_\infty$ and we have the norms $\|\cdot\|^{\mathbf{k}}_j = \|\cdot\|_{k_0+j} : E_\infty \to [0,\infty)$ and $\|\cdot\|'^{\mathbf{k}}_{j} = \|\cdot\|'_{k_0+j} : E'_\infty \to [0,\infty)$. \\
Define
\begin{align*}
V' &\definedas \phi\inv(U') \subseteq E_\infty \\
\tilde{f} &\definedas \phi'^{-1}\circ f|_{U'}\circ \phi : V' \to E'_\infty \\
\tilde{\psi} &\definedas \phi\inv\circ \psi|_{U'\times E}\circ (\phi\times \phi') : V' \times E'_\infty \to E_\infty\text{.}
\end{align*}
Then there exist shifts $\mathbf{k}, \mathbf{l}, \mathbf{m}, \mathbf{n} \subseteq \N_0$ with the following properties:
\begin{enumerate}
  \item $V' \subseteq (E_\infty, \|\cdot\|_j)$ is open for all $j \in \N_0$.
Equivalently, $V' \subseteq (E_\infty, \|\cdot\|_0)$ is open.
  \item $\tilde{f} : (E_\infty, \|\cdot\|^{\mathbf{n}}_j) \supseteq V' \to (E'_\infty, \|\cdot\|'_j)$ is continuous and satisfies a tame estimate for all $j \in \N_0$.
  \item $D\tilde{f} : (E_\infty\times E_\infty, \|\cdot\|^{\mathbf{l}}_j \oplus \|\cdot\|^{\mathbf{l}}_j) \supseteq V' \times E_\infty \to (E'_\infty, \|\cdot\|'_j)$ is continuous and satisfies a tame estimate for all $j\in \N_0$.
  \item $D^2\tilde{f} : (E_\infty\times E_\infty\times E_\infty, \|\cdot\|^{\mathbf{m}}_j \oplus \|\cdot\|^{\mathbf{m}}_j \oplus \|\cdot\|^{\mathbf{m}}_j) \supseteq V' \times E_\infty\times E_\infty \to (E'_\infty, \|\cdot\|'_j)$ is continuous and satisfies a tame estimate for all $j\in \N_0$.
  \item $\tilde{\psi} : (E_\infty\times E'_\infty, \|\cdot\|^{\mathbf{k}}_j \oplus \|\cdot\|'^{\mathbf{k}}_j) \supseteq V' \times E'_\infty \to (E_\infty, \|\cdot\|_j)$ is continuous and satisfies a tame estimate for all $j\in \N_0$.
\end{enumerate}
Also, w.\,l.\,o.\,g.~one can assume that $V'$ is convex.
Now $\mathbf{n}\circ \mathbf{l}\circ \mathbf{m} \geq \mathbf{n},\mathbf{l},\mathbf{m}$, so one can replace $\mathbf{n}$, $\mathbf{l}$ and $\mathbf{m}$ by $\mathbf{n}\circ \mathbf{l}\circ \mathbf{m}$, \ie one can assume that $\mathbf{n} = \mathbf{l} = \mathbf{m}$.
Assuming this to hold, by replacing $\mathbb{E}$ by $\mathbb{E}^{\mathbf{n}}$ and replacing $\mathbf{k}$ by $\mathbf{k}\circ \mathbf{n}$ one can furthermore assume that $\mathbf{n} = \mathbf{l} = \mathbf{m} = \mathbf{id} = (j)_{j\in\N_0}$.

In the first part of the proof below I will follow the proof from \cite{MR656198} very closely, only modifying the proofs where necessary to accommodate for the changed assumptions in \cref{Theorem_Nash_Moser}. \\
First, due the (bi-)linearity of $D\tilde{f}$, $\tilde{\psi}$ and $D^2\tilde{f}$ in the second (and third) factor, one has the following variant of the tame estimates:
\begin{lemma}\label{Lemma_Modified_tame_estimates}
In the notation as above, there exist constants $\delta > 0$, $a_j,b_j, c_j, d_j \in [0,\infty)$, for $j \in \N_0$, \st
\begin{align*}
\|\tilde{f}(x)\|'_j &\leq a_j\|x\|_j \\
\|D\tilde{f}(x)e\|'_j &\leq b_j\left(\|x\|_j\|e\|_0 + \|e\|_j\right) \\
\|D^2\tilde{f}(x)(e,\tilde{e})\|'_j &\leq c_j\left(\|x\|_j\|e\|_0\|\tilde{e}\|_0 + \|e\|_j\|\tilde{e}\|_0 + \|e\|_0\|\tilde{e}\|_j\right) \\
\|\tilde{\psi}(x)e'\|_j &\leq d_j\left(\|x\|^{\mathbf{k}}_j\|e'\|'^{\mathbf{k}}_0 + \|e'\|'^{\mathbf{k}}_j\right)
\end{align*}
for all $x \in E_\infty$ with $\|x\|^{\mathbf{k}}_0 < \delta$, $e,\tilde{e} \in E_\infty$, $e' \in E'_\infty$, and $j\in\N_0$.
\end{lemma}
\begin{proof}
The first inequality is just the tameness condition for $\tilde{f}$.
For the remaining ones apply \cite{MR656198}, Part II, Lemmas 2.1.7 and 2.1.8, to $L = D\tilde{f}$ and $b = r = s = 0$, $B = D^2\tilde{f}$ and $b = r = s = t = 0$, and $L = \tilde{\psi}$ and $b = 0$, $r = s = k_0$.
\end{proof}

The following proposition settles local injectivity of $\tilde{f}$ in a neighbourhood of $0$ and will also provide the main ingredient in the formula for the derivative of the local inverse, once surjectivity of $\tilde{f}$ onto a neighbourhood of $0$ has been shown.
\begin{proposition}\label{Proposition_Nash_Moser_injectivity}
In the setting as above, there exist constants $\delta > 0$, $c'_j, c''_j \in [0,\infty)$, for $j \in \N_0$, \st
\[
\|y - x\|_j \leq c'_j\bigl(\bigl(\|x\|^{\mathbf{k}}_j + \|y\|^{\mathbf{k}}_{j}\bigr)\bigl\|\tilde{f}(y) - \tilde{f}(x)\bigr\|'^{\mathbf{k}}_0 + \bigl\|\tilde{f}(y) - \tilde{f}(x)\bigr\|'^{\mathbf{k}}_j\bigr)
\]
and
\begin{align*}
\bigl\|y - x - \tilde{\psi}(x)\bigl(\tilde{f}(y) - \tilde{f}(x)\bigr)\bigr\|_j &\leq c''_j\bigl(1 + \|x\|^{\mathbf{k}\circ\mathbf{k}}_0 + \|y\|^{\mathbf{k}\circ\mathbf{k}}_0\bigr) \,\cdot \\
&\quad\; \cdot\, \bigl(\bigl(\|x\|^{\mathbf{k}\circ\mathbf{k}}_j + \|y\|^{\mathbf{k}\circ\mathbf{k}}_{j}\bigr)\bigl\|\tilde{f}(y) - \tilde{f}(x)\bigr\|'^{\mathbf{k}\circ\mathbf{k}}_0 \;+ \\
&\quad\; +\; \bigl\|\tilde{f}(y) - \tilde{f}(x)\bigr\|'^{\mathbf{k}\circ\mathbf{k}}_j\bigr)\bigl\|\tilde{f}(y) - \tilde{f}(x)\bigr\|'^{\mathbf{k}}_0
\end{align*}
for all $j \in \N_0$ and $x,y \in E_\infty$ with $\|x\|^{\mathbf{k}}_0, \|y\|^{\mathbf{k}}_0 < \delta$.
\end{proposition}
\begin{proof}
This is an amalgamation of \cite{MR656198}, Part III, Theorem 1.3.1 and Corollaries 1.3.2 and 1.3.3. \\
Following the layout of the proof there, by Taylor's theorem, for $x, y \in V'$,
\[
\tilde{f}(y) = \tilde{f}(x) + D\tilde{f}(x)(y-x) + \int_0^1(1-t)D^2\tilde{f}((1-t)x + ty)(y-x, y-x)\, \d t
\]
and hence, using $\tilde{\psi}(x)D\tilde{f}(x) = \id_{E_\infty}$,
\[
y-x = \underbrace{\tilde{\psi}(x)\bigl(\tilde{f}(y) - \tilde{f}(x)\bigr)}_{\defines\;\alpha(x,y)} - \underbrace{\tilde{\psi}(x)\int_0^1(1-t)D^2\tilde{f}((1-t)x + ty)(y-x,y-x)\, \d t}_{\defines\;\beta(x,y)}\text{.}
\]
Applying \cref{Lemma_Modified_tame_estimates}, for $j \in \N_0$,
\begin{align}
\|\alpha(x,y)\|_j &= \bigl\|\tilde{\psi}(x)\bigl(\tilde{f}(y) - \tilde{f}(x)\bigr)\bigr\|_j \nonumber\\
&\leq d_j\bigl(\|x\|^{\mathbf{k}}_j\bigl\|\tilde{f}(y) - \tilde{f}(x)\bigr\|'^{\mathbf{k}}_0 + \bigl\|\tilde{f}(y) - \tilde{f}(x)\bigr\|'^{\mathbf{k}}_j\bigr)\text{.}\label{Eqn_4}
\end{align}
Similarly, applying \cref{Lemma_Modified_tame_estimates} twice, for $j \in \N_0$,
{\allowdisplaybreaks
\begin{align*}
\|\beta(x,y)\|_j &= \left\|\tilde{\psi}(x)\int_0^1(1-t)D^2\tilde{f}((1-t)x + ty)(y-x,y-x)\, \d t\right\|_j \\
&\leq d_j\biggl(\|x\|^{\mathbf{k}}_j\left\| \int_0^1(1-t)D^2\tilde{f}((1-t)x + ty)(y-x,y-x)\, \d t \right\|'^{\mathbf{k}}_0 \;+ \\
&\quad\; \qquad\;\; +\; \left\| \int_0^1(1-t)D^2\tilde{f}((1-t)x + ty)(y-x,y-x)\, \d t \right\|'^{\mathbf{k}}_j\biggr) \\
&\leq d_j\biggl(\|x\|^{\mathbf{k}}_j \int_0^1(1-t)\bigl\|D^2\tilde{f}((1-t)x + ty)(y-x,y-x)\|'_{k_0}\, \d t \;+ \\
&\quad\; \qquad\;\; +\; \int_0^1(1-t)\bigl\|D^2\tilde{f}((1-t)x + ty)(y-x,y-x)\bigr\|'_{k_0+j}\, \d t \biggr) \\
&\leq d_j\biggl(\|x\|^{\mathbf{k}}_j \int_0^1(1-t)c_{k_0}(\|(1-t)x + ty\|_{k_0}(\|y-x\|_0)^2 \;+ \\
&\qquad\qquad\qquad\quad +\; 2\|y-x\|_{k_0}\|y-x\|_0)\, \d t \;+ \\
&\quad\; \qquad\;\; +\; \int_0^1(1-t)c_{k_0+j}(\|(1-t)x + ty\|_{k_0+j}(\|y-x\|_0)^2 \;+ \\
&\qquad\qquad\qquad\quad +\; 2\|y-x\|_{k_0+j}\|y-x\|_0)\, \d t \biggr) \\
&\leq d_j\biggl((c_{k_0} + c_{k_0+j} + 2)\bigl(1 + \|x\|^{\mathbf{k}}_0 + \|y\|^{\mathbf{k}}_{0}\bigr)\bigl(\|x\|^{\mathbf{k}}_j + \|y\|^{\mathbf{k}}_{j}\bigr) \,\cdot \\
&\qquad\quad \cdot\, \|y-x\|^{\mathbf{k}}_{0}\|y-x\|_0 + 2\|y-x\|^{\mathbf{k}}_{j}\|y-x\|_0 \biggr) \\
&\leq \underbrace{d_j(c_{k_0} + c_{k_0+j} + 2)(1 + 2\delta)}_{\defines \; d'_j}\,\cdot \\
&\quad\; \cdot\,\bigl(\bigl(\|x\|^{\mathbf{k}}_j + \|y\|^{\mathbf{k}}_{j}\bigr)\|y-x\|^{\mathbf{k}}_{0} + \|y-x\|^{\mathbf{k}}_{j}\bigr)\|y-x\|_0\text{.}
\end{align*}
}
So
\begin{equation}\label{Eqn_5}
\|\beta(x,y)\|_j \leq d'_j\bigl(\bigl(\|x\|^{\mathbf{k}}_j + \|y\|^{\mathbf{k}}_{j}\bigr)\|y-x\|^{\mathbf{k}}_{0} + \|y-x\|^{\mathbf{k}}_{j}\bigr)\|y-x\|_0\text{.}
\end{equation}
Applying this for $j = 0$, then given any $\delta > 0$, for all $x,y \in V'$ with $\|x\|^{\mathbf{k}}_0, \|y\|^{\mathbf{k}}_0 < \delta$ one has the estimates
\begin{align*}
\|\alpha(x,y)\|_0 &\leq d_0(1 + \delta)\bigl\|\tilde{f}(y) - \tilde{f}(x)\bigr\|'^{\mathbf{k}}_0 \\
\|\beta(x,y)\|_0 &\leq d'_0(1+2\delta)\|y-x\|^{\mathbf{k}}_0\|y-x\|_0 \\
&\leq d'_0(1+2\delta)2\delta\|y-x\|_0\text{.}
\end{align*}
And hence from $\|y-x\|_0 - \|\beta(x,y)\|_0 \leq \|\alpha(x,y)\|_0$ it follows that
\begin{align*}
\left(1 - 2d'_0(1+2\delta)\delta\right)\|y-x\|_0 \leq d_0(1 + \delta)\bigl\|\tilde{f}(y) - \tilde{f}(x)\bigr\|'^{\mathbf{k}}_0\text{.}
\end{align*}
Choosing $\delta > 0$ \st $\delta < 1$ and $\delta < \frac{1}{12d'_0}$, this implies that
\begin{align}
\|y-x\|_0 &\leq 4d_0\bigl\|\tilde{f}(y) - \tilde{f}(x)\bigr\|'^{\mathbf{k}}_0\text{.}\label{Eqn_8}
\end{align}
This shows the first estimate in the statment of the proposition, in case $j = 0$ with $c_0 \definedas 4d_0$.
From now on, $\delta$ is fixed, satisfying the above assumptions.
For general $j \in \N_0$ one can then combine this with the above estimate for $\|\beta(x,y)\|_j$ to obtain
\begin{align*}
\|\beta(x,y)\|_j &\leq 4d_0d'_j(1+2\delta)\bigl(\|x\|^{\mathbf{k}}_j + \|y\|^{\mathbf{k}}_{j}\bigr)\bigl\|\tilde{f}(y) - \tilde{f}(x)\bigr\|'^{\mathbf{k}}_0\text{.}
\end{align*}
Setting $c'_j \definedas (1+d_j)(1 + 4d_0d'_j(1+2\delta))$, combining this inequality with \labelcref{Eqn_4} shows that
\begin{align}
\|y-x\|_j &\leq \|\alpha(x,y)\|_j + \|\beta(x,y)\|_j \nonumber\\
&\leq c'_j\bigl(\bigl(\|x\|^{\mathbf{k}}_j + \|y\|^{\mathbf{k}}_{j}\bigr)\bigl\|\tilde{f}(y) - \tilde{f}(x)\bigr\|'^{\mathbf{k}}_0 + \bigl\|\tilde{f}(y) - \tilde{f}(x)\bigr\|'^{\mathbf{k}}_j\bigr)\label{Eqn_7}
\end{align}
for all $x,y\in V'$ with $\|x\|^{\mathbf{k}}_0,\|y\|^{\mathbf{k}}_0 < \delta$.
This shows the first part of the proposition.
For the second part, note that
\[
\bigl\|y - x - \tilde{\psi}(x)\bigl(\tilde{f}(y) - \tilde{f}(x)\bigr)\bigr\|_j = \|\beta(x,y)\|_j
\]
and combining \labelcref{Eqn_5}, \labelcref{Eqn_8} and \labelcref{Eqn_7} produces
\begin{align*}
\|\beta(x,y)\|_j &\leq 4d_0d'_j\bigl(\bigl(\|x\|^{\mathbf{k}}_j + \|y\|^{\mathbf{k}}_{j}\bigr)\|y-x\|^{\mathbf{k}}_{0} + \|y-x\|^{\mathbf{k}}_{j}\bigr)\bigl\|\tilde{f}(y) - \tilde{f}(x)\bigr\|'^{\mathbf{k}}_0 \\
&\leq 4d_0d'_j\bigl(\bigl(\|x\|^{\mathbf{k}}_j + \|y\|^{\mathbf{k}}_{j}\bigr)
c'_{k_0}\bigl(\bigl(\|x\|^{\mathbf{k}\circ\mathbf{k}}_0 + \|y\|^{\mathbf{k}\circ\mathbf{k}}_0\bigr)\bigl\|\tilde{f}(y) - \tilde{f}(x)\bigr\|'^{\mathbf{k}}_0 \;+ \\
&\qquad\qquad\qquad\qquad\qquad\qquad\; +\; \bigl\|\tilde{f}(y) - \tilde{f}(x)\bigr\|'^{\mathbf{k}\circ\mathbf{k}}_0\bigr) \;+ \\
&\qquad\qquad\; +\; c'_{k_0+j}\bigl(\bigl(\|x\|^{\mathbf{k}\circ\mathbf{k}}_j + \|y\|^{\mathbf{k}\circ\mathbf{k}}_{j}\bigr)\bigl\|\tilde{f}(y) - \tilde{f}(x)\bigr\|'^{\mathbf{k}}_0 \;+ \\
&\qquad\qquad\qquad\qquad\; +\;\bigl\|\tilde{f}(y) - \tilde{f}(x)\bigr\|'^{\mathbf{k}\circ\mathbf{k}}_j\bigr)
\bigr) \,\cdot \\
&\quad\; \cdot\,\bigl\|\tilde{f}(y) - \tilde{f}(x)\bigr\|'^{\mathbf{k}}_0 \\
&\leq 4d_0d'_j(c'_{k_0} + c'_{k_0+j})\bigl(1 + \|x\|^{\mathbf{k}\circ\mathbf{k}}_0 + \|y\|^{\mathbf{k}\circ\mathbf{k}}_0\bigr) \,\cdot \\
&\quad\; \cdot\, \bigl(\bigl(\|x\|^{\mathbf{k}}_j + \|y\|^{\mathbf{k}}_{j}\bigr)\bigl\|\tilde{f}(y) - \tilde{f}(x)\bigr\|'^{\mathbf{k}\circ\mathbf{k}}_0 \;+ \\
&\quad\; +\; \bigl(\|x\|^{\mathbf{k}\circ\mathbf{k}}_j + \|y\|^{\mathbf{k}\circ\mathbf{k}}_{j}\bigr)\bigl\|\tilde{f}(y) - \tilde{f}(x)\bigr\|'^{\mathbf{k}}_0 \;+ \\
&\quad\; +\; \bigl\|\tilde{f}(y) - \tilde{f}(x)\bigr\|'^{\mathbf{k}\circ\mathbf{k}}_j\bigr)\bigl\|\tilde{f}(y) - \tilde{f}(x)\bigr\|'^{\mathbf{k}}_0 \\
&\leq 8d_0d'_j(c'_{k_0} + c'_{k_0+j})\bigl(1 + \|x\|^{\mathbf{k}\circ\mathbf{k}}_0 + \|y\|^{\mathbf{k}\circ\mathbf{k}}_0\bigr) \,\cdot \\
&\quad\; \cdot\, \bigl(\bigl(\|x\|^{\mathbf{k}\circ\mathbf{k}}_j + \|y\|^{\mathbf{k}\circ\mathbf{k}}_{j}\bigr)\bigl\|\tilde{f}(y) - \tilde{f}(x)\bigr\|'^{\mathbf{k}\circ\mathbf{k}}_0 \;+ \\
&\quad\; +\; \bigl\|\tilde{f}(y) - \tilde{f}(x)\bigr\|'^{\mathbf{k}\circ\mathbf{k}}_j\bigr)\bigl\|\tilde{f}(y) - \tilde{f}(x)\bigr\|'^{\mathbf{k}}_0\text{.}
\end{align*}
Setting $c''_j \definedas 8d_0d'_j(c'_{k_0} + c'_{k_0+j})$ finishes the proof.
\end{proof}
\begin{corollary}\label{Corollary_Nash_Moser_injectivity}
In the setting as above, there exists a shift $\mathbf{l} \subseteq \N_0$ and an open neighbourhood $W \subseteq (E_\infty, \|\cdot\|^{\mathbf{l}}_0)$ of $0$ \st
\[
\tilde{f}|_W : W \to E'_\infty
\]
is injective.
\end{corollary}
\begin{proof}
Let $W$ be a ball of radius $\delta$ around zero in $(E_\infty, \|\cdot\|^\mathbf{k}_0)$, where $\delta$ is as in the proposition.
Then for $x,y \in W$,
\[
\|y - x\|_0 \leq c'_0(1+2\delta)\|\tilde{f}(y) - \tilde{f}(x)\|'^{\mathbf{k}}_0\text{.}
\]
\end{proof}

To show local surjectivity of $\tilde{f}$, I will follow the modified Newton's method from \cite{MR546504}.
It is at this point that the condition of $E$ being weakly tame comes into play.
Concretely this means that one can assume that $\mathbb{E}$ admits smoothing operators, \ie by \cref{Definition_Smoothing_operators,Remark_Strongly_cts_family_of_smoothing_operators} one can assume that there exists a continuous map
\begin{align*}
S : [0,\infty) &\to L_{\mathrm{c}}(E_0,E_\infty) \\
t &\mapsto S_t
\end{align*}
and constants $p, C_{n,m} \in [0,\infty)$, $n,m\in\N_0$, \st
\begin{align*}
\|S^{m}_{n,t}\|_{L_{\mathrm{c}}(E_m,E_n)} &\leq C^m_n (1 + e^{(p + (n-m))t}) & &\forall\, m, n \in\N_0, t \geq 0 \\
\|\iota^m_n - S^{m}_{n,t}\|_{L_{\mathrm{c}}(E_m,E_n)} &\leq C^m_n e^{(p - (m-n))t} & &\forall\, m, n \in\N_0, m-n\geq p, t \geq 0\text{,}
\end{align*}
where $S^{m}_{n,t} \definedas \iota^\infty_n\circ S_t\circ \iota^m_0 : E_m \to E_n$. \\
By abuse of notation, I will also denote by $S_t : E_\infty \to E_\infty$ the map $S_t\circ \iota^\infty_0$.
Then the above inequalities read
\begin{align}
\|S_te\|_n &\leq C^m_n (1 + e^{(p + (n-m))t})\|e\|_m & &\forall\, e \in E_\infty, m, n \in\N_0, t \geq 0 \nonumber\\
\|(\id_{E_\infty} - S_t)e\|_n &\leq C^m_n e^{(p - (m-n))t}\|e\|_m & &\forall\, e \in E_\infty, m, n \in\N_0,\label{Smoothing_estimates_std_form} \\
& & & \quad m-n\geq p, t \geq 0\text{.}\nonumber
\end{align}
Or, using \cref{Remark_Alternative_definition_smoothing_operators_II}, these can also be written as
\begin{align}
\|S_te\|_n &\leq \tilde{C}^m_n (1 + e^{(n-m)t})\|e\|^{\mathbf{p}}_{m} & &\forall\, e \in E_\infty, m, n \in\N_0, t \geq 0 \nonumber\\
\|(\id_{E_\infty} - S_t)e\|_n &\leq \tilde{C}^m_n e^{-(m-n)t}\|e\|^{\mathbf{p}}_{m} & &\forall\, e \in E_\infty, m, n \in\N_0,\label{Eqn_sm_op_shifted} \\
& & & \quad m-n\geq 0, t \geq 0\text{,}\nonumber
\end{align}
where $\mathbf{p} = (p + j)_{j\in\N_0}$.

The modified Newton's procedure then goes as follows:
Given $y \in E_\infty$, inductively define a sequence $(x_r)_{r\in\N_0}$ by
\begin{align*}
x_0 &\definedas 0 \\
x_{r+1} &\definedas x_r + \Delta x_r \\
\Delta x_r &\definedas S_{t_r}\tilde{\psi}(x_r)z_r \\
z_r &\definedas y - \tilde{f}(x_r) \\
t_r &\definedas \left(\frac{3}{2}\right)^r\text{.}
\end{align*}
The goal is to show that there exists a shift $\mathbf{l} \in \N_0$ and constants $\tau_j \in [0,\infty)$ \st
\begin{align*}
\|x_r\|^{\mathbf{k}}_j &\leq \tau_j\|y\|'^{\mathbf{l}}_j \\
\|\Delta x_r\|^{\mathbf{k}}_j &\overset{r\to\infty}{\longrightarrow} 0 \\
\|z_r\|'^{\mathbf{k}}_j &\overset{r\to\infty}{\longrightarrow} 0\text{.}
\end{align*}
This shows that the sequence $(x_r)_{r\in\N_0}$ is well defined, for $y \in E_\infty$ with $\|y\|^{\mathbf{l}}_0 < \delta$, where $\delta > 0$ is some constant.
For taking the special case $j = 0$, $\|x_r\|^{\mathbf{k}}_0 \leq \tau_0\delta$, so $x_r \in V'$ for $\delta$ small enough.
$\|\Delta x_r\|^{\mathbf{k}}_j \overset{r\to\infty}{\longrightarrow} 0$ for all $j\in\N_0$ then shows that $(x_r)_{r\in\N_0}$ converges in $(E_\infty, \|\cdot\|^{\mathbf{k}}_j)$ for all $j \in \N_0$, \ie there exists $x \in E_\infty$ with $x_r \to x$ and by definition of $z_r$ and since $\tilde{f}$ is continuous, $\|z_r\|'^{\mathbf{k}}_j \overset{r\to\infty}{\longrightarrow} 0$ implies $\tilde{f}(x) = y$.
So setting $g(y) \definedas x$ defines the required local inverse of $\tilde{f}$, satisfying
\[
\|g(y)\|^{\mathbf{k}}_j \leq \tau_j\|y\|'^{\mathbf{l}}_j\text{,}
\]
\ie $g$ is tame.

This will be proved through a series of lemmas, where the first one, \cref{Lemma_NM_surjectivity_1}, roughly corresponds to Lemma 1 in \cite{MR546504} and the second one, \cref{Lemma_NM_surjectivity_2}, to part of the proof of Lemma 2 in \cite{MR546504}.

The following points might be worth taking note of, since the proof relies quite heavily on them and it is actually quite important to know what each quantity/constant depends on:
\begin{itemize}
  \item The quantity $\rho$ in \cref{Lemma_NM_surjectivity_1,Lemma_NM_surjectivity_2} \emph{only} depends on $p$, the constant coming from the smoothing operators, and $k_0$, the constant defining the shift $\mathbf{k}$ appearing in the tameness condition for $\tilde{\psi}$.
  \item In the estimates for $\|x_r\|^{\mathbf{k}}_j$, $\|\Delta x_r\|^{\mathbf{k}}_j$, $\|z_r\|'^{\mathbf{k}}_j$ and $\|z_{r+1}\|'^{\mathbf{k}}_j$, \emph{only} products of a $\|\cdot\|'^{\mathbf{k}}_j$- and a $\|\cdot\|'^{\mathbf{k}}_0$-norm (such as $\|y\|'^{\mathbf{k}}_j\|z_r\|'^{\mathbf{k}}_0$ or $\bigl(\|z_r\|'^{\mathbf{k}}_0\bigr)^2$) appear but \emph{no} products of two $\|\cdot\|'^{\mathbf{k}}_j$-norms (such as $\bigl(\|z_r\|'^{\mathbf{k}}_j\bigr)^2$).
\end{itemize}

\begin{lemma}\label{Lemma_NM_surjectivity_1}
In the setting as above, given any $\delta_0 > 0$ small enough, there exist for every $j \in \N_0$ constants $\gamma_j,\bar{\gamma}_j \in [1,\infty)$, depending on $\delta_0$, \st if $\|y\|'^{\mathbf{k}}_0 < \delta_0$ and for some $r_0 \in \N_0$ also $\|x_r\|^{\mathbf{k}}_0 < \delta_0$ for $r = 0, \dots, r_0$, then
\begin{align*}
\|x_r\|^{\mathbf{k}}_j &\leq \gamma_je^{\rho t_r}\|y\|'^{\mathbf{k}}_j
\intertext{and}
\|\Delta x_r\|^{\mathbf{k}}_j &\leq \bar{\gamma}_je^{2\rho t_r}\bigl(\|y\|'^{\mathbf{k}}_j\|z_r\|'^{\mathbf{k}}_0 + \|z_r\|'^{\mathbf{k}}_j\bigr)
\intertext{as well as}
\|z_r\|'^{\mathbf{k}}_j &\leq \hat{\gamma}_je^{\rho t_r}\|y\|'^{\mathbf{k}}_j
\end{align*}
for all $r = 0, \dots, r_0+1$, where $\rho \definedas 2(p+k_0) + 1$.
\end{lemma}
\begin{proof}
For any $j \in \N_0$, if $\delta_0$ is small enough, then the assumption $\|x_r\|^{\mathbf{k}}_0 < \delta_0$ implies $x_r \in V'$ and by \cref{Lemma_Modified_tame_estimates} $\|z_r\|'_j = \|y + \tilde{f}(x_r)\|'_j \leq \|y\|'_j + a_j\|x_r\|_j$.
Hence for any $i,j \in \N_0$ with $j \geq i$, using the estimates on the smoothing operators (from \cref{Remark_Alternative_definition_smoothing_operators}) and \cref{Lemma_Modified_tame_estimates},
\begin{align*}
\|\Delta x_r\|_j &= \|S_{t_r}\tilde{\psi}(x_r)z_r\|_j \\
&\leq C^{j-i}_je^{(p+i)t_r}\|\tilde{\psi}(x_r)z_r\|_{j-i} \\
&\leq C^{j-i}_jd_{j-i}e^{(p+i)t_r}\bigl(\|x_r\|^{\mathbf{k}}_{j-i}\|z_r\|'^{\mathbf{k}}_0 + \|z_r\|'^{\mathbf{k}}_{j-i}\bigr) \\
&\leq C^{j-i}_jd_{j-i}e^{(p+i)t_r}\bigl(\|x_r\|^{\mathbf{k}}_{j-i}(\|y\|'^{\mathbf{k}}_0 + a_{k_0}\|x_r\|^{\mathbf{k}}_0) \;+ \\
&\qquad\qquad\qquad\qquad\quad +\; \|y\|'^{\mathbf{k}}_{j-i} + a_{j+k_0-i}\|x_r\|^{\mathbf{k}}_{j-i}\bigr) \\
&\leq C^{j-i}_jd_{j-i}(1+a_{k_0})(1+a_{j+k_0-i})(1 + \|y\|'^{\mathbf{k}}_0 + \|x_r\|^{\mathbf{k}}_0) \,\cdot \\
&\quad\; \cdot\, e^{(p+i)t_r}\bigl(\|x_r\|^{\mathbf{k}}_{j-i} + \|y\|'^{\mathbf{k}}_{j-i}\bigr)
\end{align*}
Now take $i = k_0$ and define $\gamma'_j \definedas C^{j-k_0}_jd_{j-k_0}(1+a_{k_0})(1+a_{j})(1+2\delta_0)$.
Then $\|\Delta x_r\|_j \leq \gamma'_je^{(p+k_0)t_r}(\|x_r\|_j + \|y\|_j)$.
It is
\begin{align*}
\|x_{r+1}\|_j + \|y\|'_j &= \|x_r + \Delta x_r\|_j + \|y\|'_j \\
&\leq \|x_r\|_j + \|\Delta x_r\|_j + \|y\|'_j \\
&\leq \|x_r\|_j + \gamma'_je^{(p+k_0)t_r}(\|x_r\|_j + \|y\|'_j) + \|y\|'_j \\
&\leq (\gamma'_j + 1)e^{(p+k_0)t_r}(\|x_r\|_j + \|y\|'_j)\text{,}
\end{align*}
so by induction $\|x_{r}\|_j \leq \|x_r\|_j + \|y|'_j \leq (\gamma'_j + 1)^re^{(p+k_0)\sum_{s=0}^{r-1} t_s}\|y\|'_j$.
By definition of $t_r$, $\sum_{s=0}^{r-1} = \sum_{s=0}^{r-1} (3/2)^s = \frac{(3/2)^r-1}{3/2 - 1} = 2((3/2)^r - 1) \leq 2(3/2)^r = 2t_r$.
Hence $\|x_r\|^{\mathbf{k}}_j \leq (\gamma'_j+1)^re^{2(p+k_0)t_r}\|y\|'_j = e^{r\ln(\gamma'_j+1) - t_r} e^{(2(p+k_0) + 1)t_r}\|y\|'_j$.
Now
\[
r\ln(\gamma'_j+1) - t_r = r\ln(\gamma'_j+1) - (3/2)^r \overset{r\to\infty}{\longrightarrow} 0\text{,}
\]
hence on can define $\gamma_{j-k_0} \definedas \max \{1 + e^{r\ln(\gamma'_j+1) - t_r} \;|\; r \geq 0\} \in [1,\infty)$.
This settles the first part of the statement.
For the second part, by \labelcref{Eqn_sm_op_shifted} and \cref{Lemma_Modified_tame_estimates}, for any $i, j \in \N_0$ with $j \geq i$, by \labelcref{Smoothing_estimates_std_form}
\begin{align*}
\|\Delta x_r\|_j &= \|S_{t_r}\tilde{\psi}(x_r)z_r\|_j \\
&\leq C^{j-i}_{j}e^{(p + i)t_r}\|\tilde{\psi}(x_r)z_r\|_{j-i} \\
&\leq C^{j-i}_{j}d_{j-i}e^{(p+i)t_r}\bigl(\|x_r\|^{\mathbf{k}}_{j-i}\|z_r\|'^{\mathbf{k}}_0 + \|z_r\|'^{\mathbf{k}}_{j-i}\bigr)\text{.}
\end{align*}
In particular, if $j \geq k_0$, then one can insert $i = k_0$ and obtain
\begin{align*}
\|\Delta x_r\|_j &\leq C^{j-k_0}_{j}d_{j-k_0}e^{(p+k_0)t_r}\bigl(\|x_r\|^{\mathbf{k}}_{j-k_0}\|z_r\|'^{\mathbf{k}}_0 + \|z_r\|'^{\mathbf{k}}_{j-k_0}\bigr) \\
&= C^{j-k_0}_{j}d_{j-k_0}e^{(p+k_0)t_r}\bigl(\|x_r\|_j\|z_r\|'^{\mathbf{k}}_0 + \|z_r\|'_j\bigr)\text{.}
\end{align*}
Replacing $j$ by $j+k_0$ produces, via the first part of the statement,
\begin{align*}
\|\Delta x_r\|^{\mathbf{k}}_j &\leq C^{j}_{j+k_0}d_{j}e^{(p+k_0)t_r}\bigl(\|x_r\|^{\mathbf{k}}_j\|z_r\|'^{\mathbf{k}}_0 + \|z_r\|'^{\mathbf{k}}_j\bigr) \\
&\leq C^{j}_{j+k_0}d_{j}e^{(p+k_0)t_r}\bigl(\gamma_je^{\rho t_r}\|y\|'^{\mathbf{k}}_j\|z_r\|'^{\mathbf{k}}_0 + \|z_r\|'^{\mathbf{k}}_j\bigr) \\
&\leq C^{j}_{j+k_0}d_{j}\gamma_je^{2\rho t_r}\bigl(\|y\|'^{\mathbf{k}}_j\|z_r\|'^{\mathbf{k}}_0 + \|z_r\|'^{\mathbf{k}}_j\bigr)
\end{align*}
from which the second part of the statement follows with $\bar{\gamma}_j \definedas C^{j}_{j+k_0}d_{j}\gamma_j$.
For the third part of the statement, as before, one estimates, using the first part of the statement
\begin{align*}
\|z_{r+1}\|'^{\mathbf{k}}_j &\leq \|y\|'^{\mathbf{k}}_j + a_{k_0+j}\|x\|^{\mathbf{k}}_j \\
&\leq \|y\|'^{\mathbf{k}}_j + a_{k_0+j}\gamma_je^{\rho t_r}\|y\|'^{\mathbf{k}}_j \\
&\leq	(1 + a_{k_0+j}\gamma_j)e^{\rho t_r}\|y\|'^{\mathbf{k}}_j\text{,}
\end{align*}
from which the third part of the statement follows with $\hat{\gamma}_j \definedas (1 + a_{k_0+j}\gamma_j)$.
\end{proof}

\begin{lemma}\label{Lemma_NM_surjectivity_2}
In the setting as above, given any $\delta_0 > 0$ small enough and any $\mu\in\N_0$, for every $j\in\N_0$ there exists a constant $\lambda_{\mu,j} \in [1,\infty)$, depending on $\delta_0$, \st if $\|y\|'^{\mathbf{k}}_0 < \delta_0$ and for some $r_0 \in \N_0$ also $\|x_r\|^{\mathbf{k}}_0 < \delta_0$ for $r = 0, \dots, r_0$, then
\begin{align*}
\|z_{r+1}\|'^{\mathbf{k}}_j &\leq \lambda_{\mu,j}\bigl(e^{-\mu t_{r+1}}\|y\|'^{\mathbf{m}_\mu}_j + e^{7\rho t_r}\bigl(\|y\|'^{\mathbf{k}}_j\bigl(\|z_r\|'^{\mathbf{k}}_0\bigr)^2 + \|z_r\|'^{\mathbf{k}}_0\|z_r\|'^{\mathbf{k}}_j\bigr)\bigr)
\end{align*}
for all $r = 0, \dots, r_0+1$, where $\rho \definedas 2(p+k_0) + 1$ as before and $\mathbf{m}_\mu \definedas (p + 2k_0 + 3\rho + 2\mu + j)_{j\in\N_0}$.
\end{lemma}
\begin{proof}
As in the proof of \cref{Proposition_Nash_Moser_injectivity},
\begin{align*}
\tilde{f}(x_{r+1}) &= \tilde{f}(x_r) + D\tilde{f}(x_r)\Delta x_r + \int_{0}^1 (1-t)D^2\tilde{f}(x_r + t\Delta x_r)(\Delta x_r, \Delta x_r)\,\d t\text{,}
\intertext{hence}
z_{r+1} &= y - \tilde{f}(x_{r+1}) \\
&= \underbrace{y - \tilde{f}(x_r)}_{=\; z_r} - D\tilde{f}(x_r)\underbrace{\Delta x_r}_{\mathclap{=\; S_{t_r}\tilde{\psi}(x_r)z_r}} - \int_{0}^1 (1-t)D^2\tilde{f}(x_r + t\Delta x_r)(\Delta x_r, \Delta x_r)\,\d t\text{,}
\end{align*}
so
\begin{equation}\label{Eqn_zrplus1}
z_{r+1} = D\tilde{f}(x_r)(\id_{E_\infty} - S_{t_r})\tilde{\psi}(x_r)z_r - \int_{0}^1 (1-t)D^2\tilde{f}(x_r + t\Delta x_r)(\Delta x_r, \Delta x_r)\,\d t\text{.}
\end{equation}
Now repeatedly using \cref{Lemma_Modified_tame_estimates}, applying the estimates for the smoothing operators $S_{t_r}$ in the form provided by \cref{Eqn_sm_op_shifted} and using the estimates $\|z_r\|_j \leq \|y\|_j + a_j\|x_r\|_j$ as in the proof of \cref{Lemma_NM_surjectivity_1}, one can estimate that
\begin{align*}
\alpha_j(x_r,z_r) &\definedas \bigl\|D\tilde{f}(x_r)(\id_{E_\infty} - S_{t_r})\tilde{\psi}(x_r)z_r\|'_j \\
&\leq b_j\bigl(\|x_r\|_j\|(\id_{E_\infty} - S_{t_r})\tilde{\psi}(x_r)z_r\|_0 \;+ \\
&\qquad\quad +\; \|(\id_{E_\infty} - S_{t_r})\tilde{\psi}(x_r)z_r\|_j\bigr)
\end{align*}
and for any constant $\alpha \in \N_0$,
\begin{align*}
\|(\id_{E_\infty} - S_{t_r})\tilde{\psi}(x_r)z_r\|_j &\leq \tilde{C}^{\alpha+j}_{j}e^{-\alpha t_r}\|\tilde{\psi}(x_r)z_r\|^{\mathbf{p}}_{\alpha+j} \\
&\leq \tilde{C}^{\alpha+j}_{j}e^{-\alpha t_r}d_{p+\alpha+j}(\|x_r\|^{\mathbf{p}\circ\mathbf{k}}_{\alpha+j}\|z_r\|'^{\mathbf{k}}_0 + \|z_r\|'^{\mathbf{p}\circ\mathbf{k}}_{\alpha+j}) \\
&\leq \tilde{C}^{\alpha+j}_{j}e^{-\alpha t_r}d_{p+\alpha+j}\bigl(\|x_r\|^{\mathbf{p}\circ\mathbf{k}}_{\alpha+j}(\|y\|'^{\mathbf{k}}_0 + a_{k_0}\|x_r\|^{\mathbf{k}}_0) \;+ \\
&\qquad\qquad\qquad\qquad\quad +\; \|y\|'^{\mathbf{p}\circ\mathbf{k}}_{\alpha+j} + a_{p+k_0+\alpha+j}\|x_r\|^{\mathbf{p}\circ\mathbf{k}}_{\alpha+j}\bigr) \\
&\leq e^{-\alpha t_r}\rho_{\alpha,j}\bigl(1 + \|x_r\|^{\mathbf{k}}_0 + \|y\|'^{\mathbf{k}}_0\bigr)\bigl(\|x_r\|^{\mathbf{p}\circ\mathbf{k}}_{\alpha+j} + \|y\|'^{\mathbf{p}\circ\mathbf{k}}_{\alpha+j}\bigr)\text{,}
\end{align*}
where $\rho_{\alpha,j} \definedas \tilde{C}^{\alpha+j}_{j}d_i(1 + a_{k_0} + a_{p+k_0+\alpha+j})$.
Combining this with \cref{Lemma_NM_surjectivity_1} gives
\begin{align*}
\|(\id_{E_\infty} - S_{t_r})\tilde{\psi}(x_r)z_r\|_j &\leq e^{-\alpha t_r}\rho_{\alpha,j}\bigl(1 + \gamma_0e^{\rho t_r}\|y\|'^{\mathbf{k}}_0 + \|y\|'^{\mathbf{k}}_0\bigr)\,\cdot \\
&\quad\; \cdot\, \bigl(\gamma_{p+\alpha+j}e^{\rho t_r}\|y\|'^{\mathbf{p}\circ\mathbf{k}}_{\alpha+j} + \|y\|'^{\mathbf{p}\circ\mathbf{k}}_{\alpha+j}\bigr) \\
&\leq e^{-\alpha t_r}\rho_{\alpha,j}\bigl(1 + \|y\|'^{\mathbf{k}}_0\bigr)\,\cdot \\
&\quad\; \cdot\, 4\gamma_0\gamma_{p+\alpha+j}e^{2\rho t_r}\|y\|'^{\mathbf{p}\circ\mathbf{k}}_{\alpha+j}
\end{align*}
and hence furthermore
\begin{align*}
\|x_r\|_j\|(\id_{E_\infty} - S_{t_r})\tilde{\psi}(x_r)z_r\|_0 &\leq e^{-\alpha t_r}\rho_{\alpha,0}\bigl(1 + \|y\|'^{\mathbf{k}}_0\bigr)\,\cdot \\
&\quad\; \cdot\, 4\gamma_0\gamma_{p+\alpha}\gamma_je^{3\rho t_r}\|y\|'^{\mathbf{p}\circ\mathbf{k}}_{\alpha}\|y\|'_j\text{,}
\end{align*}
provided that $j \geq k_0$.
Replacing $j$ by $j + k_0$ and combining the above gives
\begin{align*}
\alpha_{k_0 + j}(x_r,z_r) &\leq e^{-\alpha t_r}\bigl(1 + \|y\|'^{\mathbf{k}}_0\bigr)^2 \tau_{\alpha,j}e^{3\rho t_r}\|y\|'^{\mathbf{p}\circ\mathbf{k}\circ\mathbf{k}}_{\alpha+j} \\
&= e^{-(3\rho+2\mu)t_r}\bigl(1 + \|y\|'^{\mathbf{k}}_0\bigr)^2\tau_{3\rho+2\mu,j}e^{3\rho t_r}\|y\|'^{\mathbf{p}\circ\mathbf{k}\circ\mathbf{k}}_{3\rho+2\mu+j} \\
&\leq e^{-\mu t_{r+1}}\bigl(1 + \|y\|'^{\mathbf{k}}_0\bigr)^2 \tau_{3\rho+2\mu,j}\|y\|'^{\mathbf{p}\circ\mathbf{k}\circ\mathbf{k}}_{3\rho+2\mu+j} \\
&= e^{-\mu t_{r+1}}\bigl(1 + \|y\|'^{\mathbf{k}}_0\bigr)^2 \overline{\tau}_{\mu,j}\|y\|'^{\mathbf{m}_\mu}_{j}\text{.}
\end{align*}
for any $\mu\in\N_0$, where
\begin{align*}
\tau_{\alpha,j} &\definedas 4b_{k_0+j}(\rho_{\alpha,0} + \rho_{\alpha,k_0+j})(\gamma_0\gamma_{p+k_0+\alpha+j} + \gamma_0\gamma_{p+\alpha}\gamma_{k_0+j}) \\
\overline{\tau}_{\mu,j} &\definedas \tau_{3\rho+2\mu,j} \\
\mathbf{m}_{\mu} &\definedas  (p+2k_0+3\rho+2\mu+j)_{j\in\N_0}\text{.}
\end{align*}
Now the second term in \cref{Eqn_zrplus1} can be estimated as follows, using \cref{Lemma_Modified_tame_estimates} as before:
\begin{align*}
\beta_j(x_r,z_r) &\definedas \left\|\int_{0}^1 (1-t)D^2\tilde{f}(x_r + t\Delta x_r)(\Delta x_r, \Delta x_r)\,\d t\right\|_j \\
&\leq \int_0^1 \bigl\|D^2\tilde{f}(x_r + t\Delta x_r)(\Delta x_r, \Delta x_r)\bigr\|_j \,\d t \\
&\leq c_j\int_0^1\bigl(\|x_r + t\Delta x_r\|_j\bigl(\|\Delta x_r\|_0\bigr)^2 + 2\|\Delta x_r\|_j\|\Delta x_r\|_0 \bigr)\,\d t \\
&\leq c_j\int_0^1\bigl(\bigl(\|x_r\|_j + t\|\Delta x_r\|_j\bigr)\bigl(\|\Delta x_r\|_0\bigr)^2 + 2\|\Delta x_r\|_j\|\Delta x_r\|_0 \bigr)\,\d t \\
&\leq 2c_j\bigl(\|x_r\|_j\bigl(\|\Delta x_r\|_0\bigr)^2 + \|\Delta x_r\|_j\|\Delta x_r\|_0\bigl(1 + \|\Delta x_r\|_0\bigr)\bigr)\text{.}
\end{align*}
Using \cref{Lemma_NM_surjectivity_1} and the estimate $\|z_r\|_j \leq \|y\|_j + a_j\|x_r\|_j$, as before, one can furthermore estimate
\begin{align*}
\|\Delta x_r\|^{\mathbf{k}}_0 &\leq \bar{\gamma}_0e^{2\rho t_r}\bigl(1 + \|y\|'^{\mathbf{k}}_0\bigr)\|z_r\|'^{\mathbf{k}}_0 \\
&\leq \bar{\gamma}_0e^{2\rho t_r}\bigl(1 + \|y\|'^{\mathbf{k}}_0\bigr)\bigl(\|y\|'^{\mathbf{k}}_0 + a_{k_0}\|x_r\|^{\mathbf{k}}_0\bigr) \\
&\leq \bar{\gamma}_0e^{2\rho t_r}\bigl(1 + \|y\|'^{\mathbf{k}}_0\bigr)\bigl(\|y\|'^{\mathbf{k}}_0 + a_{k_0}\gamma_0e^{\rho t_r}\|y\|'^{\mathbf{k}}_0\bigr) \\
&\leq \bar{\gamma}_0(1+a_{k_0})e^{3\rho t_r}\bigl(1 + \|y\|'^{\mathbf{k}}_0\bigr)^2
\end{align*}
and consequently also
\[
1 + \|\Delta x_r\|^{\mathbf{k}}_0 \leq 2\bar{\gamma}_0(1+a_{k_0})e^{3\rho t_r}\bigl(1 + \|y\|'^{\mathbf{k}}_0\bigr)^2\text{.}
\]
Inserting this in the inequality above and using \cref{Lemma_NM_surjectivity_1} gives
\begin{align*}
\beta_{k_0+j}(x_r,z_r) &\leq 2c_{j+k_0}\bigl(\|x_r\|_{j+k_0}\bigl(\|\Delta x_r\|_{0}\bigr)^2 + \|\Delta x_r\|_{j+k_0}\|\Delta x_r\|_{0}\bigl(1 + \|\Delta x_r\|_{0}\bigr)\bigr) \\
&\leq 2c_{j+k_0}\bigl(\|x_r\|^{\mathbf{k}}_{j}\bigl(\|\Delta x_r\|^{\mathbf{k}}_{0}\bigr)^2 + \|\Delta x_r\|^{\mathbf{k}}_{j}\|\Delta x_r\|^{\mathbf{k}}_{0} \bigl(1 + \|\Delta x_r\|^{\mathbf{k}}_{0}\bigr)\bigr) \\
&\leq 2c_{j+k_0}\bigl(\gamma_je^{\rho t_r}\|y\|'^{\mathbf{k}}_j\bigl(\bar{\gamma}_0e^{2\rho t_r}\bigl(1 + \|y\|'^{\mathbf{k}}_0\bigr)\|z_r\|'^{\mathbf{k}}_0\bigr)^2 \;+ \\
&\qquad\qquad\quad +\; \bar{\gamma}_je^{2\rho t_r}\bigl(\|y\|'^{\mathbf{k}}_j\|z_r\|'^{\mathbf{k}}_0 + \|z_r\|'^{\mathbf{k}}_j\bigr) \,\cdot \\
&\qquad\qquad\qquad \cdot\, \bar{\gamma}_0e^{2\rho t_r}\bigl(1 + \|y\|'^{\mathbf{k}}_0\bigr)\|z_r\|'^{\mathbf{k}}_0 \,\cdot \\
&\qquad\qquad\qquad \cdot\, 2\bar{\gamma}_0(1+a_{k_0})e^{3\rho t_r}\bigl(1 + \|y\|'^{\mathbf{k}}_0\bigr)^2\bigr) \\
&\leq \tilde{\tau}_je^{7\rho t_r}\bigl(1 + \|y\|'^{\mathbf{k}}_0\bigr)^3\bigl(\|y\|'^{\mathbf{k}}_j\bigl(\|z_r\|'^{\mathbf{k}}_0\bigr)^2 + \|z_r\|'^{\mathbf{k}}_0\|z_r\|'^{\mathbf{k}}_j\bigr)\text{,}
\end{align*}
where $\tilde{\tau}_j \definedas 2c_{j+k_0}\bar{\gamma}_0^2(\gamma_j + 2\bar{\gamma}_j(1+a_{k_0}))$.
Finally one can now combine the above estimates for $\alpha_{k_0+j}(x_r, z_r)$ and $\beta_{k_0+j}(x_r, z_r)$ to obtain
\begin{align*}
\|z_{r+1}\|^{\mathbf{k}}_j &= \|z_{r+1}\|'_{k_0+j} \\
&\leq \alpha_{k_0+j}(x_r,z_r) + \beta_{k_0+j}(x_r+z_r) \\
&\leq (\overline{\tau}_{\mu,j} + \tilde{\tau}_j)\bigl(1 + \|y\|'^{\mathbf{k}}_0\bigr)^3 \,\cdot \\
&\quad\; \cdot\, \bigl(e^{-\mu t_{r+1}}\|y\|'^{\mathbf{m}_\mu}_j + e^{7\rho t_r}\bigl(\|y\|'^{\mathbf{k}}_j\bigl(\|z_r\|'^{\mathbf{k}}_0\bigr)^2 + \|z_r\|'^{\mathbf{k}}_0\|z_r\|'^{\mathbf{k}}_j\bigr)\bigr)\text{.}
\end{align*}
This finishes the proof by defining $\lambda_{\mu,j} \definedas (\overline{\tau}_{\mu,j} + \tilde{\tau}_j)(1 + \delta_0)^3$.
\end{proof}

\begin{proposition}\label{Proposition_Nash_Moser_surjectivity}
In the setting as above for $\mu \in \N_0$ large enough (\eg $\mu \geq 16\rho$ where $\rho \definedas 2(p+k_0)+1$) there exists a shift $\mathbf{m} \subseteq \N_0$ (\eg $\mathbf{m} \definedas (p+2k_0+3\rho+2\mu+j)_{j\in\N_0}$) \st for $\delta_0 > 0$ small enough (depending on $\mu$) there exist constants $\tau_j, \tau'_j, \tau''_j \in [1,\infty)$ (depending on $\delta_0$), for $j\in\N_0$, \st if $\|y\|'^{\mathbf{m}}_0 < \delta_0$ then $x_r$ is well defined for all $r \in \N_0$ and
\begin{align*}
\|x_r\|^{\mathbf{k}}_j &\leq \tau_j\|y\|'^{\mathbf{m}}_j \\
\|z_r\|'^{\mathbf{k}}_j &\leq \tau'_j e^{-\frac{\mu}{4}t_r}\|y\|'^{\mathbf{m}}_j \\
\|\Delta x_r\|^{\mathbf{k}}_j &\leq \tau''_j e^{-\frac{\mu}{8}t_r}\|y\|'^{\mathbf{m}}_j\text{.}
\end{align*}
\end{proposition}
\begin{proof}
\begin{claim}
Given any $\mu \geq 14\rho$, there exist constants $\tau'_0, \tau''_0 > 0$ (arbitarily large, depending on $\mu$) and $\delta_0 > 0$ arbitrarily small \st if $\|y\|'^{\mathbf{m}_\mu}_0 < \delta_0$ and for some $r_0 \in \N_0$ also $\|x_r\|^{\mathbf{k}}_0 < \delta_0$ for $r = 0, \dots, r_0$, then
\begin{align*}
\|z_r\|'^{\mathbf{k}}_0 &\leq \tau'_0 e^{-\mu t_r}\|y\|'^{\mathbf{m}_\mu}_0 \\
\|\Delta x_r\|^{\mathbf{k}}_0 &\leq \tau''_0e^{-\frac{\mu}{2} t_r}\|y\|'^{\mathbf{m}_\mu}_0
\end{align*}
for all $r = 0, \dots, r_0+1$.
\end{claim}
\begin{proof}
Let $\delta_0 > 0$ be \st \cref{Lemma_NM_surjectivity_2} applies.
One obtains
\begin{align*}
\|z_{r+1}\|'^{\mathbf{k}}_0 &\leq \lambda_{\mu,0}\bigl(e^{-\mu t_{r+1}}\|y\|'^{\mathbf{m}_\mu}_0 + e^{7\rho t_r}\bigl(1 + \|y\|'^{\mathbf{k}}_0\bigr)\bigl(\|z_r\|'^{\mathbf{k}}_0\bigr)^2\bigr) \\
&\leq \lambda_{\mu,0}(1 + \delta_0)e^{-\mu t_{r+1}}\bigl(\|y\|'^{\mathbf{m}_\mu}_0 + e^{(7\rho - \mu/2)t_r}e^{2t_r}\bigl(\|z_r\|'^{\mathbf{k}}_0\bigr)^2\bigr)\text{,}
\end{align*}
using $t_{r+1} = (3/2)t_r$, provided that the assumptions of \cref{Lemma_NM_surjectivity_2} hold.
Now assume that $\|z_r\|'^{\mathbf{k}}_0 \leq \tau'_0 e^{-\mu t_r}\|y\|'^{\mathbf{m}_\mu}_0$, for some $\tau'_0 > 0$ and that $\mu \geq 14\rho$, \ie $7\rho - \mu/2 \leq 0$.
Then by the above
\begin{align*}
\|z_{r+1}\|'^{\mathbf{k}}_0 &\leq \lambda_{\mu,0}(1 + \delta_0)e^{-\mu t_{r+1}}\bigl(\|y\|'^{\mathbf{m}_\mu}_0 + e^{(7\rho - \mu/2)t_r}e^{2t_r}\bigl(\|z_r\|'^{\mathbf{k}}_0\bigr)^2\bigr) \\
&\leq \lambda_{\mu,0}(1 + \delta_0)e^{-\mu t_{r+1}}\bigl(\|y\|'^{\mathbf{m}_\mu}_0 + e^{(7\rho - \mu/2)t_r}\tau'^2_0\bigl(\|y\|'^{\mathbf{m}_\mu}_0\bigr)^2\bigr) \\
&= \lambda_{\mu,0}(1 + \delta_0)(1 + \tau'^2_0\delta_0)e^{-\mu t_{r+1}}\|y\|'^{\mathbf{m}_\mu}_0 \\
& \leq \tau'_0 e^{-\mu t_{r+1}}\|y\|'^{\mathbf{m}_\mu}_0\text{,}
\end{align*}
provided that $\lambda_{\mu,0}(1+\delta_0)(1 + \tau'^2_0\delta_0) \leq \tau'_0$.
Now if $\delta_0 \leq 1$ then $\lambda_{\mu,0}(1+\delta_0)(1+\tau'^2_0\delta_0) \leq 2\lambda_{\mu,0}(1+\tau'^2_0\delta_0)$.
So $\lambda_{\mu,0}(1+\delta_0)(1 + \tau'^2_0\delta_0) \leq \tau'_0$ provided that $\tau'^2_0\delta_0 \leq \frac{\tau'_0}{2\lambda_{\mu,0}} - 1$, and assuming that $\tau'_0 \geq 4\lambda_{\mu,0}$ this holds if $\delta_0 \leq  \frac{1}{\tau'^2_0}$.
Hence by induction, given $\mu \geq 14\rho$, for $\tau'_0 \geq 4\lambda_{\mu,0}$ and $\delta_0 > 0$ small enough \st \cref{Lemma_NM_surjectivity_2} applies, $\delta_0 \leq 1$ and $\delta_0 \leq \frac{1}{\tau'^2_0}$, the first part of the claim, the estimate on $\|z_r\|'^{\mathbf{k}}_0$, follows. \\
For the second part combine this with \cref{Lemma_NM_surjectivity_1} to see that
\begin{align*}
\|\Delta x_r\|^{\mathbf{k}}_0 &\leq \overline{\gamma}_0e^{2\rho t_r}\bigl(1 + \|y\|'^{\mathbf{k}}_0\bigr)\|z_r\|'^{\mathbf{k}}_0 \\
&\leq \overline{\gamma}_0(1 + \delta_0)\tau'_0e^{-(\mu - 2\rho)t_r}\|y\|'^{\mathbf{m}_\mu}
\end{align*}
from which the claim follows with $\tau''_0 \definedas 2\overline{\gamma}_0\tau'_0$, since $\mu - 2\rho \geq \mu - \mu/7 \geq \mu/2$ provided that $\mu \geq 14\rho$.
\end{proof}
\begin{claim}
Given any $\mu \geq 14\rho$ for $\delta_0 > 0$ small enough (depending on $\mu$) there exists a constant $\tau_0 > 0$ (arbitrarily large, depending on $\mu$) \st if $\|y\|'^{\mathbf{m}_\mu}_0 < \delta_0$, then $x_r$ is well defined for all $r \in \N_0$ and
\[
\|x_r\|^{\mathbf{k}}_0 \leq \tau_0\|y\|'^{\mathbf{m}_\mu}_0\text{.}
\]
\end{claim}
\begin{proof}
Let $\delta_0 > 0$ and $\tau''_0 > 0$ be \st the previous claim applies.
I will show by induction that the claim holds for $\tau_0 \definedas 2\tau''_0$, after replacing $\delta_0$ by $\delta'_0 \definedas \frac{\delta_0}{1 + 2\tau''_0} > 0$.
So let $\|y\|'^{\mathbf{m}_\mu}_0 < \delta'_0 \leq \delta_0$.
By definition $x_0 = 0$, so the inequality in the claim holds for $r = 0$.
Now assume that $x_{r-1}$ is well defined, lies in $V'$ and that the inequality holds for $x_{s}$, $s = 0, \dots, r-1$.
By definition $x_r = \sum_{s=0}^{r-1}\Delta x_s$, and by the previous claim hence
\begin{align*}
\|x_r\|^{\mathbf{k}}_0 &\leq \sum_{s=0}^{r-1}\|\Delta x_s\|^{\mathbf{k}}_0 \\
&\leq \sum_{s=0}^{r-1} \tau''_0e^{-\frac{\mu}{2}t_r}\|y\|^{\mathbf{m}_\mu}_0 \\
&\leq \tau''_0\sum_{s=0}^\infty 2^{-s} \|y\|^{\mathbf{m}_\mu}_0 \\
&= 2\tau''_0\|y\|^{\mathbf{m}_\mu}_0 \\
&= \tau_0\|y\|^{\mathbf{m}_\mu}_0\text{,}
\end{align*}
since $t_r = (3/2)^r \geq r$ and $\mu/2 \geq 7\rho = 7(2(p+k_0)+1) \geq 1$.
In particular, $\|x_r\|^{\mathbf{k}}_0 \leq 2\tau''_0\delta'_0  = \frac{2\tau''_0}{1+2\tau''_0}\delta_0 < \delta_0$.
Since for $\delta_0$ small enough $\|x_r\|^{\mathbf{k}}_0 < \delta_0$ implies $x_r \in V'$, $x_r$ is well defined, lies in $V'$ and satisfies the inequality in the claim.
\end{proof}
\begin{claim}
Let $\mu \geq 16\rho$ and let $\delta_0 > 0$ be small enough \st the previous two claims and \cref{Lemma_NM_surjectivity_1,Lemma_NM_surjectivity_2} hold.
Then \cref{Proposition_Nash_Moser_surjectivity} holds as well.
\end{claim}
\begin{proof}
$\tau_0, \tau'_0, \tau''_0 > 0$ \st the inequalities in \cref{Proposition_Nash_Moser_surjectivity} hold were already defined in the previous two claims.
Using the first claim and \cref{Lemma_NM_surjectivity_1,Lemma_NM_surjectivity_2}, for general $j \in \N_0$ and $r\in\N_0$,
\begin{align*}
\|z_{r+1}\|'^{\mathbf{k}}_j &\leq \lambda_{\mu,j}\bigl(e^{-\mu t_{r+1}}\|y\|'^{\mathbf{m}_\mu}_j + e^{7\rho t_r}\tau'_0e^{-\mu t_r}\|y\|'^{\mathbf{m}_\mu}_0 \,\cdot \\
&\qquad\qquad\qquad \cdot\, \bigl(\|y\|'^{\mathbf{k}}_j\tau'_0e^{-\mu t_r}\|y\|'^{\mathbf{m}_\mu}_0 + \hat{\gamma}_je^{\rho t_r}\|y\|'^{\mathbf{k}}_j\bigr)\bigr) \\
&\leq \lambda_{\mu,j}(1 + \tau'_0(\tau'_0 + \hat{\gamma}_j))e^{8\rho t_r}\bigl(e^{-\mu t_{r+1}} + e^{-\mu t_r}\|y\|'^{\mathbf{m}_\mu}_0\bigl(1 + \|y\|'^{\mathbf{m}_\mu}_0\bigr)\bigr)\|y\|'^{\mathbf{m}_\mu}_j \\
&\leq \lambda_{\mu,j}(1 + \tau'_0(\tau'_0 + \hat{\gamma}_j))\bigl(1 + \|y\|'^{\mathbf{m}_\mu}_0\bigr)^2 e^{8\rho \frac{2}{3} t_{r+1}}(e^{-\mu t_{r+1}} + e^{-\mu \frac{2}{3}t_{r+1}})\|y\|'^{\mathbf{m}_\mu}_j \\
&\leq 2\lambda_{\mu,j}(1 + \tau'_0(\tau'_0 + \hat{\gamma}_j))\bigl(1 + \|y\|'^{\mathbf{m}_\mu}_0\bigr)^2 e^{-\frac{2}{3}(\mu-8\rho)t_{r+1}}\|y\|'^{\mathbf{m}_\mu}_j\text{.}
\end{align*}
For $\mu \geq 16 \rho$ one has $\frac{2}{3}(\mu - 8\rho) \geq \frac{1}{2}(\mu - \frac{1}{2}\mu) = \frac{\mu}{4}$.
So
\[
\|z_r\|^{\mathbf{k}}_j \leq \tau'_je^{-\frac{\mu}{4}}\|y\|'^{\mathbf{m}_\mu}_j
\]
with $\tau'_j \definedas 2\lambda_{\mu,j}(1 + \tau'_0(\tau'_0 + \hat{\gamma}_j))(1 + \delta_0)^2$.
Combining this with \cref{Lemma_NM_surjectivity_1} yields
\begin{align*}
\|\Delta x_r\|^{\mathbf{k}}_j &\leq \bar{\gamma}_je^{2\rho t_r}\bigl(\|y\|'^{\mathbf{k}}_j\|z_r\|'^{\mathbf{k}}_0 + \|z_r\|'^{\mathbf{k}}_j\bigr) \\
&\leq \overline{\gamma}_je^{2\rho t_r}\bigl(\|y\|'^{\mathbf{m}_\mu}_j\tau'_0e^{-\frac{\mu}{4}t_r}\|y\|^{\mathbf{m}_\mu}_0 + \tau'_je^{-\frac{\mu}{4}t_r}\|y\|'^{\mathbf{m}_\mu}_j\bigr) \\
&\leq \overline{\gamma}_j\bigl(\tau'_0\|y\|^{\mathbf{m}_\mu}_0 + \tau'_j\bigr)e^{-\left(\frac{\mu}{4} - 2\rho\right)t_r}\|y\|'^{\mathbf{m}_\mu}_j\text{.}
\end{align*}
For $\mu \geq 16\rho$ one has $\frac{\mu}{4} - 2\rho \geq \frac{\mu}{4} - \frac{\mu}{8} = \frac{\mu}{8}$.
So
\[
\|\Delta x_r\|^{\mathbf{k}}_j \leq \tau''_je^{-\frac{\mu}{8}t_r}\|y\|'^{\mathbf{m}_\mu}_j
\]
with $\tau''_j \definedas \overline{\gamma}_j(\tau'_0\delta_0 + \tau'_j)$.
Finally, from this and $x_r = \sum_{s=0}^{r-1} \Delta x_s$ it follows that
\begin{align*}
\|x_r\|^{\mathbf{k}}_j &\leq \sum_{s=0}^{r-1} \|\Delta x_s\|^{\mathbf{k}}_j \\
&\leq \tau''_j\sum_{s=0}^\infty e^{-\frac{\mu}{8}t_s}\|y\|'^{\mathbf{m}_\mu}_j \\
&\leq 2\tau''_j\|y\|'^{\mathbf{m}_\mu}_j\text{,}
\end{align*}
since $\mu \geq 16\rho \geq 16$, $t_s \geq s$ and hence $e^{-\frac{\mu}{8}t_s} \leq 2^{-s}$.
So
\[
\|x_r\|^{\mathbf{k}}_j \leq \tau_j\|y\|'^{\mathbf{m}_\mu}_j
\]
for $\tau_j \definedas 2\tau''_j$.
\end{proof}
\end{proof}

Combining \cref{Proposition_Nash_Moser_injectivity,Proposition_Nash_Moser_surjectivity} now easily finishes the proof of \cref{Theorem_Nash_Moser}: \\
By \cref{Corollary_Nash_Moser_injectivity}, there exists a neighbourhood $W \subseteq E_\infty$ of $0$ \st $\tilde{f}|_W : W \to E'_\infty$ is injective.
From the modified Newton's procedure, in the way described before and because of \cref{Proposition_Nash_Moser_surjectivity}, there exists a neighbourhood $W' \subseteq E'_\infty$ and a map $g : W' \to W$ as well as a shift $\mathbf{m} \subseteq \N_0$ \st $g$ satisfies the tameness conditions
\[
\|g(y)\|^{\mathbf{k}}_j \leq \tau_j \|y\|'^{\mathbf{m}}_j
\]
for some constants $\tau_j \in [0,\infty)$ and all $j \in \N_0$.
Also, w.\,l.\,o.\,g.~one can assume that $\mathbf{m} \geq \mathbf{k}$.
Replacing $W$ by $\tilde{f}\inv(W')$, because $\tilde{f}|_{W}$ is injective one can assume that $g\circ \tilde{f}|_{W} = \id_{W}$ and $\tilde{f}|_{W}\circ g = \id_{W'}$.

Now let $y_1, y_2 \in W'$.
By \cref{Proposition_Nash_Moser_injectivity} and the tameness conditions for $g$ above, applied to $y = g(y_2)$ and $x = g(y_1)$ one has for $j \in \N_0$
\begin{align*}
\|g(y_2) - g(y_1)\|_j &\leq c'_j\bigl(\bigl(\|g(y_2)\|^{\mathbf{k}}_j + \|g(y_1)\|^{\mathbf{k}}_j\bigr) \|y_2 - y_1\|'^{\mathbf{k}}_0 + \|y_2 - y_1\|'^{\mathbf{k}}_j\bigr) \\
&\leq c'_j\bigl(\tau_j\bigl(\|y_2\|'^{\mathbf{m}}_j + \|y_1\|'^{\mathbf{m}}_j\bigr)\|y_2-y_1\|'^{\mathbf{k}}_0 + \|y_2 - y_1\|'^{\mathbf{k}}_j\bigr) \\
&\leq c'_j(\tau_j+1)\bigl(\bigl(\|y_2\|'^{\mathbf{m}}_j + \|y_1\|'^{\mathbf{m}}_j\bigr)\|y_2-y_1\|'^{\mathbf{m}}_0 + \|y_2-y_1\|'^{\mathbf{m}}_j\bigr)\text{.}
\end{align*}
From this it is immediate that $g : (E'_\infty, \|\cdot\|'^{\mathbf{m}}_j) \supseteq W' \to W \subseteq (E_\infty, \|\cdot\|_j)$ is locally Lipschitz continuous for each $j \in \N_0$.

Now let $y' \in W'$, and let $\tilde{W} \subseteq E'_\infty$ be a convex balanced neighbourhood of $0$ \st $y' + u \in W'$ for all $u \in \tilde{W}$.
Let
\begin{align*}
r^g_{y'} : \tilde{W}\times [0,1] &\to Y \\
(u,t) &\mapsto \begin{cases} \frac{1}{t}(g(y'+tu) - g(y')) - \tilde{\psi}(g(y'))u & t > 0 \\ 0 & t=0 \end{cases}\text{.}
\end{align*}
Given $t \in (0,1]$, again applying \cref{Proposition_Nash_Moser_injectivity}, this time to $x = g(y')$ and $y = g(y'+tu)$, gives for $j\in\N_0$
\begin{align*}
\|r^g_{y'}(u,t)\|_j &= \bigl\|\frac{1}{t}(g(y') - g(y'+tu)) - \tilde{\psi}(g(y'))u\bigr\|_j \\
&\leq \frac{1}{t}c''_j\bigl(1 + \|g(y')\|^{\mathbf{k}\circ\mathbf{k}}_0 + \|g(y'+tu)\|^{\mathbf{k}\circ\mathbf{k}}_0\bigr) \,\cdot \\
&\quad\; \cdot\, \bigl(\bigl(\|g(y')\|^{\mathbf{k}\circ\mathbf{k}}_j + \|g(y'+tu)\|^{\mathbf{k}\circ\mathbf{k}}_{j}\bigr)\bigl\|tu\bigr\|'^{\mathbf{k}\circ\mathbf{k}}_0 \;+ \\
&\quad\; +\; \|tu\|'^{\mathbf{k}\circ\mathbf{k}}_j\bigr)\|tu\|'^{\mathbf{k}}_0 \\
&\leq tc''_j\tau_{k_0}(\tau_{k_0+j}+1)\bigl(1 + \|y'\|'^{\mathbf{m}\circ\mathbf{k}}_0 + \|y'+tu\|'^{\mathbf{m}\circ\mathbf{k}}_0\bigr)\,\cdot \\
&\quad\; \cdot\, \bigl(\bigl(\bigl\|y'\|'^{\mathbf{m}\circ\mathbf{k}}_j + \|y'+tu\|'^{\mathbf{m}\circ\mathbf{k}}_j\bigr)\|u\|'^{\mathbf{k}\circ\mathbf{k}}_0 + \|u\|'^{\mathbf{k}\circ\mathbf{k}}_j\bigr)\|u\|'^{\mathbf{k}}_0 \\
&\leq tc''_j\tau_{k_0}(\tau_{k_0+j}+1)\bigl(1 + 2\|y'\|'^{\mathbf{m}\circ\mathbf{k}}_0 + \|u\|'^{\mathbf{m}\circ\mathbf{k}}_0\bigr) \,\cdot \\
&\quad\; \cdot\, \bigl(\bigl(2\|y'\|'^{\mathbf{m}\circ\mathbf{k}}_j + \|u\|'^{\mathbf{m}\circ\mathbf{k}}_j\bigr)\|u\|'^{\mathbf{m}\circ\mathbf{k}}_0 + \|u\|'^{\mathbf{m}\circ\mathbf{k}}_j\bigr)\|u\|'^{\mathbf{m}\circ\mathbf{k}}_0\text{.}
\end{align*}
From this one can read off that $g : (E'_\infty, \|\cdot\|'^{\mathbf{m}\circ\mathbf{k}}_j) \supseteq W' \to W \subseteq (E_\infty, \|\cdot\|_j)$ is weakly Fr{\'e}chet differentiable with derivative $Dg = g^\ast \tilde{\psi}$, for each $j \in \N_0$. \\
This finishes the proof of \cref{Theorem_Nash_Moser}.

\clearpage
\section{Nonlinear Fredholm maps}\label{Section_Fredholm_maps}

\Needspace{25\baselineskip}
\subsection{Strongly smoothing and Fredholm families of morphisms}

In this section, family of morphisms will always mean $\overline{\text{sc}}^0$ family of morphisms. \\
The words ``and analogously in the tame context'' will always mean that all topological vector spaces appearing are pre-tame $\overline{\text{sc}}$-Fr{\'e}chet spaces instead of $\overline{\text{sc}}$-Fr{\'e}chet spaces and all morphism, maps and rescalings are tame.

\begin{definition}\label{Definition_Strongly_smoothing_family}
Let $E$, $E^i$, $i=1,2$, be $\overline{\text{sc}}$-Fr{\'e}chet spaces, let $U \subseteq E$ be an open subset and let $\kappa : U\times E^1 \to E^2$ be a family of morphisms.
\begin{enumerate}[label=\arabic*.,ref=\arabic*.]
  \item Let $((\mathbb{E}\oplus \mathbb{E}^1, \psi\oplus \psi^1), (\mathbb{E}^2, \psi^2), \mathcal{K} : \mathcal{U}\oplus \mathbb{E}^1 \to \mathbb{E}^2)$ be an envelope of $\kappa$, where $\mathcal{K} = (\kappa_k : U_k\times E^1_k \to E^2_k)_{k\in\N_0}$. \\
$\mathcal{K}$ is called \emph{strongly smoothing} if for every strictly monotone increasing sequence $\mathbf{k} \subseteq \N_0$ there exists an envelope $\mathcal{K}_{\mathbf{k}} : \mathcal{U}^{\mathbf{k}} \oplus \mathbb{E}^1 \to (\mathbb{E}^2)^{\mathbf{k}}$ of $\kappa$ \st
\[
\mathcal{K}\circ (\mathbb{I}^{\mathbf{k}}\times \id_{\mathbb{E}^1}) = (\mathbb{I}^2)^{\mathbf{k}}\circ \mathcal{K}_{\mathbf{k}} : \mathcal{U}^{\mathbf{k}}\times \mathbb{E}^1 \to \mathbb{E}^2\text{.}
\]
  \item $\kappa$ is called a \emph{strongly smoothing} \iff for every point $x \in U$ there exists a neighbourhood $V \subseteq U$ of $x$ \st $\kappa|_{V\times E^1} : V\times E^1 \to E^2$ has a strongly smoothing envelope.
\end{enumerate}
And analogously in the tame context.
\end{definition}

\begin{lemma}\label{Lemma_Basic_properties_strongly_smoothing_envelopes}
Let $E$, $E^i$, $i=1,2$ be $\overline{\text{sc}}$-Fr{\'e}chet spaces, let $U \subseteq E$ be an open subset and let $\kappa : U\times E^1 \to E^2$ be a family of morphisms.
Let furthermore $((\mathbb{E}\oplus \mathbb{E}^1, \psi\oplus \psi^1), (\mathbb{E}^2, \psi^2), \mathcal{K} : \mathcal{U}\oplus \mathbb{E}^1 \to \mathbb{E}^2)$ be an envelope of $\kappa$, where $\mathcal{K} = (\kappa_k : U_k\times E^1_k \to E^2_k)_{k\in\N_0}$.
\begin{enumerate}[label=\arabic*.,ref=\arabic*.]
  \item\label{Lemma_Basic_properties_strongly_smoothing_envelopes_1} If $\mathcal{K}$ is strongly smoothing, then so is any refinement of $\mathcal{K}$.
  \item\label{Lemma_Basic_properties_strongly_smoothing_envelopes_2} Given a strictly monotone increasing sequence $\mathbf{k} \subseteq \N_0$, if $\mathcal{K}$ is strongly smoothing, then so is $\mathcal{K}^{\mathbf{k}} : (\mathcal{U}\oplus \mathbb{E}^1)^{\mathbf{k}} \to (\mathbb{E}^2)^{\mathbf{k}}$.
  \item\label{Lemma_Basic_properties_strongly_smoothing_envelopes_3} Given sc-chains $\tilde{\mathbb{E}}$, $\tilde{\mathbb{E}}^i$, $i=1,2$ and continuous linear operators $\mathbb{S} : \tilde{\mathbb{E}} \to \mathbb{E}$, $\mathbb{S}^1 : \tilde{\mathbb{E}}^1 \to \mathbb{E}^1$ and $\mathbb{T} : \mathbb{E}^2 \to \tilde{\mathbb{E}}^2$, if $\mathcal{K}$ is strongly smoothing, then so is $\mathbb{T}\circ \mathcal{K}\circ (\mathbb{S}\oplus \mathbb{S}^1) : \mathbb{S}^\ast \mathcal{U}\times \tilde{\mathbb{E}}^1 \to \tilde{\mathbb{E}}^2$.
  \item\label{Lemma_Basic_properties_strongly_smoothing_envelopes_4} If $\kappa$ is strongly smoothing, then for every point $x \in U$ there exists a neighbourhood $V \subseteq U$ of $x$ with the following property: For every envelope $\mathcal{K}$ of $\kappa|_{V\times E^1}$ of product form as above there exists a refinement $\mathcal{K}'$ of $\mathcal{K}$ and shifts $\mathbf{k}, \mathbf{l} \subseteq \N_0$ with $\mathbf{k} \geq \mathbf{l}$ \st $(\mathbb{I}^2)^{\mathbf{k}}_{\mathbf{l}}\circ \mathcal{K}'^{\mathbf{k}}$ is strongly smoothing.
\end{enumerate}
\end{lemma}
\begin{proof}
\begin{enumerate}[label=\arabic*.,ref=\arabic*.]
  \item Obvious.
  \item Define $(\mathcal{K}^{\mathbf{k}})_{\mathbf{l}} \definedas \mathcal{K}_{\mathbf{k}\circ \mathbf{l}}\circ (\id_{\mathbb{E}^{\mathbf{k}\circ\mathbf{l}}}\oplus (\mathbb{I}^1)^{\mathbf{k}})$.
  \item Define $(\mathbb{T}\circ \mathcal{K}\circ (\mathbb{S}\oplus \mathbb{S}^1))_{\mathbf{k}} \definedas \mathbb{T}^{\mathbf{k}}\circ \mathcal{K}_{\mathbf{k}}\circ (\mathbb{S}^{\mathbf{k}} \oplus \mathbb{S}^1)$.
  \item Let $V \subseteq U$ be a neighbourhood of $x$ as in \cref{Proposition_Strict_envelopes} and after possibly making $V$ smaller, assume that $\kappa|_{V\times E^1}$ has a strongly smoothing envelope $\tilde{\mathcal{K}} : \tilde{\mathcal{V}}\oplus \tilde{\mathbb{E}}^1 \to \tilde{\mathbb{E}}^2$.
Let $\mathcal{K} : \mathcal{V}\oplus \mathbb{E}^1 \to \mathbb{E}^2$ be any envelope of $\kappa|_{V\times E^1}$.
By shifting and going to a refinement, assume that $\mathcal{K}$ is strict (\cf \cref{Proposition_Strict_envelopes}).
By \cref{Definition_Envelope_in_sc_Frechet_space,Proposition_Strict_envelopes} there exist weak equivalences $\mathbb{J}\oplus \mathbb{J}^1 : \mathbb{E}\oplus \mathbb{E}^1 \to \tilde{\mathbb{E}} \oplus \tilde{\mathbb{E}}^1$ and $\mathbb{K}' : \tilde{\mathbb{E}}^2 \to \mathbb{E}^2$ \st $(\mathbb{I}^2)^{\mathbf{k}\circ \mathbf{l}}\circ \mathcal{K}^{\mathbf{k}\circ \mathbf{l}}$ (which is still strict by \cref{Lemma_Rescaling_of_envelopes}) and $\mathbb{K}'\circ \left(\mathbb{J}^\ast \tilde{\mathcal{K}}\right)^{\mathbf{l}}$ are weakly equivalent.
After rescaling $\mathcal{K}$ (after wich it is still strict by \cref{Lemma_Rescaling_of_envelopes}) and rescaling, as well as taking a refinement of, $\mathcal{K}$ (after which it is still strongly smoothing by \labelcref{Lemma_Basic_properties_strongly_smoothing_envelopes_1,Lemma_Basic_properties_strongly_smoothing_envelopes_2}), $(\mathbb{I}^2)^{\mathbf{k}\circ \mathbf{l}}\circ \mathcal{K}^{\mathbf{k}\circ \mathbf{l}}$ and $\mathbb{K}'\circ \left(\mathbb{J}^\ast \tilde{\mathcal{K}}\right)^{\mathbf{l}}$ can be assumed to be strict and equivalent, hence they coincide by \cref{Lemma_Envelopes}.
Since the latter is strongly smoothing by \labelcref{Lemma_Basic_properties_strongly_smoothing_envelopes_2,Lemma_Basic_properties_strongly_smoothing_envelopes_3}, so is the former.
\end{enumerate}
\end{proof}

\begin{example}\label{Example_Strongly_smoothing_family}
Let $E$, $E^i$, $i=1,2$, be $\overline{\text{sc}}$-Fr{\'e}chet spaces and let $U \subseteq E$ be an open subset.
\begin{enumerate}[label=\arabic*.,ref=\arabic*.]
  \item\label{Example_Strongly_smoothing_family_1} Given a strongly smoothing morphism $K \in \operatorname{Sm}(E^1, E^2)$, the constant family
\begin{align*}
U\times E^1 &\to E^2 \\
(x,u) &\mapsto K(u)
\end{align*}
is strongly smoothing.
  \item\label{Example_Strongly_smoothing_family_2} Given a family of morphisms $\kappa : U \times E^1 \to E^2$, if there exists a finite dimensional $\overline{\text{sc}}$-Fr{\'e}chet space $C$ and families of morphisms $\kappa_1 : U\times E^1 \to C$, $\kappa_2 : U\times C \to E^2$ \st $\kappa = \kappa_2\circ \kappa_1$, then $\kappa$ is strongly smoothing.
\end{enumerate}
\end{example}

\begin{remark}
Note that if $\kappa : U\times E^1 \to E^2$ is a strongly smoothing family of morphisms, then $\kappa(x) : E^1 \to E^2$ is strongly smoothing in the sense of \cref{Definition_Strongly_smoothing_morphism} for all $x \in U$.
\end{remark}

\begin{proposition}\label{Lemma_Basic_properties_strongly_smoothing_families}
Let $E$, $E'$ and $E^i$, $i=0,\dots,3$, be $\overline{\text{sc}}$-Fr{\'e}chet spaces, let $U \subseteq E$ and $U' \subseteq E'$ be open subsets together with an $\overline{\text{sc}}^0$ map $f : U' \to U$ and let
\begin{align*}
\kappa : U\times E^1 &\to E^2 & \phi : U \times E^0 &\to E^1 & \psi : U\times E^2 &\to E^3
\end{align*}
be families of morphisms.
\begin{enumerate}[label=\arabic*.,ref=\arabic*.]
  \item If $\kappa$ is strongly smoothing then so are $f^\ast \kappa$, $\kappa\circ \phi$ and $\psi\circ \kappa$.
  \item The map
\begin{align*}
U &\to \operatorname{Sm}(E^1, E^2) \\
x &\mapsto \kappa(x)
\end{align*}
is well defined and continuous.
\end{enumerate}
And analogously in the tame context.
\end{proposition}
\begin{proof}
\begin{enumerate}[label=\arabic*.,ref=\arabic*.]
  \item First, observe that $f^\ast \kappa = \kappa\circ (f\times \id_{E^1})$, $\kappa\circ \phi = \kappa\circ (\id_U, \phi)$ and $\psi\circ \kappa = \psi\circ (\id_U, \kappa)$, where $(\id_U, \phi) : U\times E^0 \to U\times E^1$, $(x,u) \mapsto (x, \phi(x,u))$, etc.
Given $x\in U$, $x' \in U'$, going through the proof of \cref{Proposition_Composition_sc_continuous}, \labelcref{Proposition_Composition_sc_continuous_3}, and using \cref{Lemma_Basic_properties_strongly_smoothing_envelopes}, one can assume that there are open neighbourhoods $V' \subseteq U'$ and $V \subseteq U$ of $x'$ and $x$, respectively, \st $f(V') \subseteq V$, together with strict envelopes $\mathcal{F} : \mathbb{E}' \supseteq \mathcal{V}' \to \mathbb{E}$ of $f|_{V'}$, $\mathcal{K} : \mathbb{E}\oplus \mathbb{E}^1 \supseteq \mathcal{V}\oplus \mathbb{E}^1 \to \mathbb{E}^2$ of $\kappa|_{V\times E^1}$, $\mathcal{G} : \mathbb{E}\oplus \mathbb{E}^0 \supseteq \mathcal{V}\oplus \mathbb{E}^0 \to \mathbb{E}^1$ of $\phi|_{V\times E^0}$ and $\mathcal{H} : \mathbb{E}\oplus \mathbb{E}^2 \supseteq \mathcal{V}\oplus \mathbb{E}^2 \to \mathbb{E}^3$ of $\psi|_{V\times E^2}$.
Furthermore, $\mathcal{K}$ can be assumed to be strongly smoothing. \\
Let $\mathbf{k} \subseteq \N_0$ be a strictly monotone increasing sequence. \\
Then an envelope of $f^\ast \kappa|_{V'\times E^1}$ is given by $\mathcal{F}^\ast \mathcal{K} \definedas \mathcal{K}\circ (\mathcal{F} \times \id_{\mathbb{E}^1})$ and one can define $(\mathcal{F}^\ast \mathcal{K})_{\mathbf{k}} \definedas \mathcal{K}_{\mathbf{k}}\circ (\mathcal{F}^{\mathbf{k}}\times \id_{\mathbb{E}^1})$. \\
An envelope of $\kappa\circ \phi|_{V\times E^0}$ is given by $\mathcal{K}\circ (\id_{\mathcal{V}}, \mathcal{G})$ and one can define $(\mathcal{K}\circ (\id_{\mathcal{V}}, \mathcal{G}))_{\mathbf{k}} \definedas \mathcal{K}_{\mathbf{k}}\circ (\id_{\mathcal{V}^{\mathbf{k}}}, \mathcal{G}\circ (\mathbb{I}^{\mathbf{k}}\times \id_{\mathbb{E}^0})$. \\
Similarly, an envelope of $\psi\circ \kappa|_{V\times E^1}$ is given by $\mathcal{H}\circ (\id_{\mathcal{V}}, \mathcal{K})$ and one can define $(\mathcal{H}\circ (\id_{\mathcal{V}}, \mathcal{K}))_{\mathbf{k}} \definedas \mathcal{H}^{\mathbf{k}}\circ (\id_{\mathcal{V}^{\mathbf{k}}}, \mathcal{K}_{\mathbf{k}})$.
  \item This is a local question, so one can assume that $\kappa$ has a strongly smoothing envelope $\mathcal{K} : \mathcal{U} \oplus \mathbb{E}^1 \to \mathbb{E}^2$ as in \cref{Definition_Strongly_smoothing_family}.
For any strictly monotone increasing sequence $\mathbf{k} \subseteq \N_0$ let $\mathcal{K}_{\mathbf{k}} : \mathcal{U}^{\mathbf{k}}\oplus \mathbb{E}^1 \to (\mathbb{E}^2)^{\mathbf{k}}$ be as in \cref{Definition_Strongly_smoothing_family}.
For $x \in U_\infty$ and $u \in E^1_0$ and any $k \in \N_0$, let $\mathbf{k} \definedas (k + j)_{j\in\N_0}$.
Then
\begin{align*}
\kappa_0\circ (\iota^\infty_0\times \id_{E^1_0})(x,u) &= \kappa_0\circ (\iota^{k}_0\times \id_{E^1_0})(\iota^\infty_{k}(x),u) \\
&= \left(\mathcal{K}\circ (\mathbb{I}^{\mathbf{k}}\times \id_{\mathbb{E}^1})\right)_0(\iota^\infty_{k}(x), u) \\
&= \left((\mathbb{I}^2)^{\mathbf{k}}\circ \mathcal{K}_{\mathbf{k}}\right)_0\circ (\iota^\infty_{k}\times \id_{E^1_0})(x,u) \\
&= (\iota^2)^{k}_0\circ \left((\mathcal{K}_{\mathbf{k}})_0\circ (\iota^\infty_{k}\times \id_{E^1_0})\right)(x,u)\text{.}
\end{align*}
Hence for every $k \in \N_0$, $\kappa_0\circ (\iota^\infty_0\times \id_{E^1_0}) : U_\infty \times E^1_0 \to E^2_0$ factors via a continuous map through $(\iota^2)^{k}_0 : E^2_{k} \to E^2_0$.
So there exists a continuous map $\overline{\kappa}_0 : U_\infty \times E^1_0 \to E^2_\infty$ \st $\kappa_0\circ (\iota^\infty_0\times \id_{E^1_0}) = (\iota^2)^\infty_0\circ \overline{\kappa}_0$.
By \cref{Lemma_Weakly_continuous_implies_strongly_continuous}, the map $U_\infty \to L_{\mathrm{c}}(E^1_1, E^2_\infty)$, $x \mapsto \overline{\kappa}_0\circ (\iota^1)^1_0$, is continuous, hence defines the desired continuous map $U \to \operatorname{Sm}(E^1, E^2)$.
\end{enumerate}
\end{proof}

\begin{proposition}\label{Proposition_Invertibility_perturbation_by_strongly_smoothing_family}
Let $E$ and $E^i$, $i = 1,2$, be $\overline{\text{sc}}$-Fr{\'e}chet spaces, let $U \subseteq E$ be an open subset, let $x_0 \in U$ and let $\phi : U\times E^1 \to E^2$ and $\kappa : U\times E^1 \to E^2$ be families of morphisms. \\
If $\phi$ is invertible and $\kappa$ is strongly smoothing with $\kappa(x_0) = 0$ then there exists a neighbourhood $V \subseteq U$ of $x_0$ and a strongly smoothing family of morphisms $\kappa' : V\times E^2 \to E^1$ with $\kappa'(x_0) = 0$ \st $(\phi + \kappa)|_V$ is invertible and $(\phi + \kappa)|_V\inv = \phi\inv|_V + \kappa'$. \\
And analogously in the tame context.
\end{proposition}
\begin{proof}
\begin{claim}
One can assume w.\,l.\,o.\,g.~that $E^1 = E^2 \defines E'$ and $\phi = \id_{E'}$.
\end{claim}
\begin{proof}
$(\phi + \kappa)\circ \phi\inv = \phi\circ\phi\inv + \kappa\circ \phi\inv = \id_{E^2} + \kappa\circ \phi\inv$ and $\kappa\circ \phi\inv$ is strongly smoothing by \cref{Lemma_Basic_properties_strongly_smoothing_families} and satisfies $(\kappa\circ \phi\inv)(x_0) = \kappa(x_0)\phi\inv(x_0) = 0$.
So assuming that the result holds for $\phi = \id_{E^2}$ there exists a neighbourhood $V \subseteq U$ of $x_0$ and $\tilde{\kappa}' : V\times E^2 \to E^2$ \st $\id_{E^2} + \kappa\circ \phi\inv|_V$ is invertible with inverse $(\id_{E^2} + \kappa\circ\phi\inv)|_V\inv = \id_{E^2} + \tilde{\kappa}'$.
Then $(\phi + \kappa)|_V\inv = \phi\inv|_V\circ (\id_{E^2} + \kappa\circ \phi\inv)|_V\inv = \phi\inv|_V\circ (\id_{E^2} + \tilde{\kappa}') = \phi\inv|_V + \phi\inv|_V\circ\tilde{\kappa}'$.
Define $\kappa' \definedas \phi\inv|_V\circ\tilde{\kappa}'$, which is strongly smoothing again by \cref{Lemma_Basic_properties_strongly_smoothing_families}.
\end{proof}
After applying this claim, an analogous argument to the one used in the proof of \cref{Lemma_Basic_properties_strongly_smoothing_families}, choosing envelopes $\mathcal{K} : \mathcal{U}\oplus \mathbb{E}' \to \mathbb{E}'$ and $\mathcal{K}_{\mathbf{k}} : \mathcal{U}^{\mathbf{k}}\oplus \mathbb{E}' \to \mathbb{E}'^{\mathbf{k}}$ for $\mathbf{k} \definedas (k+j)_{j\in\N_0}$ (which is tame) and $k\in\N_0$ arbitrary, shows that $\kappa$ defines and is defined by a sequence of continuous maps
\[
\kappa^k_j : U_k\times E'_j \to E'_k
\]
for $k\in\N_0$ and $j = 0, \dots, k$ which satisfy the obvious compatibility conditions under the inclusions in $\mathbb{E}$ and $\mathbb{E}'$.
W.\,l.\,o.\,g.~one can also assume that $\mathcal{U}$ is strict.
By \cref{Lemma_Weakly_continuous_implies_strongly_continuous}, for any $0 < j \leq k$, it follows from $\kappa^k_j = \kappa^k_{j-1}\circ (\id\times \iota'^j_{j-1})$ that the map
\begin{align*}
\tilde{\kappa}^k_j : U_k &\to L_{\mathrm{c}}(E'_j, E'_k) \\
x &\mapsto \kappa^k_j(x)
\end{align*}
is continuous.
Also, by assumption, $\tilde{\kappa}^k_j(x_0) = 0$ for all $0 < j \leq k$, so in particular $\lambda \definedas \tilde{\kappa}^1_1 : U_1 \to L_{\mathrm{c}}(E'_1, E'_1)$ is continuous with $\lambda(x_0) = 0$. \\
Hence there is a neighbourhood $V_1 \subseteq U_1$ \st $\|\lambda(x)\| < \frac{1}{2}$ for all $x \in V_1$.
$\|\lambda(x)\|$ here denotes the operator norm of $\lambda(x)$.
It follows that the map
\begin{align*}
\mu : V_1 &\to L_{\mathrm{c}}(E'_1, E'_1) \\
x &\mapsto \sum_{\ell = 0}^\infty (-\lambda(x))^\ell
\end{align*}
is well defined and continuous, and furthermore that $\|\mu(x)\| < 2$ for all $x \in V_1$.
Now define $V_k \definedas (\iota^k_1)\inv(V_1) \subseteq U_k$ and
\begin{align*}
\kappa'^k_j : V_k \times E'_j &\to E'_k \\
(x,u) &\mapsto -\kappa^k_1(x)\mu(\iota^k_1(x)) \iota'^j_1(u)
\end{align*}
for all $1 \leq j \leq k$.
The $\kappa'^k_j$ are obviously continuous as compositions of continuous functions and
\[
(\id_{E'_1} + \kappa^1_1)|_V\inv = \id_{E'_1} + \kappa'^1_1
\]
as a straightforward calculation shows.
Replacing $\mathbb{E}$ and $\mathbb{E}'$ by $\mathbb{E}^{\mathbf{1}}$ and $\mathbb{E}'^{\mathbf{1}}$, respectively (to account for the fact that one always needed $j \geq 1$ in the above), the $\kappa'^k_j$ define strongly smoothing envelopes $\mathcal{K}' : \mathcal{V}\oplus \mathbb{E}' \to \mathbb{E}'$ and $\mathcal{K}'_{\mathbf{k}} : \mathcal{V}^{\mathbf{k}}\oplus \mathbb{E}' \to \mathbb{E}'^{\mathbf{k}}$, for all strictly monotone increasing sequences $\mathbf{k} \subseteq \N_0$, as in the definition of a strongly smoothing family of morphisms $\kappa' : V\times E' \to E'$ with $(\id_{E'} + \kappa)|_V\inv = \id_{E'} + \kappa'$. \\
Finally, one needs to show that in the tame context $\kappa'$ is tame, which will follow by showing that $\mathcal{K}'$ and the $\mathcal{K}'_{\mathbf{k}}$ are tame provided that $\mathcal{K}$ and the $\mathcal{K}_{\mathbf{k}}$ are.
It is fairly straightforward to see, directly from \cref{Definition_Tame_nonlinear_map}, that this amounts to the existence of constants $C_{k} > 0$ \st
\[
\|\kappa^k_j(x)u\|'_k \leq C_{k}(\|x\|_k + \|u\|_j)
\]
for all $0 \leq j\leq k$ and $(x,u) \in U_k\times E'_j$.
Now for $1 \leq j\leq k$ and $(x,u) \in V_k\times E'_j$,
\begin{align*}
\|\kappa'^k_j(x)u\|'_k &= \left\|-\kappa^k_1(x)\mu(\iota^k_1(x)) \iota'^j_1(u)\right\|'_k \\
&\leq C_{k}\left(\|x\|_k + \left\|\mu(\iota^k_1(x)) \iota'^j_1(u)\right\|_1\right) \\
&\leq C_{k}\left(\|x\|_k + 2\|\iota'^j_1(u)\|_1\right) \\
&\leq 2C_{k}\left(\|x\|_k + \|u\|_j\right)\text{.}
\end{align*}
\end{proof}

\begin{definition}\label{Definition_Family_of_Fredholm_morphisms}
Let $E$, $E^i$, $i=1,2$, be $\overline{\text{sc}}$-Fr{\'e}chet spaces, let $U \subseteq E$ be an open subset and let $\phi : U\times E^1 \to E^2$ be a family of morphisms.
\begin{enumerate}[label=\arabic*.,ref=\arabic*.]
  \item $\phi$ is called a \emph{Fredholm} \iff it is locally invertible modulo strongly smoothing families of morphisms, \ie \iff for every $x_0 \in U$ there exist a neighbourhood $V \subseteq U$ of $x_0$ and families of morphisms
\begin{align*}
\psi : V\times E^2 &\to E^1 & \kappa : V\times E^1 &\to E^1 & \kappa' : V\times E^2 &\to E^2
\end{align*}
with $\kappa$ and $\kappa'$ strongly smoothing \st
\begin{align*}
\psi \circ \phi|_{V\times E^1} = \id_{E^1} + \kappa
\intertext{and}
\phi|_{V\times E^1} \circ \psi = \id_{E^2} + \kappa'\text{.}
\end{align*}
$\psi$ is then called a \emph{local Fredholm inverse} to $\phi$.
  \item The map
\begin{align*}
\ind : U &\to \Z \\
x &\mapsto \ind(\phi(x))\text{,}
\end{align*}
where $\ind(\phi(x))$ denotes the Fredholm index of the Fredholm morphism $\phi(x) : E^1 \to E^2$ in the sense of \cref{Definition_sc_Fredholm}, is called the \emph{(Fredholm) index of $\phi$}.
  \item The map
\begin{align*}
\corank\phi : U &\to \N_0 \\
x &\mapsto \dim\coker(\phi(x))\text{,}
\end{align*}
where $\coker(\phi(x))$ denotes the (finite dimensional) cokernel of the Fredholm morphism $\phi(x) : E^1 \to E^2$, is called the \emph{corank of $\phi$}. \\
$\phi$ is said to have \emph{constant rank} if $\corank\phi : U \to \N_0$ is constant.
\end{enumerate}
And analogously in the tame context.
\end{definition}

\begin{example}
Every invertible family of morphisms $\phi$ is Fredholm of constant rank with $\corank \phi \equiv 0$ and $\ind \phi \equiv 0$.
\end{example}

\begin{proposition}\label{Lemma_Basic_properties_families_of_Fredholm_morphisms}
Let $E$, $E'$ and $E^i$, $i = 1, \dots, 3$, be $\overline{\text{sc}}$-Fr{\'e}chet spaces, let $U \subseteq E$ and $U' \subseteq E'$ be open subsets together with an $\overline{\text{sc}}^0$ map $f : U' \to U$ and let
\begin{align*}
\phi : U\times E^1 &\to E^2\text{,} \\
\psi : U\times E^2 &\to E^3
\intertext{and}
\kappa : U\times E^1 &\to E^2
\end{align*}
be families of morphisms with $\phi$ and $\psi$ Fredholm and $\kappa$ strongly smoothing.
\begin{enumerate}[label=\arabic*.,ref=\arabic*.]
  \item\label{Lemma_Basic_properties_families_of_Fredholm_morphisms_1} $f^\ast \phi : U' \times E^1 \to E^2$ is Fredholm with $\ind(f^\ast \phi) = f^\ast \ind (\phi)$.
  \item\label{Lemma_Basic_properties_families_of_Fredholm_morphisms_2} $\psi\circ \phi : U\times E^1 \to E^3$ is Fredholm with $\ind (\psi\circ \phi) = \ind \psi + \ind \phi$.
  \item\label{Lemma_Basic_properties_families_of_Fredholm_morphisms_3} $\phi + \kappa : U\times E^1 \to E^2$ is Fredholm with $\ind (\phi + \kappa) = \ind \phi$.
\end{enumerate}
And analogously in the tame context.
\end{proposition}
\begin{proof}
Using \cref{Lemma_Basic_properties_strongly_smoothing_families}, the proofs are completely straightforward.
\end{proof}

\begin{lemma}\label{Lemma_Fredholm_family_standard_form}
Let $E$, $E^i$, $i=1,2$, be $\overline{\text{sc}}$-Fr{\'e}chet spaces, let $U \subseteq E$ be an open subset and let $\phi : U\times E^1 \to E^2$ be a family of morphisms. \\
$\phi$ is Fredholm \iff for every $x_0 \in U$ there exists a neighbourhood $V\subseteq U$ of $x_0$ and families of morphisms
\begin{align*}
\psi : V\times E^2 &\to E^1 & \kappa : V\times E^1 &\to E^1 & \kappa' : V\times E^2 &\to E^2
\end{align*}
with $\kappa$ and $\kappa'$ strongly smoothing \st
\[
\psi \circ \phi|_{V\times E^1} = \id_{E^1} + \kappa \quad\text{and}\quad \phi|_{V\times E^1} \circ \psi = \id_{E^2} + \kappa'
\]
and furthermore for any splitting
\[
E = \ker(\phi(x_0)) \oplus X
\]
one can assume that
\[
E' = \ker(\psi(x_0)) \oplus \im(\phi(x_0))
\]
and
\begin{equation}\label{Equation_Fredholm_family_standard_form}
\begin{aligned}
\ker(\phi(x_0)) &= \ker(\id_{E^1} + \kappa(x_0)) \\
\im(\phi(x_0)) &= \im(\id_{E^1} + \kappa'(x_0)) \\
\ker(\psi(x_0)) &= \ker(\id_{E^2} + \kappa'(x_0)) \\
\im(\psi(x_0)) &= \im(\id_{E^2} + \kappa(x_0)) \\
&= X\text{.}
\end{aligned}
\end{equation}
The same holds in the tame context.
\end{lemma}
\begin{proof}
For $x_0 \in U$ let $\psi : V \times E^2 \to E^2$ be a local Fredholm inverse to $\phi$, in a neighbourhood $V \subseteq U$ of $x_0$.
Then $\phi(x_0) : E^1 \to E^2$ is a Fredholm morphism with Fredholm inverse $\psi(x_0) : E^2 \to E^1$.
Choose any splitting $E' = \ker F' \oplus X'$.
By \cref{Theorem_sc_Fredholm_morphism}, applied to $F \definedas \phi(x_0)$ and $F' \definedas \psi(x_0)$, there exist splittings
\begin{align*}
E &= \underbrace{C \oplus \overline{C}}_{=\; \ker F} \oplus \underbrace{A \oplus B \oplus Y}_{=\; X} \\
E' &= \underbrace{A' \oplus \overline{A}'}_{=\; \ker F'} \oplus \underbrace{B' \oplus \overline{B}'}_{=\; C'} \oplus Y'\text{,}
\end{align*}
where all of these subspaces are finite dimensional with the exception of $Y$ and $Y'$.
Furthermore, $F'' : E' \to E$ defined by
\begin{align*}
F'' &\definedas \psi(x_0) - \psi(x_0)\circ \iota^{C'}_{E'}\circ \pr^{E'}_{C'} \;+ \\
&\quad\; +\; \iota^{A}_{E}\circ \left(\pr^{E'}_{A'}\circ \phi(x_0)\circ \iota^{A}_{E}\right)\inv \circ \pr^{E'}_{A'} \;+ \\ 
&\quad\; +\; \iota^{B}_{E}\circ \left(\pr^{E'}_{B'}\circ \phi(x_0)\circ \iota^{B}_{E}\right)\inv\circ \pr^{E'}_{B'}
\end{align*}
is a well defined Fredholm inverse to $\phi(x_0)$ that satisfies the conditions \labelcref{Equation_Fredholm_family_standard_form} with $\psi(x_0)$, $\kappa(x_0)$ and $\kappa'(x_0)$ replaced by $F''$, $F''\circ \phi(x_0) - \id_{E^1}$ and $\phi(x_0)\circ F' - \id_{E^2}$, respectively.
So if one can show that there exists a neighbourhood $V \subseteq U$ of $x_0$ \st
\begin{align*}
\psi' : V\times E^2 &\to E^2 \\
\psi'(x) &\definedas \psi(x) - \psi(x)\circ \iota^{C'}_{E'}\circ \pr^{E'}_{C'} \;+ \\
&\quad\; +\; \iota^{A}_{E}\circ \left(\pr^{E'}_{A'}\circ \phi(x)\circ \iota^{A}_{E}\right)\inv \circ \pr^{E'}_{A'} \;+ \\ 
&\quad\; +\; \iota^{B}_{E}\circ \left(\pr^{E'}_{B'}\circ \phi(x)\circ \iota^{B}_{E}\right)\inv\circ \pr^{E'}_{B'}
\end{align*}
is a well-defined family of morphisms that is a Fredholm inverse to $\phi$, then the proof is finished.
But the last part is immediate, because by \cref{Example_Strongly_smoothing_family}, \labelcref{Example_Strongly_smoothing_family_2}, $\psi' - \psi|_{V\times E^2}$ is strongly smoothing.
And from
\begin{align*}
\pr^{E'}_{A'}\circ \phi(x) \circ \iota^A_E &= \pr^{E'}_{A'}\circ \phi(x_0)\circ \iota^A_E + \underbrace{\pr^{E'}_{A'}\circ (\phi(x) - \phi(x_0))\circ \iota^A_E}_{\defines\; \lambda(x)} \\
\pr^{E'}_{B'}\circ \phi(x) \circ \iota^B_E &= \pr^{E'}_{B'}\circ \phi(x_0)\circ \iota^B_E + \underbrace{\pr^{E'}_{B'}\circ (\phi(x) - \phi(x_0))\circ \iota^B_E}_{\defines\; \mu(x)}
\end{align*}
well definedness follows by \cref{Proposition_Invertibility_perturbation_by_strongly_smoothing_family} (or much more elementary because $A$, $A'$ and $B$, $B'$ are finite dimensional) because the constant families defined by $\pr^{E'}_{A'}\circ \phi(x_0)\circ \iota^A_E : A \to A'$ and $\pr^{E'}_{B'}\circ \phi(x_0)\circ \iota^B_E : B \to B'$ are invertible and the families $\lambda : U\times A \to A'$ and $\mu : U\times B \to B'$ are strongly smoothing by \cref{Example_Strongly_smoothing_family}, \labelcref{Example_Strongly_smoothing_family_2}, and vanish at $x_0$.
\end{proof}

\begin{theorem}\label{Theorem_scbar_Fredholm}
Let $E$, $E^i$, $i=1,2$, be $\overline{\text{sc}}$-Fr{\'e}chet spaces, let $U \subseteq E$ be an open subset and let $\phi : U\times E^1 \to E^2$ be a family of morphisms.
Then the following are equivalent:
\begin{enumerate}[label=\arabic*.,ref=\arabic*.]
  \item\label{Theorem_scbar_Fredholm_1} $\phi$ is Fredholm.
  \item\label{Theorem_scbar_Fredholm_2}
\begin{enumerate}[label=(\alph*),ref=(\alph*)]
  \item\label{Theorem_scbar_Fredholm_2a} For every $x_0 \in U$,
\[
\phi(x_0) : E^1 \to E^2
\]
is Fredholm.
  \item\label{Theorem_scbar_Fredholm_2b} For one/any pair of splittings
\[
E^1 = X \oplus \ker \phi(x_0) \qquad\text{and}\qquad E^2 = \im \phi(x_0) \oplus C
\]
there exists a neighbourhood $V \subseteq U$ of $x_0$ \st
\[
\pr^{E^2}_{\im \phi(x_0)}\circ \phi\circ \iota^X_{E^1}|_{V\times X} : V\times X \to \im \phi(x_0)
\]
is invertible.
\end{enumerate}
\end{enumerate}
And analogously in the tame context.
\end{theorem}
\begin{proof}
The direction \labelcref{Theorem_scbar_Fredholm_1} $\Rightarrow$ \labelcref{Theorem_scbar_Fredholm_2} will show that \labelcref{Theorem_scbar_Fredholm_2}~\labelcref{Theorem_scbar_Fredholm_2b} holds for \emph{any} pair of splittings of $E^1 = X \oplus \ker \phi(x_0)$ and $E^2 = \im \phi(x_0)\oplus C$ (which exist by \cref{Example_Finite_dimensional_split_sc_subspace} and \cref{Proposition_sc_Fredholm_morphism}) and the direction \labelcref{Theorem_scbar_Fredholm_2} $\Rightarrow$ \labelcref{Theorem_scbar_Fredholm_1} will show that if there exists \emph{one} such pair of splittings \st \labelcref{Theorem_scbar_Fredholm_2}~\labelcref{Theorem_scbar_Fredholm_2b} holds, then $\phi$ is Fredholm.
\begin{enumerate}[label=\arabic*.,ref=\arabic*.]
  \item[\labelcref{Theorem_scbar_Fredholm_1} $\Rightarrow$ \labelcref{Theorem_scbar_Fredholm_2}]
Given $x_0\in U$, let $V \subseteq U$, $\psi : V\times E^2 \to E^1$, $\kappa : V\times E^1 \to E^1$ and $\kappa' : V\times E^2 \to E^2$ be as in \cref{Lemma_Fredholm_family_standard_form}.
Defining
\begin{align*}
\bar{\phi} : V\times X &\to \im\phi(x_0) \\
\bar{\phi}(x) &\definedas \pr^{E^2}_{\im \phi(x_0)}\circ \phi(x)\circ \iota^X_{E^1}\text{,} \\
\bar{\psi} : V\times \im\phi(x_0) &\to \im\psi(x_0) = X \\
\bar{\psi}(x) &\definedas \pr^{E^1}_{X}\circ \psi(x)\circ \iota^{\im\phi(x_0)}_{E^2}\text{,} \\
\bar{\kappa} : V\times X &\to X \\
\bar{\kappa}(x) &\definedas \pr^{E^1}_{X}\circ \kappa(x)\circ \iota^{X}_{E^1}
\intertext{and}
\bar{\kappa}' : V\times \im\phi(x_0) &\to \im\phi(x_0) \\
\bar{\kappa}'(x) &\definedas \pr^{E^2}_{\im\phi(x_0)}\circ \kappa'(x)\circ \iota^{\im\phi(x_0)}_{E^2}\text{,}
\end{align*}
it has to be shown that, after possibly shrinking $V$, $\bar{\phi}$ is invertible.
Then by \cref{Proposition_sc_Fredholm_morphism}, \labelcref{Proposition_sc_Fredholm_morphism_2}~\labelcref{Proposition_sc_Fredholm_morphism_2b}, $\id_{\im(\psi(x_0))} + \bar{\kappa}(x_0)$ and $\id_{\im(\phi(x_0))} + \bar{\kappa}'(x_0)$ are isomorphisms, so by \cref{Proposition_Invertibility_perturbation_by_strongly_smoothing_family}, after shrinking $V$, one can assume that $\id_{\im(\psi(x_0))} + \bar{\kappa} = \id_{\im(\psi(x_0))} + \bar{\kappa}(x_0) + (\bar{\kappa} - \bar{\kappa}(x_0))$ and $\id_{\im(\phi(x_0))} + \bar{\kappa}' = \id_{\im(\phi(x_0))} + \bar{\kappa}'(x_0) + (\bar{\kappa}' - \bar{\kappa}'(x_0))$ are invertible.
Then
\begin{align*}
\bar{\phi}\circ \bar{\psi} &= \pr^{E^2}_{\im \phi(x_0)}\circ \phi\circ \iota^X_{E^1}\circ \pr^{E^1}_X\circ \psi\circ \iota^{\im\phi(x_0)}_{E^2} \\
&= \pr^{E^2}_{\im\phi(x_0)}\circ \phi\circ \psi\circ \iota^{\im\phi(x_0)}_{E^2} \;+ \\
&\quad\; +\; \pr^{E^2}_{\im\phi(x_0)}\circ \phi\circ (\iota^X_{E^1}\circ\pr^{E^1}_X - \id_{E^1})\circ \psi\circ \iota^{\im\phi(x_0)}_{E^2} \\
&= \id_{\im\phi(x_0)} + \bar{\kappa}' + \pr^{E^2}_{\im\phi(x_0)}\circ \phi\circ (\iota^X_{E^1}\circ\pr^{E^1}_X - \underbrace{\id_{E^1}}_{\mathclap{=\; \iota^X_{E^1}\circ \pr^{E^1}_X + \iota^{\ker\phi(x_0)}_{E^1}\circ \pr^{E^1}_{\ker\phi(x_0)}}})\circ \psi\circ \iota^{\im\phi(x_0)}_{E^2} \\
&= \id_{\im\phi(x_0)} + \bar{\kappa}' - \underbrace{\pr^{E^2}_{\im\phi(x_0)}\circ \phi\circ\iota^{\ker\phi(x_0)}_{E^1}\circ \pr^{E^1}_{\ker\phi(x_0)}\circ \psi\circ \iota^{\im\phi(x_0)}_{E^2}}_{\defines\; \tilde{\kappa}'}
\end{align*}
$\tilde{\kappa}' : V\times \im\phi(x_0) \to \im\phi(x_0)$ is strongly smoothing by \cref{Example_Strongly_smoothing_family}, \labelcref{Example_Strongly_smoothing_family_2}, and $\tilde{\kappa}'(x_0) = 0$ by choice of $\psi$.
Again by \cref{Proposition_Invertibility_perturbation_by_strongly_smoothing_family}, after possibly shrinking $V$, $\bar{\phi}\circ\bar{\psi}$ is invertible, so $\bar{\phi}$ has a right inverse.
By a completely analogous argument, involving $\bar{\kappa}$ instead of $\bar{\kappa}'$, one shows that, after possibly shrinking $V$ again, $\bar{\psi}\circ \bar{\phi}$ is invertible, hence $\bar{\phi}$ has a left inverse.
So $\bar{\phi}$ is invertible.
  \item[\labelcref{Theorem_scbar_Fredholm_2} $\Rightarrow$ \labelcref{Theorem_scbar_Fredholm_1}] Formally the same proof as that of the corresponding statements in \cref{Proposition_Fredholm_I}, \cref{Proposition_Fredholm_operator_between_sc_chains} or \cref{Proposition_sc_Fredholm_morphism}.
\end{enumerate}
\end{proof}

\begin{corollary}\label{Corollary_Fredholm_index_constant}
Let $E$, $E^i$, $i=1,2$, be $\overline{\text{sc}}$-Fr{\'e}chet spaces, let $U \subseteq E$ be an open subset and let $\phi : U\times E^1 \to E^2$ be a family of Fredholm morphisms.
Then the following hold:
\begin{enumerate}[label=\arabic*.,ref=\arabic*.]
  \item\label{Corollary_Fredholm_index_constant_1} The maps
\begin{align*}
\dim \ker \phi : U &\to \N_0 \\
x &\mapsto \dim\ker(\phi(x))
\intertext{and}
\corank\phi : U &\to \N_0 \\
x &\mapsto \dim\coker(\phi(x))
\end{align*}
are upper semicontinuous.
  \item\label{Corollary_Fredholm_index_constant_2} The Fredholm index $\ind\phi = \dim\ker\phi - \corank\phi : U \to \Z$ of $\phi$ is continuous.
\end{enumerate}
\end{corollary}
\begin{proof}
By \cref{Theorem_scbar_Fredholm}, one can follow the standard proof found \eg in the Lecture notes \cite{Mrowka_Geometry_of_Manifolds} by T. Mrowka, Chapter 16, esp.~Lemmas 16.18--16.20: \\
So let $x_0, x \in V \subseteq U$, $X \subseteq E^1$ and $C \subseteq E^2$ be as in \cref{Theorem_scbar_Fredholm}.
W.\,r.\,t.~the decomposition
\[
E^1 = X \oplus \ker \phi(x_0) \qquad\text{and}\qquad E^2 = \im \phi(x_0) \oplus C
\]
one can write $\phi(x)$ in matrix form as
\[
\phi(x) = \begin{pmatrix} A(x) & B(x) \\ C(x) & D(x) \end{pmatrix}\text{,}
\]
where
\begin{align*}
A(x) &= \pr^{E^2}_{\im\phi(x_0)}\circ \phi(x)\circ \iota^X_{E^1} & C(x) &= \pr^{E^2}_{C}\circ \phi(x)\circ \iota^X_{E^1} \\
B(x) &= \pr^{E^2}_{\im\phi(x_0)}\circ \phi(x)\circ \iota^{\ker\phi(x_0)}_{E^1} & D(x) &= \pr^{E^2}_{C}\circ \phi(x)\circ \iota^{\ker \phi(x_0)}_{E^1}
\end{align*}
$B(x_0)$, $C(x_0)$ and $D(x_0)$ all vanish and $A(x)$ is invertible for all $x \in V$, so one can define
\begin{align*}
G(x) &\definedas \begin{pmatrix} \id_X & -A(x)\inv\circ B(x) \\ 0 & \id_{\ker \phi(x_0)} \end{pmatrix} &\text{and } && H(x) &\definedas \begin{pmatrix} \id_{\im \phi(x_0)} & 0 \\ -C(x)\circ A(x)\inv & \id_{C} \end{pmatrix}
\intertext{with}
G(x)\inv &= \begin{pmatrix} \id_X & A(x)\inv\circ B(x) \\ 0 & \id_{\ker \phi(x_0)} \end{pmatrix} &\text{and } && H(x)\inv &= \begin{pmatrix} \id_{\im \phi(x_0)} & 0 \\ C(x)\circ A(x)\inv & \id_{C} \end{pmatrix}
\end{align*}
Then
\[
H(x)\circ \phi(x)\circ G(x) = \begin{pmatrix} A(x) & 0 \\ 0 & D'(x) \end{pmatrix}\text{,}
\]
where
\[
D'(x) \definedas D(x) - C(x)\circ A(x)\inv\circ B(x) : \ker \phi(x_0) \to C\text{.}
\]
It follows that $(\ind \phi)(x) = \ind (\phi(x)) = \ind (D'(x)) = \dim \ker \phi(x_0) - \dim C = \dim \ker \phi(x_0) - \dim \coker \phi(x_0) = \ind (\phi(x_0)) = (\ind \phi)(x_0)$ by the dimension formula for linear maps between finite dimensional vector spaces.
This shows \labelcref{Corollary_Fredholm_index_constant_2}. \\
For \labelcref{Corollary_Fredholm_index_constant_1}, first observe that $\dim\ker\phi(x) = \dim \ker (D'(x))$.
But $D' : V\times \ker\phi(x_0)\to C$ is a continuous family of linear maps between finite dimensional vector spaces, for which the result is known.
The result about $\corank \phi$ then follows from the results for $\ind\phi$ and $\dim\ker\phi$.
\end{proof}

\begin{example}\label{Example_Strongly_continuous_sc_Fredholm_family}
Let $E$ and $E^i$, $i = 1,2$, be $\overline{\text{sc}}$-Fr{\'e}chet spaces and let $U \subseteq E$ be an open subset and let $\phi : U\times E^1 \to E^2$ be a family of morphisms. \\
Assume that for every $x_0 \in U$, there exists a neighbourhood $V \subseteq U$ of $x_0$ \st $\phi|_V$ has an envelope $((\mathbb{E}\oplus \mathbb{E}^1, \rho\oplus\rho^1), (\mathbb{E}^2,\rho^2), \mathcal{F} : \mathcal{V}\oplus \mathbb{E}^1\to \mathbb{E}^2)$ with the following properties (for simplicity the inclusions of the sc-chains $\mathbb{E}$ and $\mathbb{E}^i$, $i=1,2$, will be dropped from the notation):
\begin{enumerate}[label=\arabic*.,ref=\arabic*.]
  \item\label{Example_Strongly_continuous_sc_Fredholm_family_1} $\mathcal{F} = (\phi_k : V_k \times E^1_k\to E^2_k)$ is strict.
  \item\label{Example_Strongly_continuous_sc_Fredholm_family_2} For all $k \in \N_0$, $\phi_k : V_k \times E^1_k\to E^2_k$ is strongly continuous, \ie
\begin{align*}
\tilde{\phi}_k : V_k &\to L_{\mathrm{c}}(E^1_k, E^2_k) \\
x &\mapsto \phi_k(x)
\end{align*}
is continuous.
  \item\label{Example_Strongly_continuous_sc_Fredholm_family_3} For all $x \in V_0$, $\phi_0(x) : E^1_0 \to E^2_0$ is Fredholm.
  \item\label{Example_Strongly_continuous_sc_Fredholm_family_4} $\mathcal{F}$ is \emph{regularising} in the following sense: For all $k \in \N_0\cup \{\infty\}$, if $x \in V_k$ and $u \in E^1_0$ with $\phi_0(x)u \in E^2_j$ for some $0 \leq j \leq k$, then $u \in E^1_j$.
\end{enumerate}
Then $\phi$ is Fredholm.

Assume furthermore that $E$, $E^i$, $i=1,2$, are pre-tame $\overline{\text{sc}}$-Fr{\'e}chet spaces, that $\phi$ is tame, and that the following \emph{elliptic inequalities} hold:
For all $k \in \N_0$ there exist constants $c_k, d_k \in [0,\infty)$ \st if $x \in V_k$ and $u \in E^1_k$, then
\[
\|u\|^1_k \leq c_k\|x\|_0\|\phi_k(x)u\|^2_k + d_k\|x\|_k\|\phi_k(x)u\|^2_0\text{.}
\]
Then $\phi$ is a tame Fredholm family of morphisms.
\begin{proof}
I will check \cref{Theorem_scbar_Fredholm}, \labelcref{Theorem_scbar_Fredholm_2}. \\
First, the proof of Lemma 3.5 in \cite{1209.4040} shows that for all $k \in \N_0\cup\{\infty\}$ and $0 \leq j \leq k$, if $x \in V_k$ then $\ker \phi_j(x) = \ker \phi_0(x) \subseteq E^1_k$, $\im \phi_j(x) = \im \phi_0(x) \cap E^2_j$, and $\coker \phi_j(x) \cong \coker \phi_0(x)$ via the inclusion $(\iota^2)^j_0 : E^2_j \hookrightarrow E^2_0$. \\
So in particular $\phi_j(x) : E^1_j \to E^2_j$ is Fredholm with $\ind \phi_j(x) = \ind \phi_0(x)$.
By \labelcref{Example_Strongly_continuous_sc_Fredholm_family_3,Example_Strongly_continuous_sc_Fredholm_family_4}, together with \cref{Proposition_Fredholm_operator_between_sc_chains}, $\mathcal{F}(x) : \mathbb{E}^1 \to \mathbb{E}^2$ is a Fredholm operator between sc-chains for all $x \in V_\infty$.
This shows \cref{Theorem_scbar_Fredholm}, \labelcref{Theorem_scbar_Fredholm_2}~\labelcref{Theorem_scbar_Fredholm_2a}. \\
Let $x_0 \in V_\infty$.
By the above there exist splittings $\mathbb{E}^1 = \ker \phi_\infty(x_0) \oplus \mathbb{X}$ and $\mathbb{E}^2 = \im \mathcal{F}(x) \oplus C$, where $C \subseteq E^2_\infty$ is finite dimensional.
In particular $E^1_0 = \ker \phi_0(x_0)\oplus X_0$ and $E^2_0 = \im \phi_0(x_0)\oplus C$, and $\pr^{E^2_0}_{\im\phi_0(x_0)}\circ \phi_0(x_0)\circ \iota_{E^1_0}^{X_0} : X_0 \to \im \phi_0(x_0)$ is an isomorphism.
Define, for $k \in \N_0\cup\{\infty\}$,
\begin{align*}
\gamma_k : V_k\times X_k &\to \im \phi_k(x_0) \\
x &\mapsto \pr^{E^2_k}_{\im\phi_k(x_0)}\circ \phi_k(x)\circ \iota_{E^1_k}^{X_k}
\intertext{and}
\tilde{\gamma}_k : V_k &\to L_{\mathrm{c}}(X_k, \im\phi_k(x_0)) \\
x &\mapsto \pr^{E^2_k}_{\im\phi_k(x_0)}\circ \phi_k(x)\circ \iota_{E^1_k}^{X_k}\text{.}
\end{align*}
By \labelcref{Example_Strongly_continuous_sc_Fredholm_family_2}, the maps $\tilde{\gamma}_k$ are continuous, hence there exists a neighbourhood $V'_0 \subseteq V_0$ of $x_0$ \st $\tilde{\gamma}_0(x)$ is invertible for all $x \in V'_0$ and the map
\begin{align*}
V'_0 &\to L_{\mathrm{c}}(\im\phi_0(x_0), X_0) \\
x &\mapsto \tilde{\gamma}_0(x)\inv
\end{align*}
is continuous.
Define $\psi_0 : V'_0 \times \im\phi_0(x_0) \to X_0$ by $\psi_0(x,u) \definedas \tilde{\gamma}_0(x)\inv u$ and, for $k\in\N_0\cup\{\infty\}$, set $V'_k \definedas (\iota^k_0)\inv(V'_0)$. \\
Then for every $k\in\N_0$ and $x \in V'_k$, $\tilde{\gamma}_k(x)$ is invertible:
For given $u \in X_k \subseteq X_0$, $\tilde{\gamma}_k(x)u = 0$ means $\phi_k(x)u \in C$, but $\phi_k(x)u = \phi_0(x)u$, so $\tilde{\gamma}_0(x)u = 0$, hence $u = 0$ because $\tilde{\gamma}_0(x)u$ is injective, and $\tilde{\gamma}_k(x)$ is injective.
And given $v \in \im \phi_k(x_0) = \im \phi_0(x_0) \cap E^2_k$, because $\tilde{\gamma}_0(x)$ is surjective there exist $\tilde{v} \in C$ and $u \in X_0$ with $\phi_0(x)u = v - \tilde{v} \in E^2_k$.
By \labelcref{Example_Strongly_continuous_sc_Fredholm_family_4}~hence $u \in X_0 \cap E^1_k = X_k$ and hence $\phi_k(x)u = \phi_0(x)u = v - \tilde{v}$, so $\tilde{\gamma}_k(x)$ is surjective.
It follows from the open mapping theorem that $\tilde{\gamma}_k(x) : X_k \to \im \phi_k(x_0)$ is an isomorphism.
Denoting, for Banach spaces $A,B$, by $\operatorname{Iso}(A,B) \subseteq L_{\mathrm{c}}(A,B)$ the open subset of invertible operators, it follows that $\tilde{\gamma}_k : V'_k \to \operatorname{Iso}(X_k, \im \phi_k(x_0))$ is well defined.
Since the inversion $\operatorname{Iso}(A,B) \to \operatorname{Iso}(B,A)$, $\lambda \mapsto \lambda\inv$, is continuous, it follows that
\begin{align*}
\psi_k : V'_k \times \im \phi_k(x_0) &\to X_k \\
(x,u) &\mapsto \tilde{\gamma}_k(x)\inv u
\end{align*}
is continuous.
The maps $\psi_k$ define an envelope for a (thus well defined) family of morphisms $\psi : V' \times \im \phi(x_0) \to X$, $V' \definedas \rho(V'_\infty)$, $X \definedas \rho^1(X_\infty)$, that is an inverse to $\pr^{E^2}_{\im\phi(x_0)}\circ \phi \circ \iota^X_{E^1}|_{V\times X} : V\times X \to \im \phi(x_0)$.
This shows  \cref{Theorem_scbar_Fredholm}, \labelcref{Theorem_scbar_Fredholm_2}~\labelcref{Theorem_scbar_Fredholm_2b}, so $\phi$ is Fredholm.

Now assume that the elliptic inequalities hold.
The goal is to show that then the family of morphisms $\psi$ just constructed is tame. \\
So let $x \in V'_k$ and let $v \in \im \phi_k(x_0)$.
Since one can write
\[
\phi_k(x) = \iota^{\im \phi_k(x_0)}_{E^2_k}\circ \pr^{E^2_k}_{\im \phi_k(x_0)}\circ \phi_k(x) + \iota^C_{E^2_k}\circ \pr^{E^2_k}_C\circ \phi_k(x)\circ \iota^{X_k}_{E^1_k}\text{,}
\]
the elliptic inequalities give
\begin{align*}
\|\psi_k(x)v\|_{X_k} &= \|\iota^{X_k}_{E^1_k}\psi_k(x)v\|^1_k \\
&\leq c_k\|x\|_0\|\phi_k(x)\iota^{X_k}_{E^1_k}\psi_k(x)\|_k + d_k\|x\|_k\|\phi_k(x)\iota^{X_k}_{E^1_k}\psi_k(x)v\|^2_0 \\
&\leq c_k\|x\|_0\Bigl(\Bigl\|\iota^{\im \phi_k(x_0)}_{E^2_k}\underbrace{\pr^{E^2_k}_{\im \phi_k(x_0)} \phi_k(x)\iota^{X_k}_{E^1_k}\psi_k(x)}_{=\; \id_{\im \phi_k(x_0)}}v\Bigr\|^2_k \;+ \\
&\qquad\quad +\; \Bigl\|\iota^C_{E^2_k}\pr^{E^2_k}_C\phi_k(x)\iota^{X_k}_{E^1_k}\psi_k(x)v\Bigr\|^2_k\Bigr) + d_k\|x\|_k\|\phi_0(x)\iota^{X_0}_{E^1_0}\psi_0(x)v\|^2_0 \\
&\leq c_k\|x\|_0\Bigl(\|\iota^{\im\phi_k(x_0)}_{E^2_k}v\|^2_k + c'_k\Bigl\|\iota^C_{E^2_0}\pr^{E^2_0}_C\phi_0(x)\iota^{X_0}_{E^1_0}\psi_0(x)v\Bigr\|^2_0\Bigr) \;+ \\
&\quad\; +\; d_k\|x\|_k\|\phi_0(x)\iota^{X_0}_{E^1_0}\psi_0(x)v\|^2_0\text{,}
\intertext{where the compatibility conditions of the $\phi_k$ and $\psi_k$ under the inclusions in $\mathbb{E}^1$ and $\mathbb{E}^2$ have been used ($C$ defines an sc-scale where all the inclusions are the identity) as well as the fact that on the finite dimensional vector space $C$ all the induced norms from the $\|\cdot\|^2_k$ are equivalent. Hence}
\|\psi_k(x)v\|_{X_k} &\leq \Bigl(c_k\|x\|_0 + \bigl(c_kc'_k\|x\|_0\left\|\iota^C_{E^2_0}\pr^{E^2_0}_C\phi_0(x)\iota^{X_0}_{E^1_0}\psi_0(x)\right\|_{L_{\mathrm{c}}(\im \phi_0(x_0), C)} \;+ \\
&\qquad +\; d_k\left\|\phi_0(x)\iota^{X_0}_{E^1_0}\psi_0(x)\right\|_{L_{\mathrm{c}}(\im \phi_0(x_0), E^2_0)}\bigr)\|v\|_{\im \phi_0(x_0)}\Bigr)\,\cdot \\
&\quad\; \cdot\, (1 + \|x\|_k + \|v\|_{\im \phi_k(x_0)})\text{.}
\end{align*}
Now by choosing $V'_0$ a small enough bounded subset of $E_0$ one can assume (by strong continuity of $\phi_0$ and $\psi_0$) that the operator norms of $\phi_0(x)$ and $\psi_0(x)$ are bounded by a constant independent of $x \in V'_0$.
Then for every bounded subset $B_0 \subseteq \im\phi_0(x_0)$ there exists a constant $C > 0$ \st
\[
\|\psi_k(x, v)\|_{X_k} \leq C (1 + \|x\|_k + \|v\|_{\im\phi_k(x_0)})
\]
for all $(x,v) \in V'_k\times B_k$ where $B_k \definedas B_0 \cap E^2_k$.
So by definition the map $\psi$ is tame.
\end{proof}
\end{example}

\Needspace{25\baselineskip}
\subsection{Fredholm maps}

\begin{definition}\label{Definition_scbar_Fredholm}
Let $E$ and $E'$ be $\overline{\text{sc}}$-Fr{\'e}chet spaces, let $U\subseteq E$ be an open subset and let $f : U \to E'$ be $\overline{\text{sc}}^1$.
\begin{enumerate}[label=\arabic*.,ref=\arabic*.]
  \item\label{Definition_scbar_Fredholm_1} $f$ is called \emph{Fredholm} \iff $Df : U\times E \to E'$ is a Fredholm family of morphisms. \\
If $E$ and $E'$ are pre-tame $\overline{\text{sc}}$-Fr{\'e}chet spaces, then $f$ is called \emph{tame Fredholm} \iff $f$ is tame and $Df : U\times E \to E'$ is a tame Fredholm family of morphisms.
  \item\label{Definition_scbar_Fredholm_2} If $f$ is Fredholm, then the map
\begin{align*}
\ind f : U &\to \Z \\
x &\mapsto \ind Df(x)
\end{align*}
is called the \emph{index} of $f$.
  \item\label{Definition_scbar_Fredholm_3} If $f$ is Fredholm, then the map
\begin{align*}
\corank f : U &\to \N_0 \\
x &\mapsto \dim\coker(Df(x))
\end{align*}
is called the \emph{corank of $f$}. \\
$f$ is said to have \emph{constant rank} if $\corank f : U \to \N_0$ is constant.
\end{enumerate}
\end{definition}

\begin{example}\label{Example_Diffeomorphisms_are_Fredholm}
Every $\overline{\text{sc}}$-diffeomorphism $f$ ($\overline{\text{sc}}^\infty$ map with an $\overline{\text{sc}}^\infty$ inverse) is Fredholm of constant rank with $\corank f \equiv 0$ and $\ind f \equiv 0$.
\end{example}

\begin{example}\label{Example_Linear_Fredholm_implies_nonlinear_Fredholm}
Every (linear) Fredholm morphism defines a (nonlinear) Fredholm map, as does every affine map with linear part a Fredholm morphism.
\end{example}

\begin{example}\label{Example_Strongly_continuous_sc_Fredholm_map}
This example covers, in abstract form, applications such as nonlinear Cauchy-Riemann operators, for which the elliptic inequalities are very well studied (\cf \cite{MR2045629}, Appendix B).
See also \cite{MR656198}, Sections II.2.2 and II.3.3. \\
Let $E$ and $E'$, be $\overline{\text{sc}}$-Fr{\'e}chet spaces, let $U \subseteq E$ be an open subset and let $f : U \to E'$ be a map. \\
Assume that for every $x_0 \in U$, there exists a neighbourhood $V \subseteq U$ of $x_0$ \st $f|_V$ has an envelope $((\mathbb{E}, \rho), (\mathbb{E}',\rho'), \mathcal{F} : \mathcal{V}\to \mathbb{E}')$ with the following properties:
\begin{enumerate}[label=\arabic*.,ref=\arabic*.]
  \item\label{Example_Strongly_continuous_sc_Fredholm_map_1} $\mathcal{F} = (f_k : V_k\to E'_k)$ is strict.
  \item\label{Example_Strongly_continuous_sc_Fredholm_map_2} For all $k \in \N_0$, $f_k : V_k\to E'_k$ is strongly continuously weakly Fr{\'e}chet differentiable.
  \item\label{Example_Strongly_continuous_sc_Fredholm_map_3} The envelope $((\mathbb{E}\oplus \mathbb{E}, \rho\oplus \rho), (\mathbb{E}', \rho'), \mathcal{DF} : \mathcal{V}\times \mathbb{E} \to \mathbb{E}')$ of $Df|_V : V\times E \to E'$ satisfies the conditions from \cref{Example_Strongly_continuous_sc_Fredholm_family}.
\end{enumerate}
Then $f$ is Fredholm.
\end{example}

\begin{proposition}\label{Lemma_Fredholm_index_nonlinear_map}
Let $E$, $E'$ and $E''$ be $\overline{\text{sc}}$-Fr{\'e}chet spaces, let $U \subseteq E$ and $U' \subseteq E'$ be open subsets and let $f : U \to E'$ and $f' : U' \to E''$ be $\overline{\text{sc}}^1$ maps with $f(U) \subseteq U'$.
\begin{enumerate}[label=\arabic*.,ref=\arabic*.]
  \item\label{Lemma_Fredholm_index_nonlinear_map_1} If $f$ is Fredholm, then $\ind f : U \to \Z$ is continuous and $\corank f : U \to \N_0$ is upper semicontinuous.
  \item\label{Lemma_Fredholm_index_nonlinear_map_2} If $f$ and $f'$ are Fredholm, then so is $f'\circ f : U \to E''$ and
\[
\ind (f'\circ f) = \ind f + f^\ast\ind f'\text{.}
\]
\end{enumerate}
The same holds in the tame context.
\end{proposition}
\begin{proof}
\labelcref{Lemma_Fredholm_index_nonlinear_map_1}~follows from \cref{Corollary_Fredholm_index_constant} and \labelcref{Lemma_Fredholm_index_nonlinear_map_2}~follows from $D(f'\circ f) = (f^\ast Df')\circ Df : U\times E \to E''$ (\cf \cref{Theorem_Chain_rule_sc_differentiable}) and \cref{Lemma_Basic_properties_families_of_Fredholm_morphisms}.
\end{proof}

\Needspace{25\baselineskip}
\subsection{Applications of the Nash-Moser inverse function theorem}

\Needspace{15\baselineskip}
\subsubsection{The constant rank theorem}

\begin{theorem}[Constant rank theorem]\label{Theorem_constant_rank_theorem}
Let $E$ and $E'$ be weakly tame $\overline{\text{sc}}$-Fr{\'e}chet spaces, let $U\subseteq E$ be an open subset and let $f : U \to E'$ be $\overline{\text{sc}}^k$ and tame up to order $k$ for some $k \in \N_0\cup\{\infty\}$ with $k \geq 2$. \\
Given $x_0 \in U$, if there exists a neighbourhood $V\subseteq U$ of $x_0$ \st $f|_V$ is tame Fredholm and has constant rank, then there exist open neighbourhoods $W\subseteq U$ of $x_0$ and $W' \subseteq E'$ of $f(x_0)$ with $f(W) \subseteq W'$, together with $\underline{\text{sc}}^k$ diffeomorphisms, tame up to order $k$,
\begin{align*}
\Phi : W &\to \tilde{W} \subseteq \im Df(x_0)\oplus \ker Df(x_0) \\
\Psi : W' &\to \tilde{W}' \subseteq \im Df(x_0)\oplus \coker Df(x_0)
\end{align*}
\st
\[
\Psi\circ f|_W\circ \Phi\inv = (\id_{\im Df(x_0)}\oplus 0)|_{\tilde{W}}\text{.}
\]
\end{theorem}
\begin{proof}
First, one can w.\,l.\,o.\,g.~assume that $V = U$, $x_0 = 0$ and $f(x_0) = 0$.
Second, because $f$ is Fredholm, one can pick splittings
\[
E = X\oplus \ker Df(x_0) \quad\text{and}\quad E' = \im Df(x_0)\oplus C'\text{,}
\]
where $C' \subseteq E'$ is identified with $\coker Df(x_0)$ via the quotient projection $E' \to \coker Df(x_0)$.
By \cref{Theorem_scbar_Fredholm}, after shrinking $U$, one can assume that
\begin{align*}
A : U\times X &\to \im Df(x_0) \\
A(x) &\definedas \pr^{E'}_{\im Df(x_0)}\circ Df(x)\circ \iota^X_E
\end{align*}
has a tame inverse $A\inv : U\times \im Df(x_0) \to X$.
Now define
\begin{align*}
\Phi' : U &\to \im Df(x_0)\oplus \ker Df(x_0) \\
x &\mapsto (\pr^{E'}_{\im Df(x_0)}(f(x)), \pr^E_{\ker Df(x_0)}(x))\text{.}
\end{align*}
One can then write in matrix form
\begin{align*}
Df(x) &= \begin{blockarray}{ccc}
X & \ker Df(x_0) & \\
\begin{block}{(cc)c}
A(x) & B(x) & \im Df(x_0) \\
C(x) & D(x) & C' \\
\end{block}
\end{blockarray} \\
D\Phi'(x) &= \begin{blockarray}{ccc}
X & \ker Df(x_0) & \\
\begin{block}{(cc)c}
A(x) & B(x) & \im Df(x_0) \\
0 & \id_{\ker Df(x_0)} & \ker Df(x_0) \\
\end{block}
\end{blockarray} \\
D\Phi'(x)\inv &= \begin{blockarray}{ccc}
\im Df(x_0) & \ker Df(x_0) & \\
\begin{block}{(cc)c}
A\inv(x) & -A\inv(x)\circ B(x) & X \\
0 & \id_{\ker Df(x_0)} & \ker Df(x_0) \\
\end{block}
\end{blockarray}
\end{align*}
so $D\Phi' : U\times E \to \im Df(x_0)\oplus \ker Df(x_0)$ has a tame inverse.
By \cref{Theorem_Nash_Moser}, there exist neighbourhoods $W \subseteq U$ and $\tilde{W} \subseteq \im Df(x_0) \oplus \ker Df(x_0)$ \st
\[
\Phi \definedas \Phi'|_W : W \to \tilde{W}
\]
is an $\underline{\text{sc}}^k$ diffeomorphism.
Furthermore, for all $x \in W$,
\begin{align*}
D(f\circ \Phi\inv)(\Phi(x)) &= Df(x)\circ D\Phi'(x)\inv \\
&= \begin{blockarray}{ccc}
\im Df(x_0) & \ker Df(x_0) & \\
\begin{block}{(cc)c}
\id_{\im Df(x_0)} & 0 & \im Df(x_0) \\
(C\circ A\inv)(x) & (D - C\circ A\inv\circ B)(x) & C' \\
\end{block}
\end{blockarray}
\end{align*}
Since $\Phi$ is an $\underline{\text{sc}}^k$ diffeomorphism, it follows that $\corank f|_W = (\corank f\circ \Phi\inv)\circ \Phi$, so $(\corank f)(x) = \corank((D - C\circ A\inv\circ B)(x))$ for all $x \in W$.
Now by construction $(D - C\circ A\inv\circ B)(x_0) = 0$, so $\corank f|_W$ is constant \iff $(D - C\circ A\inv\circ B)(x) = 0$ for all $x \in W$.
Thus, for all $\tilde{x} \in \tilde{W}$,
\begin{align*}
D(f\circ \Phi\inv)(\tilde{x}) &= \begin{blockarray}{ccc}
\im Df(x_0) & \ker Df(x_0) & \\
\begin{block}{(cc)c}
\id_{\im Df(x_0)} & 0 & \im Df(x_0) \\
(C\circ A\inv)(\Phi\inv(\tilde{x})) & 0 & C' \\
\end{block}
\end{blockarray}
\end{align*}
So after possibly shrinking $\tilde{W}$, one can assume that $\tilde{W} = \tilde{W}_1\times \tilde{W}_2$ for some neighbourhoods $\tilde{W}_1$, $\tilde{W}_2$ of $0$ in $\im Df(x_0)$ and $\ker Df(x_0)$, respectively, and
\begin{align*}
f\circ \Phi\inv : \im Df(x_0)\oplus \ker Df(x_0) \supseteq \tilde{W}_1\times \tilde{W}_2 &\to \im Df(x_0) \oplus C' \\
(u,z) &\mapsto (u, \varphi(u))
\end{align*}
where
\begin{align*}
\varphi : \tilde{W}_1 &\to C' \\
\varphi &\definedas \pr^{E'}_{C'}\circ f\circ \Phi\inv \circ \iota^{\tilde{W}_1}_{\im Df(x_0)\oplus \ker Df(x_0)}\text{.}
\end{align*}
Now define $W' \definedas \tilde{W}' \definedas \tilde{W}_1\times C'$ and
\begin{align*}
\Psi : \tilde{W}_1\times C' &\to \tilde{W}_1\times C' \\
(u,y) &\mapsto (u, y - \varphi(u))\text{.}
\end{align*}
\end{proof}

\begin{corollary}
Let $E$ and $E'$ be weakly tame $\overline{\text{sc}}$-Fr{\'e}chet spaces, let $U\subseteq E$ be an open subset and let $f : U \to E'$ be $\overline{\text{sc}}^\infty$ and tame up to arbitrary order. \\
If $f$ is tame Fredholm and both $\ind f$ and $\corank f$ are constant, then for any $y \in E'$, $f\inv(y) \subseteq E$ canonically carries a smooth structure. \\
In particular, if $f\inv(y)$ is $2^{\text{nd}}$-countable, then it is a smooth manifold of dimension $\dim f\inv(y) = \ind f + \corank f$.
\end{corollary}

\Needspace{15\baselineskip}
\subsubsection{Finite dimensional reduction and the Sard-Smale theorem}

In this section it will be shown that the results for Fredholm maps between Banach spaces from \cite{MR2045629}, Sections A.4 and A.5, carry over almost ad verbatim.
For brevities' sake I will make certain simplifications like restricting to smooth maps.
Note that for maps between tame $\overline{\text{sc}}$-Fr{\'e}chet spaces that are tame up to arbitrary order, $\overline{\text{sc}}^\infty$ and $\underline{\text{sc}}^\infty$ coincide by \cref{Proposition_Underline_sc_implies_overline_sc}.

\begin{theorem}[Finite dimensional reduction, \cf \cite{MR2045629}, Theorem A.4.3]\label{Theorem_Finite_dimensional_reduction}
Let $E$ and $E'$ be weakly tame $\overline{\text{sc}}$-Fr{\'e}chet spaces, let $U\subseteq E$ be an open subset and let $f : U \to E'$ be a tame Fredholm map that is $\overline{\text{sc}}^\infty$ and tame up to arbitrary order. \\
Given $x_0 \in U$, there exist the following:
\begin{enumerate}[label=\arabic*.,ref=\arabic*.]
  \item A splitting $E' = \im Df(x_0) \oplus C$,
  \item open neighbourhoods $V, W \subseteq U$ of $x_0$,
  \item an $\overline{\text{sc}}^\infty$ diffeomorphism $g : V \to W$, tame up to arbitrary order, with
\[
g(x_0) = x_0 \quad\text{and}\quad Dg(x_0) = \id_{E}\text{,}
\]
  \item an $\overline{\text{sc}}^\infty$ map $\tilde{k} : V \to C$, tame up to arbitrary order, with
\[
\tilde{k}(x_0) = 0 \quad\text{and}\quad D\tilde{k}(x_0) = 0\text{.}
\]
\end{enumerate}
Furthermore, setting $k \definedas \iota^C_{E'}\circ \tilde{k} : V \to E'$,
\[
f\circ g(x) = f(x_0) + Df(x_0)(x-x_0) + k(x) \quad \forall\, x \in V\text{.}
\]
\end{theorem}
\begin{proof}
First, by shifting, one can assume w.\,l.\,o.\,g.~that $x_0 = 0$ and $f(x_0) = 0$.
Second, let $D \definedas Df(x_0)$ and pick splittings $E = \ker D \oplus X$ and $E' = \im D \oplus C$ as in \cref{Proposition_sc_Fredholm_morphism}.
Then $\tilde{D} \definedas \pr^{E'}_{\im D}\circ D\circ \iota^X_E : X \to \im D$ is an isomorphism and one can define $Q \definedas \iota^X_E\circ \tilde{D}\inv\circ \pr^{E'}_{\im D} : E' \to E$ and
\begin{align*}
\psi : U &\to E \\
x &\mapsto \pr^E_{\ker D} x + Q\circ f(x)\text{.}
\end{align*}
$\psi$ is a tame Fredholm map, $\overline{\text{sc}}^\infty$ and tame up to arbitrary order, satisfying $\psi(0) = 0$ and $D\psi(0)u = \pr^E_{\ker D}u + QDu = \pr^E_{\ker D}u + \pr^E_Xu = u$ for all $u \in E$, \ie $D\psi(0) = \id_E$.
So by \cref{Theorem_scbar_Fredholm} there exists a neighbourhood $V \subseteq U$ of $0$ \st $D\psi|_{V\times E} : V\times E \to E$ is invertible with tame inverse.
After possibly shrinking $V$ and setting $W \definedas \psi(V)$, by the Nash-Moser inverse function theorem, \cref{Theorem_Nash_Moser}, together with \cref{Proposition_Underline_sc_implies_overline_sc}, there is a well-defined $\overline{\text{sc}}^\infty$ diffeomorphism, tame up to arbitrary order,
\begin{align*}
g : V &\to W \\
x &\mapsto (\psi|_V)\inv(x)\text{.}
\end{align*}
Define $\tilde{k} \definedas \pr^{E'}_C\circ f\circ g$.
One computes $D\circ \psi(x) = D\circ \pr^E_{\ker D} x + D\circ Q\circ f(x) = D\circ Q\circ f(x)$, so $D = D\circ Q\circ f\circ g$ and hence
\begin{align*}
f\circ g &= f\circ g - D + D \\
&= f\circ g - D\circ Q\circ f\circ g + D \\
&= (\id_{E'} - D\circ Q)\circ f\circ g + D \\
&= \iota^C_{E'}\circ \tilde{k} + D\text{.}
\end{align*}
\end{proof}

\begin{definition}\label{Definition_Strongly_smoothing_map}
Let $E$ and $E'$ be $\overline{\text{sc}}$-Fr{\'e}chet spaces, let $U\subseteq E$ be an open subset and let $k : U \to E'$ be $\overline{\text{sc}}^1$. \\
$k$ is called \emph{strongly smoothing} \iff $Dk : U \times E \to E'$ is a strongly smoothing family of morphisms. \\
And analogously in the tame context.
\end{definition}

\begin{proposition}\label{Lemma_Perturbations_of_Fredholm_maps}
Let $E$, $E'$ and $E''$ be $\overline{\text{sc}}$-Fr{\'e}chet spaces, let $U \subseteq E$, $V \subseteq E'$ and $W \subseteq E''$ be open subsets, and let $f : U \to V$, $g : V \to W$ and $k : U \to E'$ be $\overline{\text{sc}}^1$.
\begin{enumerate}[label=\arabic*.,ref=\arabic*.]
  \item If at least one of $f$ and $g$ is strongly smoothing, then so is $g\circ f$.
  \item If $f$ is Fredholm and $k$ is strongly smoothing, then $f + k : U \to E'$ is Fredholm as well with $\ind(f+k) = \ind(f)$.
\end{enumerate}
And analogously in the tame context.
\end{proposition}
\begin{proof}
Clear from the definitions and \cref{Lemma_Basic_properties_strongly_smoothing_families,Lemma_Basic_properties_families_of_Fredholm_morphisms}.
\end{proof}

\begin{corollary}\label{Corollary_Standard_form_of_Fredholm_maps}
Let $E$ and $E'$ be weakly tame $\overline{\text{sc}}$-Fr{\'e}chet spaces, let $U\subseteq E$ be an open subset and let $f : U \to E'$ be $\overline{\text{sc}}^1$. \\
Then the following are equivalent:
\begin{enumerate}[label=\arabic*.,ref=\arabic*.]
  \item $f$ is tame Fredholm, $\overline{\text{sc}}^\infty$, and tame up to arbitrary order
  \item for every $x_0 \in U$, $Df(x_0) : E \to E'$ is Fredholm and there exist
\begin{enumerate}[label=(\alph*),ref=(\alph*)]
  \item open neighbourhoods $V, W \subseteq U$ of $x_0$,
  \item an $\overline{\text{sc}}^\infty$ diffeomorphism $g : V \to W$, tame up to arbitrary order, with
\[
g(x_0) = x_0 \quad\text{and}\quad Dg(x_0) = \id_{E}\text{,}
\]
  \item a strongly smoothing $\overline{\text{sc}}^\infty$ map $k : V \to E'$, tame up to arbitrary order, with
\[
k(x_0) = 0 \quad\text{and}\quad Dk(x_0) = 0\text{,}
\]
\end{enumerate}
\st
\[
f\circ g(x) = f(x_0) + Df(x_0)(x-x_0) + k(x) \quad\forall\, x\in V\text{.}
\]
\end{enumerate}
\end{corollary}
\begin{proof}
One direction is provided by \cref{Theorem_Finite_dimensional_reduction} and \cref{Example_Strongly_smoothing_family}, the other direction follows from \cref{Example_Linear_Fredholm_implies_nonlinear_Fredholm}, \cref{Lemma_Perturbations_of_Fredholm_maps}, \cref{Example_Diffeomorphisms_are_Fredholm} and \cref{Lemma_Fredholm_index_nonlinear_map}, using that being Fredholm is a local property.
\end{proof}

\begin{definition}
Let $E$, $E^i$, $i=1,2$, be $\overline{\text{sc}}$-Fr{\'e}chet spaces, let $U \subseteq E$ be an open subset and let $\phi : U\times E^1 \to E^2$ be a Fredholm family of morphisms. \\
A point $x \in U$ is called a \emph{regular point} of $\phi$ \iff $\corank\phi(x) = 0$.
\end{definition}

\begin{lemma}
Let $E$, $E^i$, $i=1,2$, be $\overline{\text{sc}}$-Fr{\'e}chet spaces, let $U \subseteq E$ be an open subset and let $\phi : U\times E^1 \to E^2$ be a Fredholm family of morphisms. \\
Then the set of regular points of $\phi$ is an open subset of $U$.
\end{lemma}
\begin{proof}
This follows from \cref{Corollary_Fredholm_index_constant}, \labelcref{Corollary_Fredholm_index_constant_1}.
\end{proof}

\begin{definition}\label{Definition_Regular_point_and_value}
Let $E$ and $E'$ be $\overline{\text{sc}}$-Fr{\'e}chet spaces, let $U\subseteq E$ be an open subset and let $f : U \to E'$ be a Fredholm map.
\begin{enumerate}[label=\arabic*.,ref=\arabic*.]
  \item A point $x \in U$ is called a \emph{regular point} of $f$ \iff $x$ is a regular point of $Df : U\times E \to E'$.
  \item A point $y \in E'$ is called a \emph{regular value} of $f$ \iff $x \in U$ is a regular point of $f$ for every $x \in f\inv(y)$.
\end{enumerate}
\end{definition}

\begin{theorem}[Sard-Smale, \cf \cite{MR2045629}, Theorem A.5.1]\label{Theorem_Sard_Smale}
Let $E$ and $E'$ be weakly tame $\overline{\text{sc}}$-Fr{\'e}chet spaces with $E$ $2^\text{nd}$-countable, let $U\subseteq E$ be an open subset and let $f : U \to E'$ be a tame Fredholm map that is $\overline{\text{sc}}^\infty$ and tame up to arbitrary order. \\
Then the set of regular values of $f$ is a generic subset of $E'$ (\ie a countable intersection of open and dense subsets).
\end{theorem}
\begin{proof}
For an arbitrary subset $A \subseteq U$, denote
\[
\mathcal{V}_{\mathrm{reg}}(f; A) \definedas \{y \in E' \;|\; \corank f(x) = 0 \;\;\forall\, x \in f\inv(y)\cap A\}\text{.}
\]
I will follow the proof of Theorem A.5.1 in \cite{MR2045629} and show the following:
\begin{claim}
For every point $x_0 \in U$ there exists a neighbourhood $V \subseteq U$ of $x_0$ \st for every subset $A \subseteq V$ \st $A$ is closed in $E$, $\mathcal{V}_{\mathrm{reg}}(f; A)\subseteq E'$ is an open and dense subset of $E'$.
\end{claim}
\begin{proof}
First, let $D \definedas Df(x_0)$ and pick splittings $E = \ker D \oplus X$ and $E' = \im D \oplus C$ as in \cref{Proposition_sc_Fredholm_morphism}.
Shifting and using \cref{Theorem_Finite_dimensional_reduction}, one can assume w.\,l.\,o.\,g.~that $x_0 = 0$ and $f(x_0) = 0$, and that there exists a neighbourhood $V \subseteq U$ of $0$ \st
\[
f(x) = Dx + k(x)\quad\forall\, x \in V\text{,}
\]
where $k : V \to E'$ can be written as $k = \iota^C_{E'}\circ \tilde{k}$ for some map $\tilde{k} : V \to C$ with $\tilde{k}(0) = 0$ and $D\tilde{k}(0) = 0$.
For simplicities sake I will drop the inclusion $\iota^C_{E'}$ from the notation and write $k = \tilde{k}$.
Furthermore after possibly shrinking $V$, because $\ker D$ is finite dimensional, hence locally bounded, one can assume that $\pr^E_{\ker D}(V) \subseteq \ker D$ is bounded. \\
It needs to be shown that $\mathcal{V}_{\mathrm{reg}}(f; A) \subseteq E'$ is open and dense for every subset $A \subseteq V$ that is closed in $E$. \\
So let $A \subseteq V$ be such a subset, let $y \in E'$ and let $(y_n)_{n\in\N} \subseteq E' \setminus \mathcal{V}_{\mathrm{reg}}(f; A)$ be a sequence with $y_n \to y$.
Write $y_n = y^0_n + y^1_n$ for $y^0_n \in \im D$ and $y^1_n \in C$.
By assumption there exists a sequence $(x_n)_{n\in \N} \subseteq A$ with $f(x_n) = y_n$ and $\corank f(x_n) \geq 1$, and one can write $x_n = x^0_n + x^1_n$ for $x^0_n \in \ker D$ and $x^1_n \in X$.
Then
\begin{align*}
y_n &= y^0_n + y^1_n \\
&= f(x_n) \\
&= D(x_n) + k(x_n) \\
&= D(x^1_n) + k(x_n) \\
&\in \im D \oplus C\text{,}
\end{align*}
so $D(x^1_n) = y^0_n$.
By \cref{Proposition_sc_Fredholm_morphism}, $\tilde{D} \definedas \pr^{E'}_{\im D}\circ D \circ \iota^X_E : X \to \im D$ is an isomorphism, hence $x^1_n = \tilde{D}\inv(y^0_n) = \tilde{D}\inv\circ \pr^{E'}_{\im D}(y_n) \in X$ and $(x^1_n)_{n\in\N}$ converges.
Since $(x^0_n)_{n\in\N} \subseteq \pr^E_{\ker D}(V)$ is bounded by assumption, after passing to a subsequence one can assume that $(x^0_n)_{n\in \N}$ is convergent as well.
Since $A$ is closed, $x \definedas \lim_{n\to\infty} x_n \in A$ and $f(x) = y$.
Since $\corank f(x_n) \geq 1$, by \cref{Lemma_Fredholm_index_nonlinear_map} also $\corank f(x) \geq 1$ and hence $y \in E' \setminus \mathcal{V}_{\mathrm{reg}}(f; A)$.
So $\mathcal{V}_{\mathrm{reg}}(f; A)$ is open. \\
To show that $\mathcal{V}_{\mathrm{reg}}(f; A)$ is also dense in $E'$, it will a fortiori be shown that $\mathcal{V}_{\mathrm{reg}}(f; V)$ is dense in $E'$.
To do so, let $y \in E'$ be arbitrary.
The goal is to define a sequence $(y_n)_{n\in\N} \subseteq \mathcal{V}_{\mathrm{reg}}(f; V)$ with $y_n \to y$.
As before, write $y = y^0 + y^1$ with $y^0 \in \im D$ and $y^1 \in C$.
Fixing $y^0$, the sequence we will construct will be of the form $y_n = y^0 + y^1_n$ for some sequence $(y^1_n)_{n\in\N} \subseteq C$ with $y^1_n \to y^1$.
For $x \in V$ write $x = x^0 + x^1$ for $x^0 \in \ker D$ and $x^1 \in X$.
Then as above, $f(x) = y_n$ implies $y^0 = \tilde{D}(x^1)$ and $y^1_n = k(x^0 + x^1)$, or $y^1_n = k(x^0 + \tilde{D}\inv(y^0))$.
So $y_n = y^0 + y^1_n$ is a regular value of $f|_V$ \iff $y^1_n$ is a regular value of the map
\begin{align*}
\ker D \supseteq (V - \tilde{D}\inv(y^0))\cap \ker D &\to C \\
x &\mapsto k(x + \tilde{D}\inv(y_0))\text{.}
\end{align*}
Since $\ker D$ and $C$ are finite dimensional, by Sard's theorem the set of regular values of this map is dense in $C$.
Hence one can find the required sequence $(y^1_n)_{n\in\N} \subseteq C$ with the above properties.
\end{proof}
Using the above claim, the remainder of the proof follows the proof of Theorem A.5.1 in \cite{MR2045629} ad verbatim:
Since $E$ is $2^\text{nd}$-countable and metrisable, cover $U$ by countably many open subsets $V_i$, $i \in \N$, as in the claim, and each $V_i$ in turn by countably many subsets $A_i^j \subseteq V_i$, $j\in\N$, each of which is closed in $X$ (taken for example to be small enough closed balls \wrt some chosen compatible metric on $E$).
Then the set of regular values of $f$ is
\[
\mathcal{V}_{\mathrm{reg}}(f) = \bigcap_{i,j\in\N} \mathcal{V}_{\mathrm{reg}}(f; A_i^j)\text{.}
\]
\end{proof}

\clearpage
\section{Summary and a convenient category of Fr{\'e}chet spaces for nonlinear Fredholm theory}\label{Section_Summary}

The purpose of this last part is to collect the main definitions, theorems and examples of this article, packaged neatly into an application friendly framework.
This section may actually serve as either a starting point to reading this article or a way to avoid reading most of this article.

\begin{definition}[\cref{Definition_sc_Frechet_space,Definition_Pretame_sc_Frechet_space,Definition_Weakly_tame_sc_Frechet_space}]
\leavevmode
\begin{enumerate}[label=\arabic*.,ref=\arabic*.]
  \item A \emph{tame scale space}, or \emph{tsc-space} is a weakly tame $\overline{\text{sc}}$-Fr{\'e}chet space.
  \item A \emph{morphism} between tsc-spaces is a continuous linear map that defines a tame morphism between pre-tame $\overline{\text{sc}}$-Fr{\'e}chet spaces.
  \item A morphism $F : E \to E'$ between tsc-spaces $E$ and $E'$ is called an \emph{isomorphism} if there exists a morphism $G : E' \to E$ \st $F\circ G = \id_{E'}$ and $G\circ F = \id_E$.
\end{enumerate}
\end{definition}

\begin{remark}
In particular, every tsc-space is a pre-tame $\overline{\text{sc}}$-Fr{\'e}chet space and has an underlying $\overline{\text{sc}}$-Fr{\'e}chet space, which in turn has an underlying Fr{\'e}chet space and every (iso-)morphism of tsc-spaces defines an (iso-)morphism between the underlying ($\overline{\text{sc}}$-)Fr{\'e}chet spaces.
\end{remark}

\begin{example}[\cref{Example_Finite_dimensional_sc_Frechet_space}]
\leavevmode\\
Every finite dimensional vector space (over $\mathds{k} = \R,\C$) uniquely defines a tsc-space.
\end{example}

\begin{example}[\cref{Example_Sobolev_chain_I,Example_Sobolev_chain_II,Example_Sobolev_chain_III,Example_Sobolev_chain_IV,Example_Ck_chain_I,Example_Ck_chain_II,Example_Sobolev_and_Ck_chain,Example_Sobolev_and_Ck_sc_space,Example_PDO_I,Example_PDO_II,Example_PDO_III}]
\leavevmode\\
Let $\Sigma$ be a closed manifold and let $\pi : F \to \Sigma$ be a (real or complex) vector bundle.
Then $\Gamma(F)$, the vector space of smooth sections of $F$ with the $\mathcal{C}^\infty$-topology, canonically carries the structure of a tsc-space.
If $\pi' : F' \to \Sigma$ is another vector bundle and $P : \Gamma(F) \to \Gamma(F')$ is a (pseudo)differential operator, then $P$ canonically defines a morphism of tsc-spaces.
\end{example}

\begin{proposition}[\cref{Lemma_Morphisms_of_sc_Frechet_spaces,Remark_Subspaces_sums_pretame_context}]
\leavevmode\\
Let $E$, $E'$ and $E''$ be tsc-spaces, let $S, T : E \to E'$ and $S' : E' \to E''$ be morphisms between tsc-spaces, and let $\lambda, \mu \in \mathds{k}$.
Then
\[
\lambda S + \mu T : E \to E' \quad\text{and}\quad S' \circ S : E \to E''
\]
are morphisms between tsc-spaces as well.
\end{proposition}

\begin{definition}[\cref{Definition_sc_subspace,Definition_Sum_of_scFrechet_spaces,Remark_Subspaces_sums_pretame_context}]
\leavevmode\\
Let $E$, $E^i$, $i=1,\dots, k$ for some $k \in \N_0$, be tsc-spaces
\begin{enumerate}[label=\arabic*.,ref=\arabic*.]
  \item The \emph{direct sum} of the $E^i$, $i=1,\dots,k$, is the tsc-space $E^1\oplus \cdots\oplus E^k$ defined by the direct sum of the $E^i$, $i=1,\dots,k$, as pre-tame $\overline{\text{sc}}$-Fr{\'e}chet spaces.
  \item A \emph{subspace} $E'$ of a tsc-space $E$ is a $\overline{\text{sc}}$-subspace $E' \subseteq E$ of pre-tame $\overline{\text{sc}}$-Fr{\'e}chet spaces \st there exists a tame morphism $P : E \to E'$ with $P\circ J = \id_{E'}$, where $J : E' \hookrightarrow E$ is the canonical inclusion.
  \item A subspace $E'$ of $E$ is called \emph{split} if there exists another subspace $E''$ of $E$ \st the canonical morphism
\begin{align*}
J : E' \oplus E'' &\to E \\
(e,e') &\mapsto e + e'
\end{align*}
is an isomorphism. \\
$E''$ is then called a \emph{complement} to $E'$ and vice versa.
\end{enumerate}
\end{definition}

\begin{remark}
That subspaces and direct sums of tsc-spaces are well defined tsc-spaces follows from \cref{Proposition_Tame_subspaces_and_direct_sums}.
\end{remark}

\begin{example}[\cref{Example_Finite_dimensional_split_subchain,Example_Finite_dimensional_split_sc_subspace}]
\leavevmode\\
If $E$ is a tsc-space and $C \subseteq E$ is a finite dimensional subspace (as vector spaces), then $C$ canonically defines a split subspace of $E$.
\end{example}

\begin{example}[\cref{Example_Splitting_along_submanifold_I,Example_Splitting_along_submanifold_II}]
\leavevmode\\
Let $\Sigma$ be a closed manifold, let $C \subseteq \Sigma$ be a closed submanifold and let $\pi : F \to \Sigma$ be a real (or complex) vector bundle. \\
There is a well-defined morphism
\begin{align*}
R : \Gamma(F) &\to \Gamma(F|_C) \\
u &\mapsto u|_C
\end{align*}
between tsc-spaces and also a well-defined subspace
\[
\Gamma(F;C) \definedas \ker R = \{u \in \Gamma(F) \;|\; u|_C = 0\}\text{.}
\]
$\Gamma(F;C)$ is a split subspace of $\Gamma(F)$ with a complement isomorphic to $\Gamma(F|_C)$, \ie
\[
\Gamma(F) \cong \Gamma(F;C)\oplus \Gamma(F|_C)\text{.}
\]
\end{example}

\begin{definition}[\cref{Definition_Strongly_smoothing_morphism,Remark_Subspaces_sums_pretame_context}]
\leavevmode\\
Let $K : E \to E'$ be a morphism between tsc-spaces.
$K$ is called \emph{strongly smoothing} if it is strongly smoothing as a tame morphism between pre-tame $\overline{\text{sc}}$-Fr{\'e}chet spaces.
\end{definition}

\begin{remark}
Note that a morphism $K : E \to E'$ between tsc-spaces is strongly smoothing \iff the underlying morphism between $\overline{\text{sc}}$-Fr{\'e}chet spaces is strongly smoothing in the sense of \cref{Definition_Strongly_smoothing_morphism}.
This follows from \cref{Lemma_Characterisation_strongly_smoothing}.
\end{remark}

\begin{example}[\cref{Example_Strongly_smoothing_morphisms}]
\leavevmode\\
If $K : E \to E'$ is a morphism of tsc-spaces \st there exists a finite dimensional vector space $C$ and morphisms $L_1 : E \to C$ and $L_2 : C \to E'$ with $K = L_2\circ L_1$, then $K$ is strongly smoothing.
\end{example}

\begin{proposition}[\cref{Corollary_Basic_properties_strongly_smoothing_morphisms,Remark_Strongly_smoothing_implies_compact_scFrechet}]
\leavevmode\\
Let $E$, $E'$ and $E''$ be tsc-spaces, let $K, L : E \to E'$, $S : E'' \to E$ and $T : E' \to E''$ be morphisms and let $\lambda, \mu \in \mathds{k}$.
\begin{enumerate}[label=\arabic*.,ref=\arabic*.]
  \item If $K$ is strongly smoothing, then its underlying morphism between Fr{\'e}chet spaces is compact.
  \item If $K$ and $L$ are strongly smoothing, then so are
\[
\lambda K + \mu L : E \to E'\text{,}\qquad K\circ S : E'' \to E \quad\text{and}\quad
T\circ K : E \to E''\text{.}
\]
\end{enumerate}
\end{proposition}

\begin{definition}[\cref{Definition_sc_Fredholm}]
\leavevmode\\
Let $E$ and $E'$ be tsc-spaces.
\begin{enumerate}[label=\arabic*.,ref=\arabic*.]
  \item A morphism $F : E \to E'$ is called \emph{Fredholm} \iff $F$ is invertible modulo strongly smoothing morphisms \ie \iff there exists a morphism $F' : E' \to E$ and strongly smoothing morphisms $K : E \to E$ and $K' : E' \to E'$ \st
\begin{align*}
F' \circ F &= \id_E + K \\
F \circ F' &= \id_{E'} + K'\text{.}
\end{align*}
Such a morphism $F' : E' \to E$ is then called a \emph{Fredholm inverse} of $F$.
  \item The \emph{Fredholm index} $\ind F$ of a Fredholm morphism $F : E \to E'$ is defined as the Fredholm index of $F$ as a Fredholm operator between the underlying Fr{\'e}chet spaces.
\end{enumerate}
\end{definition}

\begin{example}[\cref{Example_Elliptic_operator_I,Example_Elliptic_operator_II}]
\leavevmode\\
Let $\Sigma$ be a closed manifold, and let $\pi : F \to \Sigma$ and $\pi' : F' \to \Sigma$ be vector bundles.
If $P : \Gamma(F) \to \Gamma(F')$ is an elliptic (pseudo)differential operator, then $P$ defines a Fredholm morphism of tsc-spaces.
\end{example}

\begin{proposition}[\cref{Lemma_Basic_properties_sc_Fredholm_operators}]
\leavevmode\\
Let $E$, $E'$ and $E''$ be tsc-spaces, let $F : E \to E'$ and $F' : E' \to E''$ be Fredholm morphisms and let $K : E \to E'$ be a strongly smoothing morphism.
Then:
\begin{enumerate}[label=\arabic*.,ref=\arabic*.]
  \item $F'\circ F : E \to E''$ is Fredholm with $\ind(F'\circ F) = \ind F' + \ind F$.
  \item $F + K : E \to E'$ is Fredholm with $\ind(F + K) = \ind F$.
\end{enumerate}
\end{proposition}

\begin{theorem}[\cref{Proposition_sc_Fredholm_morphism}]
\leavevmode\\
Let $E$ and $E'$ be tsc-spaces and let $F : E \to E'$ be a morphism.
Then the following are equivalent:
\begin{enumerate}[label=\arabic*.,ref=\arabic*.]
  \item $F$ is Fredholm.
  \item \begin{enumerate}[label=(\alph*),ref=(\alph*)]
  \item $F$ is a Fredholm operator between the Fr{\'e}chet spaces underlying $E$ and $E'$.
  \item There exist splittings
\[
E = \ker F \oplus X \quad\text{and}\quad E' = \im F \oplus C\text{.}
\]
In particular, $\im F$ defines a split subspace of $E'$ and the quotient projection $E' \to \coker F = E'/_{\im F}$ is a well defined morphism between tsc-spaces.
  \item For any splitting as above,
\[
\pr^{E'}_{\im F}\circ F\circ \iota^X_E : X \to \im F
\]
is an isomorphism of tsc-spaces.
\end{enumerate}
\end{enumerate}
\end{theorem}

\begin{definition}[\cref{Definition_sc_continuous,Definition_sc1,Definition_sck,Definition_Tame_nonlinear_map,Definition_Tame_nonlinear_map_scFrechet_spaces}]
\leavevmode\\
Let $E$ and $E'$ be tsc-spaces, let $U \subseteq E$ and $V \subseteq E'$ be open subsets, and let $f : U \to V \subseteq E'$ be a map.
\begin{enumerate}[label=\arabic*.,ref=\arabic*.]
  \item $f$ is called \emph{$\text{tsc}^\infty$} if $f$ is $\overline{\text{sc}}^\infty$ and tame up to arbitrary order as a map between open subsets of pre-tame $\overline{\text{sc}}$-Fr{\'e}chet spaces.
  \item $f$ is called a \emph{diffeomorphism} if $f$ is $\text{tsc}^\infty$ and if there exists a $\text{tsc}^\infty$-map $g : V \to U$ \st $g\circ f = \id_U$ and $f\circ g = \id_V$.
  \item $f$ is called a \emph{local diffeomorphism} if $f$ is $\text{tsc}^\infty$ and if for every $x \in U$ there exist open neighbourhoods $U' \subseteq U$ and $V' \subseteq V$ of $x$ and $f(x)$, respectively, \st $f(U') = V'$ and $f|_{U'} : U' \to V'$ is a diffeomorphism.
\end{enumerate}
\end{definition}

\begin{proposition}[\cref{Proposition_Composition_sc_continuous,Theorem_Chain_rule_sc_differentiable,Proposition_Standard_properties_tame}]
\leavevmode\\
Let $E$, $E'$ and $E''$ be tsc-spaces, let $U \subseteq E$, $V \subseteq E'$ and $W \subseteq E''$ be open subsets, and let $f : U \to V$ and $f' : V \to W$ be maps.
If $f$ and $f'$ are $\text{tsc}^\infty$, then so is $f'\circ f$.
\end{proposition}

\begin{remark}
A map is $\text{tsc}^\infty$ \iff it is $\underline{\text{sc}}^\infty$ and tame up to arbitrary order as a map between open subsets of pre-tame $\overline{\text{sc}}$-Fr{\'e}chet spaces.
This follows from \cref{Proposition_Tame_maps,Proposition_Underline_sc_implies_overline_sc}.
\end{remark}

\begin{remark}
In particular, every $\text{tsc}^\infty$-map defines an arbitrarily often weakly continuously weakly Fr{\'e}chet differentiable map between open subsets of the underlying Fr{\'e}chet spaces and hence has a well defined differential.
\end{remark}

\begin{definition}[\cref{Definition_Family_of_morphisms}]
\leavevmode\\
Let $E$, $E'$ and $E^i$, $i=1,2,3$, be tsc-spaces, and let $U \subseteq E$ and $U' \subseteq E'$ be open subsets.
\begin{enumerate}[label=\arabic*.,ref=\arabic*.]
  \item A \emph{$\text{tsc}^\infty$-family (over $U$) of morphisms (from $E^1$ to $E^2$)} is a $\text{tsc}^\infty$-map
\begin{align*}
\phi : U\times E^1 &\to E^2 \\
(x,u) &\mapsto \phi(x)u
\end{align*}
that is linear in the second factor.
  \item Given $\text{tsc}^\infty$-families of morphisms $\phi : U\times E^1 \to E^2$ and $\psi : U\times E^2 \to E^3$, their \emph{composition} is the $\text{tsc}^\infty$-family of morphisms
\begin{align*}
\psi\circ \phi : U\times E^1 &\to E^3 \\
(x,u) &\mapsto \psi(x)\phi(x)u\text{.}
\end{align*}
  \item A $\text{tsc}^\infty$-family of morphisms $\phi : U\times E^1 \to E^2$ is called \emph{invertible} if there exists a $\text{tsc}^\infty$-family of morphisms $\psi : U\times E^2 \to E^1$ with $\psi\circ \phi = \id_{E^1}$ and $\phi\circ \psi = \id_{E^2}$.
$\psi$ will then be called the \emph{inverse} to $\phi$ and denoted by $\phi\inv \definedas \psi$.
  \item Given a $\text{tsc}^\infty$-family of morphisms $\phi : U\times E^1 \to E^2$ and a $\text{tsc}^\infty$-map $f : U' \to U$, the \emph{pullback} of $\phi$ by $f$ is the $\text{tsc}^\infty$-family of morphisms $f^\ast \phi \definedas \phi\circ (f\times \id_{E^1})$, \ie
\begin{align*}
f^\ast \phi : U'\times E^1 &\to E^2 \\
(x,u) &\mapsto \phi(f(x))u\text{.}
\end{align*}
\end{enumerate}
\end{definition}

\begin{remark}
If $\phi : U\times E^1 \to E^2$ is a $\text{tsc}^\infty$-family of morphisms, then $\phi(x) : E^1 \to E^2$ is a morphism of tsc-spaces for all $x \in U$.
This follows directly from the definitions and \cref{Proposition_Standard_properties_tame}.
\end{remark}

\begin{proposition}
Let $E$, $E'$ and $E''$ be tsc-spaces, let $U \subseteq E$, $V \subseteq E'$ and $W \subseteq E''$ be open subsets, and let $f : U \to V$ and $g : V \to W$ be $\text{tsc}^\infty$-maps.
Then:
\begin{enumerate}
  \item $Df : U\times E \to E'$ is a $\text{tsc}^\infty$-family of morphisms.
  \item $g\circ f : U\to W$ is a $\text{tsc}^\infty$-map with
\[
D(g\circ f) = (f^\ast Dg)\circ Df\text{.}
\]
\end{enumerate}
\end{proposition}

\begin{example}[\cref{Subsection_Reparametrisation_action,Example_Reparametrisation_Action_I,Example_Reparametrisation_Action_II,Example_Reparametrisation_Action_III}]
\leavevmode\\
Let $\Sigma$ be a closed $n$-dimensional manifold and let $\pi : F \to \Sigma$, $\pi_1 : F_1 \to \Sigma$ and $\pi_2 : F_2 \to \Sigma$ be vector bundles.
Let furthermore $B \subseteq \R^r$ (for some $r\in\N_0$) be an open subset and let
\begin{align*}
\phi : B \times \Sigma &\to B\times \Sigma \\
(b, z) &\mapsto (b, \phi_b(z))
\intertext{and}
\Phi : B \times F_1 &\to B\times F_2 \\
(b, \xi) &\mapsto (b, \Phi_b(\xi))
\end{align*}
be smooth families of maps, where $\phi$ is a family of diffeomorphisms and $\Phi$ covers $\phi\inv$.
I.\,e.~$\phi$ and $\Phi$ are smooth, $\phi_b \in \operatorname{Diff}(\Sigma)$ and $\Phi_b \in \mathcal{C}^\infty(F_1, F_2)$ for all $b \in B$, and
\[
\xymatrix{
B\times F_2 \ar@{}[rd]|-{\circlearrowleft} \ar[d]_-{\id_B\times \pi} & B\times F_1 \ar[d]^-{\id_B\times \pi} \ar[l]_-{\Phi} \\
B\times \Sigma \ar[r]^-{\phi} & B\times \Sigma
}
\]
commutes.
Furthermore, assume that $\Phi$ is linear in the fibres of $F_1$ and $F_2$, \ie for every $b\in B$, $\Phi_b : F_1 \to F_2$ defines a vector bundle morphism covering $\phi\inv_b : \Sigma \to \Sigma$. \\
Define
\begin{align*}
\Gamma_B(F) &\definedas B\times \Gamma(F)\text{,} \\
\rho &\definedas \pr_1 : \Gamma_B(F) \to B \\
\intertext{and set for $b\in B$}
\Gamma_b(F) &\definedas \{b\}\times \Gamma(F) = \rho\inv(b)\text{.}
\end{align*}
Also, there is a $\text{tsc}^\infty$-family morphisms
\begin{align*}
\psi : B\times \Gamma(F_1) &\to \Gamma(F_2) \\
(b,u) &\mapsto \Phi_b^\ast u\text{,}
\intertext{where}
\Phi_b^\ast u &\definedas \Phi_b\circ u\circ \phi_b\text{.}
\end{align*}
$\Gamma_B(F)$ is an open subset of the tsc-space $\R^r\oplus \Gamma(F)$ and there is a $\text{tsc}^\infty$-map
\begin{align*}
\Psi : \Gamma_B(F_1) &\to \Gamma_B(F_2) \\
(b,u) &\mapsto (b, \psi(b,u))\text{.}
\end{align*}
\end{example}

\begin{definition}[\cref{Definition_Strongly_smoothing_family}]
\leavevmode\\
Let $E$, $E^i$, $i=1,2$ be tsc-spaces, let $U\subseteq E$ be an open subset and let $\kappa : U\times E^1 \to E^2$ be a $\text{tsc}^\infty$-family of morphisms.
$\kappa$ is called \emph{strongly smoothing} \iff it is strongly smoothing as a family of morphisms between pre-tame $\overline{\text{sc}}$-Fr{\'e}chet spaces.
\end{definition}

\begin{proposition}[\cref{Lemma_Basic_properties_strongly_smoothing_families}]
\leavevmode\\
Let $E$, $E'$ and $E^i$, $i=0,\dots,3$, be tsc-spaces, let $U \subseteq E$ and $U' \subseteq E'$ be open subsets together with a $\text{tsc}^\infty$-map $f : U' \to U$ and let
\begin{align*}
\kappa : U\times E^1 &\to E^2 & \phi : U \times E^0 &\to E^1 & \psi : U\times E^2 &\to E^3
\end{align*}
be $\text{tsc}^\infty$-families of morphisms.
If $\kappa$ is strongly smoothing then so are
\[
f^\ast \kappa\text{,}\qquad\kappa\circ \phi \quad\text{and}\quad\psi\circ \kappa\text{.}
\]
\end{proposition}

\begin{remark}
If $\kappa : U\times E^1 \to E^2$ is a strongly smoothing $\text{tsc}^\infty$-family of morphisms, then for every $x \in U$, $\kappa(x) : E^1 \to E^2$ is a strongly smoothing morphism of tsc-spaces.
\end{remark}

\begin{example}[\cref{Example_Strongly_smoothing_family}]
\leavevmode\\
Let $E$, $E^i$, $i=1,2$, be tsc-spaces and let $U \subseteq E$ be an open subset.
\begin{enumerate}[label=\arabic*.,ref=\arabic*.]
  \item Given a strongly smoothing morphism $K : E^1 \to E^2$, the constant family
\begin{align*}
U\times E^1 &\to E^2 \\
(x,u) &\mapsto K(u)
\end{align*}
is strongly smoothing.
  \item Given a family of morphisms $\kappa : U \times E^1 \to E^2$, if there exists a finite dimensional vector space $C$ and families of morphisms $\kappa_1 : U\times E^1 \to C$, $\kappa_2 : U\times C \to E^2$ \st $\kappa = \kappa_2\circ \kappa_1$, then $\kappa$ is strongly smoothing.
\end{enumerate}
\end{example}

\begin{definition}[\cref{Definition_Family_of_Fredholm_morphisms}]
\leavevmode\\
Let $E$, $E^i$, $i=1,2$, be tsc-spaces, let $U \subseteq E$ be an open subset and let $\phi : U\times E^1 \to E^2$ be a $\text{tsc}^\infty$-family of morphisms.
\begin{enumerate}[label=\arabic*.,ref=\arabic*.]
  \item $\phi$ is called a \emph{Fredholm} \iff it is locally invertible modulo strongly smoothing $\text{tsc}^\infty$-families of morphisms, \ie \iff for every $x_0 \in U$ there exist a neighbourhood $V \subseteq U$ of $x_0$ and $\text{tsc}^\infty$-families of morphisms
\begin{align*}
\psi : V\times E^2 &\to E^1 & \kappa : V\times E^1 &\to E^1 & \kappa' : V\times E^2 &\to E^2
\end{align*}
with $\kappa$ and $\kappa'$ strongly smoothing \st
\begin{align*}
\psi \circ \phi|_{V\times E^1} = \id_{E^1} + \kappa
\intertext{and}
\phi|_{V\times E^1} \circ \psi = \id_{E^2} + \kappa'\text{.}
\end{align*}
$\psi$ is then called a \emph{local Fredholm inverse} to $\phi$.
  \item The map
\begin{align*}
\ind : U &\to \Z \\
x &\mapsto \ind(\phi(x))\text{,}
\end{align*}
where $\ind(\phi(x))$ denotes the Fredholm index of the Fredholm morphism $\phi(x) : E^1 \to E^2$, is called the \emph{(Fredholm) index of $\phi$}.
  \item The map
\begin{align*}
\corank\phi : U &\to \N_0 \\
x &\mapsto \dim\coker(\phi(x))\text{,}
\end{align*}
where $\coker(\phi(x))$ denotes the (finite dimensional) cokernel of the Fredholm morphism $\phi(x) : E^1 \to E^2$, is called the \emph{corank of $\phi$}. \\
$\phi$ is said to have \emph{constant rank} if $\corank\phi : U \to \N_0$ is constant.
\end{enumerate}
\end{definition}

\begin{theorem}[\cref{Theorem_scbar_Fredholm}]
\leavevmode\\
Let $E$, $E^i$, $i=1,2$, be tsc-spaces, let $U \subseteq E$ be an open subset and let $\phi : U\times E^1 \to E^2$ be a $\text{tsc}^\infty$-family of morphisms.
Then the following are equivalent:
\begin{enumerate}[label=\arabic*.,ref=\arabic*.]
  \item $\phi$ is Fredholm.
  \item
\begin{enumerate}[label=(\alph*),ref=(\alph*)]
  \item For every $x_0 \in U$,
\[
\phi(x_0) : E^1 \to E^2
\]
is a Fredholm morphism of tsc-spaces.
  \item For one/any pair of splittings
\[
E^1 = X \oplus \ker \phi(x_0) \qquad\text{and}\qquad E^2 = \im \phi(x_0) \oplus C
\]
there exists a neighbourhood $V \subseteq U$ of $x_0$ \st
\[
\pr^{E^2}_{\im \phi(x_0)}\circ \phi\circ \iota^X_{E^1}|_{V\times X} : V\times X \to \im \phi(x_0)
\]
is invertible.
\end{enumerate}
\end{enumerate}
\end{theorem}

\begin{proposition}[\cref{Lemma_Basic_properties_families_of_Fredholm_morphisms,Corollary_Fredholm_index_constant}]
\leavevmode\\
Let $E$, $E'$ and $E^i$, $i = 1, \dots, 3$, be tsc-spaces, let $U \subseteq E$ and $U' \subseteq E'$ be open subsets together with a $\text{tsc}^\infty$-map $f : U' \to U$ and let
\begin{align*}
\phi : U\times E^1 &\to E^2\text{,} \\
\psi : U\times E^2 &\to E^3
\intertext{and}
\kappa : U\times E^1 &\to E^2
\end{align*}
be $\text{tsc}^\infty$-families of morphisms with $\phi$ and $\psi$ Fredholm and $\kappa$ strongly smoothing.
Then the following hold:
\begin{enumerate}[label=\arabic*.,ref=\arabic*.]
  \item $f^\ast \phi : U' \times E^1 \to E^2$ is Fredholm with $\ind(f^\ast \phi) = f^\ast \ind (\phi)$.
  \item $\psi\circ \phi : U\times E^1 \to E^3$ is Fredholm with $\ind (\psi\circ \phi) = \ind \psi + \ind \phi$.
  \item $\phi + \kappa : U\times E^1 \to E^2$ is Fredholm with $\ind (\phi + \kappa) = \ind \phi$.
  \item The maps
\begin{align*}
\dim \ker \phi : U &\to \N_0 \\
x &\mapsto \dim\ker(\phi(x))
\intertext{and}
\corank\phi : U &\to \N_0 \\
x &\mapsto \dim\coker(\phi(x))
\end{align*}
are upper semicontinuous.
  \item The Fredholm index $\ind\phi = \dim \ker\phi - \corank\phi : U \to \Z$ of $\phi$ is continuous.
\end{enumerate}
\end{proposition}

\begin{definition}[\cref{Definition_scbar_Fredholm,Definition_Strongly_smoothing_map}]
\leavevmode\\
Let $E$ and $E'$ be tsc-spaces, let $U\subseteq E$ be an open subset and let $f : U \to E'$ be $\text{tsc}^\infty$.
\begin{enumerate}[label=\arabic*.,ref=\arabic*.]
  \item $f$ is called \emph{Fredholm} \iff $Df : U\times E \to E'$ is Fredholm.
  \item If $f$ is Fredholm, then the map
\begin{align*}
\ind f : U &\to \Z \\
x &\mapsto \ind Df(x)
\end{align*}
is called the \emph{(Fredholm) index} of $f$.
  \item If $f$ is Fredholm, then the map
\begin{align*}
\corank f : U &\to \N_0 \\
x &\mapsto \dim\coker(Df(x))
\end{align*}
is called the \emph{corank of $f$}. \\
$f$ is said to have \emph{constant rank} if $\corank f : U \to \N_0$ is constant.
  \item $f$ is called \emph{strongly smoothing} \iff $Df : U\times E \to E'$ is strongly smoothing.
\end{enumerate}
\end{definition}

\begin{proposition}[\cref{Lemma_Fredholm_index_nonlinear_map,Lemma_Perturbations_of_Fredholm_maps}]
\leavevmode\\
Let $E$, $E'$ and $E''$ be tsc-spaces, let $U \subseteq E$, $V \subseteq E'$ and $W \subseteq E''$ be open subsets, and let $f : U \to V$, $g : V \to W$ and $k : U \to E'$ be $\text{tsc}^\infty$.
\begin{enumerate}[label=\arabic*.,ref=\arabic*.]
  \item If $f$ is a diffeomorphism, then $f$ is Fredholm.
  \item If both $f$ and $g$ are Fredholm, then so is $g\circ f$.
Furthermore,
\[
\ind(g\circ f) = \ind f + f^\ast \ind g\text{.}
\]
  \item If $f$ is Fredholm, then $\ind f$ is continuous and $\corank f$ is upper semicontinuous.
  \item If at least one of $f$ and $g$ is strongly smoothing, then so is $g\circ f$.
  \item If $f$ is Fredholm and $k$ is strongly smoothing, then $f + k : U \to E'$ is Fredholm as well with $\ind(f+k) = \ind(f)$.
\end{enumerate}
\end{proposition}

\begin{theorem}[Nash-Moser inverse function theorem, \cref{Theorem_Nash_Moser}]
\leavevmode\\
Let $E$ and $E'$ be tsc-spaces, let $U \subseteq E$ be an open subset and let $f : U\to E'$ be $\text{tsc}^\infty$. \\
If $Df : U\times E \to E'$ is invertible, then $f$ is a local diffeomorphism. \\
More explicitely, if there exists a $\text{tsc}^\infty$-family of morphisms $\phi : U\times E' \to E$ \st $Df\circ \phi = \id_{E'}$ and $\phi\circ Df = \id_E$, then for each $x \in U$ there exists an open neighbourhood $V \subseteq U$ of $x$ \st $f(V) \subseteq E'$ is an open neighbourhood of $f(x)$ and there exists a $\text{tsc}^\infty$-map $g : f(V) \to V$ that satisfies $g\circ f|_V = \id_V$ and $f|_V\circ g = \id_{f(V)}$.
Furthermore, $Dg = g^\ast\phi$.
\end{theorem}

\begin{theorem}[Constant rank theorem, \cref{Theorem_constant_rank_theorem}]
\leavevmode\\
Let $E$ and $E'$ be tsc-spaces, let $U\subseteq E$ be an open subset and let $f : U\to E'$ be $\text{tsc}^\infty$.
Given $x_0 \in U$, if there exists a neighbourhood $V \subseteq U$ of $x_0$ \st $f|_V$ is Fredholm and has constant rank, then there exist open neighbourhoods $W \subseteq U$ of $x_0$ and $W' \subseteq E'$ of $f(x_0)$ with $f(W) \subseteq W'$ together with $\text{tsc}^\infty$-diffeomorphisms
\begin{align*}
\Phi : W &\to \tilde{W} \subseteq \im Df(x_0) \oplus \ker Df(x_0) \\
\Psi : W' &\to \tilde{W}' \subseteq \im Df(x_0) \oplus \coker Df(x_0)
\end{align*}
\st
\[
\Psi\circ f|_W\circ \Phi\inv = (\id_{\im Df(x_0)}\oplus 0)|_{\tilde{W}}\text{.}
\]
\end{theorem}

\begin{corollary}
Let $E$ and $E'$ be $2^{\text{nd}}$-countable tsc-spaces, let $U\subseteq E$ be an open subset and let $f : U \to E'$ be $\text{tsc}^\infty$. \\
If $f$ is Fredholm, has constant rank, and $\ind f$ is globally constant, then for any $y \in E'$, $f\inv(y) \subseteq E$ is a smooth manifold of dimension
\[
\dim f\inv(y) = \ind f + \corank f\text{.}
\]
\end{corollary}

\begin{theorem}[Finite dimensional reduction, \cref{Corollary_Standard_form_of_Fredholm_maps}]
\leavevmode\\
Let $E$ and $E'$ be tsc-spaces, let $U\subseteq E$ be an open subset and let $f : U \to E'$ be $\text{tsc}^\infty$.
Then the following are equivalent:
\begin{enumerate}[label=\arabic*.,ref=\arabic*.]
  \item $f$ is Fredholm.
  \item For every $x_0 \in U$, $Df(x_0) : E \to E'$ is Fredholm and there exist
\begin{enumerate}[label=(\alph*),ref=(\alph*)]
  \item open neighbourhoods $V, W \subseteq U$ of $x_0$,
  \item a $\text{tsc}^\infty$-diffeomorphism $g : V \to W$ with
\[
g(x_0) = x_0 \quad\text{and}\quad Dg(x_0) = \id_E\text{,}
\]
  \item a strongly smoothing $\text{tsc}^\infty$-map $k : V \to E'$ with
\[
k(x_0) = 0 \quad\text{and}\quad Dk(x_0) = 0\text{,}
\]
\end{enumerate}
\st
\[
f\circ g(x) = f(x_0) + Df(x_0)(x-x_0) + k(x) \quad\forall\, x \in V\text{.}
\]
\end{enumerate}
\end{theorem}

\begin{definition}[\cref{Definition_Regular_point_and_value}]
\leavevmode\\
Let $E$ and $E'$ be tsc-spaces, let $U \subseteq E$ and $V \subseteq E'$ be open subsets and let $f : U \to V$ be $\text{tsc}^\infty$ and Fredholm.
\begin{enumerate}[label=\arabic*.,ref=\arabic*.]
  \item A point $x \in U$ is called a \emph{regular point} of $f$ \iff $\corank f(x) = 0$.
  \item A point $y \in V$ is called a \emph{regular value} of $f$ \iff $x \in U$ is a regular point of $f$ for every $x \in f\inv(y)$.
\end{enumerate}
\end{definition}

\begin{theorem}[Sard-Smale,\cref{Theorem_Sard_Smale}]
\leavevmode\\
Let $E$ and $E'$ be $2^{\text{nd}}$-countable tsc-spaces, let $U \subseteq E$ be an open subset and let $f : U \to E'$ be $\text{tsc}^\infty$ and Fredholm. \\
Then the set of regular values of $f$ is a generic subset of $E'$.
\end{theorem}

\clearpage

\bibliographystyle{alpha}
\addcontentsline{toc}{section}{Bibliography}
\bibliography{Bibliography}

\begin{thebibliography}{HWZ14}

\bibitem[DGV15]{DodsonGalanisVassiliou}
C.~T.~J. Dodson, George Galanis, and Efstathios Vassiliou.
\newblock {\em Geometry in a {F}r{\'e}chet Context: A Projective Limit
  Approach}, volume 428 of {\em London Mathematical Society Lecture Note
  Series}.
\newblock Cambridge University Press, Cambridge, 2015.

\bibitem[Ham82]{MR656198}
Richard~S. Hamilton.
\newblock The inverse function theorem of {N}ash and {M}oser.
\newblock {\em Bull. Amer. Math. Soc. (N.S.)}, 7(1):65--222, 1982.

\bibitem[HWZ07]{MR2341834}
H.~Hofer, K.~Wysocki, and E.~Zehnder.
\newblock A general {F}redholm theory. {I}. {A} splicing-based differential
  geometry.
\newblock {\em J. Eur. Math. Soc. (JEMS)}, 9(4):841--876, 2007.

\bibitem[HWZ10]{MR2644764}
Helmut Hofer, Kris Wysocki, and Eduard Zehnder.
\newblock sc-smoothness, retractions and new models for smooth spaces.
\newblock {\em Discrete Contin. Dyn. Syst.}, 28(2):665--788, 2010.

\bibitem[HWZ14]{1407.3185}
Helmut~H. Hofer, Kris Wysocki, and Eduard Zehnder.
\newblock Polyfold and fredholm theory i: Basic theory in m-polyfolds, 2014.

\bibitem[Kel74]{MR0440592}
Hans~Heinrich Keller.
\newblock {\em Differential calculus in locally convex spaces}.
\newblock Lecture Notes in Mathematics, Vol. 417. Springer-Verlag, Berlin-New
  York, 1974.

\bibitem[KM97]{MR1471480}
Andreas Kriegl and Peter~W. Michor.
\newblock {\em The convenient setting of global analysis}, volume~53 of {\em
  Mathematical Surveys and Monographs}.
\newblock American Mathematical Society, Providence, RI, 1997.

\bibitem[{\L}Z79]{MR546504}
S.~{\L}ojasiewicz, Jr. and E.~Zehnder.
\newblock An inverse function theorem in {F}r\'echet-spaces.
\newblock {\em J. Funct. Anal.}, 33(2):165--174, 1979.

\bibitem[Mro04]{Mrowka_Geometry_of_Manifolds}
Tomasz Mrowka.
\newblock Lecture notes in geometry of manifolds, 2004.

\bibitem[MS04]{MR2045629}
Dusa McDuff and Dietmar Salamon.
\newblock {\em {$J$}-holomorphic curves and symplectic topology}, volume~52 of
  {\em American Mathematical Society Colloquium Publications}.
\newblock American Mathematical Society, Providence, RI, 2004.

\bibitem[Mun75]{MR0464128}
James~R. Munkres.
\newblock {\em Topology: a first course}.
\newblock Prentice-Hall, Inc., Englewood Cliffs, N.J., 1975.

\bibitem[Omo97]{MR1421572}
Hideki Omori.
\newblock {\em Infinite-dimensional {L}ie groups}, volume 158 of {\em
  Translations of Mathematical Monographs}.
\newblock American Mathematical Society, Providence, RI, 1997.
\newblock Translated from the 1979 Japanese original and revised by the author.

\bibitem[Osb14]{MR3154940}
M.~Scott Osborne.
\newblock {\em Locally convex spaces}, volume 269 of {\em Graduate Texts in
  Mathematics}.
\newblock Springer, Cham, 2014.

\bibitem[Weh12]{1209.4040}
Katrin Wehrheim.
\newblock Fredholm notions in scale calculus and hamiltonian floer theory,
  2012.

\bibitem[Yam74]{MR0488118}
Sadayuki Yamamuro.
\newblock {\em Differential calculus in topological linear spaces}.
\newblock Lecture Notes in Mathematics, Vol. 374. Springer-Verlag, Berlin-New
  York., 1974.

\end{thebibliography}

\end{document}